\documentclass[10pt]{article}
\fussy 
\usepackage[margin=1in]{geometry} 
\usepackage{authblk}

\usepackage{amsthm,amsmath,amssymb}
\usepackage{graphicx}
\usepackage{enumerate}
\usepackage[numbers]{natbib}
\usepackage{hyperref}
\hypersetup{
    colorlinks,%
    citecolor=black,%
    filecolor=black,%
    linkcolor=black,%
    urlcolor=black
}
\usepackage{hypernat}
\usepackage{mathabx}

\usepackage{tikz}
\usetikzlibrary{shapes.geometric, positioning}
\usepackage{soul}

\usepackage{multirow}
\usepackage{float}
\usepackage{enumitem}

\newcommand{\E}{{\mathbb E}}
\newcommand{\R}{{\mathbb R}}
\renewcommand{\P}{{\mathbb P}}
\newcommand{\T}{{\mathcal{T}}}
\newcommand{\F}{{\mathcal F}}

\newcommand{\tcd}{{\mathcal{F}^d_{\text{TC}}}}
\newcommand{\tcds}{{\mathcal{F}^{d, s}_{\text{TC}}}}
\newcommand{\tcdd}{{\mathcal{F}^{d, d}_{\text{TC}}}}
\newcommand{\tcp}{{\mathcal{F}^{d}_{\text{TCP}}}}
\newcommand{\infmars}{{\mathcal{F}^{d, s}_{\infty-\text{mars}}}}

\newcommand{\ind}{\mathbf{1}}

\newcommand{\zerovec}{{\mathbf{0}}}

\def\qt#1{\qquad\text{#1}}
\def\argmin{\mathop{\rm argmin}}

\newtheorem{theorem}{Theorem}%
\newtheorem{proposition}{Proposition}%
\newtheorem{lemma}{Lemma}%
\newtheorem{example}{Example}%
\newtheorem{remark}{Remark}%
\newtheorem{definition}{Definition}

\begin{document}

\title{Totally Concave Regression}
\author{Dohyeong Ki\thanks{\texttt{dohyeong\_ki@berkeley.edu}} }
\author{Adityanand Guntuboyina\thanks{\texttt{aditya@stat.berkeley.edu}}}

\affil{Department of Statistics, University of California, Berkeley}

\date{} 

\maketitle
\vspace{-1.5em}

\begin{abstract}
Shape constraints in nonparametric regression provide a powerful
framework for estimating regression functions under realistic
assumptions without tuning parameters. However, most existing
methods—except additive models—impose too weak restrictions, often
leading to overfitting in high dimensions. Conversely, additive models
can be too rigid, failing to capture covariate interactions. This
paper introduces a novel multivariate shape-constrained regression
approach based on total concavity, originally studied by
T. Popoviciu. Our method allows interactions while mitigating the
curse of dimensionality, with convergence rates that depend only
logarithmically on the number of covariates. We characterize and
compute the least squares estimator over totally concave functions,
derive theoretical guarantees, and demonstrate its practical
effectiveness through empirical studies on real-world datasets. 
\vspace{0.5em} 
\end{abstract}

\section{Introduction}
Before describing our multivariate generalization of univariate concave
regression, let us first review 
the univariate case. Consider a response 
variable $y$ and a single covariate $x$ on which we observe data
$(x^{(1)}, y_1), \dots, (x^{(n)}, y_n)$. Throughout, we assume
covariates are scaled to take values in $[0, 1]$. Univariate concave 
regression \citep{Hildreth54, HanPled76} fits the best concave 
function on $[0, 1]$ to the data using least squares:  
\begin{equation*}
  \hat{f}_{\text{concave}}^1 \in \argmin_{f \text{ is concave}}
  \sum_{i=1}^n \big(y_i - f(x^{(i)}) \big)^2.
\end{equation*}
The estimate $\hat{f}_{\text{concave}}^1$
($1$ indicates covariate dimension) suits settings with
diminishing returns where the rate of change of $y$ on $x$ decreases
as $x$ grows.  The convex analogue
$\hat{f}_{\text{convex}}^1$ covers increasing returns; we focus on
concave regression but our ideas also apply to
convex regression. The estimate
$\hat{f}_{\text{concave}}^1$ is transparent (relying only on
concavity), tuning-free, interpretable (yielding continuous piecewise
affine fits), and has good theoretical properties under standard
regression models when the true regression function is concave (see,
e.g., \cite{DuembgenEtAl04, groeneboom2014nonparametric, 
  guntuboyina2018nonparametric}). 
  
Our goal is to extend $\hat{f}_{\text{concave}}^1$ to
multiple regression where, instead of a single covariate 
$x$, one has $d$ covariates $x_1, \dots, x_d$ with $d \ge 1$. We
assume each $x_j$ takes values in $[0, 1]$. The data are
$(\mathbf{x}^{(1)}, y_1), \dots, 
(\mathbf{x}^{(n)}, y_n)$ with $y_i \in \R$ and each 
$\mathbf{x}^{(i)} \in 
[0, 1]^d$. Our extension, which we call \textit{totally concave}
regression, involves the least squares estimator (LSE) over the class of
totally concave functions. Total concavity is a multivariate notion of
concavity introduced by \cite{popoviciu1933quelques} (see
\cite{gal2010shape} for a book-length reference). As explained in
Section \ref{comparison-concavity}, it is quite different from the
three usual multivariate notions of concavity: general concavity
(which is just standard concavity: $f((1-\alpha)\mathbf{x} +
\alpha \mathbf{y}) \ge (1 - \alpha) f(\mathbf{x}) + \alpha
f(\mathbf{y})$ for all $\alpha \in [0, 1]$ and $\mathbf{x}, \mathbf{y}
\in [0, 1]^d$), axial concavity (which refers to concavity in each
coordinate when all other coordinates are held fixed), and additive
concavity (which corresponds to functions $\sum_{j=1}^d f_j(x_j)$
where each $f_j$ is a univariate concave function). 

Totally concave functions are obtained by the inclusion of simple axially concave interactions to additive concave functions. Consider the additive concave function: 
\begin{equation}\label{addmod}
  f(x_1, \dots, x_d) = f_1(x_1) + \dots + f_d(x_d)
\end{equation}
where each $f_j:[0, 1] \rightarrow \R$ is concave. To formulate
interactions in \eqref{addmod}, recall the basic  
representation theorem for univariate concave functions (e.g.,
\citet[Section 1.6]{niculescu2006convex}), which states that, for every
concave function $g:[0, 1] \rightarrow \R$ with 
\begin{equation}\label{1dreg}
    g'(0+) := \lim_{t \rightarrow 0+} \frac{g(t) - g(0)}{t} < \infty \quad \text{and} \quad g'(1-) := \lim_{t \rightarrow 1-} \frac{g(1) - g(t)}{1 - t} > -\infty,
\end{equation}
there exists a unique finite
 Borel measure $\nu$ on $(0, 1)$ and $a_0,
a_1 \in \mathbb{R}$ such that  
\begin{equation}\label{1drepconcave}
  g(x) = a_0 + a_1 x - \int_{(0, 1)} (x - t)_+ \, d\nu(t) \qt{for all $x
    \in [0, 1]$}
\end{equation}
where $(x - t)_+ := \max(x - t, 0)$. Intuitively, \eqref{1drepconcave}
states that every concave function on $[0, 1]$ satisfying
\eqref{1dreg} is a linear combination of the primitives $x \mapsto -(x -
t)_+, t \in [0, 1]$. The condition \eqref{1dreg} is essential for
validity of \eqref{1drepconcave} and is satisfied by all
functions we consider. 

Using the representation
\eqref{1drepconcave} for each $f_j$, we can rewrite 
\eqref{addmod} as
\begin{equation}\label{addrep}
  f^*(x_1, \dots, x_d) = \beta_0 + \sum_{j=1}^d \beta_j x_j -
\sum_{j=1}^d  \int_{(0, 1)} 
  (x_j - t_j)_+ \, d\nu_j(t_j) 
\end{equation}
where each $\nu_j$ is a finite Borel measure on $(0, 1)$ and $\beta_0,
\dots, \beta_d$ are real-valued coefficients. The 
right-hand side of \eqref{addrep} is a linear model 
in $x_1, \dots, x_d$ (with coefficients 
$\beta_1, \dots, \beta_d$) and the modified
variables $-(x_j - t_j)_+$ for $t_j \in (0,
1)$ (with coefficients given by $\nu_j$). 

We now introduce to \eqref{addrep} interactions leading to total
concavity. As \eqref{addrep} describes a linear model, it is natural
to add interaction terms by taking products of the variables appearing
in \eqref{addrep}.  When $d = 2$, there are four kinds of product
terms: $x_1 x_2$, $(x_1 - t_1)_+ x_2$ for $t_1 \in (0, 1)$, $x_1 (x_2
- t_2)_+$ for $t_2 \in (0, 1)$, and $(x_1 - t_1)_+ (x_2 - t_2)_+$ for
$t_1, t_2 \in (0, 1)$.  Appending these terms to the right-hand side
of \eqref{addrep} results in the following modified model for $d = 2$:
\begingroup
\allowdisplaybreaks
\begin{align}\label{intermod2d}
\begin{split}
  &f^*(x_1, x_2) = \beta_0 + \beta_1 x_1 + \beta_2 x_2 
  - \int_{(0, 1)} (x_1 - t_1)_+ \, d\nu_1(t_1) 
  - \int_{(0, 1)} (x_2 - t_2)_+ \, d\nu_2(t_2) 
  + \beta_{12} x_1 x_2 \\
  & \qquad - \int_{(0, 1)} (x_1 - t_1)_+ x_2 \, d\widetilde{\nu}_{1}(t_1) 
  - \int_{(0, 1)} x_1 (x_2 - t_2)_+ \, d\widetilde{\nu}_{2}(t_2) 
  - \int_{(0, 1)^2} (x_1 - t_1)_+ (x_2 - t_2)_+ \, d\nu_{12}(t_1, t_2)
\end{split}
\end{align}
\endgroup
where $\beta_{12} \in \R$ and $\widetilde{\nu}_1, \widetilde{\nu}_2,
\nu_{12}$ are finite \textit{signed} Borel measures (on $(0, 1), (0, 1), (0,
1)^2$ respectively) representing the coefficients of the interaction
terms.

Total concavity arises from imposing
\textit{nonnegativity} on the signed measures $\widetilde{\nu}_1,
\widetilde{\nu}_2, \nu_{12}$, so that they simply become measures. This
constraint ensures that each 
interaction term in \eqref{intermod2d} satisfies axial concavity
because $(x_1, x_2) \mapsto -\beta(x_1 - t_1)_+
x_2$, $(x_1, x_2) \mapsto -\beta x_1 (x_2 - t_2)_+$, and $(x_1, x_2)
\mapsto -\beta(x_1 - t_1)_+ (x_2 - t_2)_+$ are axially concave if and
only if $\beta \geq 0$.   
 
Let $\F_{\text{TC}}^2$ be the set 
of functions on $[0, 1]^2$ of the form \eqref{intermod2d} with
$\beta_0, \beta_1, \beta_2, \beta_{12} \in \R$ and finite 
Borel measures $\nu_1, \nu_2, \widetilde{\nu}_1, \widetilde{\nu}_2, 
\nu_{12}$; this is the class of totally concave functions on $[0, 1]^2$.  

For $d \ge 2$, one can define totally concave functions in a similar
fashion by adding interaction terms to the additive model
\eqref{addrep}. If we include product-form
interaction terms of all orders $s$ for $s = 2,
\dots, d$, we obtain the class $\tcd$ of all totally concave functions on
$[0, 1]^d$. If instead we only add interactions up to order $s$ for a
fixed $s \le d$, then we obtain the subclass $\tcds$ of 
totally concave functions on $[0, 1]^d$, which do not include
interaction terms of order greater than $s$. Formal definitions of $\tcd$
and $\tcds$ are given in Section \ref{measdef}. 

In the classical references
\citet{popoviciu1933quelques} and \citet[Chapter 2]{gal2010shape}, total concavity is defined in a different way in terms of nonpositivity of
certain divided differences. We recall this classical definition in
Section \ref{subsec:tcp} and prove that it is equivalent to the measure-based definition under a mild regularity condition. To the
best of our knowledge, this equivalence has not been previously
observed. For
sufficiently smooth functions, total concavity can be
characterized using derivatives (similar to other notions of
multivariate concavity). Specifically, we prove that a smooth function
$f$ belongs to $\tcd$ if and only if 
\begin{equation}\label{mixparder2}
 f^{(p_1, \dots, p_d)} := \frac{\partial^{p_1 + \cdots + p_d}
   f}{\partial x_1^{p_1} \cdots \partial x_d^{p_d}}  \le 0 \qt{for all
   $(p_1, \dots, p_d) \in \mathbb{Z}_{\ge 0}^d$ with $\max_{1 \le j \le d} p_j = 2$}.  
\end{equation}
In words, total concavity is nonpositivity of all partial
derivatives of \textit{maximum} order $\max_{1 \leq j \leq d} p_j$
equal to 2. In contrast, the 
other notions of multivariate concavity are characterized via partial
derivatives with \textit{total} order $\sum_{j=1}^d p_j$ at most 2
(see Section \ref{comparison-concavity} for details). 

The goal of this paper is a systematic study of the least squares
estimator $\hat{f}_{\text{TC}}^{d, s}$ over the class $\tcds$ for each
fixed $s$, and to explore its uses for practical regression. The
estimators $\hat{f}_{\text{TC}}^{d, s}$ with $2 \le s \le d$  are novel
multivariate extensions of univariate concave regression.  

The function class $\tcds$ is convex (albeit infinite-dimensional),
and hence, $\hat{f}_{\text{TC}}^{d, s}$ is a solution to an
infinite-dimensional convex optimization problem. Section
\ref{excomp} describes reduction to a finite-dimensional convex optimization problem
(specifically, nonnegative least squares 
\citep{slawski2013non, meinshausen2013sign} with an appropriate 
design matrix), to which standard optimization software is applicable. The fitted
functions are of the form: 
\begin{equation}\label{intro-fittedfunctions} 
\hat{\beta}_0 + \sum_{S : 1 \le |S|
      \le s} \hat{\beta}_S \bigg[\prod_{j \in S} 
    x_j \bigg] - \sum_{S: 1 \le |S| \le s} \sum_{\mathbf{t} = (t_j, j \in S)
    \in T_S} \hat{\beta}^{(\mathbf{t})} \bigg[\prod_{j \in S} (x_j - t_j)_+\bigg], 
\end{equation}
where $T_S \subseteq [0, 1)^{|S|} \setminus \{(0, \dots, 0)\}$ is a finite subset, $\hat{\beta}_0, \hat{\beta}_S \in \R$,
and $\hat{\beta}^{(\mathbf{t})} \ge 0$ for all $\mathbf{t} \in T_S$. These functions are similar to functions produced by the MARS method of  \cite{friedman1991multivariate} (see also
\cite{ki2024mars}). When $n$ is   
large, the number of parameters in this 
optimization problem can be huge, and we provide some strategies to cope with this. 

In Section \ref{rates}, we prove bounds on the rate of convergence of
 $\hat{f}_{\text{TC}}^{d, s}$ as an estimator of the true regression
 function 
 $f^*$ under the well-specified assumption $f^* \in \tcds$ and
under a standard fixed-design regression model with squared error
loss (we also prove a random-design result for a regularized
version of $\hat{f}_{\text{TC}}^{d, s}$). The rate of convergence
turns out to be $n^{-4/5} (\log 
n)^{3(2s-1)/5}$ for $2 \le s \le d$, meaning that the usual curse of
dimensionality is largely avoided by these estimators even when $s$ is
taken to be as large as $d$. Intuition behind this rate, mitigating the curse of dimensionality, can be drawn from the
characterization \eqref{mixparder2} of total concavity in terms of
nonpositivity of partial  derivatives of maximum order two. The number
of these partial derivatives increases exponentially in 
$d$, and the total order of these partial derivatives can also be as
large as $2d$. Thus, with $d$ increasing, the constraints on the class
also become proportionally restrictive, leading to the overall rate
affected by $d$ only in the logarithmic factor. This relatively fast
rate makes total concavity  
a promising constraint to use in multiple regression.
Note, however, that even 
though the constraints become proportionally restrictive with
$d$ increasing, the function class $\tcds$
contains many non-smooth functions with points of
non-differentiability. In contrast to totally concave regression,
rates of convergence for generally and axially
concave regressions (based on general and axial concavity) are slower
and affected by the curse of dimensionality (see, e.g.,
\cite{lim2014convergence, balazs2015near, han2016multivariate, kur2024convex}).

Beyond achieving theoretical convergence rates immune to the curse of dimensionality, we argue that totally concave regression is also effective in practice. In Section \ref{realdata}, we apply it to real-world regression problems and compare its performance to standard methods.

In improving practical prediction performance, we found it helpful to
further regularize $\hat{f}_{\text{TC}}^{d, s}$. Section
\ref{overfitting} introduces a regularized variant that imposes an
upper bound on the size of the measures defining totally concave
functions. This approach, combining total concavity with the
smoothness constraint of \citet{ki2024mars}, helps mitigate
overfitting near the boundary of the covariate domain, a common issue
with shape-constrained estimators. This variant also enables
theoretical analysis on random-design accuracy (Section
\ref{rdtheory}). 

For better performance when $d$ is large, Section \ref{extensions} 
discusses hybrid approaches that impose total concavity on a subset 
of covariates while assuming linearity on the others. Their utility 
in data analysis is explored in Section \ref{realdata}.

This paper builds on the work of \citet{ki2024mars} and \citet{fang2021multivariate}. The former introduces a smoothness-constrained estimator that shares key features with totally concave regression, while the latter proposes entirely monotonic regression, a monotonicity-based analogue of totally concave regression. Connections to these works are given in Section \ref{comparison-ki}.

As observed in \eqref{mixparder2}, total concavity can be framed as a 
nonpositivity constraint on all partial derivatives of maximum order two. 
Nonparametric regression methods with $L^p$ constraints on partial 
derivatives of maximum order $r$ (for fixed $r \geq 1$) have been widely 
studied (e.g., \cite{donoho2000high, lin2000tensor, benkeser2016highly, 
van2023higher, ki2024mars}). Under suitable assumptions, they can
achieve convergence rates that mitigate the curse of
dimensionality. In approximation theory, function classes with $L^p$
norm constraints on such mixed partial derivatives have also received
significant attention (see, e.g., \cite{bungartz2004sparse,
  dung2018hyperbolic, temlyakov2018multivariate}). 

The rest of the paper is organized thus. Section
\ref{totconcdef} introduces total concavity, while
Section \ref{comparison-concavity} compares it with standard
concavity notions. Section \ref{excomp} presents the
totally concave LSEs, and discusses existence
and computation. Sections \ref{overfitting} and \ref{rates} address
overfitting issues and convergence rates, respectively. Section
\ref{comparison-ki} relates our work to \citet{ki2024mars} and
\citet{fang2021multivariate}. Section \ref{extensions} details
variants for high-dimensional settings, and Section \ref{realdata}
presents applications. Computational details, additional plots, and
all proofs appear in Appendices \ref{axcon}, \ref{additional-plots}, and \ref{proofs}.

\section{Total Concavity: Definitions and Characterizations} \label{totconcdef} 
We present two definitions of total concavity.
The first, based on measures, is an extension of \eqref{intermod2d} to
$d \ge 2$. The second involves nonpositivity of divided
differences \cite[Chapter 2]{gal2010shape}. We prove  equivalence of
the two definitions (under a minor regularity condition).

\subsection{Measure-Based Definition}\label{measdef}
Before describing the extension of \eqref{intermod2d} to $d \ge 2$,
let us note the following simplification of the right-hand side of
\eqref{intermod2d}. The term $(x_1 - t_1)_+ (x_2 - t_2)_+$ coincides
with $x_1 (x_2 - t_2)_+$ when $t_1 = 0$ and with $(x_1 - t_1)_+ x_2$
when $t_2 = 0$. Thus, we can subsume $\widetilde{\nu}_1, \widetilde{\nu}_2,
\nu_{12}$ under a single measure on $[0, 1)^2 \setminus \{(0,
0)\}$, which leads to the following form that is equivalent to
\eqref{intermod2d}: 
\begin{align*}
  f^*(x_1, x_2) &= \beta_0 + \beta_1 x_1 + \beta_2 x_2 - \int_{(0, 1)}
  (x_1 - t_1)_+ \, d\nu_1(t_1) - \int_{(0, 1)}
  (x_2 - t_2)_+ \, d\nu_2(t_2)  \\ 
  &\quad+ \beta_{12} x_1 
  x_2 - \int_{[0, 1)^2 \setminus \{(0, 0)\}} (x_1 - t_1)_+ (x_2 -
    t_2)_+ \, d\nu_{12}(t_1, t_2). 
\end{align*}
Here, $\beta_0, \beta_1, \beta_2, \beta_{12}$ are real numbers, 
$\nu_1$, $\nu_2$ are finite measures on $(0, 1)$, and $\nu_{12}$ is a finite
measure on $[0, 1)^2 \setminus \{(0, 0)\}$. The class of all
such functions is denoted by $\mathcal{F}_{\text{TC}}^2$. The
$d$-dimensional analogue of this function class consists of all
functions of the form: 
\begin{equation}\label{intermodd}
  f^*(x_1, \dots, x_d) = \beta_0 + \sum_{\emptyset \neq S \subseteq [d]} \beta_S \prod_{j \in S} x_j -\sum_{\emptyset \neq S \subseteq [d]} \int_{[0, 1)^{|S|} \setminus \{\zerovec\}} \prod_{j \in S} (x_j - t_j)_+ \, d\nu_S(t_j, j \in S)     
\end{equation}
where $[d] := \{1, \dots, d\}$, $\zerovec := (0, \dots, 0)$, $\beta_S$
is a coefficient for the interaction  
$\prod_{j \in S} x_j$, and $\nu_S$ is a finite Borel measure on $[0,
1)^{|S|} \setminus \{\zerovec\}$ ($|S|$ denotes the size of $S$) representing
coefficients for the interactions $-\prod_{j \in S}(x_j -
t_j)_+$ for each nonempty $S \subseteq [d]$. We 
denote by
$\tcd$ the class of all functions of the form \eqref{intermodd} and
call it the class of totally concave functions on $[0, 1]^d$.

Functions of the form \eqref{intermodd} are obtained by
adding interaction terms of all orders to the additive model
\eqref{addrep} 
along with sign constraints ensuring axial concavity of each
interaction term. In practice, interactions
of order greater than $2$ or $3$ are rarely used. It is therefore natural
to consider a subclass of $\tcd$ where we only include interactions
of order at most $s$ for a pre-specified $s \in \{1, \dots, d\}$
(typically $s = 2$ or $3$). 
This leads to functions of the form 
\begin{equation}\label{intermodds}
  \beta_0 + \sum_{S : 1 \le |S| \le s} \beta_S \prod_{j \in S} x_j - \sum_{S : 1 \le |S| \le s} \int_{[0, 1)^{|S|} \setminus \{\zerovec\}} 
  \prod_{j \in S} (x_j - t_j)_+ \, d\nu_S(t_j, j \in S).                  
\end{equation}
We denote by $\tcds$ the class of all functions of the form
\eqref{intermodds}. $\tcdd = \tcd$ and $\tcds$ is strictly smaller than
$\tcd$ for $s < d$. When $s = 1$, $\tcds$ is the class of additive
concave functions.

\subsection{Total Concavity in the sense of Popoviciu}\label{subsec:tcp}
We now present the second definition of total concavity, originally
due to \cite{popoviciu1933quelques}. We refer to this as \textit{total
  concavity in the sense of Popoviciu}. 

\begin{definition}[Total concavity in the sense of
  Popoviciu]\label{defn-tc-popoviciu} 
  We say $f: [0, 1]^d \rightarrow \R$ is totally
  concave in the 
  sense of Popoviciu if for every $(p_1, \dots, p_d) \in \mathbb{Z}_{\ge 0}^d$ with $\max_k
p_k = 2$, we have 
\begin{equation}\label{divdiff}
    \sum_{i_1 = 1}^{p_1 + 1} \cdots \sum_{i_d = 1}^{p_d + 1} \frac{f(x_{i_1}^{(1)}, \dots, x_{i_d}^{(d)})}{\prod_{j_1 \neq i_1} (x_{i_1}^{(1)} - x_{j_1}^{(1)}) \times \cdots \times \prod_{j_d \neq i_d} (x_{i_d}^{(d)} - x_{j_d}^{(d)})} 
    \le 0
\end{equation}
for every $0 \le x_1^{(k)} < \cdots < x_{p_k + 1}^{(k)} \le 1$ for $k
= 1, \dots, d$.
\end{definition}

The left-hand side of \eqref{divdiff} is a divided difference of $f$
of order $(p_1, \dots, p_d)$  
(see \citet[Section 6.6]{isaacson1994analysis} or 
\citet[Chapter 2]{gal2010shape}
for more information on these tensor-product kind of divided differences). In words, $f$ is totally concave in
the sense of Popoviciu if the divided difference of $f$ of order
$(p_1, \dots, p_d)$ is nonpositive for all $(p_1, \dots, p_d)$ 
with $\max_{k} p_k = 2$.
We denote by $\tcp$ the class of all functions $f$
on $[0, 1]^d$ that are totally concave in the sense of Popoviciu.

\subsection{Equivalence with the Measure-Based
  Definition}\label{total-concavity-equiv}

The two definitions of total concavity are equivalent under a minor
regularity condition: 
  

\begin{proposition}\label{prop:tc-equiv}
    Every function $f \in \tcd$ is totally concave in the sense of Popoviciu, i.e., $f \in \tcp$. 
    Also, every function $f \in \tcp$ satisfying the following additional regularity conditions belongs to $\tcd$: 
    for each nonempty subset $S$ of $[d]$,
   \begin{equation}\label{eq:finite-derivative-at-zero}
    \sup_{t_k > 0,\, k \in S} 
    \bigg\{
    \frac{1}{\prod_{k \in S} t_k} 
    \sum_{R \subseteq S} (-1)^{|S| - |R|} \cdot
    f\Big( \big( t_k \cdot \ind\{k \in R\}\big)_{k=1}^d \Big)
    \bigg\} < +\infty, 
\end{equation}
\begin{equation}\label{eq:finite-derivative-at-one}
    \inf_{t_k < 1,\, k \in S} 
    \bigg\{
    \frac{1}{\prod_{k \in S} (1 - t_k)} 
    \sum_{R \subseteq S} (-1)^{|S| - |R|} \cdot
    f\Big( \big( \ind\{k \in R\} 
    + t_k \cdot \ind\{k \in S \setminus R\}\big)_{k=1}^d \Big)
    \bigg\} > -\infty.
  \end{equation}
\end{proposition}
In \eqref{eq:finite-derivative-at-zero} and
\eqref{eq:finite-derivative-at-one}, $(a_k)_{k=1}^d$ denotes the
vector with $k^{\text{th}}$ component $a_k$. Thus, 
$(t_k \cdot \ind\{k \in R\})_{k=1}^d$ has entries $t_k$ if $k \in R$ and $0$ 
otherwise, while 
$(\ind\{k \in R\} + t_k \cdot \ind\{k \in S \setminus R\})_{k=1}^d$ has entries 
$1$ if $k \in R$, $t_k$ if $k \in S \setminus R$, and $0$ otherwise. 
See the beginning of Appendix \ref{proofs} for other equivalent representations of 
\eqref{eq:finite-derivative-at-zero} and \eqref{eq:finite-derivative-at-one}.

The conditions \eqref{eq:finite-derivative-at-zero} and
\eqref{eq:finite-derivative-at-one} reduce to \eqref{1dreg} when $d =
1$, so that Proposition \ref{prop:tc-equiv} reduces to the univariate
representation theorem \eqref{1drepconcave}. Totally concave functions
with restricted interaction—functions in 
$\tcds$ (see \eqref{intermodds})—can also be described as totally
concave functions in the 
Popoviciu sense that additionally satisfy the following \textit{interaction restriction condition}: 
\begin{equation}\label{int-rest-cond}
    \sum_{R \subseteq S} (-1)^{|S| - |R|} \cdot
    f\Big( \big( y_k \cdot \ind\{k \in R\} 
    + x_k \cdot \ind\{k \in S \setminus R\}\big)_{k=1}^d \Big) = 0
\end{equation}
for every subset $S$ of $[d]$ with
$|S| > s$, and for all $0 \le x_k < y_k \le 1$, $k \in S$. 
\begin{proposition}\label{prop:tc-equiv-restricted-interaction}
    The function class $\tcds$ consists of all functions $f \in \tcp$ satisfying 
    \eqref{eq:finite-derivative-at-zero} and \eqref{eq:finite-derivative-at-one} 
    for all nonempty $S \subseteq [d]$, and \eqref{int-rest-cond} for every 
    $S \subseteq [d]$ with $|S| > s$.
  \end{proposition}

\subsection{Smoothness Characterization}

The next result shows that for smooth functions, total concavity is 
characterized by \eqref{mixparder2}, which is basically a continuous 
analogue of the divided difference constraint in 
Definition~\ref{defn-tc-popoviciu}.  

\begin{proposition}\label{prop:alt-characterization-smooth}
    Let $f: [0, 1]^d \rightarrow \mathbb{R}$ be such that $f^{(\mathbf{p})}$ 
    exists and is continuous for every 
    $\mathbf{p} = (p_1, \dots, p_d) \in \{0, 1, 2\}^d$. Then, $f \in \tcd$ if 
    $f^{(\mathbf{p})} \le 0$ for all $\mathbf{p} \in \{0, 1, 2\}^d$ with 
    $\max_j p_j = 2$. If, in addition, $f^{(\mathbf{p})} = 0$ for all 
    $\mathbf{p} \in \{0, 1\}^d$ with $\sum_{j} p_j > s$, then 
    $f \in \tcds$.
\end{proposition}

\section{Other Notions of Multivariate Concavity}\label{comparison-concavity}
Every totally concave function is axially concave (i.e., it is concave
in each coordinate when all other coordinates are held fixed). In
fact, $f$ is axially concave if and only if \eqref{divdiff} holds for
$(p_1, \dots, p_d) = (2, 0, \dots, 0), \dots, (0, \dots, 0, 2)$ (see 
Appendix \ref{axcon_char} for an explanation), 
which is clearly weaker than total concavity, which requires 
\eqref{divdiff} for all $(p_1, \dots, p_d)$ with $\max_k p_k = 2$. On the 
other hand, every additive concave function satisfies the condition in 
Definition \ref{defn-tc-popoviciu} and is therefore totally
concave. There is no strict relation between total
concavity and general concavity (general concavity has the
usual definition: $f((1-\alpha)\mathbf{x} + \alpha \mathbf{y}) \ge (1
- \alpha) f(\mathbf{x}) + \alpha f(\mathbf{y})$ for all $\alpha \in
[0, 1]$ and $\mathbf{x}, \mathbf{y} \in [0, 1]^d$). There exist
totally concave functions (e.g., $(x_1, x_2) \mapsto x_1 x_2$) which
are not generally concave, and there also exist generally concave
functions (e.g., $(x_1, x_2) \mapsto (x_1 + x_2 - 0.5)_+$) which are
not totally concave. These relations are summarized in Figure
\ref{pictconc}.  
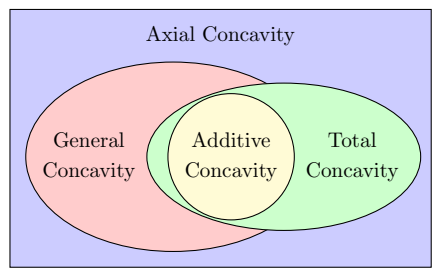
\begin{figure}[ht]
\begin{center}
\begin{tikzpicture}[scale=0.7]
    \draw[fill=blue!20, draw=black] (-4.2,-2.1) rectangle (3.8,2.8);
    \node[scale = 0.8] at (-0.2,2.3) {Axial Concavity};

    \draw[fill=red!20, draw=black] (-1.1,0) ellipse (2.8cm and 1.8cm);
    \node[scale = 0.8, text width=2cm, align=center] at (-2.7,0) {General\\Concavity};

    \draw[fill=green!20, draw=black] (1,0) ellipse (2.6cm and 1.4cm);
    \node[scale = 0.8, text width=2cm, align=center] at (2.3,0) {Total\\Concavity};

    \draw[fill=yellow!20, draw=black] (0,0) circle (1.2cm);
    \node[scale = 0.8, text width=2cm, align=center] at (0,0) {Additive\\Concavity};
\end{tikzpicture}
\end{center}
\caption{Four types of multivariate concavity.}
\label{pictconc}
\end{figure}

For smooth $f$, these can be characterized as (below $Hf$ is
the Hessian matrix of $f$): 
\begin{itemize}
\item \textbf{General Concavity}: 
  $-\mathit{H} f \succcurlyeq 0$ ($\succcurlyeq
0$ stands for positive semi-definiteness),
\item \textbf{Axial Concavity}: 
  $-\text{diag}(\mathit{H}f) \ge 0$ 
  (here, $\text{diag}(\mathit{H}f)$ indicates the diagonal of $\mathit{H}f$ and $\ge 0$
  represents element-wise nonnegativity),
\item \textbf{Additive Concavity}: 
  $-\mathit{H}f = -\text{diag}(\mathit{H}f) \ge 0$,
\item \textbf{Total Concavity}: 
  $-f^{(p_1, \dots, p_d)} \ge 0$ for
  every $(p_1, \dots, p_d)$ with $\max_j p_j = 2$. 
\end{itemize}
General, axial, and additive concavity can be characterized via the
Hessian matrix. Total concavity, on the other hand, is fundamentally
different because it involves higher 
order derivatives and cannot be characterized via the
Hessian. 

Nonparametric regression with general concavity has received 
attention \citep{kuosmanen2008representation, SS11,
  lim2012consistency, balazs2016convex} including determination of
its rates of convergence \citep{han2016multivariate,
  kur2024convex}, which are all affected by the usual curse of
dimensionality. Additive concavity for regression has been studied
in \cite{meyer2013semi, meyer2018framework, pya2015shape,
  chen2016generalized}; unfortunately, additive models do not allow
for interaction effects. Axial concavity is a straightforward concept
of multivariate concavity that, surprisingly, has not garnered much
attention, perhaps with
the exception of  \citet{iwanaga2016estimating}. By contrast, axial
monotonicity (recognized as the standard notion for
multivariate monotonicity) has been extensively studied (see, e.g., 
\cite{RWD88, chatterjee2018matrix, han2019isotonic}).

\section{Totally Concave Least Squares Estimators}\label{excomp}

Given data $({\mathbf{x}}^{(i)}, y_i), 1 \leq i \leq n$ with
${\mathbf{x}}^{(i)} \in [0, 1]^d$ and $y_i \in \R$, the
LSE over $\tcds$ is
\begin{equation}\label{tcreg}
  \hat{f}_{\text{TC}}^{d, s} \in \argmin_{f \in \tcds} \sum_{i=1}^n
  \big(y_i - f({\mathbf{x}}^{(i)}) \big)^2.
\end{equation}
The following result shows that 
$\hat{f}_{\text{TC}}^{d, s}$ also minimizes least squares over all
functions $f \in \tcp$ satisfying the interaction 
restriction condition \eqref{int-rest-cond} for every subset $S$ of
$[d]$ with $|S| > s$.
\begin{proposition}\label{prop:reduction-from-popoviciu}
    Every minimizer $\hat{f}_{\text{TC}}^{d, s}$ of \eqref{tcreg} satisfies
    \begin{equation*}
      \hat{f}_{\text{TC}}^{d, s} \in \argmin_f \bigg\{\sum_{i=1}^n 
      \big(y_i - f({\mathbf{x}}^{(i)}) \big)^2 : f \in \tcp \text{ and satisfies } \eqref{int-rest-cond} \bigg\}. 
    \end{equation*}
    In particular, $\hat{f}_{\text{TC}}^{d, d}$ also minimizes least
    squares over the entire class $\tcp$. 
  \end{proposition}


As totally concave functions
\eqref{intermodds} are parameterized by  
measures, the least
squares problem \eqref{tcreg}  is an infinite-dimensional
optimization problem. The following result guarantees a finite-dimensional
reduction where each $\nu_S$ can be taken to be a
discrete measure supported on the lattice $L_S$ generated by ${\mathbf{x}}^{(1)}, \dots, {\mathbf{x}}^{(n)}$ (and $\zerovec = (0, \dots, 0)$): 
\begin{equation}\label{lattices-gen-by-data}
  L_S := \prod_{j \in S} \Big(\{0\} \cup \big\{x_j^{(i)}: i = 1, \dots, n \big\} \Big) \setminus \{\zerovec\}.
\end{equation}


\begin{proposition}\label{prop:existence-computation}
   The estimator $\hat{f}_{\text{TC}}^{d, s}$ exists and can be
   computed as the least squares minimizer over the class of all
   functions of the form:  
   \begin{equation}\label{eq:discrete-measure-subclass}
     \beta_0 + \sum_{S : 1 \le |S| \le s} \beta_S \bigg[\prod_{j \in S}
    x_j \bigg] - \sum_{S: 1 \le |S| \le s} \sum_{\mathbf{t} = (t_j, j \in S)
    \in L_S} \beta^{(\mathbf{t})} \bigg[\prod_{j \in S} (x_j - t_j)_+\bigg]
\end{equation}
where $\beta_0, \beta_S, 1 \leq |S| \leq s$ are real numbers, and $\beta^{(\mathbf{t})}, \mathbf{t} \in L_S$ are nonnegative. 
\end{proposition}


Functions of the form \eqref{eq:discrete-measure-subclass} arise by
constraining the measures $\nu_S$ in \eqref{intermodds} to be discrete and
supported on the lattices $L_S$. Restriction to these functions results in a
finite-dimensional convex optimization problem solvable with standard
optimization tools. Our implementation uses the software
\textsf{MOSEK} via its \textsf{R} interface 
\textsf{Rmosek}. Since this optimization involves $O(n^s)$ variables, computation can
be expensive for large $n$. For such cases, we employ an
approximate LSE by replacing each $L_S$ with a
proxy lattice obtained by taking $N_j$ equally spaced points  for the
$j^{\text{th}}$ coordinate, where each $N_j$ is pre-selected.  


\section{Overfitting Reduction}\label{overfitting}

Shape-constrained regression estimators tend to 
overfit near the boundary of the covariate domain, even in the
univariate case 
\citep{woodroofe1993penalized, sun1996adaptive, 
kulikov2006behavior, lim2012consistency, mazumder2019computational, 
bertsimas2021sparse, liao2024overfitting}.  
Totally concave regression shows a similar tendency.
The only restriction on the measures $\nu_S$ is nonnegativity, and hence, 
they are allowed to take arbitrarily large weights near the boundary to 
achieve nearly perfect fits in that region. This issue can be mitigated by 
constraining the sum of the sizes $\nu_S([0, 1)^{|S|} \setminus \{\zerovec\})$: 
\begin{equation}\label{modified-estimator}
  \hat{f}_{\text{TC}, V}^{d, s} \in \argmin_f \bigg\{\sum_{i=1}^n
  \big(y_i - f(\mathbf{x}^{(i)}) \big)^2: f \in \tcds 
  \text{ and } V(f) \le V \bigg\},
\end{equation}
where, for $f \in \tcds$ of the form \eqref{intermodds},
\begin{equation}\label{complexity-measure}
    V(f) := \sum_{S : 1 \le |S| \le s} 
    \nu_S \big([0, 1)^{|S|} \setminus \{\zerovec\}\big),
\end{equation}
and $V$ is a tuning parameter chosen via cross-validation in practice. 
Estimators with such additional constraints, similar to
$\hat{f}_{\text{TC}, V}^{d, s}$, have been used to reduce boundary overfitting 
\citep{woodroofe1993penalized, sun1996adaptive, liao2024overfitting}. 


\section{Theoretical Results}\label{rates}
We study accuracy of our estimators under the model (for both
fixed and random designs):
\begin{equation}\label{regression-setting}
  y_i = f^*(\mathbf{x}^{(i)}) + \xi_i,
\end{equation}
where $f^*$ is the unknown regression function and $\xi_i$ are mean-zero 
errors. Note that although we state our accuracy results for a fixed $f^*$
satisfying certain conditions, they actually hold uniformly over all
$f^*$ in that class, thereby implying minimax upper bounds.

\subsection{Fixed Design}\label{fdtheory}
Here, we assume that $\mathbf{x}^{(1)}, \dots, \mathbf{x}^{(n)}$ form the 
following equally-spaced lattice:
\begin{equation}\label{equally-spaced-lattice}
  \{\mathbf{x}^{(1)}, \dots, \mathbf{x}^{(n)}\} 
  = \prod_{k = 1}^{d} \Big\{0, \frac{1}{n_k}, \dots, 
  \frac{n_k - 1}{n_k} \Big\}
\end{equation}
for some integers $n_k \ge 2$, $k = 1, \dots, d$, and define loss and risk as
\begin{equation*}
  \|\hat{f} - f^*\|_{n}^2 
  := \frac{1}{n} \sum_{i = 1}^{n} \big(\hat{f}(\mathbf{x}^{(i)}) 
  - f^*(\mathbf{x}^{(i)})\big)^2 \ \text{ and } \ R_F(\hat{f}, f^*) 
  := \E\big[\|\hat{f} - f^*\|_{n}^2\big],
\end{equation*}
where the expectation is taken over $y_1, \dots, y_n$. The following result 
shows that the risk of  $\hat{f}_{\text{TC}}^{d, s}$ is of order $n^{-4/5}$ 
(up to log factors). This result involves the following quantity:
\begin{equation*}
  V_{\text{design}}(f) := \inf \big\{V(g): g \in \tcd \text{ such that } 
  g(\mathbf{x}^{(i)}) 
  = f(\mathbf{x}^{(i)}) \text{ for } i = 1, \dots, n\big\}, 
\end{equation*}
where $V(g)$ is as defined in \eqref{complexity-measure} for $s = d$. We 
provide a simple upper bound for $V_{\text{design}}(f)$ when $f \in \tcp$ in 
Lemma \ref{lem:complexity-upper-bound}. In all results in this
subsection, $C_d$ is a constant that depends only on $d$, whose
value may change from line to line. 

\begin{theorem}\label{thm:risk-bound-fixed} 
  Assume that $f^* \in \tcp$ and satisfies \eqref{int-rest-cond} for
  every $S \subseteq [d]$ with  
  $|S| > s$. Also, assume \eqref{equally-spaced-lattice} and 
  $\xi_i \overset{\text{i.i.d.}}{\sim} N(0, \sigma^2)$. Then, with 
  $V := V_{\text{design}}(f^*)$, 
  \begin{equation*}
    R_F(\hat{f}_{\text{TC}}^{d, s}, f^*) 
    \le C_d \bigg( \frac{\sigma^2 V^{1/2}}{n} \bigg)^{4/5} 
    \bigg[ \log\Big(\Big(1 + \frac{V}{\sigma}\Big) n \Big)\bigg]^{3(2s - 1)/5} 
    + C_d \cdot \frac{\sigma^2}{n} (\log n)^{5d/4} \bigg[ \log\Big(\Big(1 
    + \frac{V}{\sigma}\Big) n \Big)\bigg]^{3(2s - 1)/4}
  \end{equation*}
  if $s \ge 2$, and for $s = 1$, 
  \begin{equation*}
    R_F(\hat{f}_{\text{TC}}^{d, 1}, f^*) 
    \le C_d \bigg( \frac{\sigma^2 V^{1/2}}{n} \bigg)^{4/5} 
    + C_d \cdot \frac{\sigma^2}{n} (\log n)^{d} \bigg[ \log\Big(\Big(1 
    + \frac{V}{\sigma}\Big) n \Big)\bigg]^{2} 
    + C_d \cdot \frac{\sigma^2}{n} (\log n)^{5d/4}. 
  \end{equation*}
\end{theorem}

The risk of the LSE in univariate concave
regression is of order $n^{-4/5}$
\citep{guntuboyina2015global, chatterjee2015risk, han2016multivariate, 
chatterjee2016improved, bellec2018sharp}. Theorem \ref{thm:risk-bound-fixed} 
shows that our estimators achieve nearly the same rate, 
differing only by logarithmic factors, even when $s = d$. Thus,
totally concave regression largely avoids  
the dimensionality curse, making it attractive for 
regression in higher dimensions. 

Shape-constrained estimators exhibit adaptivity by achieving 
nearly parametric convergence rates at specific points in the 
function space (\cite{samworth2018recent, guntuboyina2018nonparametric}). For 
$\hat{f}_{\text{TC}}^{d, s}$, adaptivity manifests when $f^* \in \tcds$ belongs 
to $\mathcal{R}^{d, s}$, the class of continuous rectangular piecewise 
multi-affine functions. A function $f$ is in $\mathcal{R}^{d, s}$ if it is 
continuous and there exists an axis-aligned split of $[0,1]^d$ into rectangles 
of the form $\prod_{k = 1}^{d} [u^{(k)}_{j_k - 1}, u^{(k)}_{j_k}]$ where 
$0 = u^{(k)}_0 < \dots < u^{(k)}_{m_k} = 1$ for $k = 1, \dots, d$, such that 
on each rectangle,
\begin{equation}\label{piecewise-multi-affine}
  f(x_1, \dots, x_d) = \beta_0^{(\mathbf{j})} + \sum_{S: 1 \le |S| \le s} 
  \beta_S^{(\mathbf{j})} \bigg[ \prod_{k \in S} x_k \bigg]
\end{equation}
for some constants $\beta_0^{(\mathbf{j})}$ and $\beta_S^{(\mathbf{j})}$. For 
$f \in \mathcal{R}^{d, s}$, we denote by $m(f)$ the minimum number of 
rectangles needed for such a representation. In the univariate case 
($d = s = 1$), $\mathcal{R}^{d, s}$ simply consists of continuous piecewise 
affine functions. The next theorem implies that if 
$f^* \in \tcds \cap \mathcal{R}^{d, s}$, then $\hat{f}_{\text{TC}}^{d, s}$ 
achieves the rate of order $n^{-1}$ (with logarithmic factors), which is much 
faster than the rate $n^{-4/5}$ given by Theorem \ref{thm:risk-bound-fixed}.
\begin{theorem}\label{thm:adaptive}
 Assume \eqref{equally-spaced-lattice} and 
  $\xi_i \overset{\text{i.i.d.}}{\sim} N(0, \sigma^2)$. Then, for every
  $f^* \in \tcds \cap \mathcal{R}^{d, s}$, 
\begin{equation*}
  R_F(\hat{f}_{\text{TC}}^{d, s}, f^*) \le 
  \begin{cases}
    C_d \cdot m(f^*) \cdot \sigma^2 (\log n)^{(5d + 6s - 3)/4} / n 
    & \text{if } s \ge 2 \\
    C_d \cdot m(f^*) \cdot \sigma^2 (\log n)^{\max(d + 2, 5d/4)} / n
    & \text{if } s = 1.
  \end{cases}
\end{equation*}
\end{theorem}
The next result generalizes Theorem \ref{thm:adaptive} to cases
where $f^*$ need not be rectangular piecewise multi-affine but may be
well-approximated by that class. It establishes a sharp oracle
inequality (in the sense of \cite{bellec2018sharp}) and applies even
when $f^*$ is not totally concave. 



\begin{theorem}\label{thm:sharp-oracle}
  Assume \eqref{equally-spaced-lattice} and 
  $\xi_i \overset{\text{i.i.d.}}{\sim} N(0, \sigma^2)$. For every
  function $f^*: [0, 1]^d \rightarrow \R$,
  \begin{equation*}
      R_F(\hat{f}_{\text{TC}}^{d, s}, f^*) 
      \le \inf_{f \in \tcds \cap \mathcal{R}^{d, s}} \Big\{\|f - f^*\|_n^2 
      + C_d \cdot m(f) \cdot \frac{\sigma^2}{n} (\log n)^{(5d + 6s - 3)/4}
      \Big\}
      \ \text{ if } s \ge 2, \text{ and }
  \end{equation*}
  \begin{equation*}
      R_F(\hat{f}_{\text{TC}}^{d, 1}, f^*) 
      \le \inf_{f \in \F^{d, 1}_{\text{TC}} \cap \mathcal{R}^{d, 1}} 
      \Big\{\|f - f^*\|_n^2 + C_d \cdot m(f) \cdot \frac{\sigma^2}{n} 
      (\log n)^{\max(d + 2, 5d/4)} \Big\}
      \ \text{ if } s = 1. 
  \end{equation*}
\end{theorem}

\subsection{Random Design}\label{rdtheory}

Here, our design assumption and loss function are (below $b$, $B$ are fixed
positive constants):  
\begin{equation}\label{rd_assumption}
  \mathbf{x}^{(i)} \overset{\text{i.i.d.}}{\sim} \text{ density } p_0:
  [0, 1]^d \rightarrow [b, B] \ \text{ and } \ \|\hat{f} -
  f^*\|_{p_0, 2}^2  
    := \int \big(\hat{f}(x) - f^*(x)\big)^2 p_0(x) \, dx.
\end{equation}
We also analyze the regularized estimator 
$\hat{f}_{\text{TC}, V}^{d, s}$ (defined in
\eqref{modified-estimator}). This aligns with existing literature (e.g., \cite{han2016multivariate, mazumder2019computational, 
kur2019optimality}) where random-design analyses study constrained
variants of original LSEs. 

For our first result, we assume that the errors $\xi_i$ are
independent of $\mathbf{x}^{(i)}$ and satisfy:  
\begin{equation}\label{rd_errors}
 \xi_i \text{ are i.i.d. with } \E \xi_i = 0  \ \text{ and } \ 
 \|\xi_i\|_{5, 1} := 
 \int_0^{\infty} \left(\P \{|\xi_i| > t\} \right)^{1/5} dt < \infty. 
\end{equation}
This finite $L^{5, 1}$ norm condition—slightly stronger than the finite
$L^5$ norm condition (see, e.g., \citet[Chapter
1.4]{loukas2014classical})—is much weaker than the
Gaussian error
assumption used in our fixed design results. The extra
regularization in $\hat{f}_{\text{TC}, V}^{d, s}$ enables us to work under
this weaker assumption, using tools from
\citet{han2019convergence}. The following result is our random design 
analogue of Theorem \ref{thm:risk-bound-fixed}, and it yields the same 
qualitative conclusions. 

\begin{theorem}\label{thm:risk-bound-random}
    Assume $f^* \in \tcds$ with $V(f^*) \le V$,
    \eqref{rd_assumption} with $b = 0$, and \eqref{rd_errors}. Then, 
    \begin{equation*}
        \|\hat{f}_{\text{TC}, V}^{d, s} - f^*\|_{p_0, 2}^2 
        = O_p \big(n^{-4/5} (\log n)^{8(s - 1)/5}\big),
    \end{equation*}
    where the constants underlying $O_p(\cdot)$ depend on $d, V$, 
    $B$, and the moments of $\xi_i$.
\end{theorem}

Our next result shows adaptivity of $\hat{f}_{\text{TC}, V}^{d, s}$
when $f^* \in \tcds \cap \mathcal{R}^{d, s}$. For this, we assume 
that $\xi_i$ are sub-exponential errors, independent of 
$\mathbf{x}^{(i)}$ and satisfying the following:
\begin{equation}\label{sub-exponential}
 \xi_i \text{ are i.i.d. with } \E \xi_i = 0 \ \text{ and } \ 
 2K^2 \cdot \E\bigg[ \exp\Big(\frac{|\xi_i|}{K}\Big) - 1 
 - \frac{|\xi_i|}{K} \bigg] \le \sigma^2 
 \ \text{ for some } K, \sigma > 0. 
\end{equation}

\begin{theorem}\label{thm:adaptive-random}
  Assume $f^* \in \tcds \cap \mathcal{R}^{d, s}$ with $V(f^*) \le V$,
  \eqref{rd_assumption}, and \eqref{sub-exponential}. Then,
  \begin{equation*}
    \|\hat{f}_{\text{TC}, V}^{d, s} - f^*\|_{p_0, 2}^2 
    = O_p \big(n^{-1/2} (\log n)^{\max(5d/8 + (s - 1)/e, d/2 + 1/2e)}\big),
  \end{equation*}
  where the constants underlying $O_p(\cdot)$ depend on 
  $d, V, b, B, K, \sigma$, and $m(f^*)$.
\end{theorem}

In all our results (fixed and random design), the covariate domain is
$[0, 1]^d$. By scaling, the theorems can be extended
to any axis-aligned rectangle. However, extensions to more general shapes
(e.g., polytopes) do not follow from our proofs, which rely
on rectangular geometry—decompositions into subrectangles
and entropy bounds for total concavity classes on each subrectangle. 
It is unclear how to adapt our arguments beyond rectangles.

We also note that total concavity, like axial and additive concavity,
is inherently axis-aligned, unlike general concavity. While
results for least squares under general concavity exist on
non-rectangular domains (see \cite{han2016multivariate,
  kur2019optimality, kur2024convex}), establishing such results for
the axis-aligned notion of total concavity appears challenging. 

\section{Relations to \citet{ki2024mars} and 
\citet{fang2021multivariate}}\label{comparison-ki}

This paper builds on work in \citet{ki2024mars} and
\citet{fang2021multivariate}. Our function class $\tcds$ consists of
functions \eqref{intermodds} involving finite measures $\nu_S$. If 
the nonnegativity assumption on $\nu_S$ is removed, thereby allowing
them to be finite 
\textit{signed} measures instead, and a regularization constraint involving
the sum of the total variations of the signed measures is imposed,
then the method MARS-LASSO of \citet{ki2024mars} is
obtained. MARS-LASSO requires tuning and does not preserve any shape
constraints, while totally concave regression is tuning-free and
shape-constrained. Also, the MARS-LASSO estimator is an infinite-dimensional
LASSO estimator, while totally concave regression yields an
infinite-dimensional nonnegative least squares (NNLS) estimator. While
there has been prior work on general infinite-dimensional LASSO
estimators (\cite{rosset2007L1, bredies2013inverse}), there appears to
be no general treatment for 
infinite-dimensional NNLS estimators. 

Despite the differences outlined above, some of our results closely
resemble those established for the MARS-LASSO estimator in
\cite{ki2024mars}. Our existence and computation result (Proposition
\ref{prop:existence-computation}) parallels those in \cite[Section
2]{ki2024mars}, as both rely on the same discretization idea. While
the convergence rates match, proof techniques differ substantially due
to a key distinction: in totally concave regression, the sizes of the
measures are unconstrained, whereas MARS-LASSO imposes a
regularization constraint restricting the sizes of its measures. This
mirrors the relationship between LASSO and NNLS in finite-dimensional
regression, where similar results emerge from different assumptions
and proofs. 
Our regularized totally concave regression estimator \eqref{modified-estimator} from Section \ref{overfitting} incorporates both total concavity and the smoothness constraint of MARS-LASSO. 

Entirely monotonic regression introduced in
\cite{fang2021multivariate} is a monotonicity analogue of totally
concave regression. While total concavity requires partial derivatives of 
max order two to be nonpositive, entire monotonicity constrains
partial derivatives of max order one to be nonnegative. Entirely monotonic
regression fits discontinuous rectangular piecewise constant functions
in contrast to the continuous functions \eqref{intro-fittedfunctions}
produced by totally concave regression. Rates of convergence for
entirely monotonic regression are of order $n^{-2/3}$ (up to
logarithmic factors), whereas totally concave regression achieves the
rates of order $n^{-4/5}$ (up to logarithmic factors). Thus, totally
concave regression leads to more accurate estimators for smooth
functions (on the other hand, if the true regression function is
discontinuous, then totally concave regression may lead to
inconsistent estimates).

\section{Variants for Handling Large $d$}\label{extensions}

The computational complexity of $\hat{f}^{d, s}_{\text{TC}}$ increases
with $d$. When $d$ is large, it thus makes sense to assume total
concavity only on a subset of the covariates, while imposing the
stronger (and simpler) assumption of 
linearity on the remaining covariates. This leads to the model:
\begin{equation}\label{vari1}
  f^*(x_1, \dots, x_d) = f_1^*(x_1, \dots, x_p) + \beta_{p+1} x_{p+1}
  + \dots + \beta_d x_d
\end{equation}
where $f_1^* \in \mathcal{F}^{p, s}_{\text{TC}}$ and $\beta_{p+1},
\dots, \beta_d \in \R$.  We denote this class by $\F^{p,
  s}_{\text{TC}}(x_1, \dots, x_p) \oplus \mathcal{L}(x_{p + 1}, \dots,
x_d)$ and the LSE by $\hat{f}^{d,
  p, s}_{\text{TC-L}}$. 
Using \eqref{intermodds}, we rewrite \eqref{vari1} as
\begin{equation*}
    \beta_0 + \sum_{\substack{S \subseteq [p] \\ 1 \le |S| \le s}} \beta_S \prod_{j \in S} x_j - \sum_{\substack{S \subseteq [p] \\ 1 \le |S| \le s}} \int_{[0, 1)^{|S|} \setminus \{\zerovec\}} \prod_{j \in S} (x_j - t_j)_+ \, d\nu_S(t_j, j \in S) + \sum_{j = p + 1}^{d} \beta_j x_j.
\end{equation*}
We further extend this model using interactions between
covariates $x_1,\dots,x_p$ and a subset of the remaining covariates,
$x_{p+1},\dots,x_q$ (excluding $x_{q+1},\dots,x_d$). These interactions
 take the form $\prod_{j \in S} x_j \cdot \prod_{k \in T}
x_k$ and $\prod_{j \in S} (x_j - t_j)_+ \cdot \prod_{k \in T} x_k$,
where $S \subseteq [p]$ and $T \subseteq [q] \setminus [p]$, with $|S
\cup T| \leq s$ to maintain the interaction order limit. This results
in functions  
\begin{align*}
    \beta_0 &+ \sum_{\substack{S \subseteq [p], T \subseteq [q] \setminus [p] \\ 1 \le |S \cup T| \le s}} \beta_{S, T} \bigg[\prod_{j \in S} x_j \cdot \prod_{k \in T} x_k \bigg] \\ 
    &- \sum_{\substack{S \subseteq [p], T \subseteq [q] \setminus [p] \\ |S| \ge 1, |S \cup T| \le s}} \int_{[0, 1)^{|S|} \setminus \{\zerovec\}} \bigg[\prod_{j \in S} (x_j - t_j)_+ \cdot \prod_{k \in T} x_k \bigg] \, d\nu_{S, T}(t_j, j \in S) + \sum_{j = q + 1}^{d} \beta_j x_j.
\end{align*}
We denote by the collection of these functions $\F^{q, s}_{\text{TC}}(x_1,
\dots, x_p; x_{p + 1}, \dots, x_q) \oplus \mathcal{L}(x_{q + 1},
\dots, x_d)$ and the corresponding LSE
$\hat{f}^{d, q, p, s}_{\text{TC-L-I}}$. Similarly to
\eqref{complexity-measure}, these functions can also be regularized by
constraining the sizes of measures: 
\begin{align*}
    V(f) := \sum_{\substack{S \subseteq [p], T \subseteq [q] \setminus [p] \\ |S| \ge 1, |S \cup T| \le s}} \nu_{S, T} \big([0, 1)^{|S|} \setminus \{\zerovec\}\big). 
\end{align*}
The regularized variant $\hat{f}^{d, q, p, s}_{\text{TC-L-I}, V}$ of
$\hat{f}^{d, q, p, s}_{\text{TC-L-I}}$ is then defined as the
LSE over the class of functions $f \in \F^{q,
  s}_{\text{TC}}(x_1, \dots, x_p; x_{p + 1}, \dots, x_q) \oplus
\mathcal{L}(x_{q + 1}, \dots, x_d)$ that satisfy the additional
constraint $V(f) \le V$.

These function classes are substantially
smaller than $\tcds$ and thus provide stronger regularization, making them
better suited for higher-dimensional regression problems.

\section{Real Data Examples}\label{realdata}
We apply our methods to three real datasets and compare them to existing approaches. 
Implementations of our methods are available in the \textsf{R} package
\textsf{regmdc} (\url{https://github.com/DohyeongKi/regmdc}), which
also includes related methods from \citet{fang2021multivariate} and \citet{ki2024mars}. 
We first normalize each covariate to $[0, 1]$, and then,
after computing our estimators, we invert the transformation and
present results on the original scale.  

In regression, one typically begins with linear models and then adds non-linear terms in each covariate.
Among various non-linear terms, quadratic terms are one of the simplest and common choices.
A function that is quadratic in a variable $x_i$ is effectively convex or concave in $x_i$.
Additive models with convex/concave components thus represent
nonparametric generalizations of quadratic regression. On the other hand, adding simple
product-form interaction terms to a quadratic model lets us move beyond the
additive framework. 

Totally concave regression (and its variants)
provides a natural, nonparametric, shape-constrained alternative that
encompasses these additive models and parametric models with interactions, yet remains
simple and more interpretable 
than black-box approaches (e.g., random forests, deep neural
networks). 
Also, in contrast to black-box approaches, which often require extensive tuning, 
totally concave regression is fully shape-constrained and tuning-free. Our
real data examples show that totally concave regression performs comparably to
parametric models and improves further when regularization is
introduced.


\subsection{Earnings Data}\label{earnings}
Consider the classical regression problem (with $d = 2$) from
labour economics for which $y$ is the log of weekly earnings, $x_1$ is
years of education, and $x_2$ is years of experience, for adults in
the workforce. Labor economics literature suggests that the regression
function should be axially concave (see, e.g., \citet[Section
III]{lemieux2006mincer}). On the other hand, classical regression
models for this problem (e.g., \citet{Mincerbook, murphy1992structure}) 
are additive and thus inevitably miss crucial
interaction effects (see, e.g., \citet[Section
V]{lemieux2006mincer}). Totally concave regression
presents a natural approach here as it is capable of capturing interactions
while staying in the broad framework of axial concavity.   

We use the dataset \textsf{ex1029} from the \textsf{R} library 
\textsf{Sleuth3}, which contains data on full-time male workers in 1987. 
We pre-process the data (following \citet{ulrick2007measuring}) by restricting 
to non-black workers with $x_1 \geq 8$ and $x_2 \in (0, 40)$ and work with
$n = 20,967$ observations. We fit the following models and compare their 
performance: 
\begin{enumerate}
    \item \cite{Mincerbook}: $\E[y|\mathbf{x}] = \beta_0 + \beta_1 x_1 + \beta_2 x_2 + \beta_3 x_2^2$, where $\mathbf{x} = (x_1, x_2)$. 
    \item \cite{murphy1992structure} (MW): $\E[y|\mathbf{x}] = \beta_0 + \beta_1 x_1 + \beta_2 x_2 + \beta_3 x_2^2 + 
    \beta_4 x_2^3 + \beta_5 x_2^4$.
    \item Additive Concave Regression: $\E[y|\mathbf{x}] = f_1(x_1) + f_2(x_2)$ where $f_1$ and $f_2$ are concave. The LSE over all additive concave functions is $\hat{f}_{\text{TC}}^{d, 1}$. 
    \item Axially Concave Regression. We fit the LSE over the class of axially concave functions
      (implementation details are in Appendix \ref{axcon-implementation}).
    \item Totally Concave Regression. We compute the estimator $\hat{f}_{\text{TC}}^{d, s}$ for $d = s = 2$. 
\end{enumerate}

Model performance can vary significantly with sample size, making it
essential to assess prediction accuracy across different training data
sizes. For each $k = 0, \dots, 8$, we randomly select $0.9/2^k$ of the
observations for training and use the remaining $(1 - 0.9/2^k)$ for
testing. For instance, $k = 0$ corresponds to a 90-10 train-test
split, while $k = 4$ uses $0.9/16$ of the observations ($n = 1,179$)
for training and the rest for testing. We measure performance using
mean squared errors on the test set and repeat each experiment 100 times.

Plot A of Figure \ref{fig:earnings-top-and-bottom2-ratios-full}
displays the proportion of repetitions in which each method outperforms 
the other four ($x$-axis represents training data proportion); 
e.g., when the proportion of training data is $0.9/16$, the 
proportions of repetitions in which Mincer, MW, additive, axially, and
totally concave regressions perform the 
best are 0.01, 0.29, 0.06, 0.02, and 0.62. Plot B of Figure 
\ref{fig:earnings-top-and-bottom2-ratios-full} shows the proportion of 
repetitions in which each model ranks among the two
\textit{worst-performing} models. 

\begin{figure}[ht]
  \begin{center}
  \includegraphics{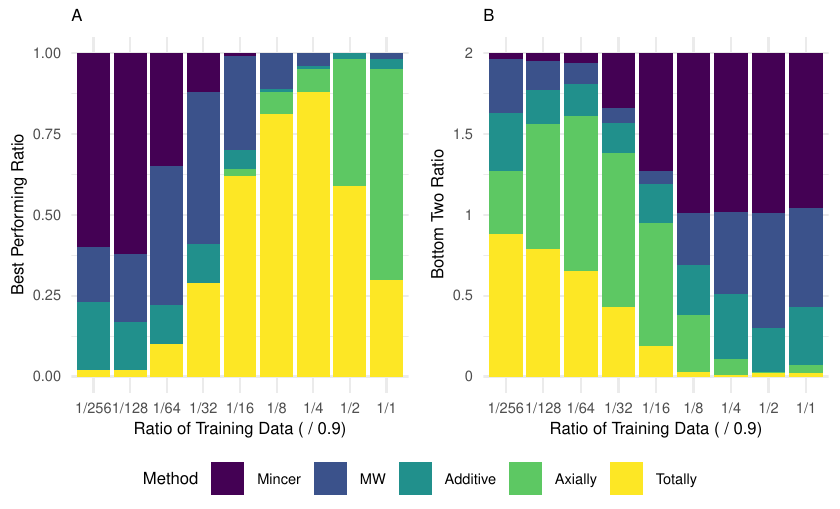}
  \end{center}
  \caption{Best performing and bottom two ratios of each method for various 
  sub-sample sizes. The $x$-axis and $y$-axis represent the proportion of 
  training data and the proportion of repetitions (out of 100) where (A) each 
  model outperforms all the other four models and (B) each model is one of 
  the two worst-performing models.}
  \label{fig:earnings-top-and-bottom2-ratios-full}
\end{figure}

From these plots, it is
clear that restrictive models like Mincer's parametric approach
perform the best with limited data, but totally concave regression
emerges as the preferred model at moderate sample sizes. It maintains
this advantage until sample sizes become very large, where axially
concave regression takes the lead. This pattern highlights totally
concave regression's balanced approach: it captures interactions while
maintaining enough structure to perform well with moderately-sized
datasets. Although the transition points may vary across datasets,
totally concave regression appears to occupy a sweet spot between
overly restrictive parametric models and the more flexible axially
concave regression. 


See Appendix \ref{additional-plots} for
additional plots for the analysis. 

The above comparison only addresses prediction accuracy without computational
considerations. Axially
concave regression becomes computationally prohibitive when $d \geq
3$. Totally concave regression's computational burden can be 
reduced by limiting interaction orders (as in $\tcds$ vs $\tcd$) or
using approximate versions (Sections \ref{excomp} and
\ref{extensions}). It remains unclear how to adapt such strategies
to axially concave regression. Due to this computational limitation,
axially concave regression will not be considered in the following
subsections.

\subsection{Housing Price Data}\label{housing-price}
We use the \textsf{hprice2} dataset from the \textsf{R} package \textsf{wooldridge}, 
which contains housing data from 506 Boston census tracts (1970 US census). 
The dataset includes median housing price (\textsf{price}), per-capita crime rate 
(\textsf{crime}), average number of rooms per dwelling (\textsf{room}), 
nitric oxides level (\textsf{nox}), weighted distance to major employment 
centers (\textsf{distance}), and student-teacher ratio (\textsf{stratio}). 
Here are some simple parametric regressions for $y =
\log(\textsf{price})$ using $x_1 = \textsf{crime}$, $x_2 = \textsf{room}$, 
$x_3 = \log(\textsf{nox})$, $x_4 = \log(\textsf{distance})$, and $x_5
= \textsf{stratio}$
as regressors: 



\begin{enumerate}
    \item Linear:
        $\E[y|\mathbf{x}] = \beta_0 + \beta_1 x_1 + \beta_2 x_2 + \beta_3 x_3 + \beta_4 x_4 + \beta_5 x_5$, where $\mathbf{x} = (x_1, \dots, x_5)$.
    \item Quadratic: $\E[y|\mathbf{x}] = \beta_0 + \beta_1 x_1 + \beta_{11} x_1^2 + \beta_2 x_2 + \beta_{22} x_2^2 + \beta_3 x_3 + \beta_4 x_4 + \beta_5 x_5$.
    \item Interaction 1 (Int 1): $\E[y|\mathbf{x}] = \beta_0 + \beta_1 x_1 + \beta_{11} x_1^2 + \beta_2 x_2 + \beta_{22} x_2^2  + \beta_{12} x_1 x_2 + \beta_3 x_3 + \beta_4 x_4 + \beta_5 x_5$.
    \item Interaction 2 (Int 2): $\E[y|\mathbf{x}] = \beta_0 + \beta_1 x_1 + \beta_{11} x_1^2 + \beta_2 x_2 + \beta_{22} x_2^2  + \beta_{12} x_1 x_2 + \beta_{112} x_1^2 x_2 + \beta_{122} x_1 x_2^2 + \beta_3 x_3 + \beta_4 x_4 + \beta_5 x_5$.
    \item Interaction 3 (Int 3): $\E[y|\mathbf{x}] = \beta_0 + \beta_1 x_1 + \beta_{11} x_1^2 + \beta_2 x_2 + \beta_{22} x_2^2  + \beta_{12} x_1 x_2 + \beta_3 x_3 + \beta_{13} x_1 x_3 + \beta_{113} x_1^2 x_3 + \beta_{23} x_2 x_3 + \beta_{223} x_2^2 x_3 + \beta_4 x_4 + \beta_5 x_5$.
    \item Interaction 4 (Int 4): $\E[y|\mathbf{x}] = \beta_0 + \beta_1
      x_1 + \beta_{11} x_1^2 + \beta_2 x_2 + \beta_{22} x_2^2  +
      \beta_{12} x_1 x_2 + \beta_{112} x_1^2 x_2 + \beta_{122} x_1
      x_2^2 + \beta_3 x_3 + \beta_{13} x_1 x_3 + \beta_{113} x_1^2 x_3
      + \beta_{23} x_2 x_3 + \beta_{223} x_2^2 x_3 + \beta_4 x_4 +
      \beta_5 x_5$.
    \end{enumerate}
We compare the above models to nonparametric approaches based on total
concavity. The quadratic model gives convex (as opposed to concave)
fits, so we work with convex variants of our function classes (these
are given by functions of the form \eqref{intermodds} where the negative sign
on the last term is replaced by the positive
sign). We indicate this by the subscript \textit{convex}.
\begin{enumerate}
    \item Additive ($\hat{f}^{5, 2, 1}_{\text{TC-L}}$):
      $\E[y|\mathbf{x}] \in \F^{2, 1}_{\text{TC, convex}}(x_1, x_2)
      \oplus \mathcal{L}(x_3, x_4, x_5)$.
    \item Our Model 1 ($\hat{f}^{5, 2, 2}_{\text{TC-L}}$) (Ours 1): $\E[y|\mathbf{x}] \in \F^{2, 2}_{\text{TC, convex}}(x_1, x_2) \oplus \mathcal{L}(x_3, x_4, x_5)$.
    \item Our Model 2 ($\hat{f}^{5, 3, 2, 2}_{\text{TC-L-I}}$) (Ours 2): $\E[y|\mathbf{x}] \in \F^{3, 2}_{\text{TC, convex}}(x_1, x_2; x_3) \oplus \mathcal{L}(x_4, x_5)$.
\end{enumerate}
We evaluated these models using 100 random training-test splits (90\% - 10\%) and ranked them by test mean squared errors (with rank 1 being the best, and rank
9 the worst). Figure \ref{fig:hprice-rank-plot-no-complexity} shows
the empirical cumulative distributions of these rankings across the
100 splits. For plots of this kind (including 
Figures \ref{fig:hprice-rank-plot-complexity} and \ref{fig:401k-rank-plot}
appearing later), an ideal procedure is one whose rank
distribution function dominates (is everywhere larger than) all the others.

\begin{figure}[ht]
\begin{center}
\includegraphics{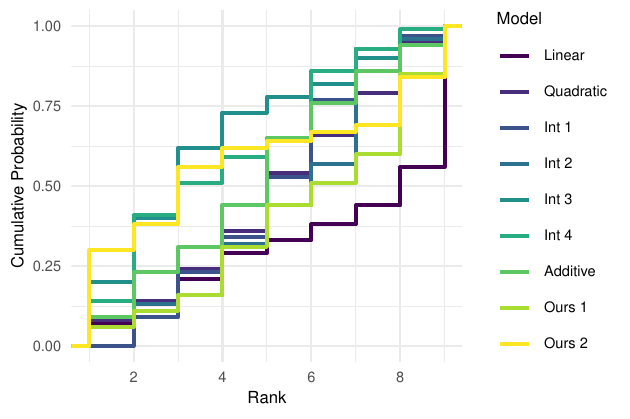}
\end{center}
\caption{Empirical cumulative distributions of model ranks across 100 splits into training and test sets. 
Ranks are determined by mean squared errors on the test sets.}
\label{fig:hprice-rank-plot-no-complexity}
\end{figure}

The model Ours 2 showed both promise and limitations—ranking first in
30\% of splits but in the bottom two for another 
31\%. By contrast, Int 3, Int 4, and Additive demonstrated better
consistency, rarely performing poorly though achieving fewer top
rankings as well. The remaining models showed no notable advantages. 


Figure \ref{fig:hprice-overfitting} illustrates why Ours 2 performed 
poorly sometimes. Comparing fitted functions from Ours 2 and Int 3 for
a split where Ours 2 ranked worst reveals severe overfitting at one corner of the covariate domain. The dramatically different scales of
$\log(\textsf{price})$ between the two panels highlight this problem. 


\begin{figure}[ht]
\begin{center}
\includegraphics{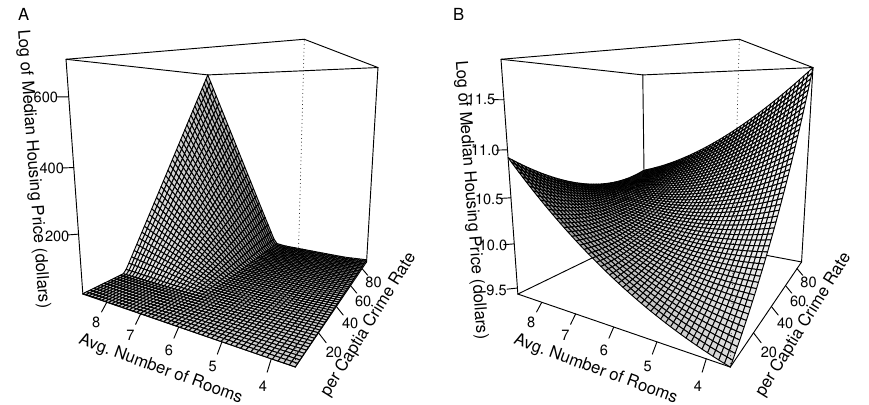}
\end{center}
\caption{Fitted regression functions of (A) Ours 2 and (B) Int 3 for a random split where Ours 2 performs the worst. 
The regression functions are plotted against \textsf{crime} and \textsf{room}, with $\log(\textsf{nox})$, $\log(\textsf{distance})$, and \textsf{stratio} fixed at their median.}
\label{fig:hprice-overfitting}
\end{figure}

To address overfitting, we introduced regularized variants of Ours 2
($\hat{f}^{5, 3, 2, 2}_{\text{TC-L-I}, V}$) and Additive
($\hat{f}^{5, 2, 1}_{\text{TC-L}, V}$) (the regularization parameter
$V$ is tuned by 10-fold cross validation, separately for each random
split). We excluded consistently underperforming models (Int 1, Int 2,
and Ours 1) for clarity and redid the comparison with the regularized variants.


\begin{figure}[ht]
\begin{center}
\includegraphics{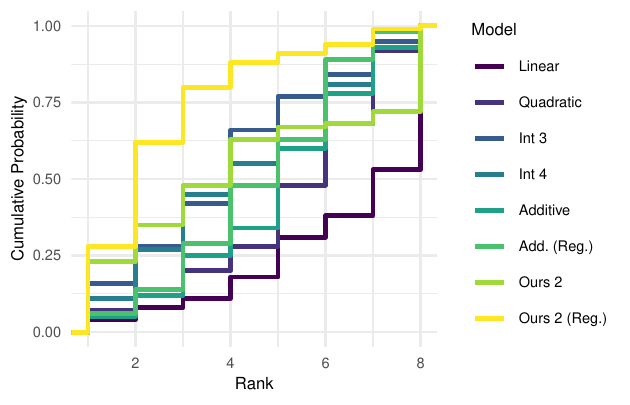}
\end{center}
\caption{Empirical cumulative model rank distributions across 100
  training-test splits. 
Add. (Reg.) and Ours 2 (Reg.) are the regularized variants \eqref{modified-estimator} of Additive and Ours 2.}
\label{fig:hprice-rank-plot-complexity}
\end{figure}

Figure \ref{fig:hprice-rank-plot-complexity} shows that the
regularized variant of  
Ours 2 clearly outperformed all other models. Its cumulative
distribution dominates all competitors, achieving top-two rankings in
62\% of splits. The regularization also effectively addressed the
overfitting issue—while the original Ours 2 ranked worst in 28\% of
splits, its regularized version appeared in the bottom four only 12\%
of the time. This analysis demonstrates
that variants of totally concave regression can outperform standard
regression approaches in practical applications.

\subsection{Retirement Saving Plan (401(k)) Data}\label{401k}

We analyzed the 401(k) dataset \textsf{k401ksubs} from the
R package \textsf{wooldridge}, containing $n = 9,275$ observations. Our
target variable was net total financial assets ($y$), and 
predictors were annual income ($x_1$), age ($x_2$),
401(k) eligibility ($x_3$; 1 if eligible, 0 otherwise), and
family size. We converted the family size variable (\textsf{fsize}) into five dummy
variables: \textsf{fsize1} ($x_4$) through \textsf{fsize4} ($x_8$) (each equal to 1 if
family size equals the corresponding number, 0 otherwise) and \textsf{fsize5}
(equal to 1 if family size $\geq$ 5). We compare the following models:

\begin{enumerate}
    \item Quadratic: 
    $\E[y|\mathbf{x}] = \beta_1 x_1 + \beta_{11} x_1^2 + \beta_2 x_2 + \beta_{22} x_2^2 + \beta_3 x_3 + \beta_4 x_4 + \cdots + \beta_8 x_8$.
    \item Int 1: 
    $\E[y|\mathbf{x}] = \beta_1 x_1 + \beta_{11} x_1^2 + \beta_2 x_2 + \beta_{22} x_2^2 + \beta_{12} x_1 x_2 + \beta_3 x_3 + \beta_4 x_4 + \cdots + \beta_8 x_8$.
    \item Int 2: 
    $\E[y|\mathbf{x}] = \beta_0 + \beta_1 x_1 + \beta_{11} x_1^2 + \beta_2 x_2 + \beta_{22} x_2^2  + \beta_{12} x_1 x_2 + \beta_{112} x_1^2 x_2 + \beta_{122} x_1 x_2^2 + \beta_3 x_3 + \beta_4 x_4 + \cdots + \beta_8 x_8$.
    \item Int 3: 
    $\E[y|\mathbf{x}] = \beta_0 + \beta_1 x_1 + \beta_{11} x_1^2 + \beta_2 x_2 + \beta_{22} x_2^2  + \beta_{12} x_1 x_2 + \beta_3 x_3 + \beta_{13} x_1 x_3 + \beta_{113} x_1^2 x_3 + \beta_{23} x_2 x_3 + \beta_{223} x_2^2 x_3 + \beta_4 x_4 + \cdots + \beta_8 x_8$.
    \item Int 4: 
    $\E[y|\mathbf{x}] = \beta_0 + \beta_1 x_1 + \beta_{11} x_1^2 + \beta_2 x_2 + \beta_{22} x_2^2  + \beta_{12} x_1 x_2 + \beta_{112} x_1^2 x_2 + \beta_{122} x_1 x_2^2 + \beta_3 x_3 + \beta_{13} x_1 x_3 + \beta_{113} x_1^2 x_3 + \beta_{23} x_2 x_3 + \beta_{223} x_2^2 x_3 + \beta_4 x_4 + \cdots + \beta_8 x_8$.
    \item Additive ($\hat{f}^{8, 2, 1}_{\text{TC-L}}$): 
    $\E[y|\mathbf{x}] \in \F^{2, 1}_{\text{TC, convex}}(x_1, x_2) \oplus \mathcal{L}(x_3, \dots, x_8)$.
    \item Ours 1 ($\hat{f}^{8, 2, 2}_{\text{TC-L}}$): 
    $\E[y|\mathbf{x}] \in \F^{2, 2}_{\text{TC, convex}}(x_1, x_2) \oplus \mathcal{L}(x_3, \dots, x_8)$.
    \item Ours 2 ($\hat{f}^{8, 3, 2, 2}_{\text{TC-L-I}}$): 
    $\E[y|\mathbf{x}] \in \F^{3, 2}_{\text{TC, convex}}(x_1, x_2; x_3)  \oplus \mathcal{L}(x_4, \dots, x_8)$.
    \item Regularized Variant of Ours 2 ($\hat{f}^{8, 3, 2, 2}_{\text{TC-L-I}, V}$) (Ours 2 (Reg.)).
\end{enumerate}

We focus on total convexity (as opposed to total
concavity) as convexity provides better fits (as can be easily
verified by inspecting the fits of the simple quadratic model). To
manage computational demands, we used approximate versions 
for Additive, Ours 1, Ours 2, and Ours 2 (Reg.) with proxy
lattices (replacing $L_S$ in
\eqref{lattices-gen-by-data}) with $N_1 = N_2 = 50$. Since total convexity is only imposed
on $x_1$ and $x_2$, there is no need for $N_3, \dots, N_8$.

\begin{figure}[ht]
\begin{center}
\includegraphics{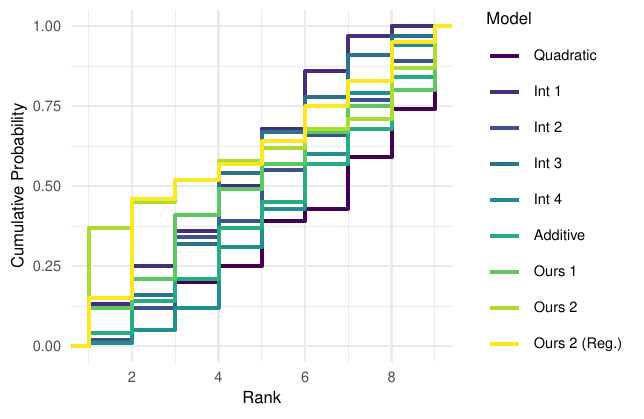}
\end{center}
\caption{Empirical cumulative model rank distributions for 401(k) data}
\label{fig:401k-rank-plot}
\end{figure}

We evaluated performance across 100 random training-test splits (90\%
- 10\%) and ranked models by test mean squared errors. Figure
\ref{fig:401k-rank-plot} 
shows the empirical rank distributions. The results (from Figure
\ref{fig:401k-rank-plot}) align with our findings on housing price data. Ours 2 showed 
high potential but inconsistency—ranking first in 37 splits but in
the bottom two for 29 others. Int 1 and Int 3 demonstrated greater
consistency but fewer top placements. The regularized variant of Ours 2
effectively balanced these extremes, ranking worst in only 5 splits
while maintaining strong performance. The regularization effect, however, was
less dramatic than in the housing price 
example, likely because the approximation with proxy lattices already
provided some overfitting protection by constraining boundary
weights. These results confirm totally concave regression's practical 
utility and demonstrate that approximate versions offer viable
alternatives when computational demands are high.

\section*{Acknowledgments}
The authors gratefully acknowledge support from the NSF grant 
DMS-2210504.



\bibliographystyle{chicago}
\bibliography{main}

@preamble{"\def\noopsort#1{} "}

@string{annstat = {Ann. Statist.}}

@string{appleconlett = {Appl. Econ. Lett.}}

@string{economj = {Econom. J.}}

@string{ejs = {Electron. J. Statist.}}

@string{informscomp = {Informs J. Comput.}}

@string{jasa = {J. Am. Statist. Ass.}}

@string{jima = {Inf. Inference J. IMA}}

@string{jnonp = {J. Nonparam. Stat.}}

@string{jrssb = {J. R. Statist. Soc. B}}

@string{jstplinf = {J. Statist. Planng Inf.}}

@string{mmstat = {Math. Methods Stat.}}

@string{opres = {Ops Res.}}

@string{probtheorel = {Probab. Theory Relat. Fields}}

@string{statcomp = {Statist. Comput.}}

@string{statsci = {Statist. Sci.}}

@article{aistleitner2015functions,
  title={{Functions of bounded variation, signed measures, and a general Koksma--Hlawka inequality}},
  author={Aistleitner, Christoph and Dick, Josef},
  journal={Acta Arithmetica},
  volume={167},
  number={2},
  pages={143--171},
  year={2015},
  publisher={Institute of Mathematics Polish Academy of Sciences}
}

@article{amelunxen2014living,
  title={Living on the edge: Phase transitions in convex programs with random data},
  author={Amelunxen, Dennis and Lotz, Martin and McCoy, Michael B and Tropp, Joel A},
  journal=jima,
  volume={3},
  number={3},
  pages={224--294},
  year={2014},
  publisher={OUP}
}

@phdthesis{balazs2016convex,
  title={Convex Regression: Theory, Practice, and Applications},
  author={Bal{\'a}zs, G{\'a}bor},
  year={2016},
  school={University of Alberta}
}

@InProceedings{balazs2015near,
  title = 	 {{Near-optimal max-affine estimators for convex regression}},
  author = 	 {Balazs, Gabor and György, András and Szepesvari, Csaba},
  booktitle = 	 {Proceedings of the Eighteenth International Conference on Artificial Intelligence and Statistics},
  pages = 	 {56--64},
  year = 	 {2015}
}

@article{bellec2018sharp,
  title={Sharp oracle inequalities for least squares estimators in shape restricted regression},
  author={Bellec, Pierre C},
  journal=annstat,
  volume={46},
  number={2},
  pages={745--780},
  year={2018},
  publisher={JSTOR}
}

@INPROCEEDINGS{benkeser2016highly,
  author={Benkeser, David and van der Laan, Mark},
  booktitle={IEEE International Conference on Data Science and Advanced Analytics (DSAA)}, 
  title={The Highly Adaptive Lasso Estimator}, 
  year={2016},
  volume={},
  number={},
  pages={689-696}
}

@article{bertsimas2021sparse,
  title={Sparse convex regression},
  author={Bertsimas, Dimitris and Mundru, Nishanth},
  journal=informscomp,
  volume={33},
  number={1},
  pages={262--279},
  year={2021},
  publisher={INFORMS}
}

@article{bredies2013inverse,
  title={Inverse problems in spaces of measures},
  author={Bredies, Kristian and Pikkarainen, Hanna Katriina},
  journal={ESAIM: Control, Optimisation and Calculus of Variations},
  volume={19},
  number={1},
  pages={190--218},
  year={2013},
  publisher={EDP Sciences}
}

@article{bungartz2004sparse,
  title={Sparse grids},
  author={Bungartz, Hans-Joachim and Griebel, Michael},
  journal={Acta Numerica},
  volume={13},
  pages={147--269},
  year={2004},
  publisher={Cambridge University Press}
}

@article{chatterjee2014new,
  title={A new perspective on least squares under convex constraint},
  author={Chatterjee, Sourav},
  journal=annstat,
  volume={42},
  number={6},
  pages={2340--2381},
  year={2014},
  publisher={Institute of Mathematical Statistics},
}

@article{chatterjee2016improved,
  title={An improved global risk bound in concave regression},
  author={Chatterjee, Sabyasachi},
  journal=ejs,
  volume={10},
  number={1},
  pages={1608--1629},
  year={2016},
  publisher={The Institute of Mathematical Statistics and the Bernoulli Society}
}

@article{chatterjee2015risk,
  title={On risk bounds in isotonic and other shape restricted regression problems},
  author={Chatterjee, Sabyasachi and Guntuboyina, Adityanand and Sen, Bodhisattva},
  journal=annstat,
  volume={43},
  number={4},
  pages={1774 -- 1800},
  year={2015}
}

@article{chatterjee2018matrix,
  title={On matrix estimation under monotonicity constraints},
  author={Sabyasachi Chatterjee and Adityanand Guntuboyina and Bodhisattva Sen},
  journal={Bernoulli},
  volume={24},
  number={2},
  pages={1072--1100},
  year={2018},
  publisher={Bernoulli Society for Mathematical Statistics and Probability}
}

@article{chen2016generalized,
  title={Generalized additive and index models with shape constraints},
  author={Chen, Yining and Samworth, Richard J},
  journal=jrssb,
  volume={78},
  number={4},
  pages={729--754},
  year={2016},
  publisher={Oxford University Press}
}

@article{donoho2000high,
  title={High-dimensional data analysis: The curses and blessings of dimensionality},
  author={Donoho, David L},
  journal={AMS Math Challenges Lecture},
  volume={1},
  number={2000},
  pages={32},
  year={2000}
}

@article {DuembgenEtAl04,
    AUTHOR = {D{\"u}mbgen, L. and Freitag, S. and Jongbloed, G.},
    TITLE = {Consistency of concave regression with an application to current-status data},
   JOURNAL = mmstat,
    VOLUME = {13},
      YEAR = {2004},
    NUMBER = {1},
     PAGES = {69--81}
}

@book{dung2018hyperbolic,
  title={Hyperbolic Cross Approximation},
  author={D{\~u}ng, Dinh and Temlyakov, Vladimir and Ullrich, Tino},
  year={2018},
  publisher={Birkh\"{a}user},
  address={Cham},
  series={Advanced Courses in Mathematics - CRM Barcelona}
}

@article{fang2021multivariate,
  title={{Multivariate extensions of isotonic regression and total variation denoising via entire monotonicity and Hardy--Krause variation}},
  author={Fang, Billy and Guntuboyina, Adityanand and Sen, Bodhisattva},
  journal=annstat,
  volume={49},
  number={2},
  pages={769--792},
  year={2021},
  publisher={Institute of Mathematical Statistics}
}

@article{friedman1991multivariate,
  title={Multivariate adaptive regression splines},
  author={Friedman, Jerome H},
  journal=annstat,
  volume={19},
  number={1},
  pages={1--67},
  year={1991},
  publisher={Institute of Mathematical Statistics}
}

@book{gal2010shape,
  title={Shape-Preserving Approximation by Real and Complex Polynomials},
  author={Gal, Sorin G},
  year={2010},
  publisher={Birkhäuser},
  address={Boston, MA}
}

@book{loukas2014classical,
  title={Classical Fourier Analysis},
  edition = {Third},
  series = {Graduate Texts in Mathematics},
  author={Loukas Grafakos},
  year={2014},
  publisher={Springer},
  address={New York, NY}
}

@book{groeneboom2014nonparametric,
  title={Nonparametric Estimation under Shape Constraints: Estimators, Algorithms and Asymptotics},
  author={Groeneboom, Piet and Jongbloed, Geurt},
  volume={38},
  year={2014},
  publisher={Cambridge University Press},
  series={Cambridge Series in Statistical and Probabilistic Mathematics},
  address={Cambridge}
}

@article{guntuboyina2020adaptive,
  title={Adaptive risk bounds in univariate total variation denoising and trend filtering},
  author={Guntuboyina, Adityanand and Lieu, Donovan and Chatterjee, Sabyasachi and Sen, Bodhisattva},
  journal=annstat,
  volume={48},
  number={1},
  pages={205--229},
  year={2020},
  publisher={Institute of Mathematical Statistics},
}

@article{guntuboyina2015global,
  title={Global risk bounds and adaptation in univariate convex regression},
  author={Guntuboyina, Adityanand and Sen, Bodhisattva},
  journal=probtheorel,
  volume={163},
  number={1},
  pages={379--411},
  year={2015},
  publisher={Springer}
}

@article{guntuboyina2018nonparametric,
  title={Nonparametric shape-restricted regression},
  author={Guntuboyina, Adityanand and Sen, Bodhisattva},
  journal=statsci,
  volume={33},
  number={4},
  pages={568--594},
  year={2018},
  publisher={Institute of Mathematical Statistics}
}

@article{han2019isotonic,
  title={Isotonic regression in general dimensions},
  author={Han, Qiyang and Wang, Tengyao and Chatterjee, Sabyasachi and Samworth, Richard J},
  journal=annstat,
  volume={47},
  number={5},
  pages={2440--2471},
  year={2019},
  publisher={Institute of Mathematical Statistics}
}

@article{han2016multivariate,
  title={Multivariate convex regression: global risk bounds and adaptation},
  author={Han, Qiyang and Wellner, Jon A},
  journal={arXiv preprint arXiv:1601.06844},
  year={2016},
}

@article{han2019convergence,
  title={Convergence rates of least squares regression estimators with heavy-tailed errors},
  author={Han, Qiyang and Wellner, Jon A},
  journal=annstat,
  volume={47},
  number={4},
  pages={2286--2319},
  year={2019},
  publisher={Institute of Mathematical Statistics}
}

@article{HanPled76,
    AUTHOR = {Hanson, D. L. and Pledger, Gordon},
     TITLE = {Consistency in concave regression},
   JOURNAL = annstat,
    VOLUME = {4},
      YEAR = {1976},
    NUMBER = {6},
     PAGES = {1038--1050}
}

@article{Hildreth54,
  title={Point estimates of ordinates of concave functions},
  author={Hildreth, Clifford},
  journal=jasa,
  volume={49},
  number={267},
  pages={598--619},
  year={1954},
  publisher={Taylor \& Francis}
}

@book{isaacson1994analysis,
  title={Analysis of Numerical Methods},
  author={Isaacson, Eugene and Keller, Herbert Bishop},
  year={1994},
  edition={Revised},
  publisher={Dover}
}

@article{iwanaga2016estimating,
  title={Estimating product-choice probabilities from recency and frequency of page views},
  author={Iwanaga, Jiro and Nishimura, Naoki and Sukegawa, Noriyoshi and Takano, Yuichi},
  journal={Knowledge-Based Systems},
  volume={99},
  pages={157--167},
  year={2016},
  publisher={Elsevier}
}

@article{ki2024mars,
  title={{MARS} via {LASSO}},
  author={Ki, Dohyeong and Fang, Billy and Guntuboyina, Adityanand},
  journal=annstat,
  volume={52},
  number={3},
  pages={1102--1126},
  year={2024},
  publisher={Institute of Mathematical Statistics}
}

@article{ki2021mars,
      title={{MARS via LASSO}}, 
      author={Dohyeong Ki and Billy Fang and Adityanand Guntuboyina},
      year={2021},
      journal={arXiv preprint arXiv:2111.11694v1},
      url={https://arxiv.org/pdf/2111.11694v1}, 
}

@article{kulikov2006behavior,
  title={{The behavior of the NPMLE of a decreasing density near the boundaries of the support}},
  author={Kulikov, Vladimir N and Lopuha{\"a}, Hendrik P},
  journal=annstat,
  volume={34},
  number={2},
  pages={742--768},
  year={2006}
}

@article{kuosmanen2008representation,
  title={Representation theorem for convex nonparametric least squares},
  author={Kuosmanen, Timo},
  journal=economj,
  volume={11},
  number={2},
  pages={308--325},
  year={2008},
  publisher={Oxford University Press Oxford, UK}
}

@article{kur2019optimality,
  title={Optimality of maximum likelihood for log-concave density estimation and bounded convex regression},
  author={Kur, Gil and Dagan, Yuval and Rakhlin, Alexander},
  journal={arXiv preprint arXiv:1903.05315},
  year={2019}
}

@article{kur2024convex,
  title={Convex regression in multidimensions: Suboptimality of least squares estimators},
  author={Kur, Gil and Gao, Fuchang and Guntuboyina, Adityanand and Sen, Bodhisattva},
  journal=annstat,
  volume={52},
  number={6},
  pages={2791--2815},
  year={2024},
  publisher={Institute of Mathematical Statistics}
}

@incollection{lemieux2006mincer,
  title={The “{M}incer equation” thirty years after schooling, experience, and earnings},
  author={Lemieux, Thomas},
  booktitle={Jacob Mincer a Pioneer of Modern Labor Economics},
  pages={127--145},
  year={2006},
  publisher={Springer},
  address={Boston, MA}
}

@article{liao2024overfitting,
  title={Overfitting Reduction in Convex Regression},
  author={Liao, Zhiqiang and Dai, Sheng and Lim, Eunji and Kuosmanen, Timo},
  journal={arXiv preprint arXiv:2404.09528},
  year={2024}
}

@article{lim2014convergence,
  title={On convergence rates of convex regression in multiple dimensions},
  author={Lim, Eunji},
  journal=informscomp,
  volume={26},
  number={3},
  pages={616--628},
  year={2014},
  publisher={INFORMS}
}

@article{lim2012consistency,
  title={Consistency of multidimensional convex regression},
  author={Lim, Eunji and Glynn, Peter W},
  journal=opres,
  volume={60},
  number={1},
  pages={196--208},
  year={2012},
  publisher={INFORMS},
}

@article{lin2000tensor,
  title={{Tensor product space ANOVA models}},
  author={Lin, Yi},
  journal=annstat,
  volume={28},
  number={3},
  pages={734--755},
  year={2000},
  publisher={Institute of Mathematical Statistics}
}

@book{lojasiewicz1988introduction,
  title={An Introduction to the Theory of Real Functions},
  author={{\L}ojasiewicz, Stanis{\l}aw},
  year={1988},
  publisher={Wiley},
  address={Chichester}
}

@article{mazumder2019computational,
  title={A computational framework for multivariate convex regression and its variants},
  author={Mazumder, Rahul and Choudhury, Arkopal and Iyengar, Garud and Sen, Bodhisattva},
  journal=jasa,
  volume={114},
  number={525},
  pages={318--331},
  year={2019},
  publisher={Taylor \& Francis}
}

@article{meinshausen2013sign,
  title={Sign-constrained least squares estimation for high-dimensional regression},
  author={Meinshausen, Nicolai},
  journal=ejs,
  volume={7},
  pages={1607--1631},
  year={2013}
}

@article{meyer2013semi,
  title={Semi-parametric additive constrained regression},
  author={Meyer, Mary C},
  journal=jnonp,
  volume={25},
  number={3},
  pages={715--730},
  year={2013},
  publisher={Taylor \& Francis},
}

@article{meyer2018framework,
  title={A framework for estimation and inference in generalized additive models with shape and order restrictions},
  author={Meyer, Mary C},
  journal=statsci,
  volume={33},
  number={4},
  pages={595--614},
  year={2018},
  publisher={JSTOR}
}

@book{Mincerbook,
	Address = {New York, NY},
	Author = {Mincer, Jacob},
	Booktitle = {Schooling, Experience and Earnings},
	Publisher = {Columbia University Press},
	Title = {Schooling, Experience and Earnings},
	Year = {1974}
}

@article{murphy1992structure,
  title={The structure of wages},
  author={Murphy, Kevin M and Welch, Finis},
  journal={Quarterly Journal of Economics},
  volume={107},
  number={1},
  pages={285--326},
  year={1992},
  publisher={MIT Press}
}

@book{niculescu2006convex,
  title={Convex Functions and Their Applications},
  author={Niculescu, Constantin and Persson, Lars-Erik},
  year={2006},
  publisher={Springer},
  edition={First},
  series={CMS Books in Mathematics},
  address={New York, NY}
}

@phdthesis{popoviciu1933quelques,
  title={Sur quelques propri{\'e}t{\'e}s des fonctions d'une ou de deux variables r{\'e}elles},
  author={Popoviciu, Tiberiu},
  year={1933},
  school={Institutul de Arte Grafice" Ardealul}
}

@article{pya2015shape,
  title={Shape constrained additive models},
  author={Pya, Natalya and Wood, Simon N},
  journal=statcomp,
  volume={25},
  pages={543--559},
  year={2015},
  publisher={Springer}
}

@book{RWD88,
    AUTHOR = {Robertson, Tim and Wright, F. T. and Dykstra, R. L.},
    TITLE = {Order Restricted Statistical Inference},
    SERIES = {Wiley Series in Probability and Mathematical Statistics},
    PUBLISHER = {Wiley},
    ADDRESS = {Chichester},
    YEAR = {1988},
    PAGES = {xx+521},
}

@inproceedings{rosset2007L1,
  title={$\ell_1$ regularization in infinite dimensional feature spaces},
  author={Rosset, Saharon and Swirszcz, Grzegorz and Srebro, Nathan and Zhu, Ji},
  booktitle={Learning Theory. COLT 2007},
  pages={544--558},
  year={2007},
  series={Lecture Notes in Computer Science},
  volume={4539},
  publisher={Springer Berlin, Heidelberg}
}

@article{samworth2018recent,
  title={Recent Progress in Log-Concave Density Estimation},
  author={Samworth, Richard J},
  journal=statsci,
  volume={33},
  number={4},
  pages={493--509},
  year={2018}
}

@article{SS11,
	Author = {Seijo, E. and Sen, B.},
	Journal = annstat,
	Pages = {1633-1657},
	Title = {Nonparametric least squares estimation of a multivariate convex regression function},
	Volume = {39},
    number = {3},
	Year = {2011}
}

@article{slawski2013non,
  title={Non-negative least squares for high-dimensional linear models: Consistency and sparse recovery without regularization},
  author={Slawski, Martin and Hein, Matthias},
  journal=ejs,
  volume={7},
  pages={3004},
  year={2013},
  publisher={Institute of Mathematical Statistics}
}

@article{sun1996adaptive,
  title={{Adaptive smoothing for a penalized NPMLE of a non-increasing density}},
  author={Sun, Jiayang and Woodroofe, Michael},
  journal=jstplinf,
  volume={52},
  number={2},
  pages={143--159},
  year={1996},
  publisher={Elsevier}
}

@book{temlyakov2018multivariate,
  address={Cambridge},    
  title={Multivariate Approximation},
  author={Temlyakov, Vladimir},
  volume={32},
  year={2018},
  series={Cambridge Monographs on Applied and Computational Mathematics},
  publisher={Cambridge University Press}
}

@article{ulrick2007measuring,
  title={Measuring the returns to education nonparametrically},
  author={Ulrick, Shawn W},
  journal=appleconlett,
  volume={14},
  number={13},
  pages={1005--1011},
  year={2007},
  publisher={Taylor \& Francis}
}

@article{van2023higher,
  title={Higher order spline highly adaptive lasso estimators of functional parameters: Pointwise asymptotic normality and uniform convergence rates},
  author={van der Laan, Mark},
  journal={arXiv preprint arXiv:2301.13354},
  year={2023}
}

@article{woodroofe1993penalized,
  title={A penalized maximum likelihood estimate of $f(0+)$ when $f$ is non-increasing},
  author={Woodroofe, Michael and Sun, Jiayang},
  journal={Statistica Sinica},
  volume={3},
  number={2},
  pages={501--515},
  year={1993},
  publisher={JSTOR},
}

\newpage
\appendix

\section{More on Axially Concave Regression}\label{axcon}

\subsection{On Axial Concavity}\label{axcon_char}
In Section \ref{comparison-concavity}, we claimed that $f$ is axially
concave if and only if \eqref{divdiff} holds for
$(p_1, \dots, p_d) = (2, 0, \dots, 0), \dots, (0, \dots, 0, 2)$. Here, we
provide a detailed explanation. 

For $(p_1, \dots, p_d) = (2, 0, \dots, 0)$, the condition \eqref{divdiff} 
becomes
\begin{equation*}
  \frac{f(x_1^{(1)}, x^{(2)}, \dots, x^{(d)})}{(x^{(1)}_1 -
    x_2^{(1)})(x^{(1)}_1 - x_3^{(1)})}  +   \frac{f(x_2^{(1)},
    x^{(2)}, \dots, x^{(d)})}{(x^{(1)}_2 - 
    x_1^{(1)})(x^{(1)}_2 - x_3^{(1)})}  
+   \frac{f(x_3^{(1)}, x^{(2)}, \dots, x^{(d)})}{(x^{(1)}_3 -
    x_1^{(1)})(x^{(1)}_3 - x_2^{(1)})}  \leq 0
\end{equation*}
for every $0 \leq x_1^{(1)} < x_2^{(1)} < x_3^{(1)} \leq 1$ and
$x^{(2)}, \dots, x^{(d)} \in [0, 1]$, which can also be written as
\begin{equation*}
  \frac{f(x_2^{(1)}, x^{(2)}, \dots, x^{(d)}) - f(x_1^{(1)}, x^{(2)},
    \dots, x^{(d)})}{x_2^{(1)} - x_1^{(1)}} \geq 
  \frac{f(x_3^{(1)}, x^{(2)}, \dots, x^{(d)}) - f(x_2^{(1)}, x^{(2)},
    \dots, x^{(d)})}{x_3^{(1)} - x_2^{(1)}}. 
\end{equation*}
Thus, the condition \eqref{divdiff} for $(p_1, \dots, p_d) = (2,
0, \dots, 0)$ is equivalent to $f(\cdot, x^{(2)}, \dots, x^{(d)})$ 
being concave for every fixed $x^{(2)}, \dots, x^{(d)} \in [0, 1]$. 

By the same argument, for any $(p_1, \dots, p_d)$ with $p_j = 2$ and 
$p_k = 0$ for $k \neq j$, the condition \eqref{divdiff} is equivalent to 
$f$ being concave in the $j^{\text{th}}$ coordinate while holding all other 
coordinates fixed. Therefore, $f$ is axially concave if and only if 
\eqref{divdiff} holds for 
$(p_1, \dots, p_d) = (2, 0, \dots, 0), \dots, (0, \dots, 0, 2)$.

\subsection{Implementation Details}\label{axcon-implementation}
Let $\hat{f}_{\text{AC}}^{d}$ denote the least squares estimator over the class of axially concave functions on $[0, 1]^d$:
\begin{equation}\label{axconreg}
  \hat{f}_{\text{AC}}^{d} \in \argmin_f \bigg\{ \sum_{i=1}^n
  \big(y_i - f({\mathbf{x}}^{(i)}) \big)^2: f \text{ is axially concave}\bigg\}.
\end{equation}
In this section, we provide details for the computation of $\hat{f}_{\text{AC}}^{d}$ that we deployed in Section \ref{realdata}. 

For each $k = 1, \dots, d$, write 
\begin{equation*}
    \big\{0, x_k^{(1)}, \dots, x_k^{(n)}, 1\big\} = \big\{u_0^{(k)}, \dots, u_{n_k}^{(k)}\big\}
\end{equation*}
where $0 = u_0^{(k)} < \cdots < u_{n_k}^{(k)} = 1$.
Then, for each $i = 1, \dots, n$, there exists $\mathbf{b}(i) = (b_1(i), \dots, b_d(i)) \in \widebar{I}_0 := \prod_{k = 1}^{d} \{0, \dots, n_k\}$ such that $x_k^{(i)} = u_{b_k(i)}^{(k)}$ for every $k = 1, \dots, d$. 
Also, for each real-valued function $f$ on $[0, 1]^d$, let $\theta_f \in \R^{|\widebar{I}_0|}$ denote the vector of evaluations of $f$ at $(u_{j_1}^{(1)}, \dots, u_{j_d}^{(d)})$ for $(j_1, \dots, j_d) \in \widebar{I}_0$; that is,
\begin{equation*}
    (\theta_f)_{j_1, \dots, j_d} = f\big(u_{j_1}^{(1)}, \dots, u_{j_d}^{(d)}\big)
\end{equation*}
for $(j_1, \dots, j_d) \in \widebar{I}_0$.
Note that $f({\mathbf{x}}^{(i)}) = (\theta_f)_{\mathbf{b}(i)}$ for every $i = 1, \dots, d$.

Suppose $f$ is an axially concave function on $[0, 1]^d$. 
It is clear from the axial concavity of $f$ that 
\begin{equation}\label{axcon-fin-cond}
\begin{split}
    &\frac{1}{u^{(l)}_{j_l + 1} - u^{(l)}_{j_l}} \cdot \big((\theta_f)_{j_1, \dots, j_{l - 1}, j_l + 1, j_{l + 1}, \dots, j_d} - (\theta_f)_{j_1, \dots, j_{l - 1}, j_l, j_{l + 1}, \dots, j_d}\big) \\
    &\qquad \ge \frac{1}{u^{(l)}_{j_l + 2} - u^{(l)}_{j_l + 1}} \cdot \big((\theta_f)_{j_1, \dots, j_{l - 1}, j_l + 2, j_{l + 1}, \dots, j_d} - (\theta_f)_{j_1, \dots, j_{l - 1}, j_l + 1, j_{l + 1}, \dots, j_d}\big)
\end{split}    
\end{equation}
for all $0 \le j_k \le n_k$ ($k \neq l$) and $0 \le j_l \le n_l - 2$.
Conversely, if $\theta \in \R^{|\widebar{I}_0|}$ satisfies \eqref{axcon-fin-cond},
then there exists an axially concave function $f$ on $[0, 1]^d$ such that $\theta_f = \theta$.
This can be readily proved as follows.

Let $f$ denote the function on $[0, 1]^d$ defined by 
\begin{equation}\label{axcon-interpolation}
    f(x_1, \dots, x_d) = \frac{1}{\prod_{k = 1}^{d} (u_{j_k + 1}^{(k)} - u_{j_k}^{(k)})} \cdot 
    \sum_{\delta \in \{0, 1\}^d} \bigg[\prod_{k = 1}^{d} \Big\{(1 - \delta_k) \big(u_{j_k + 1}^{(k)} - x_k\big) + \delta_k \big(x_k - u_{j_k}^{(k)}\big)\Big\}\bigg] \cdot \theta_{j_1 + \delta_1, \dots, j_d + \delta_d}
\end{equation}
for $(x_1, \dots, x_d) \in \prod_{k = 1}^{d} [u_{j_k}^{(k)}, u_{j_k + 1}^{(k)}]$ for each $(j_1, \dots, j_d) \in \prod_{k = 1}^{d} \{0, 1, \dots, n_k - 1\}$.
Observe that $\theta_f = \theta$.
Also, observe that, for every $l = 1, \dots, d$ and for every fixed $(x_1, \dots, x_{l - 1}, x_{l + 1}, \dots, x_d) \in [0, 1]^{d - 1}$, the function 
$g_l(t) := f(x_1, \dots, x_{l - 1}, t, x_{l + 1}, \dots, x_d)$ is continuous and piecewise affine in $t$ over $[0, 1]$.
In fact, the function $f$ is a continuous rectangular piecewise multi-affine function (introduced in Section \ref{rates}), interpolating the points $((u_{j_1}^{(1)}, \dots, u_{j_d}^{(d)}), \theta_{j_1, \dots, j_d})$ for $(j_1, \dots, j_d) \in \widebar{I}_0$. 

If $(x_1, \dots, x_l, x_{l + 1}, \dots, x_d) \in \prod_{k \neq l}[u_{j_k}^{(k)}, u_{j_k + 1}^{(k)}]$, 
then the slope of $g_l$ at $t$ is given by
\begin{align*}
    &\frac{1}{\prod_{k \neq l} (u_{j_k + 1}^{(k)} - u_{j_k}^{(k)})} \cdot 
    \sum_{\delta \in \{0, 1\}^{d - 1}} \bigg[\prod_{k \neq l}\Big\{(1 - \delta_k) \big(u_{j_k + 1}^{(k)} - x_k\big) + \delta_k \big(x_k - u_{j_k}^{(k)}\big)\Big\}\bigg] \\
    &\qquad \quad \cdot \frac{1}{u_{j_l + 1}^{(l)} - u_{j_l}^{(l)}} \cdot \big(\theta_{j_1 + \delta_1, \dots, j_{l - 1} + \delta_{l - 1}, j_l + 1, j_{l + 1} + \delta_{l + 1}, \dots, j_d + \delta_d} - \theta_{j_1 + \delta_1, \dots, j_{l - 1} + \delta_{l - 1}, j_l, j_{l + 1} + \delta_{l + 1}, \dots, j_d + \delta_d}\big)
\end{align*}
for $t \in (u_{j_l}^{(l)}, u_{j_l + 1}^{(l)})$.
By \eqref{axcon-fin-cond}, this slope is decreasing in $t$, implying that $g_l$ is concave.
Hence, $f$ is axially concave, as desired.

Now, consider the following finite-dimensional least squares problem:
\begin{equation}\label{axcon-fin-dim-prob}
    \hat{\theta}_{\text{AC}}^{d} \in \argmin_{\theta} \big\{ \|\mathbf{y} - \mathit{M} \theta\|_2: \theta \in \R^{|\widebar{I}_0|} \text{ satisfies \eqref{axcon-fin-cond}} \big\},
\end{equation}
where $\mathbf{y} = (y_1, \dots, y_n)$, and $\mathit{M}$ is the $n \times |\widebar{I}_0|$ matrix with 
\begin{equation*}
    \mathit{M}_{i, (j_1, \dots, j_d)} =
    \begin{cases}
        1 & \text{if } (j_1, \dots, j_d) = \mathbf{b}(i) \\
        0 & \text{otherwise} 
    \end{cases}
\end{equation*}
for $i = 1, \dots, n$ and for $(j_1, \dots, j_d) \in \widebar{I}_0$.
Recall that $\theta_f$ satisfies \eqref{axcon-fin-cond} for every axially concave function $f$, and there always exists some axially concave function $f$ on $[0, 1]^d$ such that $\theta_f = \theta$ for every $\theta \in \R^{|\widebar{I}_0|}$ satisfying \eqref{axcon-fin-cond}.
Also, recall that $f({\mathbf{x}}^{(i)}) = (\theta_f)_{\mathbf{b}(i)}$ for every $i = 1, \dots, d$.
It thus follows that if we define $f$ as in \eqref{axcon-interpolation} using $\hat{\theta}_{\text{AC}}^{d}$ instead of $\theta$, then such $f$ solves the original least squares problem \eqref{axconreg}.
Hence, once we solve \eqref{axcon-fin-dim-prob}, we can construct the the least squares estimator for axially concave regression using \eqref{axcon-interpolation}.
This is also what we used in Section \ref{realdata}.

However, it should be noted that solving \eqref{axcon-fin-dim-prob} can be computationally challenging when the dimension $d$ is not small. 
The number of components of the vector $\theta$ in \eqref{axcon-fin-dim-prob} is $\prod_{k = 1}^{d} (n_k + 1)$, which can grow as large as $(n + 2)^d$ in the worst case. 
Recall that, for each $k = 1, \dots, d$, $n_k + 1$ is the number of unique values observed in the $k^{\text{th}}$ component of ${\mathbf{x}}^{(1)}, \dots, {\mathbf{x}}^{(d)}$.
This computational burden is also the reason why axially concave regression was not considered for the housing price dataset and the 401(k) dataset in Sections \ref{housing-price} and \ref{401k}.
Therefore, developing more efficient methods for computing $\hat{f}_{\text{AC}}^{d}$ or computationally feasible variants is crucial for making axially concave regression practical.
Exploring these alternatives would be an interesting direction for future research.

\section{Additional Plots}\label{additional-plots}
In this section, we provide additional plots for the earnings dataset. 
Figure \ref{fig:earnings-top2-and-bottom-ratios-full} shows the proportion of 
repetitions in which each model is among the two best-performing models 
(Plot A) and the worst-performing model (Plot B). This figure is similar
to Figure \ref{fig:earnings-top-and-bottom2-ratios-full} (the only
differences are best vs. top-two and worst vs. bottom-two). As the plots in
Figure \ref{fig:earnings-top-and-bottom2-ratios-full}, those in
Figure \ref{fig:earnings-top2-and-bottom-ratios-full}  indicate that for 
small sample sizes, Mincer's and Murphy and Welch's parametric models, as 
well as additive concave regression, are often better options. For large sample 
sizes, totally and axially concave regressions generally perform best. More 
interestingly, for moderate sample sizes, totally concave regression 
consistently ranks among the top two models, whereas axially concave 
regression frequently becomes the worst-performing model. These observations 
reaffirm the findings from Figure
\ref{fig:earnings-top-and-bottom2-ratios-full}. 

\begin{figure}[ht]
    \begin{center}
    \includegraphics{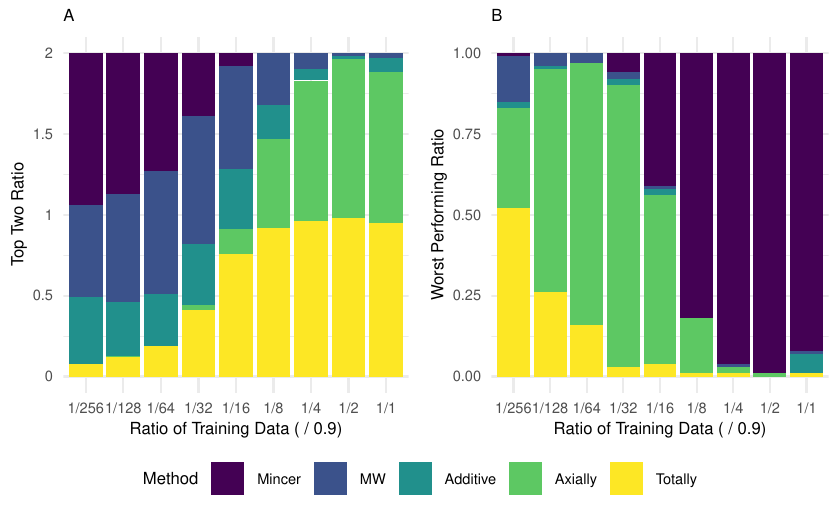}
    \end{center}
    \caption{Top two and worst performing ratios of each method for various 
    sub-sample sizes for the earnings dataset. The $x$-axis and $y$-axis 
    represent the proportion of training data and the proportion of repetitions 
    (out of 100) where (A) each model is one of the two best-performing models 
    and (B) each model is the worst-performing model.}
    \label{fig:earnings-top2-and-bottom-ratios-full}
\end{figure}

Figure \ref{fig:ex1029-rank-plot} displays the empirical cumulative 
distributions of model ranks across 100 splits into training and test sets, 
for four different training data proportions. As in the cumulative rank plots 
presented in Sections \ref{housing-price} and \ref{401k}, ranks are determined 
by mean squared errors on the test sets, with rank 1 indicating the best 
performing model and rank 5 the worst. Plot A corresponds to a training data 
proportion of $0.9/2^7$, representing a small sample size regime, while Plot D 
corresponds to $0.9/2^1$, representing a large sample size regime. Plots B 
and C correspond to intermediate training data proportions of $0.9/2^5$ and 
$0.9/2^3$, respectively.

An ideal method is one whose rank distribution function is everywhere larger 
than the others. Plot A shows that Mincer's model is the ideal method in the 
small sample size regime. As the sample size increases to the level of Plot B, 
Murphy and Welch's model becomes the most ideal. In both cases, axially 
concave regression is the worst choice. When the sample size reaches that of 
Plot C, totally concave regression is clearly the best option. Finally, 
for the largest sample size in Plot D, totally and axially concave regressions 
are the two best options that are comparable.

\begin{figure}[ht]
    \begin{center}
    \includegraphics{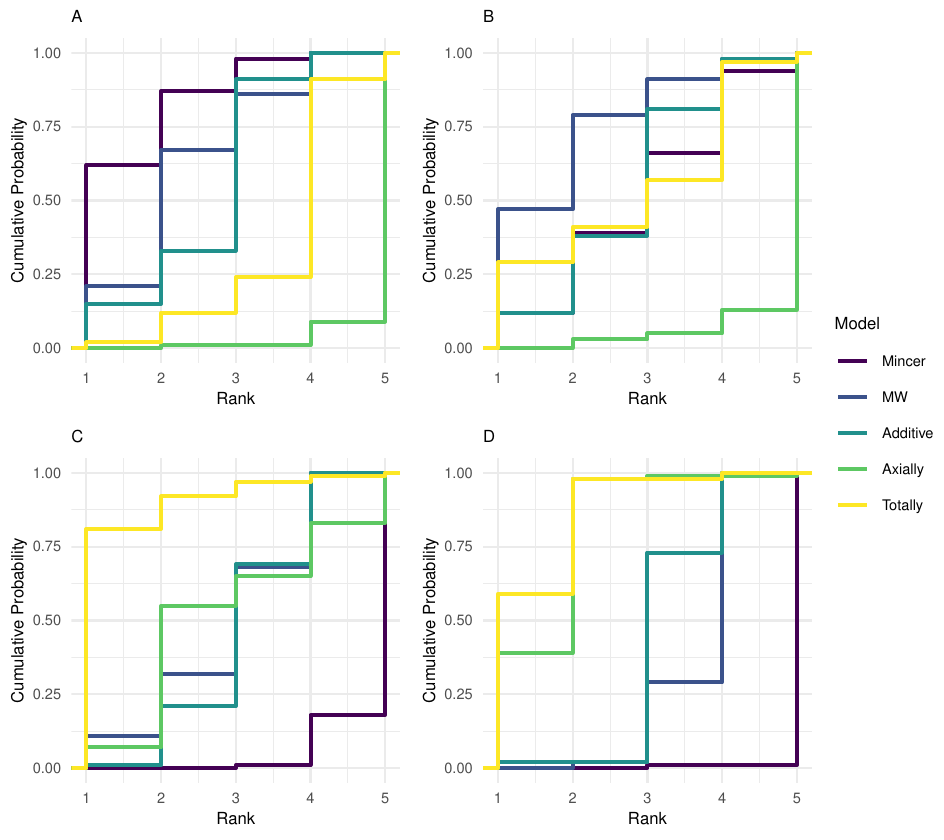}
    \end{center}
    \caption{Empirical cumulative distributions of model ranks across 100 
    splits into training and test sets, for various proportions of training data 
    for the earnings dataset. The training data proportions are 
    (A) $0.9/2^7$, (B) $0.9/2^5$, (C) $0.9/2^3$, and (D) $0.9/2^1$. Ranks are 
    determined by mean squared errors on the test sets.}
    \label{fig:ex1029-rank-plot}
\end{figure}

\section{Proofs}\label{proofs}
In our proofs, for each real-valued function $f$ on $[0, 1]^d$, we use the following notation for the divided difference of $f$ of order $(p_1, \dots, p_d)$:
\begin{equation*}
    \begin{bmatrix}
        x_1^{(1)}, \dots, x_{p_1 + 1}^{(1)} & \\ 
        \vdots                              &; \ f \\
        x_1^{(d)}, \dots, x_{p_d + 1}^{(d)} &
    \end{bmatrix}
    := \sum_{i_1 = 1}^{p_1 + 1} \cdots \sum_{i_d = 1}^{p_d + 1} \frac{f(x_{i_1}^{(1)}, \dots, x_{i_d}^{(d)})}{\prod_{j_1 \neq i_1} (x_{i_1}^{(1)} - x_{j_1}^{(1)}) \times \cdots \times \prod_{j_d \neq i_d} (x_{i_d}^{(d)} - x_{j_d}^{(d)})}
\end{equation*}  
for $0 \le x_1^{(k)} < \dots < x_{p_k+1}^{(k)} \le 1$, $k = 1, \dots, d$.
Using this notation, the conditions \eqref{eq:finite-derivative-at-zero} and \eqref{eq:finite-derivative-at-one} in Proposition \ref{prop:tc-equiv} can be represented as 
\begin{equation}\label{eq:finite-derivative-at-zero-restated}
    \sup \bigg\{
    \begin{bmatrix}
        0, t_k, & k \in S    & \multirow{2}{*}{; \ $f$ } \\ 
        0,      & k \notin S &
    \end{bmatrix}
    : t_k > 0 \ \text{ for } k \in S \bigg\} < +\infty
\end{equation}
and 
\begin{equation}\label{eq:finite-derivative-at-one-restated}
    \inf \bigg\{
    \begin{bmatrix}
        t_k, 1, & k \in S    & \multirow{2}{*}{; \ $f$ } \\ 
        0,      & k \notin S &
    \end{bmatrix}
    : t_k < 1 \ \text{ for } k \in S \bigg\} > -\infty
\end{equation}
for each nonempty subset $S$ of $[d]$.
Similarly, the interaction restriction condition \eqref{int-rest-cond} can be 
written as
\begin{equation}\label{int-rest-cond-restated}
    \begin{bmatrix}
        x_k, y_k, & k \in S    & \multirow{2}{*}{; \ $f$ } \\ 
        0,        & k \notin S &
    \end{bmatrix}
    = 0
    \quad \text{ for every } 0 \le x_k < y_k \le 1, k \in S
\end{equation}
for every subset $S$ of $[d]$ with $|S| > s$. 
It can be readily verified by induction that the conditions 
\eqref{int-rest-cond-restated} for all $S \subseteq [d]$ with $|S| > s$ are 
equivalent to the condition that the following holds for all subsets $S$ of 
$[d]$ with $|S| > s$:
\begin{equation}\label{int-rest-cond-restated-2}
    \begin{bmatrix}
        x_k, y_k, & k \in S    & \multirow{2}{*}{; \ $f$ } \\ 
        x_k,        & k \notin S &
    \end{bmatrix}
    = 0
\end{equation}
for every $0 \le x_k < y_k \le 1, k \in S$ and $x_k \in [0, 1], k \notin S$.

\subsection{Proofs of Propositions in Section \ref{totconcdef}}\label{pf:totconcdef}
\subsubsection{Proof of Proposition \ref{prop:tc-equiv-restricted-interaction}}\label{pf:tc-equiv-restricted-interaction}
An alternative representation of $\tcds$, as stated in Lemma \ref{lem:alt-characterization}, is helpful in proving Proposition \ref{prop:tc-equiv-restricted-interaction}. 
To describe this representation, we first recall the concept of \textit{entire monotonicity} of functions. 
Entire monotonicity is a multivariate generalization of univariate monotonicity, defined as follows (see also \cite{fang2021multivariate} and references therein).
Similar to total concavity in the sense of Popoviciu, entire monotonicity is defined using the divided differences of functions. 
While total concavity involves divided differences of order $\mathbf{p} = (p_1, \dots, p_d) \in \{0, 1, 2\}^d$ with $\max_k p_k = 2$, entire monotonicity is based on divided differences of order $\mathbf{p} \in \{0, 1\}^d$ with $\max_k p_k = 1$.
The proof of Lemma \ref{lem:alt-characterization} is provided in Appendix \ref{pf:alt-characterization}.

\begin{definition}[Entire Monotonicity]\label{defn-entire-mono}
  A real-valued function $f$ on $[0, 1]^d$ is said to be entirely monotone if, for every $(p_1, \dots, p_d)$ with $\max_k p_k = 1$ (i.e., $(p_1, \dots, p_d) \in \{0, 1\}^d \setminus \{\zerovec\}$), we have 
\begin{equation*}
    \begin{bmatrix}
        x_1^{(1)}, \dots, x_{p_1 + 1}^{(1)} & \\ 
        \vdots                              &; \ f \\
        x_1^{(d)}, \dots, x_{p_d + 1}^{(d)} &
    \end{bmatrix}
    \ge 0
\end{equation*}
for every $0 \le x_1^{(k)} < \cdots < x_{p_k + 1}^{(k)} \le 1$ for $k
= 1, \dots, d$.
\end{definition}


\begin{lemma}\label{lem:alt-characterization}
    The function class $\tcds$ consists of all functions of the form
    \begin{equation}\label{eq:alt-characterization}
        f(x_1, \dots, x_d) = \beta_0 - \sum_{S : 1 \le |S| \le s} \int_{\prod_{k \in S} [0, x_k]} g_{S}((t_k, k \in S)) \, d(t_k, k \in S)
    \end{equation}
    for some $\beta_0 \in \R$ and for some collection of functions $\{g_S: S \subseteq [d] \text{ and } 1 \le |S| \le s\}$, 
    where for each $S$, 
    \begin{enumerate}
        \item $g_S$ is a real-valued function on $[0, 1]^{|S|}$,
        \item $g_S$ is entirely monotone,
        \item $g_S$ is coordinate-wise right-continuous on $[0, 1]^{|S|}$, and 
        \item $g_S$ is coordinate-wise left-continuous at each point $(x_k, k \in S) \in [0, 1]^{|S|} \setminus [0, 1)^{|S|}$ with respect to all the $k^{\text{th}}$ coordinates where $x_k = 1$.
    \end{enumerate}
\end{lemma}

Here, we also briefly introduce the concept of absolutely continuous interval functions and their derivatives we need for our proof of Proposition \ref{prop:tc-equiv-restricted-interaction}.
We refer readers to \citet{lojasiewicz1988introduction} for more details about them.

We call the product of $m$ closed intervals in $\R$, $\prod_{k = 1}^{m} [x_k, y_k]$ ($x_k < y_k$ for $k = 1, \dots, d$), an $m$-dimensional closed interval or a closed interval in $\R^m$. 
In particular, if $m$ closed intervals in $\R$ are of equal length, we refer to the product as an $m$-dimensional cube.
We say that two closed intervals in $\R^m$ do not overlap if they do not share interior points.
We also say that two closed intervals in $\R$ are adjoining if they have a common end point but do not overlap.
For instance, $[0, 1]$ and $[1, 2]$ are two adjoining closed intervals in $\R$.
In general, two closed intervals in $\R^m$ 
\begin{equation*}
    \prod_{k = 1}^{m} [x_k, y_k] \ \text{ and } \ \prod_{k = 1}^{m} [z_k, w_k]
\end{equation*}
are said to be adjoining if there exists some $k_0 \in \{1, \dots, m\}$ such that $[x_k, y_k]$ and $[z_k, w_k]$ are adjoining if $k = k_0$ and $[x_k, y_k] = [z_k, w_k]$ otherwise.
For example, $[0, 1] \times [0, 1]$ and $[0, 1] \times [1, 2]$ are adjoining closed intervals in $\R^2$. 
Note that the union of two adjoining closed intervals in $\R^m$ is also a closed interval in $\R^m$.
In addition, for an $m$-dimensional closed interval $\prod_{k = 1}^{m} [x_k, y_k]$, let $|\prod_{k = 1}^{m} [x_k, y_k]|$ denote the volume of the interval given by 
\begin{equation*}
    \bigg|\prod_{k = 1}^{m} [x_k, y_k]\bigg| = \prod_{k =  1}^{m} (y_k - x_k).
\end{equation*}

Let $I_0$ be a closed interval in $\R^m$, and let $F$ be a real-valued function defined for $m$-dimensional closed intervals contained in $I_0$.
For $\mathbf{x} \in I_0$, the upper and lower derivatives of $F$ at $\mathbf{x}$ are defined by
\begin{equation*}
    \overline{D}F(\mathbf{x}) = \limsup_{\substack{C: \text{cube, } C \ni \mathbf{x} \\ |C| \rightarrow 0}} \frac{F(C)}{|C|} \ \text{ and } \
    \underline{D}F(\mathbf{x}) = \liminf_{\substack{C: \text{cube, } C \ni \mathbf{x} \\ |C| \rightarrow 0}} \frac{F(C)}{|C|}.
\end{equation*}
Note that $\limsup$ and $\liminf$ are taken over $m$-dimensional cubes containing $\mathbf{x}$.
If these two derivatives are equal, the derivative of $F$ at $\mathbf{x}$ exists and is defined as their common value:
\begin{equation*}
    DF(\mathbf{x}) = \lim_{\substack{C: \text{cube, } C \ni \mathbf{x} \\ |C| \rightarrow 0}} \frac{F(C)}{|C|}.
\end{equation*}
The function $F$ is called an additive interval function on $I_0$ if 
\begin{equation*}
    F(I \cup I') = F(I) + F(I')
\end{equation*}
for every pair of $m$-dimensional adjoining closed intervals $I, I' \subseteq I_0$.
Furthermore, $F$ is said to be absolutely continuous on $I_0$ if it is additive and satisfies the following: 
For every $\epsilon > 0$, there exists $\delta > 0$ such that, for any finite collection of non-overlapping $m$-dimensional closed intervals $P_1, \dots, P_l \subseteq I_0$ with $\sum_{k = 1}^{l} |P_k| < \delta$, it holds that $\sum_{k = 1}^{l} |F(P_k)| < \epsilon$.

The following theorem is a version of the fundamental theorem of calculus for absolutely continuous interval functions.
This theorem will play a crucial role in our proof of Proposition \ref{prop:tc-equiv-restricted-interaction}.

\begin{theorem}[Theorem 7.3.3 of \citet{lojasiewicz1988introduction}]\label{thm:lojasiewicz}
    Suppose $F$ is an absolutely continuous interval function on a closed interval $I_0$ in $\R^m$. 
    In this case, $DF$ exists almost everywhere (with respect to the Lebesgue measure) on $I_0$, and 
    \begin{equation*}
        F(I) = \int_I DF(\mathbf{x}) \, d\mathbf{x}
    \end{equation*}
    for every $m$-dimensional closed interval $I \subseteq I_0$.
\end{theorem}

In our proof of Proposition \ref{prop:tc-equiv-restricted-interaction}, we also employ the following notation.
For a real-valued function $f$ on $[0, 1]^d$ and for each nonempty $S \subseteq [d]$, let $\triangle_{S} f$ denote the function on intervals, defined by
\begin{equation}\label{eq:alt-sum}
    \triangle_{S} f\bigg(\prod_{k \in S} [x_k, y_k]\bigg) = \sum_{\delta \in \prod_{k \in S} \{0, 1\}} (-1)^{\sum_{k \in S} \delta_k} \cdot f\big(\widetilde{xy}_{\delta}\big)
\end{equation}
for $0 \le x_k < y_k \le 1$ for $k \in S$. 
Here, $\widetilde{xy}_{\delta}$ is the $d$-dimensional vector for which
\begin{align*}
    (\widetilde{xy}_{\delta})_k = 
    \begin{cases}
        \delta_k x_k + (1 - \delta_k) y_k & \text{if } k \in S \\
        0 & \text{otherwise.} 
    \end{cases}
\end{align*}
For example, if $d = 3$ and $S = \{1, 3\}$, \eqref{eq:alt-sum} simplifies to 
\begin{equation*}
    \triangle_{\{1, 3\}} f\big([u_1, u_2] \times [v_1, v_2]\big) = f(u_2, 0, v_2) - f(u_1, 0, v_2) - f(u_2, 0, v_1) + f(u_1, 0, v_1)
\end{equation*}
for $0 \le u_1 < u_2 \le 1$ and $0 \le v_1 < v_2 \le 1$.
Note that $\triangle_S f$ can be related to the divided differences of $f$ as:
\begin{equation*}
    \triangle_{S} f\bigg(\prod_{k \in S} [x_k, y_k]\bigg) = 
    \bigg[\prod_{k \in S} (y_k - x_k)\bigg] \cdot
    \begin{bmatrix}
        x_k, y_k, & k \in S    & \multirow{2}{*}{; \ $f$ } \\ 
        0,     & k \notin S &
    \end{bmatrix}
\end{equation*}
for $0 \le x_k < y_k \le 1$ for $k \in S$. 
For notational convenience, when $S = [d]$, we omit the subscript and write $\triangle := \triangle_{[d]}$.

\begin{proof}[Proof of Proposition \ref{prop:tc-equiv-restricted-interaction}]
\textbf{Step 1: $f \in \tcds \Rightarrow f \in \tcp$.}

Assume that $f \in \tcds$. 
We first show that $f \in \tcp$.
For this, we need to prove that the divided difference of $f$ of order $\mathbf{p}$ is nonpositive on $[0, 1]^d$ for every $\mathbf{p} = (p_1, \dots, p_d) \in \{0, 1, 2\}^d$ with $\max_k p_k = 2$.
The following lemma (proved in Appendix \ref{pf:tc-weak-cond}) shows that, in fact, it suffices to verify the nonpositivity of the divided differences of $f$ on some proper subsets of $[0, 1]^d$.

\begin{lemma}\label{lem:tc-weak-cond}
    Let $f$ be a real-valued function on $[0, 1]^d$. 
    Then, $f \in \tcp$ if, for each $\mathbf{p} = (p_1, \dots, p_d) \in \{0, 1, 2\}^d$ with $\max_k p_k = 2$,
    the divided difference of $f$ of order $\mathbf{p}$ is nonpositive on $N^{(\mathbf{p})} := N_1^{(\mathbf{p})} \times \cdots \times N_d^{(\mathbf{p})}$, where $N_k^{(\mathbf{p})} = [0, 1]$ if $p_k \neq 0$ and $N_k^{(\mathbf{p})} = \{0\}$ otherwise, i.e., 
    \begin{equation*}
    \begin{bmatrix}
        x_1^{(1)}, \dots, x_{p_1 + 1}^{(1)} & \\ 
        \vdots                              &; \ f \\
        x_1^{(d)}, \dots, x_{p_d + 1}^{(d)} &
    \end{bmatrix}
    \le 0
    \end{equation*}
    provided that $0 \le x_1^{(k)} < \cdots < x_{p_k + 1}^{(k)} \le 1$ if $p_k \neq 0$ and $x_1^{(k)} = 0$ otherwise.
\end{lemma}

Without loss of generality, we assume that $\mathbf{p} = (2, \dots, 2, 1, \dots, 1, 0, \dots, 0)$ where $|\{k: p_k \neq 0\}| = m_2$ and $|\{k: p_k = 2\}| = m_1$. 
Also, for notational convenience, we let $S_{\mathbf{p}} = [m_2]$.
By Lemma \ref{lem:alt-characterization}, there exist $\beta_0 \in \R$ and functions $g_S$, $1 \le |S| \le s$ satisfying the conditions in Lemma \ref{lem:alt-characterization} such that $f$ can be expressed as in \eqref{eq:alt-characterization}.
For every $0 \le x_1^{(k)} < \cdots < x_{p_k + 1}^{(k)} \le 1$, $k = 1, \dots, m_2$, we can thus represent the divided difference of $f$ on $N^{(\mathbf{p})}$ as follows:
\begingroup
\allowdisplaybreaks
\begin{align*}
    &\begin{bmatrix}
        x_1^{(1)}, \dots, x_{p_1 + 1}^{(1)}       & \multirow{6}{*}{; \ $f$ } \\ 
        \vdots                                    & \\
        x_1^{(m_2)}, \dots, x_{p_{m_2} + 1}^{(m_2)} & \\
        0                                         & \\
        \vdots                                    & \\
        0                                         &
    \end{bmatrix}
    \\
    &\quad= \sum_{i_1 = 1}^{3} \cdots \sum_{i_{m_1} = 1}^{3} \sum_{i_{m_1 + 1} = 1}^{2} \cdots \sum_{i_{m_2} = 1}^{2} \frac{1}{\prod_{j_1 \neq i_1} (x_{i_1}^{(1)} - x_{j_1}^{(1)})} \times \cdots \times \frac{1}{\prod_{j_{m_1} \neq i_{m_1}} (x_{i_{m_1}}^{(m_1)} - x_{j_{m_1}}^{(m_1)})} \\
    &\quad \qquad \qquad \qquad \qquad \qquad \qquad \qquad \cdot \frac{1}{\prod_{k = m_1 + 1}^{m_2}(x_2^{(k)} - x_1^{(k)})} \cdot (-1)^{\sum_{k = m_1 + 1}^{m_2} i_k} \cdot f\big(x_{i_1}^{(1)}, \dots, x_{i_{m_2}}^{(m_2)}, 0, \dots, 0\big)\\
    &\quad= \frac{1}{\prod_{k = 1}^{m_1}(x_3^{(k)} - x_1^{(k)})} \cdot \sum_{\delta \in \{0, 1\}^{m_1}} (-1)^{\sum_{k = 1}^{m_1} \delta_k} \cdot \frac{1}{\prod_{k = 1}^{m_1}(x_{3 - \delta_k}^{(k)} - x_{2 - \delta_k}^{(k)})} \cdot \frac{1}{\prod_{k = m_1 + 1}^{m_2}(x_2^{(k)} - x_1^{(k)})} \\
    &\quad \quad \cdot \sum_{i_1 = 2 - \delta_1}^{3 - \delta_1} \cdots \sum_{i_{m_1} = 2 - \delta_{m_1}}^{3 - \delta_{m_1}} \sum_{i_{m_1 + 1} = 1}^{2} \cdots \sum_{i_{m_2} = 1}^{2} (-1)^{\sum_{k = 1}^{m_1} (i_k + \delta_k - 1)} \cdot (-1)^{\sum_{k = m_1 + 1}^{m_2} i_k} f\big(x_{i_1}^{(1)}, \dots, x_{i_{m_2}}^{(m_2)}, 0, \dots, 0\big) \\
    &\quad= \frac{1}{\prod_{k = 1}^{m_1}(x_3^{(k)} - x_1^{(k)})} \cdot \frac{1}{\prod_{k = m_1 + 1}^{m_2}(x_2^{(k)} - x_1^{(k)})} \cdot \sum_{\delta \in \{0, 1\}^{m_1}} (-1)^{\sum_{k = 1}^{m_1} \delta_k} \cdot \frac{1}{\prod_{k = 1}^{m_1}(x_{3 - \delta_k}^{(k)} - x_{2 - \delta_k}^{(k)})} \\
    &\quad \qquad \qquad \qquad \qquad \qquad \qquad \qquad \qquad \qquad \qquad \quad \cdot \triangle_{S_{\mathbf{p}}}f \bigg(\prod_{k = 1}^{m_1} \big[x_{2 - \delta_k}^{(k)}, x_{3 - \delta_k}^{(k)}\big] \times \prod_{k = m_1 + 1}^{m_2} \big[x_1^{(k)}, x_2^{(k)} \big]\bigg) \\
    &\quad= - \frac{1}{\prod_{k = 1}^{m_1}(x_3^{(k)} - x_1^{(k)})} \cdot \frac{1}{\prod_{k = m_1 + 1}^{m_2}(x_2^{(k)} - x_1^{(k)})} \cdot \sum_{\delta \in \{0, 1\}^{m_1}} (-1)^{\sum_{k = 1}^{m_1} \delta_k} \cdot \frac{1}{\prod_{k = 1}^{m_1}(x_{3 - \delta_k}^{(k)} - x_{2 - \delta_k}^{(k)})} \\
    &\quad \qquad \qquad \qquad \qquad \qquad \qquad \qquad \cdot 
    \int_{x_{2 - \delta_1}^{(1)}}^{x_{3 - \delta_1}^{(1)}} \cdots \int_{x_{2 - \delta_{m_1}}^{(m_1)}}^{x_{3 - \delta_{m_1}}^{(m_1)}} \int_{x_1^{(m_1 + 1)}}^{x_2^{(m_1 + 1)}} \cdots \int_{x_1^{(m_2)}}^{x_2^{(m_2)}} g_{S_{\mathbf{p}}}(t_1, \dots, t_{m_2}) \, dt_{m_2} \cdots dt_1.
\end{align*}
\endgroup
Also, since $g_{S_{\mathbf{p}}}$ is entirely monotone, we have
\begingroup
\allowdisplaybreaks
\begin{align*}
    0 &\le \int_{x_2^{(1)}}^{x_3^{(1)}} \int_{x_1^{(1)}}^{x_2^{(1)}} \cdots \int_{x_2^{(m_1)}}^{x_3^{(m_1)}} \int_{x_1^{(m_1)}}^{x_2^{(m_1)}} 
    \int_{x_1^{(m_1 + 1)}}^{x_2^{(m_1 + 1)}} 
    \cdots
    \int_{x_1^{(m_2)}}^{x_2^{(m_2)}} 
    \\
    &\qquad \quad
    \prod_{k = 1}^{m_1} (v_k - u_k) \cdot
    \begin{bmatrix}
        u_1, v_1         & \multirow{6}{*}{; \ $g_{S_{\mathbf{p}}}$ } \\ 
        \vdots           & \\
        u_{m_1}, v_{m_1} & \\
        t_{m_1 + 1}      & \\
        \vdots           & \\
        t_{m_2}          &
    \end{bmatrix}
    \, dt_{m_2} \cdots dt_{m_1 + 1} du_{m_1} dv_{m_1} \cdots du_1 dv_1 \\
    &= \prod_{k = 1}^{m_1} \Big[\big(x_3^{(k)} - x_2^{(k)}\big) \big(x_2^{(k)} - x_1^{(k)}\big)\Big] \cdot \sum_{\delta \in \{0, 1\}^{m_1}} (-1)^{\sum_{k = 1}^{m_1} \delta_k} \cdot \frac{1}{\prod_{k = 1}^{m_1}(x_{3 - \delta_k}^{(k)} - x_{2 - \delta_k}^{(k)})} \\
    &\quad \qquad \qquad \qquad \qquad \qquad \cdot 
    \int_{x_{2 - \delta_1}^{(1)}}^{x_{3 - \delta_1}^{(1)}} \cdots \int_{x_{2 - \delta_{m_1}}^{(m_1)}}^{x_{3 - \delta_{m_1}}^{(m_1)}} \int_{x_1^{(m_1 + 1)}}^{x_2^{(m_1 + 1)}} \cdots \int_{x_1^{(m_2)}}^{x_2^{(m_2)}} g_{S_{\mathbf{p}}}(t_1, \dots, t_{m_2}) \, dt_{m_2} \cdots dt_1.
\end{align*}
\endgroup
Combining these results, we can show that the divided difference of order $\mathbf{p}$ is nonpositive on $N^{(\mathbf{p})} = [0, 1]^{m_2} \times \{0\}^{d - m_2}$ when $\mathbf{p} = (2, \dots, 2, 1, \dots, 1, 0, \dots, 0)$.
Through the same argument, we can prove that this is also true for every $\mathbf{p} \in \{0, 1, 2\}^d$ with $\max_k p_k = 2$.
Due to Lemma \ref{lem:tc-weak-cond}, we can therefore conclude that $f \in \tcp$.

\noindent\textbf{Step 2: $f \in \tcds \Rightarrow $ \eqref{eq:finite-derivative-at-zero} and \eqref{eq:finite-derivative-at-one} hold for all $S \neq \emptyset$ and \eqref{int-rest-cond} holds for all $S$ with $|S| > s$.}

We now show that $f$ satisfies \eqref{eq:finite-derivative-at-zero} and \eqref{eq:finite-derivative-at-one} (recall \eqref{eq:finite-derivative-at-zero-restated} and \eqref{eq:finite-derivative-at-one-restated}) for all nonempty subsets $S \subseteq [d]$ and \eqref{int-rest-cond} (recall \eqref{int-rest-cond-restated}) for all subsets $S$ with $|S| > s$.
It can be readily checked that \eqref{int-rest-cond} holds for all subsets $S$ with $|S| > s$. 
Also, for each subset $S \subseteq [d]$ with $1 \le |S| \le s$, since $g_S$ is entirely monotone, we have that 
\begin{equation*}
    \begin{bmatrix}
        0, t_k, & k \in S    & \multirow{2}{*}{; \ $f$ } \\ 
        0,      & k \notin S &
    \end{bmatrix}
    = \frac{1}{\prod_{k \in S} t_k} \int_{\prod_{k \in S}(0, t_k]} g_S((s_k, k \in S)) \, d(s_k, k \in S) \ge g_S((0, k \in S))
\end{equation*}
for every $t_k > 0$ for $k \in S$,
and that 
\begin{equation*}
    \begin{bmatrix}
        t_k, 1, & k \in S    & \multirow{2}{*}{; \ $f$ } \\ 
        0,     & k \notin S &
    \end{bmatrix}
    = \frac{1}{\prod_{k \in S} (1 - t_k)} \int_{\prod_{k \in S}(t_k, 1]} g_S((s_k, k \in S)) \, d(s_k, k \in S) \le g_S((1, k \in S))
\end{equation*}
for every $t_k < 1$ for $k \in S$.
The conditions \eqref{eq:finite-derivative-at-zero} and \eqref{eq:finite-derivative-at-one} follow directly from these results.

\noindent\textbf{Step 3: $f \in \tcp $ and \eqref{eq:finite-derivative-at-zero}, \eqref{eq:finite-derivative-at-one} hold for all $S \neq \emptyset$ $\Rightarrow f \in \tcd$.}

Next, we prove that, for every $f \in \tcp$ satisfying \eqref{eq:finite-derivative-at-zero} and \eqref{eq:finite-derivative-at-one} for all nonempty subsets $S \subseteq [d]$, it holds that $f \in \tcd$.
Observe that
\begin{equation*}
    f(x_1, \dots, x_d) = f(0, \dots, 0) + \sum_{\emptyset \neq S \subseteq [d]} \triangle_S f\bigg(\prod_{k \in S} [0, x_k]\bigg)
\end{equation*}
for every $(x_1, \dots, x_d) \in [0, 1]^d$.
Hence, by Lemma \ref{lem:alt-characterization}, it suffices to show that, for each nonempty subset $S$, there exists a function $g_S$ satisfying the conditions of Lemma \ref{lem:alt-characterization} such that 
\begin{equation}\label{eq:alt-sum-as-int}
    \triangle_S f\bigg(\prod_{k \in S} [0, x_k]\bigg) = - \int_{\prod_{k \in S} [0, x_k]} g_{S}((t_k, k \in S)) \, d(t_k, k \in S)
\end{equation}
for every $(x_k, k \in S) \in [0, 1]^{|S|}$.
Here, we prove this claim only for the case $S = [d]$, but our argument can be readily adapted to other nonempty subsets $S$ of $[d]$. 

\noindent\textbf{Step 3-1: $\triangle f$ is absolutely continuous.}

Recall the notation $\triangle = \triangle_{[d]}$.
We first show that $\triangle f$ is an absolutely continuous interval function on $[0, 1]^d$.
To show this, we use the following lemma, which we prove in Appendix \ref{pf:div-diff-mono}.

\begin{lemma}\label{lem:div-diff-mono}
    Suppose $f$ is a real-valued function on $[0, 1]^d$ whose divided difference of order $\mathbf{p} = (p_1, \dots, p_d)$ is nonpositive for $\mathbf{p} = (2, 1, \dots, 1), (1, 2, 1, \dots, 1), \dots, (1, \dots, 1, 2)$. 
    Then, we have 
    \begin{equation*}
        \begin{bmatrix}
            x_1^{(1)}, x_2^{(1)} & \\ 
            \vdots               &; \ f \\
            x_1^{(d)}, x_2^{(d)} &
        \end{bmatrix} \ge
        \begin{bmatrix}
            x_1^{(1)}, x_3^{(1)} & \\ 
            \vdots               &; \ f \\
            x_1^{(d)}, x_3^{(d)} &
        \end{bmatrix}
    \end{equation*}
    if $0 \le x_1^{(k)} < x_2^{(k)} \le x_3^{(k)} \le 1$ for $k = 1, \dots, d$, and 
    \begin{equation*}
        \begin{bmatrix}
            x_1^{(1)}, x_3^{(1)} & \\ 
            \vdots               &; \ f \\
            x_1^{(d)}, x_3^{(d)} &
        \end{bmatrix} \ge
        \begin{bmatrix}
            x_2^{(1)}, x_3^{(1)} & \\ 
            \vdots               &; \ f \\
            x_2^{(d)}, x_3^{(d)} &
        \end{bmatrix}
    \end{equation*}
    if $0 \le x_1^{(k)} \le x_2^{(k)} < x_3^{(k)} \le 1$ for $k = 1, \dots, d$.
    Also, it follows that 
    \begin{equation*}
        \begin{bmatrix}
            x_1^{(1)}, x_2^{(1)} & \\ 
            \vdots               &; \ f \\
            x_1^{(d)}, x_2^{(d)} &
        \end{bmatrix} \ge
        \begin{bmatrix}
            y_1^{(1)}, y_2^{(1)} & \\ 
            \vdots               &; \ f \\
            y_1^{(d)}, y_2^{(d)} &
        \end{bmatrix}
    \end{equation*}
    if $0 \le x_1^{(k)} \le y_1^{(k)} \le 1$ and $0 \le x_2^{(k)} \le y_2^{(k)} \le 1$ for $k = 1, \dots, d$.
\end{lemma}

Along with the conditions \eqref{eq:finite-derivative-at-zero} and \eqref{eq:finite-derivative-at-one}, Lemma \ref{lem:div-diff-mono} implies that, for every $0 \le x_k^{(1)} < x_k^{(2)} \le 1$, $k = 1, \dots, d$, 
\begin{equation}\label{eq:bdd-of-derivative}
\begin{split}
    &\bigg|\frac{1}{\prod_{k = 1}^{d}(x_k^{(2)} - x_k^{(1)})} \cdot \triangle f\bigg(\prod_{k = 1}^{d} \big[x_k^{(1)}, x_k^{(2)}\big] \bigg)\bigg| 
    = \bigg|
    \begin{bmatrix}
        x_1^{(1)}, x_2^{(1)} & \\ 
        \vdots               &; \ f \\
        x_1^{(d)}, x_2^{(d)} &
    \end{bmatrix}
    \bigg| \\
    &\qquad \quad \le \max \Bigg( 
    \bigg| \sup \bigg\{
    \begin{bmatrix}
        t_1, 1 & \\ 
        \vdots &; \ f \\
        t_d, 1 &
    \end{bmatrix}
    : t_k < 1 \bigg\} \bigg|,
    \ \bigg| \inf \bigg\{
    \begin{bmatrix}
        0, t_1 & \\ 
        \vdots &; \ f \\
        0, t_d &
    \end{bmatrix}
    : t_k > 0 \bigg\} \bigg|
    \Bigg)
    < +\infty,
\end{split}
\end{equation}
from which the absolute continuity of $\triangle f$ follows directly.
Therefore, by Theorem \ref{thm:lojasiewicz}, $D\triangle f$ exists almost everywhere (w.r.t. the Lebesgue measure) on $[0, 1]^d$, and we have 
\begin{equation}\label{eq:lojasiewicz-on-alt-sum}
    \triangle f \bigg(\prod_{k = 1}^{d}[0, x_k]\bigg) 
    = \int_{\prod_{k = 1}^{d} [0, x_k]} D \triangle f(t_1, \dots, t_d) \, d(t_1, \dots, t_d)
\end{equation}
for every $(x_1, \dots, x_d) \in [0, 1]^d$.
Furthermore, using Lemma \ref{lem:div-diff-mono} again, we can show that 
\begin{equation}\label{eq:mono-derivative}
    D\triangle f(x_1, \dots, x_d) \ge D\triangle f(x_1', \dots, x_d')
\end{equation}
whenever $x_k \le x_k'$ for $k = 1, \dots, d$, and they exist.

\noindent\textbf{Step 3-2: Definition of $g$.}

Now, we define a function $g$ on $[0, 1]^d$ as follows. 
For $\mathbf{x} = (x_1, \dots, x_d) \in [0, 1)^d$, let
\begin{equation}\label{eq:def-of-g}
    g(x_1, \dots, x_d) = \inf \big\{-D\triangle f(s_1, \dots, s_d): s_k > x_k, k = 1, \dots, d \text{ and } D\triangle f(s_1, \dots, s_d) \text{ exists} \big\}.
\end{equation}
For $\mathbf{x} = (x_1, \dots, x_d) \in [0, 1]^d \setminus [0, 1)^d$, the value of $g$ at $\mathbf{x}$ is determined sequentially as follows:
\begingroup
\allowdisplaybreaks
\begin{align*}
     &g(1, x_2, \dots, x_d) = \lim\limits_{t_1 \uparrow 1} g(t_1, x_2, \dots, x_d) \ \text{ for } x_2, \dots, x_d < 1, \\
     &g(x_1, 1, x_3, \dots, x_d) = \lim\limits_{t_2 \uparrow 1} g(x_1, t_2, x_3, \dots, x_d) \ \text{ for } x_1, x_3, \dots, x_d < 1, \ \dots \ , \\
     &g(x_1, \dots, x_{d - 1}, 1) = \lim\limits_{t_d \uparrow 1} g(x_1, \dots, x_{d - 1}, t_d) \ \text{ for } x_1, \dots, x_{d - 1} < 1, \\
     &g(1, 1, x_3, \dots, x_d) = \lim\limits_{t_1 \uparrow 1} g(t_1, 1, x_3, \dots, x_d) = \lim\limits_{t_2 \uparrow 1} g(1, t_2, x_3, \dots, x_d) \ \text{ for } x_3, \dots, x_d < 1, \ \dots \ , \\
     &g(x_1, \dots, x_{d - 2}, 1, 1) = \lim\limits_{t_{d - 1} \uparrow 1} g(x_1, \dots, x_{d - 2}, t_{d - 1}, 1) = \lim\limits_{t_d \uparrow 1} g(x_1, \dots, x_{d - 2}, 1, t_d) \ \text{ for } x_1, \dots, x_{d - 2} < 1, \ \dots \ , \\
     &g(1, \dots, 1) = \lim\limits_{t_1 \uparrow 1} g(t_1, 1, \dots, 1) = \cdots = \lim\limits_{t_d \uparrow 1} g(1, \dots, 1, t_d).
\end{align*}
\endgroup
If we prove that (a) $g$ is well-defined, (b) $g$ satisfies the conditions of Lemma \ref{lem:alt-characterization}, and (c) $g = -D\triangle f$ almost everywhere (w.r.t. the Lebesgue measure), then by \eqref{eq:lojasiewicz-on-alt-sum}, we can guarantee that $g$ is the desired function for the claim \eqref{eq:alt-sum-as-int} when $S = [d]$.

\noindent\textbf{Step 3-3: Well-definedness of $g$.}

We prove that $g$ is well-defined at every point $\mathbf{x} = (x_1, \dots, x_d) \in [0, 1]^d$ by induction on $|\{k: x_k = 1\}|$. 
From \eqref{eq:bdd-of-derivative} and \eqref{eq:def-of-g}, it follows that
\begin{equation*}
    \bigg|\sup_{(x_1, \dots, x_d) \in [0, 1)^d} g(x_1, \dots, x_d)\bigg| < + \infty.
\end{equation*}
Moreover, by \eqref{eq:def-of-g}, $g(x_1, \dots, x_d) \le g(x_1', \dots, x_d')$ for every $x_k \le x_k' < 1$, $k = 1, \dots, d$.
These boundedness and monotonicity of $g$ ensure that, for fixed $x_2, \dots, x_d < 1$, the limit $\lim_{t_1 \uparrow 1} g(t_1, x_2, \dots, x_d)$ exists, which means that $g(1, x_2, \dots, x_d)$ is well-defined. 
Through the same argument, we can show that $g$ is well-defined at every $\mathbf{x} \in [0, 1]^d$ with $|\{k: x_k = 1\}| = 1$.

We now assume that $g$ is well-defined at every $\mathbf{x} \in [0, 1]^d$ with $|\{k: x_k = 1\}| \le m - 1$ and show that it is also well-defined at every $\mathbf{x}$ with $|\{k: x_k = 1\}| = m$.
Without loss of generality, assume $\mathbf{x} = (1, \dots, 1, x_{m + 1}, \dots, x_d)$, where $x_{m + 1}, \dots, x_d < 1$.
To show that $g$ is well-defined at $\mathbf{x}$, we need to verify that 
\begin{equation*}
    \lim_{t_{1} \uparrow 1} g(t_1, 1, \dots, 1, x_{m + 1}, \dots, x_d), \dots, \lim_{t_{m} \uparrow 1} g(1, \dots, 1, t_{m}, x_{m + 1}, \dots, x_d)
\end{equation*}
exist and coincide.
To establish this, it suffices to prove that $\lim_{t_{m} \uparrow 1} g(1, \dots, 1, t_{m}, x_{m + 1}, \dots, x_d)$ exists and satisfies
\begin{equation*}
    \lim_{t_{m} \uparrow 1} g(1, \dots, 1, t_{m}, x_{m + 1}, \dots, x_d) = \lim_{n \rightarrow \infty} g\Big(1 - \frac{1}{n}, \dots, 1 - \frac{1}{n}, x_{m + 1}, \dots, x_d\Big).
\end{equation*}
The existence of the limit follows from the boundedness of $g$ and the monotonicity:
\begin{align*}
    &g(1, \dots, 1, t_m, x_{m + 1}, \dots, x_d) 
    = \lim_{t_1 \uparrow 1} \cdots \lim_{t_{m - 1} \uparrow 1} g(t_1, \dots, t_{m-1}, t_m, x_{m + 1}, \dots, x_d) \\
    &\qquad \le \lim_{t_1 \uparrow 1} \cdots \lim_{t_{m - 1} \uparrow 1} g(t_1, \dots, t_{m-1}, t_m', x_{m + 1}, \dots, x_d) 
    = g(1, \dots, 1, t_m', x_{m + 1}, \dots, x_d)
\end{align*}
for every $t_m < t_m' < 1$.
Next, note that
\begin{align*}
    &\lim_{n \rightarrow \infty} g\Big(1 - \frac{1}{n}, \dots, 1 - \frac{1}{n}, x_{m + 1}, \dots, x_d\Big) 
    \le \lim_{n \rightarrow \infty} g\Big(1, 1 - \frac{1}{n}, \dots, 1 - \frac{1}{n}, x_{m + 1}, \dots, x_d\Big) \\
    &\qquad \le \cdots \le \lim_{n \rightarrow \infty} g\Big(1, \dots, 1, 1 - \frac{1}{n}, x_{m + 1}, \dots, x_d\Big) = \lim_{t_{m} \uparrow 1} g(1, \dots, 1, t_{m}, x_{m + 1}, \dots, x_d).
\end{align*}
Thus, it remains to prove that 
\begin{equation}\label{eq:compare-limits-for-g}
    \lim_{n \rightarrow \infty} g\Big(1 - \frac{1}{n}, \dots, 1 - \frac{1}{n}, x_{m + 1}, \dots, x_d\Big) \ge \lim_{t_{m} \uparrow 1} g(1, \dots, 1, t_{m}, x_{m + 1}, \dots, x_d).
\end{equation}
Fix $\epsilon > 0$. 
Then, there exists $s_m < 1$ such that 
\begin{equation*}
    g(1, \dots, 1, s_m, x_{m + 1}, \dots, x_d) > \lim_{t_{m} \uparrow 1} g(1, \dots, 1, t_{m}, x_{m + 1}, \dots, x_d) - \frac{\epsilon}{m}.
\end{equation*}
Similarly, there exists $s_{m - 1} < 1$ such that
\begin{align*}
    g(1, \dots, 1, s_{m - 1}, s_m, x_{m + 1}, \dots, x_d) &> g(1, \dots, 1, 1, s_m, x_{m + 1}, \dots, x_d) - \frac{\epsilon}{m} \\
    &> \lim_{t_{m} \uparrow 1} g(1, \dots, 1, 1, t_{m}, x_{m + 1}, \dots, x_d) - \frac{2\epsilon}{m}.
\end{align*}
Repeating this argument, we can find $s_1, \dots, s_m < 1$ for which
\begin{equation*}
    g(s_1, \dots, s_m, x_{m + 1}, \dots, x_d) 
    > \lim_{t_{m} \uparrow 1} g(1, \dots, 1, t_{m}, x_{m + 1}, \dots, x_d) - \epsilon.
\end{equation*}
This implies that
\begin{align*}
    \lim_{n \rightarrow \infty} g\Big(1 - \frac{1}{n}, \dots, 1 - \frac{1}{n}, x_{m + 1}, \dots, x_d\Big) &\ge g(s_1, \dots, s_m, x_{m + 1}, \dots, x_d) \\
    &> \lim_{t_{m} \uparrow 1} g(1, \dots, 1, t_{m}, x_{m + 1}, \dots, x_d) - \epsilon.
\end{align*}
Since $\epsilon > 0$ can be arbitrarily small, this proves \eqref{eq:compare-limits-for-g}, and therefore $g$ is well-defined at $\mathbf{x}$.
By induction, we can see that $g$ is well-defined at every point in $[0, 1]^d$.

\noindent\textbf{Step 3-4: Entire monotonicity of $g$.}

Next, we show that $g$ is entirely monotone. 
To prove this, we use the following lemma, which derives the nonpositivity of the divided difference of $D\triangle f$ of order $\mathbf{p} = (p_1, \dots, p_d)$ from the nonpositivity of the divided difference of $f$ of order $(p_1 + 1, \dots, p_d + 1)$ for $\mathbf{p} \in \{0, 1\}^d \setminus \{\zerovec\}$. 
The proof of Lemma \ref{lem:div-diff-first-order-ineq} is deferred to Appendix \ref{pf:div-diff-first-order-ineq}.

\begin{lemma}\label{lem:div-diff-first-order-ineq}
    Fix $\mathbf{p} = (p_1, \dots, p_d) \in \{0, 1\}^d \setminus \{\zerovec\}$. 
    Suppose $f$ is a real-valued function on $[0, 1]^d$ whose divided difference of order $(p_1 + 1, \dots, p_d + 1)$ is nonpositive.
    Then, for every $0 \le x_1^{(k)} < \cdots < x_{p_k + 1}^{(k)} \le 1$, $k = 1, \dots, d$, we have
    \begin{equation*}
    \begin{bmatrix}
        x_1^{(1)}, \dots, x_{p_1 + 1}^{(1)} & \\ 
        \vdots                              &; \ D\triangle f \\
        x_1^{(d)}, \dots, x_{p_d + 1}^{(d)} &
    \end{bmatrix}
    \le 0,
    \end{equation*}
    provided that $D\triangle f(x_{i_1}^{(1)}, \dots, x_{i_d}^{(d)})$ exists for every $i_k = 1, \dots, p_k + 1$, $k = 1, \dots, d$.
\end{lemma}

Fix $\mathbf{p} = (p_1, \dots, p_d) \in \{0, 1\}^d \setminus \{\zerovec\}$. 
We prove that the divided difference of $g$ of order $\mathbf{p}$ with respect to the points $(x^{(k)}_i, 1 \le i \le p_k+1, 1 \le k \le d)$ is nonnegative for every $0 \le x_1^{(k)} < \dots < x_{p_k + 1}^{(k)} \le 1$, $k = 1, \dots, d$.
We first consider the case where $x_{p_k + 1}^{(k)} < 1$ for all $k = 1, \dots, d$.
Since $D\triangle f$ exists almost everywhere, there exists a sequence of positive real vectors $\{(\delta_{n}^{(1)}, \dots, \delta_{n}^{(d)}): n \in \mathbb{N}\}$ such that $\lim_{n \rightarrow \infty} \delta_n^{(k)} = 0$ for each $k = 1, \dots, d$, and $D \triangle f(x_{i_1}^{(1)} + \delta_n^{(1)}, \dots, x_{i_d}^{(d)} + \delta_n^{(d)})$ exists for every $i_k = 1, \dots, p_k + 1$, $k = 1, \dots, d$ and for every $n \in \mathbb{N}$.
From the monotonicity \eqref{eq:mono-derivative} of $D\triangle f$ and the definition \eqref{eq:def-of-g} of $g$, it follows that
\begin{equation*}
    g\big(x_{i_1}^{(1)}, \dots, x_{i_d}^{(d)}\big) = - \lim_{n \rightarrow \infty} D \triangle f \big(x_{i_1}^{(1)} + \delta_n^{(1)}, \dots, x_{i_d}^{(d)} + \delta_n^{(d)}\big)
\end{equation*}
for every $i_k = 1, \dots, p_k + 1$, $k = 1, \dots, d$.
This implies that
\begin{equation*}
    \begin{bmatrix}
        x_1^{(1)}, \dots, x_{p_1 + 1}^{(1)} & \\ 
        \vdots                              &; \ g \\
        x_1^{(d)}, \dots, x_{p_d + 1}^{(d)} &
    \end{bmatrix}
    = - \lim_{n \rightarrow \infty} 
    \begin{bmatrix}
        x_1^{(1)} + \delta_n^{(1)}, \dots, x_{p_1 + 1}^{(1)} + \delta_n^{(1)} & \\ 
        \vdots                              &; \ D\triangle f \\
        x_1^{(d)} + \delta_n^{(d)}, \dots, x_{p_d + 1}^{(d)} + \delta_n^{(d)} &
    \end{bmatrix}.
\end{equation*}
Since $\max_k (p_k + 1) = 2$, the divided difference of $f$ of order $(p_1 + 1, \dots, p_d + 1)$ is nonpositive.
By Lemma \ref{lem:div-diff-first-order-ineq}, we thus have 
\begin{equation*}
    \begin{bmatrix}
        x_1^{(1)} + \delta_n^{(1)}, \dots, x_{p_1 + 1}^{(1)} + \delta_n^{(1)} & \\ 
        \vdots                              &; \ D\triangle f \\
        x_1^{(d)} + \delta_n^{(d)}, \dots, x_{p_d + 1}^{(d)} + \delta_n^{(d)} &
    \end{bmatrix}
    \le 0
\end{equation*}
for every $n \in \mathbb{N}$. 
This proves that the divided difference of $g$ order $\mathbf{p}$ with respect to the points $(x^{(k)}_i, 1 \le i \le p_k+1, 1 \le k \le d)$ is nonnegative.

Now, we consider the case where $x_{p_k + 1}^{(k)} = 1$ for some $k \in \{1, \dots, d\}$.
For notational convenience, we assume that $x_{p_k + 1}^{(k)} = 1$ for $k \le m$ and $x_{p_k + 1}^{(k)} < 1$ for $k > m$.
Then, by the definition \eqref{eq:def-of-g} of $g$ and the result established in the previous paragraph, we have 
\begin{equation*}
    \begin{bmatrix}
        x_1^{(1)}, \dots, x_{p_1 + 1}^{(1)} & \\ 
        \vdots                              &; \ g \\
        x_1^{(d)}, \dots, x_{p_d + 1}^{(d)} &
    \end{bmatrix}
    = \lim_{t_1 \uparrow 1} \cdots \lim_{t_m \uparrow 1} 
    \begin{bmatrix}
        x_1^{(1)}, \dots, x_{p_1}^{(1)}, t_1 & \multirow{6}{*}{; \ $g$ } \\ 
        \vdots                               & \\
        x_1^{(m)}, \dots, x_{p_m}^{(m)}, t_m & \\ 
        x_1^{(m + 1)}, \dots, x_{p_{m + 1} + 1}^{(m + 1)}  & \\
        \vdots                               & \\
        x_1^{(d)}, \dots, x_{p_d + 1}^{(d)}  &
    \end{bmatrix}
    \ge 0,
\end{equation*}
which means that the divided difference of $g$ order $\mathbf{p}$ with respect to the points $(x^{(k)}_i, 1 \le i \le p_k+1, 1 \le k \le d)$ is nonnegative in this case as well.
Thus, we can conclude that $g$ is entirely monotone.

\noindent\textbf{Step 3-5: Coordinate-wise right and left continuity of $g$.}

It is clear from the definition that $g$ is coordinate-wise left-continuous at each point $(x_k, k \in S) \in [0, 1]^{|S|} \setminus [0, 1)^{|S|}$ with respect to all the $k^{\text{th}}$ coordinates where $x_k = 1$.
Hence, to show that $g$ satisfies the conditions of Lemma \ref{lem:alt-characterization}, it remains to verify the (coordinate-wise) right-continuity of $g$.
The right-continuity of $g$ on $[0, 1)^d$ is also clear from the definition of $g$. 
We prove that $g$ is also right-continuous at every point $\mathbf{x} = (x_1, \dots, x_d) \in [0, 1]^d \setminus [0, 1)^d$ by induction on $|\{k: x_k = 1\}|$.

We assume that $g$ is right-continuous at every $\mathbf{x} \in [0, 1]^d \setminus [0, 1)^d$ with $|\{k: x_k = 1\}| \le m - 1$ and show that it is also right-continuous at every $\mathbf{x} \in [0, 1]^d \setminus [0, 1)^d$ with $|\{k: x_k = 1\}| = m$.
Without loss of generality, let $\mathbf{x} = (1, \dots, 1, x_{m + 1}, \dots, x_d)$, where $x_{m + 1}, \dots, x_d < 1$.
Suppose, for contradiction, $g$ is not right-continuous at $\mathbf{x}$.
Then, without loss of generality, we can assume that
\begin{equation*}
    g(1, \dots, 1, x_{m + 1}, \dots, x_d) < \lim_{t_{m + 1} \downarrow x_{m + 1}} g(1, \dots, 1, t_{m + 1}, x_{m + 2}, \dots, x_d).
\end{equation*}
Let $\epsilon = \lim_{t_{m + 1} \downarrow x_{m + 1}} g(1, \dots, 1, t_{m + 1}, x_{m + 2}, \dots, x_d) - g(1, \dots, 1, x_{m + 1}, \dots, x_d) > 0$ and choose any $s_{m + 1} > x_{m + 1}$.
Since 
\begin{equation*}
    g(1, \dots, 1, 1, s_{m + 1}, x_{m + 2}, \dots, x_d) = \lim_{t_{m} \uparrow 1} g(1, \dots, 1, t_m, s_{m + 1}, x_{m + 2}, \dots, x_d),
\end{equation*}
there exists $s_m < 1$ such that
\begin{equation*}
    g(1, \dots, 1, s_m, s_{m + 1}, x_{m + 2}, \dots, x_d) > g(1, \dots, 1, 1, s_{m + 1}, x_{m + 2}, \dots, x_d) - \frac{\epsilon}{2}.
\end{equation*}
For every $t_{m + 1} \in (x_{m + 1}, s_{m + 1})$, the entire monotonicity of $g$ implies that 
\begingroup
\allowdisplaybreaks
\begin{align*}
    g(1, \dots, 1, s_m, t_{m + 1}, x_{m + 2}, \dots, x_d) &\ge g(1, \dots, 1, 1, t_{m + 1}, x_{m + 2}, \dots, x_d) \\
    + \big(g(1, \dots,  1, s_m, s_{m + 1}, &x_{m + 2}, \dots, x_d) - g(1, \dots, 1, 1, s_{m + 1}, x_{m + 2}, \dots, x_d)\big) \\
    &\ge g(1, \dots, 1, 1, t_{m + 1}, x_{m + 2}, \dots, x_d) - \frac{\epsilon}{2}.
\end{align*}
\endgroup
By the right-continuity of $g$ at $(1, \dots, 1, s_m, x_{m + 1}, \dots, x_d)$ with respect to the $(m + 1)^{\text{th}}$ coordinate, we have
\begingroup
\allowdisplaybreaks
\begin{align*}
    &g(1, \dots, 1, s_m, x_{m + 1}, \dots, x_d) = \lim_{t_{m + 1} \downarrow x_{m + 1}} g(1, \dots, 1, s_m, t_{m + 1}, x_{m + 2}, \dots, x_d) \\
    &\qquad \ge \lim_{t_{m + 1} \downarrow x_{m + 1}} g(1, \dots, 1, 1, t_{m + 1}, x_{m + 2}, \dots, x_d) - \frac{\epsilon}{2} 
    = g(1, \dots, 1, 1, x_{m + 1}, x_{m + 2}, \dots, x_d) + \frac{\epsilon}{2} \\
    &\qquad \ge g(1, \dots, 1, s_m, x_{m + 1}, \dots, x_d) + \frac{\epsilon}{2},
\end{align*}
\endgroup
which leads to contradiction.
Hence, $g$ is right-continuous at $\mathbf{x}$. 
By induction, we can conclude that $g$ is right-continuous on the entire domain $[0, 1]^d$.

\noindent\textbf{Step 3-6: $g = - D \triangle f$ almost everywhere.}

Lastly, we show that $g = -D \triangle f$ almost everywhere.
Consider the set
\begin{equation*}
    A = \big\{(x_1, \dots, x_d) \in [0, 1)^d: D \triangle f (x_1, \dots, x_d) \text{ exists and } g(x_1, \dots, x_d) \neq -D\triangle f (x_1, \dots, x_d) \big\}.
\end{equation*}
It suffices to show that $m(A) = 0$, where $m$ denotes the Lebesgue measure.
First, observe that if $x_k < x_k'$ for $k = 1, \dots, d$, and $D\triangle f(x_1, \dots, x_d)$, $D \triangle f(x_1', \dots, x_d')$ exist, then by \eqref{eq:mono-derivative} and \eqref{eq:def-of-g}, 
\begin{equation*}
    -D\triangle f(x_1, \dots, x_d) \le g(x_1, \dots, x_d) \le -D\triangle f(x_1', \dots, x_d') \le g(x_1', \dots, x_d').
\end{equation*}
This implies that, for each $\mathbf{c} = (c_1, \dots, c_d) \in \R^d$, there exists an injective map from the set 
\begin{equation*}
    A_{\mathbf{c}} := \big\{u \in \mathbb{R}: u \cdot (1, \dots, 1) + \mathbf{c} \in A\big\}
\end{equation*}
to $\mathbb{Q}$, the set of rational numbers, 
from which it follows that $A_{\mathbf{c}}$ is at most countable. 
Let $\mathbf{1} = (1, \dots, 1)$ and let $\mathbf{e}_k = (0, \dots, 0, 1, 0, \dots, 0)$, $k = 1, \dots, d$ be the standard basis vectors of $\R^d$.
Also, let $\mathit{P}$ be a $d \times d$ orthogonal matrix satisfying
\begin{equation*}
    \mathit{P} \Big(\frac{1}{\sqrt{d}} \cdot \mathbf{1}\Big)
    = \mathbf{e}_1
\end{equation*}
and let $\phi_{\mathit{P}}$ be the corresponding linear map defined by $\phi_{\mathit{P}}(\mathbf{x}) = \mathit{P} \mathbf{x}$ for $\mathbf{x} \in \R^d$.
Then, we can derive that
\begin{align*}
    m(A) &= m(\phi(A)) = \int \cdots \int 1\big\{(y_1, \dots, y_d) \in \phi_{\mathit{P}}(A)\big\} \, dy_1 \cdots dy_d \\
    &= \int \cdots \int 1\bigg\{\sum_{k = 1}^{d} y_k \mathbf{e}_k \in \phi_{\mathit{P}}(A)\bigg\} \, dy_1 \cdots dy_d 
    = \int \cdots \int 1\bigg\{\sum_{k = 1}^{d} y_k \phi_{\mathit{P}}^{-1} (\mathbf{e}_k) \in A\bigg\} \, dy_1 \cdots dy_d \\
    &= \int \cdots \int 1\bigg\{\frac{y_1}{\sqrt{d}} \cdot \mathbf{1} + \sum_{k = 2}^{d} y_k \phi_{\mathit{P}}^{-1} (\mathbf{e}_k) \in A\bigg\} \, dy_1 \cdots dy_d 
    = \int \cdots \int 0 \, dy_2 \cdots dy_d = 0,
\end{align*}
where the second-to-last inequality is from the fact that 
\begin{equation*}
    1\bigg\{y_1 \in \R: \frac{y_1}{\sqrt{d}} \cdot \mathbf{1} + \sum_{k = 2}^{d} y_k \phi_{\mathit{P}}^{-1} (\mathbf{e}_k) \in A\bigg\}
\end{equation*}
is at most countable for fixed $y_2, \dots, y_d \in \R$.

To sum up, we showed that $g$ satisfies the conditions of Lemma \ref{lem:alt-characterization} and $g = -D\triangle f$ almost everywhere. 
By \eqref{eq:lojasiewicz-on-alt-sum}, it follows that $g$ is the desired function for the claim \eqref{eq:alt-sum-as-int} when $S = [d]$. 
Using similar arguments, we can find a function $g_S$ for the claim \eqref{eq:alt-sum-as-int} for each nonempty subset $S$.
Consequently, Lemma \ref{lem:alt-characterization} proves that $f \in \tcd$.

\noindent\textbf{Step 4: $f \in \tcd$ and \eqref{int-rest-cond} holds for all $S$ with $|S| > s \Rightarrow f \in \tcds$.}

We wrap up the proof by showing that if we further assume that $f$ satisfies \eqref{int-rest-cond} for every subset $S$ with $|S| > s$, then $f \in \tcds$.
In this case, it follows directly that $\triangle_S f = 0$ for every subset $S$ with $|S| > s$. 
Thus, we can express $f$ as 
\begin{equation*}
    f(x_1, \dots, x_d) = f(0, \dots, 0) + \sum_{\emptyset \neq S \subseteq [d]} \triangle_S f\bigg(\prod_{k \in S} [0, x_k]\bigg) 
    = f(0, \dots, 0) + \sum_{S: 1 \le |S| \le s} \triangle_S f\bigg(\prod_{k \in S} [0, x_k]\bigg)
\end{equation*}
for $(x_1, \dots, x_d) \in [0, 1]^d$.
By using $g_S$ satisfying the claim \eqref{eq:alt-sum-as-int} and Lemma \ref{lem:alt-characterization} again, we can therefore conclude that $f \in \tcds$.
\end{proof}

\subsubsection{Proof of Proposition \ref{prop:alt-characterization-smooth}}
We use the following lemma in our proof of Proposition \ref{prop:alt-characterization-smooth}. 
This lemma can be proved via repeated applications of the fundamental theorem of calculus. 
For details of the proof, we refer readers to \citet[Section 11.4.3]{ki2024mars}.

\begin{lemma}[Lemma 11.25 of \citet{ki2024mars}]\label{lem:mult-fund-calc}
    Suppose $g:[0, 1]^m \rightarrow \R$ is smooth in the sense that, for every $\mathbf{p} \in \{0, 1\}^m$, it has a continuous derivative $g^{(\mathbf{p})}$ on $[0, 1]^m$. 
    For each $\mathbf{p} \in \{0, 1\}^m$, let $S_{\mathbf{p}} = \{k \in \{1, \dots, m\}: p_k = 1\}$ and ${\mathbf{t}}^{(\mathbf{p})} = (t_k, k \in S_p)$.
    Then, $g$ can be expressed as
    \begin{equation*}
        g(x_1, \dots, x_m) = g(0, \dots, 0) + \sum_{\mathbf{p} \in \{0, 1\}^m \setminus \{\zerovec\}}  \int_{\prod_{k \in S_{\mathbf{p}}}[0, x_k]} g^{(\mathbf{p})} (\widetilde{\mathbf{t}}^{(\mathbf{p})}) \, d{\mathbf{t}}^{(\mathbf{p})}
    \end{equation*}
    for $x = (x_1, \dots, x_m) \in [0, 1]^m$, 
    where $\widetilde{\mathbf{t}}^{(\mathbf{p})}$ is an $m$-dimensional extension of ${\mathbf{t}}^{(\mathbf{p})}$ where
    \begin{align*}
    \widetilde{t}^{(\mathbf{p})}_k = 
    \begin{cases}
        t_k &\mbox{if } k \in S_{\mathbf{p}}  \\
        0 &\mbox{otherwise}.
    \end{cases}
    \end{align*}
\end{lemma}

\begin{proof}[Proof of Proposition \ref{prop:alt-characterization-smooth}]
Suppose $f^{(\mathbf{p})}$ exists and is continuous on $[0, 1]^d$ for every $\mathbf{p} \in \{0, 1, 2\}^d$.
Assume further that $f^{(\mathbf{p})} = 0$ for all $\mathbf{p} \in \{0, 1\}^d$ with $\sum_{k} p_k > s$.
For each nonempty subset $S \subseteq [d]$, let $\mathbf{p}_S$ denote the $d$-dimensional vector for which $(p_S)_k = 1$ if $k \in S$, and $(p_S)_k = 0$ otherwise. 
Also, for each $S$, let $f_S$ be the restriction of $f^{(\mathbf{p}_S)}$ to $\prod_{k \in S} [0, 1] \times \prod_{k \notin S} \{0\}$. 
Note that $f_S = 0$ for every $S$ with $|S| > s$.
Since $f$ is smooth in the sense of Lemma \ref{lem:mult-fund-calc}, we can write it as
\begin{align*}
    f(x_1, \dots, x_d) &= f(0, \dots, 0) + \sum_{\emptyset \neq S \subseteq [d]} \int_{\prod_{k \in S}[0, x_k]} f_S((t_k, k \in S)) \, d(t_k, k \in S) \\
    &= f(0, \dots, 0) + \sum_{S: 1 \le |S| \le s} \int_{\prod_{k \in S}[0, x_k]} f_S((t_k, k \in S)) \, d(t_k, k \in S)
\end{align*}
for $(x_1, \dots, x_d) \in [0, 1]^d$.
Also, since $f_S$ are continuous, if $-f_S$ are entirely monotone, then we can apply Lemma \ref{lem:alt-characterization} to conclude that $f \in \tcds$.

We now show that $-f_S$ is entirely monotone for every $S$ with $1 \le |S| \le s$.
Without loss of generality, we verify this only for $S = [m]$. 
Since $f_S$ is smooth in the sense of Lemma \ref{lem:mult-fund-calc}, it can be expressed as
\begin{align*}
    f_S(x_1, \dots, x_m) &= f_S(0, \dots, 0) + \sum_{\mathbf{p} \in \{0, 1\}^m \setminus \{\zerovec\}} \int_{\prod_{k \in S_{\mathbf{p}}}[0, x_k]} f_S^{(\mathbf{p})} (\widetilde{\mathbf{t}}^{(\mathbf{p})}) \, d{\mathbf{t}}^{(\mathbf{p})}
\end{align*}
for $(x_1, \dots, x_m) \in [0, 1]^m$.
For each $\mathbf{l} = (l_1, \dots, l_m) \in \{0, 1\}^m \setminus \{\zerovec\}$, the divided difference of $f_S$ of order $\mathbf{l}$ is thus given by
\begin{equation*}
    \begin{bmatrix}
        x_1^{(1)}, \dots, x_{l_1 + 1}^{(1)} & \\ 
        \vdots                              &; \ f_S    \\
        x_1^{(m)}, \dots, x_{l_m + 1}^{(m)} &   
    \end{bmatrix}
    = \sum_{\mathbf{p}: S_{\mathbf{l}} \subseteq S_{\mathbf{p}}} \int_{\prod_{k \in S_{\mathbf{l}}}(x_1^{(k)}, x_2^{(k)}] \times \prod_{k \in S_{\mathbf{p}} \setminus S_{\mathbf{l}}} [0, x_1^{(k)}]} f_S^{(\mathbf{p})} (\widetilde{\mathbf{t}}^{(\mathbf{p})}) \, d\mathbf{t}^{(p)}.
\end{equation*}
Note that $f_S^{(\mathbf{p})}$ is the restriction of $f^{(p_1 + 1, \dots, p_m + 1, 0, \dots, 0)}$ to $[0, 1]^m \times \{0\}^{d - m}$ and that $\max_k (p_k + 1) = 2$.
Hence, by the assumption that $f^{(\mathbf{p})} \le 0$ for every $\mathbf{p} \in \{0, 1, 2\}^d$ with $\max_k p_k = 2$, we have $f_S^{(\mathbf{p})} \le 0$ for every $\mathbf{p} \in \{0, 1\}^m \setminus \{\zerovec\}$.
This ensures that the divided difference of $f_S$ of order $\mathbf{l}$ is nonpositive for every $\mathbf{l} \in \{0, 1\}^m \setminus \{\zerovec\}$, which means that $- f_S$ is entirely monotone.
\end{proof}

\subsection{Proofs of Propositions in Section \ref{excomp}}\label{pf:excomp}
\subsubsection{Proof of Proposition \ref{prop:reduction-from-popoviciu}}

Note that the least squares criterion only depends on the values of functions at the design points $\mathbf{x}^{(1)}, \dots, \mathbf{x}^{(n)}$. 
Proposition \ref{prop:reduction-from-popoviciu} is thus a direct consequence of the following lemma, which states that, for every function $f \in \tcp$ satisfying \eqref{int-rest-cond} (recall \eqref{int-rest-cond-restated}) for every subset $S$ with $|S| > s$, we can always find a function in $\tcds$ that agrees with $f$ at all $\mathbf{x}^{(1)}, \dots, \mathbf{x}^{(n)}$. 
In other words, there always exists an element of $\tcds$ that is indistinguishable from $f$ at $\mathbf{x}^{(1)}, \dots, \mathbf{x}^{(n)}$. 

\begin{lemma}\label{lem:reduction-from-popoviciu}
    Suppose $f \in \tcp$ satisfies the interaction restriction condition \eqref{int-rest-cond} for every subset $S$ of $[d]$ with $|S| > s$. 
    Then, there exists $g \in \tcds$ such that $g(\mathbf{x}^{(i)}) = f(\mathbf{x}^{(i)})$ for all $i = 1, \dots, n$.
\end{lemma}

\begin{proof}[Proof of Lemma \ref{lem:reduction-from-popoviciu}]
For each $k = 1, \dots, d$, we write 
\begin{equation*}
    \big\{0, x_k^{(1)}, \dots, x_k^{(n)}, 1\big\} = \big\{u_0^{(k)}, \dots, u_{n_k}^{(k)}\big\}
\end{equation*}
where $0 = u_0^{(k)} < \cdots < u_{n_k}^{(k)} = 1$.
As there can be duplications among $x_k^{(i)}$, $n_k$ may be smaller than $n + 1$.
Also, for each $S \subseteq [d]$ with $1 \le |S| \le s$, we let 
\begin{equation*}
  \widebar{L}_S := \bigg(\prod_{k \in S} \big\{u_0^{(k)}, \dots, u_{n_k - 1}^{(k)}\big\}\bigg) \setminus \{\zerovec\}.
\end{equation*}
Note that $\widebar{L}_S \subseteq L_S$ for each $S$.

We construct $f_{\beta, \nu} \in \tcds$ that agrees with $f$ at the design points ${\mathbf{x}}^{(1)}, \dots, {\mathbf{x}}^{(n)}$ explicitly as follows.
First, for each $S \subseteq [d]$ with $1 \le |S| \le s$, $\mathbf{l} = (l_k, k \in S) \in \prod_{k \in S} \{0, 1, \dots, n_k - 1\}$, and $k = 1, \dots, d$, we let
\begin{equation*}
    p_k(S, \mathbf{l}) = 
    \begin{cases}
        0 & \text{if } k \notin S \text{ or } l_k = 0 \\
        l_k - 1 & \text{otherwise}
    \end{cases}
    \quad \text{ and } \quad 
    q_k(S, \mathbf{l}) = 
    \begin{cases}
        0 & \text{if } k \notin S \\
        l_k + 1 & \text{otherwise.}
    \end{cases}
\end{equation*}
Observe that $0 \le q_k(S, \mathbf{l}) - p_k(S, \mathbf{l}) \le 2$ and that $q_k(S, \mathbf{l}) - p_k(S, \mathbf{l}) = 2$ if $k \in S$ and $l_k \neq 0$.
It thus follows that $\max_k q_k(S, \mathbf{l}) - p_k(S, \mathbf{l}) = 2$, unless $\mathbf{l} = (0, k \in S)$.
Next, we let $\beta_0 = f(0, \dots, 0)$ and, for each $S$, let 
\begin{equation}\label{popoviciu-discrete-beta}
    \beta_S =
    \begin{bmatrix}
        u_{p_1(S, \mathbf{l})}^{(1)}, \dots, u_{q_1(S, \mathbf{l})}^{(1)} & \\ 
        \vdots                                          &; \ f \\
        u_{p_d(S, \mathbf{l})}^{(d)}, \dots, u_{q_d(S, \mathbf{l})}^{(d)} &
    \end{bmatrix}.
\end{equation}
Also, for each $S$, we let $\nu_S$ be the discrete measure supported on $\widebar{L}_S$ ($\subseteq L_S$) for which
\begin{equation}\label{popoviciu-discrete-nu}
    \nu_S\big(\big\{\big(u_{l_k}^{(k)}, k \in S\big)\big\}\big) = - \bigg(\prod_{\substack{k \in S \\ l_k \neq 0}} \big(u_{l_k + 1}^{(k)} - u_{l_k - 1}^{(k)}\big)\bigg)\cdot 
    \begin{bmatrix}
        u_{p_1(S, \mathbf{l})}^{(1)}, \dots, u_{q_1(S, \mathbf{l})}^{(1)} & \\ 
        \vdots                                          &; \ f \\
        u_{p_d(S, \mathbf{l})}^{(d)}, \dots, u_{q_d(S, \mathbf{l})}^{(d)} &
    \end{bmatrix}
\end{equation}
for $\mathbf{l} = (l_k, k \in S) \in \prod_{k \in S} \{0, 1, \dots, n_k - 1\} \setminus \{\zerovec\}$.
Since $f \in \tcp$, we have
\begin{equation*}
    \begin{bmatrix}
        u_{p_1(S, \mathbf{l})}^{(1)}, \dots, u_{q_1(S, \mathbf{l})}^{(1)} & \\ 
        \vdots                                          &; \ f \\
        u_{p_d(S, \mathbf{l})}^{(d)}, \dots, u_{q_d(S, \mathbf{l})}^{(d)} &
    \end{bmatrix}
    \le 0
\end{equation*}
for every $\mathbf{l} \neq (0, k \in S)$, and this ensures that $\nu_S$ is a measure for each $S$.
Now, we prove that $f_{\beta, \nu} \in \tcds$, constructed from these $\beta_0$, $\beta_S$, and $\nu_S$, agrees with $f$ at the design point ${\mathbf{x}}^{(1)}, \dots, {\mathbf{x}}^{(n)}$.
In fact, we can prove a strong statement. 
We can show that $f_{\beta, \nu}(u_{i_1}^{(1)}, \dots, u_{i_d}^{(d)}) = f(u_{i_1}^{(1)}, \dots, u_{i_d}^{(d)})$ for all $i_k \in \{0, \dots, n_k\}$, $k = 1, \dots, d$.
For this, we use the following lemma, which will be proved right after this proof.  
Here, for each $S$ and $v_k \in [0, 1]$ for $k \notin S$, we denote by $f_S(\cdot; (v_k, k \notin S))$ the restriction of $f$ to the subset $x_k = v_k$, $k \notin S$ of $[0, 1]^d$. 
Also, recall the definition of $\triangle$ from \eqref{eq:alt-sum}.

\begin{lemma}\label{lem:red-from-pop-first}
Suppose we are given $S \subseteq [d]$ and $v_k \in [0, 1]$ for each $k \notin S$.
Then, we have that 
\begin{align}\label{eq:red-from-pop-first}
    &\triangle f_S\big(\cdot; (v_k, k \notin S)\big) \bigg(\prod_{k \in S} \big[0, u_{i_k}^{(k)}\big]\bigg) \nonumber \\
    &\quad = \sum_{\mathbf{l} \in \prod_{k \in S}\{0, \dots, n_k - 1\}} \prod_{k \in S} \big(u_{i_k}^{(k)} - u_{l_k}^{(k)}\big)_+ \cdot \prod_{\substack{k \in S \\ l_k \neq 0}} \big(u_{l_k + 1}^{(k)} - u_{l_k - 1}^{(k)} \big) \cdot 
    \begin{bmatrix}
        u_{p_k(S, \mathbf{l})}^{(k)}, \dots, u_{q_k(S, \mathbf{l})}^{(k)}, & k \in S    & \multirow{2}{*}{; \ $f$ } \\ 
        v_k,                                             & k \notin S &
    \end{bmatrix}
\end{align}
for every $i_k \in \{0, 1, \dots, n_k\}$, $k = 1, \dots, d$.
\end{lemma}

It follows from Lemma \ref{lem:red-from-pop-first} that, for each $S$,
\begingroup
\allowdisplaybreaks
\begin{align*}
    &\beta_S \cdot \prod_{k \in S} u_{i_k}^{(k)} - \sum_{\mathbf{l} \in \prod_{k \in S}\{0, \dots, n_k - 1\} \setminus \{\zerovec\}} \bigg(\prod_{k \in S} \big(u_{i_k}^{(k)} - u_{l_k}^{(k)}\big)_+\bigg) \cdot \nu_S\big(\big\{\big(u_{l_k}^{(k)}, k \in S\big)\big\}\big) \\
    &\quad= \sum_{\mathbf{l} \in \prod_{k \in S}\{0, \dots, n_k - 1\}} \bigg(\prod_{k \in S} \big(u_{i_k}^{(k)} - u_{l_k}^{(k)}\big)_+\bigg) \cdot \bigg(\prod_{\substack{k \in S \\ l_k \neq 0}} \big(u_{l_k + 1}^{(k)} - u_{l_k - 1}^{(k)} \big)\bigg) \cdot 
    \begin{bmatrix}
        u_{p_k(S, \mathbf{l})}^{(k)}, \dots, u_{q_k(S, \mathbf{l})}^{(k)}, & k \in S    & \multirow{2}{*}{; \ $f$ } \\ 
        0,                                               & k \notin S &
    \end{bmatrix} 
    \\
    &\quad= \triangle f_S\big(\cdot; (0, k \notin S)\big) \bigg(\prod_{k \in S} \big[0, u_{i_k}^{(k)}\big]\bigg) 
    = \triangle_S f\bigg(\prod_{k \in S} \big[0, u_{i_k}^{(k)}\big]\bigg)
\end{align*}
\endgroup
for every $i_k \in \{0, 1, \dots, n_k\}$, $k = 1, \dots, d$.
Together with the interaction restriction condition \eqref{int-rest-cond} for $S \subseteq [d]$ with $|S| > s$, this implies that
\begingroup 
\allowdisplaybreaks
\begin{align*}
    f_{\beta, \nu}\big(u_{i_1}^{(1)}, \dots, u_{i_d}^{(d)}\big) &= f(0, \dots, 0) + \sum_{S : 1 \le |S| \le s} \beta_S \prod_{k \in S} u_{i_k}^{(k)} \\
    &\qquad- \sum_{S: 1 \le |S| \le s} \sum_{\mathbf{l} \in \prod_{k \in S}\{0, \dots, n_k - 1\} \setminus \{\zerovec\}} \bigg(\prod_{k \in S} \big(u_{i_k}^{(k)} - u_{l_k}^{(k)}\big)_+\bigg) \cdot \nu_S\big(\big\{\big(u_{l_k}^{(k)}, k \in S\big)\big\}\big) \\
    &= f(0, \dots, 0) + \sum_{S: 1 \le |S| \le s} \triangle_S f\bigg(\prod_{k \in S} \big[0, u_{i_k}^{(k)}\big]\bigg) \\
    &= f(0, \dots, 0) + \sum_{\emptyset \neq S \subseteq [d]} \triangle_S f\bigg(\prod_{k \in S} \big[0, u_{i_k}^{(k)}\big]\bigg) 
    = f\big(u_{i_1}^{(1)}, \dots, u_{i_d}^{(d)}\big)
\end{align*}
\endgroup
for every $i_k \in \{0, 1, \dots, n_k\}$, $k = 1, \dots, d$.
\end{proof}

\begin{proof}[Proof of Lemma \ref{lem:red-from-pop-first}]
We prove the lemma by induction on $|S|$.
We first consider the case $|S| = 1$.
In this case, the claim of the lemma can be readily checked as follows.
Without loss of generality, here, we assume that $S = \{1\}$.
\begingroup
\allowdisplaybreaks
\begin{align*}
    \eqref{eq:red-from-pop-first} &= u_{i_1}^{(1)} \cdot 
    \begin{bmatrix}
        u_0^{(1)}, u_1^{(1)} & \multirow{4}{*}{; \ $f$ } \\ 
        v_2    & \\
        \vdots & \\
        v_d    &
    \end{bmatrix} 
    + \sum_{l_1 = 1}^{i_1 - 1} \big(u_{i_1}^{(1)} - u_{l_1}^{(1)}\big) \cdot \big(u_{l_1 + 1}^{(1)} - u_{l_1 - 1}^{(1)} \big)
    \cdot 
    \begin{bmatrix}
        u_{l_1 - 1}^{(1)}, u_{l_1}^{(1)}, u_{l_1 + 1}^{(1)} & \multirow{4}{*}{; \ $f$ } \\ 
        v_2    & \\
        \vdots & \\
        v_d    &
    \end{bmatrix}
    \\
    &= u_{i_1}^{(1)} \cdot 
    \begin{bmatrix}
        u_0^{(1)}, u_1^{(1)} & \multirow{4}{*}{; \ $f$ } \\ 
        v_2    & \\
        \vdots & \\
        v_d    &
    \end{bmatrix} 
    + \sum_{l_1 = 1}^{i_1 - 1} \big(u_{i_1}^{(1)} - u_{l_1}^{(1)}\big) \cdot \Bigg[
    \cdot 
    \begin{bmatrix}
        u_{l_1}^{(1)}, u_{l_1 + 1}^{(1)} & \multirow{4}{*}{; \ $f$ } \\ 
        v_2    & \\
        \vdots & \\
        v_d    &
    \end{bmatrix}
    - 
    \begin{bmatrix}
        u_{l_1 - 1}^{(1)}, u_{l_1}^{(1)} & \multirow{4}{*}{; \ $f$ } \\ 
        v_2    & \\
        \vdots & \\
        v_d    &
    \end{bmatrix}
    \Bigg] \\
    &= \sum_{l_1 = 1}^{i_1} \big(u_{l_1}^{(1)} - u_{l_1 - 1}^{(1)}\big) 
    \begin{bmatrix}
        u_{l_1 - 1}^{(1)}, u_{l_1}^{(1)} & \multirow{4}{*}{; \ $f$ } \\ 
        v_2    & \\
        \vdots & \\
        v_d    &
    \end{bmatrix}
    \\
    &= f\big(u_{i_1}^{(1)}, v_2, \dots, v_d \big) - f(0, v_2, \dots, v_d) = \triangle f_{\{1\}}\big(\cdot; (v_2, \dots, v_d)\big) \big(\big[0, u_{i_1}^{(1)}\big]\big).
\end{align*}
\endgroup

Next, we suppose that the the claim holds for all $S \subseteq [d]$ with $|S| = m - 1$ and prove the claim for subsets $S$ with $|S| = m$.
For notational convenience, here, we only consider the case $S = [m]$, but the same argument applies to every subset $S$ with $|S| = m$. 
Through the same computation as in the case $S = \{1\}$, we can show that 
\begingroup
\allowdisplaybreaks
\begin{align*}
    \eqref{eq:red-from-pop-first} &= \sum_{l_1 = 0}^{n_1 - 1} \cdots \sum_{l_{m - 1} = 0}^{n_{m - 1} - 1} \prod_{k = 1}^{m - 1} \big(u_{i_k}^{(k)} - u_{l_k}^{(k)}\big)_+ \cdot \prod_{\substack{k \in \{1, \dots, m - 1\} \\ l_k \neq 0}} \big(u_{l_k + 1}^{(k)} - u_{l_k - 1}^{(k)} \big) \\
    &\qquad \qquad \cdot \Bigg[u_{i_m}^{(m)} \cdot 
    \begin{bmatrix}
        u_{p_1(S, \mathbf{l})}^{(1)}, \dots, u_{q_1(S, \mathbf{l})}^{(1)} & \\ 
        \vdots                                          & \\
        u_{p_{m - 1}(S, \mathbf{l})}^{(m - 1)}, \dots, u_{q_{m - 1}(S, \mathbf{l})}^{(m - 1)} & \\ 
        u_{0}^{(m)}, u_{1}^{(m)} &; \ f \\ 
        v_{m + 1} & \\ 
        \vdots & \\ 
        v_{d} &
    \end{bmatrix} 
    \\
    &\qquad \qquad \qquad \quad + \sum_{l_m = 1}^{i_m - 1} \big(u_{i_m}^{(m)} - u_{l_m}^{(m)}\big) \cdot \big(u_{l_m + 1}^{(m)} - u_{l_m - 1}^{(m)} \big) \cdot 
    \begin{bmatrix}
        u_{p_1(S, \mathbf{l})}^{(1)}, \dots, u_{q_1(S, \mathbf{l})}^{(1)} & \\ 
        \vdots                                          & \\
        u_{p_{m - 1}(S, \mathbf{l})}^{(m - 1)}, \dots, u_{q_{m - 1}(S, \mathbf{l})}^{(m - 1)} & \\ 
        u_{l_m - 1}^{(m)}, u_{l_m}^{(m)}, u_{l_m + 1}^{(m)} &; \ f \\ 
        v_{m + 1} & \\ 
        \vdots & \\ 
        v_{d} &
    \end{bmatrix}
    \Bigg] \\
    &= \sum_{l_1 = 0}^{n_1 - 1} \cdots \sum_{l_{m - 1} = 0}^{n_{m - 1} - 1} \prod_{k = 1}^{m - 1} \big(u_{i_k}^{(k)} - u_{l_k}^{(k)}\big)_+ \cdot \prod_{\substack{k \in \{1, \dots, m - 1\} \\ l_k \neq 0}} \big(u_{l_k + 1}^{(k)} - u_{l_k - 1}^{(k)} \big) \\
    &\qquad \qquad \cdot \Bigg[
    \begin{bmatrix}
        u_{p_1(S, \mathbf{l})}^{(1)}, \dots, u_{q_1(S, \mathbf{l})}^{(1)} & \\ 
        \vdots                                          & \\
        u_{p_{m - 1}(S, \mathbf{l})}^{(m - 1)}, \dots, u_{q_{m - 1}(S, \mathbf{l})}^{(m - 1)} & \\ 
        u_{i_m}^{(m)} &; \ f \\ 
        v_{m + 1} & \\ 
        \vdots & \\ 
        v_{d} &
    \end{bmatrix}
    -
    \begin{bmatrix}
        u_{p_1(S, \mathbf{l})}^{(1)}, \dots, u_{q_1(S, \mathbf{l})}^{(1)} & \\ 
        \vdots                                          & \\
        u_{p_{m - 1}(S, \mathbf{l})}^{(m - 1)}, \dots, u_{q_{m - 1}(S, \mathbf{l})}^{(m - 1)} & \\ 
        0 &; \ f \\ 
        v_{m + 1} & \\ 
        \vdots & \\ 
        v_{d} &
    \end{bmatrix}
    \Bigg] \\
    &= \triangle f_{[m - 1]}\big(\cdot; \big(u_{i_m}^{(m)}, v_{m + 1}, \dots, v_d\big)\big)\bigg(\prod_{k = 1}^{m - 1} \big[0, u_{i_k}^{(k)}\big]\bigg) - \triangle f_{[m - 1]}(\cdot; (0, v_{m + 1}, \dots, v_d))\bigg(\prod_{k = 1}^{m - 1} \big[0, u_{i_k}^{(k)}\big]\bigg) \\
    &= \triangle f_{[m]}\big(\cdot; (v_{m + 1}, \dots, v_d)\big)\bigg(\prod_{k = 1}^{m} \big[0, u_{i_k}^{(k)}\big]\bigg),
\end{align*}
\endgroup
where the second-to-last equality is due to the induction hypothesis.
This proves the claim of the lemma for $S = [m]$ and concludes the proof.
\end{proof}

\subsubsection{Proof of Proposition \ref{prop:existence-computation}}

Proposition \ref{prop:existence-computation} directly follows from the lemma below, which states that, for every $f_{\beta, \nu} \in \tcds$, there always exists $f_{\gamma, \mu} \in \tcds$ whose measures $\mu_S$ are discrete and supported on $L_S$, and which coincides with $f_{\beta, \nu}$ at all design points ${\mathbf{x}}^{(1)}, \dots, {\mathbf{x}}^{(n)}$.

\begin{lemma}\label{lem:reduction-to-discrete}
    Suppose we are given real numbers $\beta_0$, $\beta_S$, $1 \le |S| \le s$, and finite measures $\nu_{S}$, $1 \le |S| \le s$. 
    Then, there exist real numbers $\gamma_0$, $\gamma_S$, $1 \le |S| \le s$, and discrete measures $\mu_{S}$ supported on $L_S$, $1 \le |S| \le s$ such that 
    $f_{\gamma, \mu}(\mathbf{x}^{(i)}) = f_{\beta, \nu}(\mathbf{x}^{(i)})$ for all $i = 1, \dots, n$.
\end{lemma}

\begin{proof}[Proof of Lemma \ref{lem:reduction-to-discrete}]
Since $f_{\beta, \nu} \in \tcds$, by Proposition \ref{prop:tc-equiv-restricted-interaction}, we have $f_{\beta, \nu} \in \tcp$, and it satisfies the interaction restriction condition \eqref{int-rest-cond} for every subset $S$ of $[d]$ with $|S| > s$.
For such a function, in our proof of Lemma \ref{lem:reduction-from-popoviciu}, we constructed $f_{\gamma, \mu}$ that agrees with it at the design points ${\mathbf{x}}^{(1)}, \dots, {\mathbf{x}}^{(n)}$, with explicit choices of real numbers $\gamma_0$, $\gamma_S$, and discrete measures $\nu_S$ supported on $L_S$. 
Hence, Lemma \ref{lem:reduction-to-discrete} is a direct consequence of Proposition \ref{prop:tc-equiv-restricted-interaction} and our proof of Lemma \ref{lem:reduction-from-popoviciu}.
\end{proof}

\subsection{Proofs of Theorems in Section \ref{rates}}\label{pf:rates}
In the proofs presented in this subsection, $C$ will  represent universal constants. 
When a constant depends on specific variables, this dependence will be indicated with subscripts.
The value of a constant may vary from one line to another, even if the same notation is used.
Also, for simplicity, the exact values of constants are often left unspecified.

\subsubsection{Simple Bound of $V_{\text{design}}(\cdot)$}
Here, in the following lemma, we present a simple upper bound of $V_{\text{design}}(f)$ for $f \in \tcp$.

\begin{lemma}\label{lem:complexity-upper-bound}
  If $f \in \tcp$, then
  \begin{equation*}
    V_{\text{design}}(f) \le \sum_{\emptyset \neq S \subseteq [d]} \bigg[
    \begin{bmatrix}
        0, 1/n_k, & k \in S & \multirow{2}{*}{; \ $f$ } \\ 
        0,        & k \notin S &
    \end{bmatrix}
    -
    \begin{bmatrix}
        (n_k - 1)/n_k, 1, & k \in S & \multirow{2}{*}{; \ $f$ } \\ 
        0,                & k \notin S &
    \end{bmatrix}
    \bigg].
  \end{equation*}
\end{lemma}

\begin{proof}[Proof of Lemma \ref{lem:complexity-upper-bound}]
In this proof, we reuse $f_{\beta, \nu}$ that we constructed in our proof of Lemma \ref{lem:reduction-from-popoviciu}.
Recall our specific choices of $\beta_0$, $\beta_S$, and $\nu_S$ (with $s = d$) from \eqref{popoviciu-discrete-beta} and \eqref{popoviciu-discrete-nu}. 
Note that, as
\begin{equation*}
    \{{\mathbf{x}}^{(1)}, \dots, {\mathbf{x}}^{(n)}\} = \prod_{k = 1}^{d} \Big\{0, \frac{1}{n_k}, \dots, \frac{n_k - 1}{n_k} \Big\},
\end{equation*}
here, we have
\begin{equation*}
    \big\{u_0^{(k)}, \dots, u_{n_k}^{(k)}\big\} = \Big\{0, \frac{1}{n_k}, \dots, \frac{n_k - 1}{n_k}, 1\Big\}
\end{equation*}
for each $k = 1, \dots, d$.
Since $f_{\beta, \nu}$ and $f$ have the same values at all design points ${\mathbf{x}}^{(1)}, \dots, {\mathbf{x}}^{(n)}$, the definition of $V_{\text{design}}(\cdot)$ implies that
\begin{equation*}
    V_{\text{design}}(f) \le V(f_{\beta, \nu}) = \sum_{\emptyset \neq S \subseteq [d]} \nu_S\big([0, 1)^{|S|} \setminus \{\zerovec\}\big).
\end{equation*}
Hence, it suffices to show that
\begin{equation}\label{eq:discrete-measure-size}
    \nu_S\big([0, 1)^{|S|} \setminus \{\zerovec\}\big) = 
    \begin{bmatrix}
        0, 1/n_k, & k \in S & \multirow{2}{*}{; \ $f$ } \\ 
        0,        & k \notin S &
    \end{bmatrix}
    -
    \begin{bmatrix}
        (n_k - 1)/n_k, 1, & k \in S & \multirow{2}{*}{; \ $f$ } \\ 
        0,                & k \notin S &
    \end{bmatrix}
\end{equation}
for each nonempty subset $S \subseteq [d]$.
Here, we prove this equation only for $S = [m]$ for notational convenience, but the other cases can also be proved exactly the same.
When $S = [m]$, our choice \eqref{popoviciu-discrete-nu} of $\nu_S$ leads to
\begingroup
\allowdisplaybreaks
\begin{align*}
    &\nu_S\big([0, 1)^{|S|} \setminus \{\zerovec\}\big) -  
    \begin{bmatrix}
        0, 1/n_k, & k \in S & \multirow{2}{*}{; \ $f$ } \\ 
        0,        & k \notin S &
    \end{bmatrix}
    \\
    &\ = - \sum_{l_1 = 0}^{n_1 - 1} \cdots \sum_{l_m = 0}^{n_m - 1} \bigg(\prod_{\substack{k \in \{1, \dots, m\} \\ l_k \neq 0}} \frac{2}{n_k}\bigg)\cdot 
    \begin{bmatrix}
        p_1(S, \mathbf{l})/n_1, \dots, q_1(S, \mathbf{l})/n_1 & \multirow{6}{*}{; \ $f$ } \\ 
        \vdots                              & \\
        p_m(S, \mathbf{l})/n_m, \dots, q_m(S, \mathbf{l})/n_m & \\ 
        0                                   & \\
        \vdots                              & \\ 
        0                                   &
    \end{bmatrix}
    \\
    &\ = - \sum_{l_1 = 0}^{n_1 - 1} \cdots \sum_{l_{m - 1} = 0}^{n_{m - 1} - 1}  \bigg(\prod_{\substack{k \in \{1, \dots, m - 1\} \\ l_k \neq 0}} \frac{2}{n_k}\bigg) \cdot 
    \Bigg[
    \begin{bmatrix}
        p_1(S, \mathbf{l})/n_1, \dots, q_1(S, \mathbf{l})/n_1                         & \\ 
        \vdots                                                      & \\
        p_{m - 1}(S, \mathbf{l})/n_{m - 1}, \dots, q_{m - 1}(S, \mathbf{l})/n_{m - 1} & \\ 
        0, 1/n_m                                                    &; \ f \\ 
        0                                                           & \\
        \vdots                                                      & \\ 
        0                                                           &
    \end{bmatrix}
    \\
    &\ \qquad \qquad \qquad \qquad \qquad \qquad \qquad \qquad \qquad \quad + \sum_{l_m = 1}^{n_m - 1} \frac{2}{n_m} \cdot
    \begin{bmatrix}
        p_1(S, \mathbf{l})/n_1, \dots, q_1(S, \mathbf{l})/n_1                         & \\ 
        \vdots                                                      & \\
        p_{m - 1}(S, \mathbf{l})/n_{m - 1}, \dots, q_{m - 1}(S, \mathbf{l})/n_{m - 1} & \\ 
        (l_m - 1)/n_m, l_m/n_m, (l_m + 1)/n_m                       &; \ f \\ 
        0                                                           & \\
        \vdots                                                      & \\ 
        0                                                           &
    \end{bmatrix}
    \Bigg]
    \\
    &\ = - \sum_{l_1 = 0}^{n_1 - 1} \cdots \sum_{l_{m - 1} = 0}^{n_{m - 1} - 1}  \bigg(\prod_{\substack{k \in \{1, \dots, m - 1\} \\ l_k \neq 0}} \frac{2}{n_k}\bigg) \cdot 
    \Bigg[
    \begin{bmatrix}
        p_1(S, \mathbf{l})/n_1, \dots, q_1(S, \mathbf{l})/n_1                   & \\ 
        \vdots                                                & \\
        p_{m - 1}(S, \mathbf{l})/n_{m - 1}, \dots, q_{m - 1}(S, \mathbf{l})/n_{m - 1} & \\ 
        0, 1/n_m                                              &; \ f \\ 
        0                                                     & \\
        \vdots                                                & \\ 
        0                                                     &
    \end{bmatrix}
    \\
    & + \sum_{l_m = 1}^{n_m - 1} 
    \begin{bmatrix}
        p_1(S, \mathbf{l})/n_1, \dots, q_1(S, \mathbf{l})/n_1                   & \\ 
        \vdots                                                & \\
        p_{m - 1}(S, \mathbf{l})/n_{m - 1}, \dots, q_{m - 1}(S, \mathbf{l})/n_{m - 1} & \\ 
        l_m/n_m, (l_m + 1)/n_m                                &; \ f \\ 
        0                                                     & \\
        \vdots                                                & \\ 
        0                                                     &
    \end{bmatrix}
    -
    \begin{bmatrix}
        p_1(S, \mathbf{l})/n_1, \dots, q_1(S, \mathbf{l})/n_1                   & \\ 
        \vdots                                                & \\
        p_{m - 1}(S, \mathbf{l})/n_{m - 1}, \dots, q_{m - 1}(S, \mathbf{l})/n_{m - 1} & \\ 
        (l_m - 1)/n_m, l_m/n_m                                &; \ f \\ 
        0                                                     & \\
        \vdots                                                & \\ 
        0                                                     &
    \end{bmatrix}
    \Bigg] \\
    &\ = - \sum_{l_1 = 0}^{n_1 - 1} \cdots \sum_{l_{m - 1} = 0}^{n_{m - 1} - 1}  \bigg(\prod_{\substack{k \in \{1, \dots, m - 1\} \\ l_k \neq 0}} \frac{2}{n_k}\bigg) \cdot 
    \begin{bmatrix}
        p_1(S, \mathbf{l})/n_1, \dots, q_1(S, \mathbf{l})/n_1                   & \\ 
        \vdots                                                & \\
        p_{m - 1}(S, \mathbf{l})/n_{m - 1}, \dots, q_{m - 1}(S, \mathbf{l})/n_{m - 1} & \\ 
        (n_m - 1)/n_m, 1                                      &; \ f \\ 
        0                                                     & \\
        \vdots                                                & \\ 
        0                                                     &
    \end{bmatrix}.
\end{align*}
\endgroup
Repeating this argument with the $(m - 1)^{\text{th}}, \dots, 1^{\text{st}}$ coordinates in turn, we can obtain
\begingroup
\allowdisplaybreaks
\begin{align*}
    &\nu_S\big([0, 1)^{|S|} \setminus \{\zerovec\}\big) -  
    \begin{bmatrix}
        0, 1/n_k, & k \in S & \multirow{2}{*}{; \ $f$ } \\ 
        0,        & k \notin S &
    \end{bmatrix}
    \\
    &\qquad= -
    \begin{bmatrix}
        (n_1 - 1)/n_1, 1                                      & \multirow{6}{*}{; \ $f$ } \\ 
        \vdots                                                & \\
        (n_m - 1)/n_m, 1                                      & \\ 
        0                                                     & \\
        \vdots                                                & \\ 
        0                                                     &
    \end{bmatrix} 
    = -
    \begin{bmatrix}
        (n_k - 1)/n_k, 1, & k \in S & \multirow{2}{*}{; \ $f$ } \\ 
        0,        & k \notin S &
    \end{bmatrix},
\end{align*}
\endgroup
which proves \eqref{eq:discrete-measure-size} for the case $S = [m]$.
\end{proof}

\subsubsection{Proof of Theorem \ref{thm:risk-bound-fixed}}
The following theorem (Theorem \ref{thm:chatterjee-risk-bound}), established in \citet{chatterjee2014new}, plays a central role in our proof of Theorem \ref{thm:risk-bound-fixed}. 
Before presenting the theorem, we briefly outline the problem setting considered in \cite{chatterjee2014new}.

Consider the model
\begin{equation}\label{gaussian-random-vector-model}
    y_i = \theta_i^* + \xi_i \quad \qt{for $i = 1, \dots, n$},
\end{equation}
where $\mathbf{y} = (y_1, \dots, y_n)$ is the observation vector, $\theta^* = (\theta^*_1, \dots, \theta^*_n)$ is the unknown vector to be estimated, and $\xi = (\xi_1, \dots, \xi_n)$ is a vector of Gaussian errors with $\xi_i \overset{\text{i.i.d.}}{\sim} N(0, \sigma^2)$.
Also, let $K \subseteq \R^n$ be a closed convex set.
The least squares estimator of $\theta^*$ over $K$ based on the observations $y_1, \dots, y_n$ is then defined as 
\begin{equation}\label{discrete-convex-lse}
    \hat{\theta}_K = \argmin_{\theta \in K} \|\mathbf{y} - \theta\|_2,
\end{equation}
where $\|\cdot\|_2$ denotes the Euclidean norm.

\citet{chatterjee2014new} studied the risk of this estimator $\hat{\theta}_K$, measured by the expected squared error $\E[\|\hat{\theta}_K - \theta^*\|_2^2]$, where the expectation is taken over the Gaussian errors $\xi_i$.
A key tool in their analysis is the function $h$ defined as
\begin{equation}\label{eq:chatterjee-function}
    h(t) = \E\Big[\sup_{\theta \in K: \|\theta - \theta^*\|_2 \le t} \langle \xi, \theta - \theta^* \rangle \Big] - \frac{t^2}{2}
\end{equation}
for $t \ge 0$.
If no $\theta \in K$ satisfies $\|\theta - \theta^*\|_2 \le t$, $h(t)$ is set to $-\infty$.
\citet[Theorem 1]{chatterjee2014new} proved that this function $h$ has a unique maximizer, denoted by $t^*$.
Furthermore, \citet[Corollary 1.2]{chatterjee2014new}, restated as Theorem \ref{thm:chatterjee-risk-bound} below, provides an upper bound on the expected squared error $\E[\|\hat{\theta}_K - \theta^*\|_2^2]$ in terms of $t^*$.
    
\begin{theorem}[Corollary 1.2 of \citet{chatterjee2014new}]\label{thm:chatterjee-risk-bound}
    The expected squared error can be bounded as
    \begin{equation*}
        \E\big[\|\hat{\theta}_K - \theta^*\|_2^2\big] \le C(\sigma^2 + {t^*}^2)
    \end{equation*}
    for some universal constant $C$.
\end{theorem}

However, determining the maximizer $t^*$ of the function $h$ is often quite challenging. 
Fortunately, we can utilize the following result in \cite{chatterjee2014new} to obtain an upper bound for $t^*$, 
which can then be used instead to bound the expected squared error. 

\begin{proposition}[Proposition 1.3 of \citet{chatterjee2014new}]\label{prop:chatterjee-maximizer}
    If $h(t_1) \ge h(t_2)$ for some $0 \le t_1 < t_2$, then  $t^* \le t_2$.
\end{proposition}

\begin{remark}[Well-specified case]
    If $\theta^* \in K$, it holds that $h(0) = 0$. 
    Thus, in this case, for every $t_2 > 0$ with $h(t_2) \le 0$, we have $t^* \le t_2$.
\end{remark}

We now establish a connection between $\hat{f}_{\text{TC}}^{d, s}$ and $\hat{\theta}_K$ with an appropriate closed convex set $K$.
Consider the index set
\begin{equation*}
    I_0 = \prod_{k = 1}^{d} \{0, 1, \dots, n_k - 1\}.
\end{equation*}
We first re-index ${\mathbf{x}}^{(1)}, \dots, {\mathbf{x}}^{(n)}$ using $I_0$, so that ${\mathbf{x}}^{(\mathbf{i})} = (i_1/n_1, \dots, i_d/n_d)$ for $\mathbf{i} = (i_1, \dots, i_d) \in I_0$. 
Similarly, the observations $y_1, \dots, y_n$ are re-indexed as $(y_{\mathbf{i}}, \mathbf{i} \in I_0)$.
Next, define
\begin{equation}\label{discrete-tc-set}
    T = \big\{(f({\mathbf{x}}^{(\mathbf{i})}), \mathbf{i} \in I_0): f \in \tcds\big\}.
\end{equation}
Note that $T$ is a closed convex subset of $\R^{|I_0|}$.
Also, let $\hat{\theta}_T \in \R^{|I_0|}$ denote the least squares estimator of $\theta^*$ over $T$:
\begin{equation*}
    \hat{\theta}_T = \argmin_{\theta \in T} \|\mathbf{y} - \theta\|_2,
\end{equation*}
where $\mathbf{y} = (y_{\mathbf{i}}, \mathbf{i} \in I_0)$.
Then, by the definition \eqref{discrete-tc-set} of $T$, it is clear that the values of $\hat{f}_{\text{TC}}^{d, s}$ at the design points are equal to the components of $\hat{\theta}_T$.
That is,
\begin{equation*}
    \hat{f}_{\text{TC}}^{d, s}\Big(\frac{i_1}{n_1}, \dots, \frac{i_d}{n_d}\Big) = (\hat{\theta}_T)_{\mathbf{i}}
\end{equation*}
for all $\mathbf{i} = (i_1, \dots, i_d) \in I_0$.
Consequently, the risk $R_F(\hat{f}_{\text{TC}}^{d, s}, f^*)$ of $\hat{f}_{\text{TC}}^{d, s}$ can be expressed in terms of the expected squared error of $\hat{\theta}_T$:
\begin{equation*}
    R_F(\hat{f}_{\text{TC}}^{d, s}, f^*) = \frac{1}{n} \cdot \E\big[\|\hat{\theta}_T - \theta^*\|_2^2 \big],
\end{equation*}
where $\theta^* = (f^*({\mathbf{x}}^{(\mathbf{i})}), \mathbf{i} \in I_0)$.
This relationship provides a way to analyze the risk $R_F(\hat{f}_{\text{TC}}^{d, s}, f^*)$ of $\hat{f}_{\text{TC}}^{d, s}$ by studying the expected squared error $\E[\|\hat{\theta}_T - \theta^*\|_2^2]$ of $\hat{\theta}_T$.

We also employ the following concepts and notations in our proof of Theorem \ref{thm:risk-bound-fixed}.
For $\theta \in \R^{|I_0|}$ and a vector of nonnegative integers $\mathbf{p} = (p_1, \dots, p_d)$, the discrete difference of $\theta$ of order $\mathbf{p}$ at $(i_1, \dots, i_d)$, where $(i_1, \dots, i_d) \in I_0$ and $i_k \ge p_k$ for $k = 1, \dots, d$, is defined as 
\begin{equation}\label{discrete-differences}
    (D^{(\mathbf{p})} \theta)_{i_1, \dots, i_d} =
    \frac{\prod_{k = 1}^{d} p_k !}{\prod_{k = 1}^{d} n_k^{p_k}} \cdot 
    \begin{bmatrix}
        (i_1 - p_1)/n_1, \dots, i_1/n_1 & \\ 
        \vdots                          &; \ f \\
        (i_d - p_d)/n_d, \dots, i_d/n_d &
    \end{bmatrix},
\end{equation} 
where $f: [0, 1]^d \rightarrow \R$ is any function satisfying $f(j_1/n_1, \dots, j_d/n_d) = \theta_{\mathbf{j}}$ for all $\mathbf{j} = (j_1, \dots, j_d) \in I_0$.
Since \eqref{discrete-differences} only depends on the values of $f$ at $(j_1/n_1, \dots, j_d/n_d)$ for $\mathbf{j} = (j_1, \dots, j_d) \in I_0$, the discrete difference $(D^{(\mathbf{p})} \theta)_{i_1, \dots, i_d}$ is well-defined; it is independent of the choice of $f$.
These discrete differences $(D^{(\mathbf{p})} \theta)_{i_1, \dots, i_d}$ are consistent with the definitions provided in \citet[Section 9]{ki2024mars}.

\begin{example}[d = 2]
Let $\theta = (\theta_{\mathbf{i}}, \mathbf{i} \in I_0) \in \R^{|I_0|}$.
For $\mathbf{p} = (1, 0)$: 
\begin{equation*}
    (D^{(1, 0)} \theta)_{i_1, i_2} = \frac{1}{n_1} 
    \begin{bmatrix}
        (i_1 - 1)/n_1, i_1/n_1 & \multirow{2}{*}{; \ $f$ } \\ 
        i_2/n_2                & 
    \end{bmatrix}
    = f\Big(\frac{i_1}{n_1}, \frac{i_2}{n_2}\Big) - f\Big(\frac{i_1 - 1}{n_1}, \frac{i_2}{n_2}\Big) 
    = \theta_{i_1, i_2} - \theta_{i_1 - 1, i_2}
\end{equation*}
provided that $i_1 \ge 1$.
For $\mathbf{p} = (1, 1)$:
\begin{align*}
    (D^{(1, 1)} \theta)_{i_1, i_2} &= \frac{1}{n_1 n_2} 
    \begin{bmatrix}
        (i_1 - 1)/n_1, i_1/n_1 & \multirow{2}{*}{; \ $f$ } \\ 
        (i_2 - 1)/n_2, i_2/n_2 & 
    \end{bmatrix}
    \\
    &= f\Big(\frac{i_1}{n_1}, \frac{i_2}{n_2}\Big) - f\Big(\frac{i_1 - 1}{n_1}, \frac{i_2}{n_2}\Big) 
    - f\Big(\frac{i_1}{n_1}, \frac{i_2 - 1}{n_2}\Big) + f\Big(\frac{i_1 - 1}{n_1}, \frac{i_2 - 1}{n_2}\Big) \\
    &= \theta_{i_1, i_2} - \theta_{i_1 - 1, i_2} - \theta_{i_1, i_2 - 1} + \theta_{i_1 - 1, i_2 - 1}
\end{align*}
provided that $i_1, i_2 \ge 1$. 
For $\mathbf{p} = (2, 0)$:
\begingroup
\allowdisplaybreaks
\begin{align*}
    (D^{(2, 0)} \theta)_{i_1, i_2} &= \frac{2}{n_1^2} 
    \begin{bmatrix}
        (i_1 - 2)/n_1, (i_1 - 1)/n_1, i_1/n_1 & \multirow{2}{*}{; \ $f$ } \\ 
        i_2/n_2                & 
    \end{bmatrix}
    \\
    &= f\Big(\frac{i_1}{n_1}, \frac{i_2}{n_2}\Big) - 2 f\Big(\frac{i_1 - 1}{n_1}, \frac{i_2}{n_2}\Big) + f\Big(\frac{i_1 - 2}{n_1}, \frac{i_2}{n_2}\Big)
    = \theta_{i_1, i_2} - 2 \theta_{i_1 - 1, i_2} + \theta_{i_1 - 2, i_2}
\end{align*}
\endgroup
provided that $i_1 \ge 2$.
\end{example}

If $\theta \in T$, then we can select $f$ satisfying the required condition from $\tcds$. 
Let $f_0$ denote such a function. 
Since the divided difference of $f_0$ of order $\mathbf{p}$ is nonpositive for every $\mathbf{p} \in \{0, 1, 2\}^d$ with $\max_k p_k = 2$, it then follows that $(D^{(\mathbf{p})} \theta)_{\mathbf{i}} \le 0$ for all $\mathbf{p} \in \{0, 1, 2\}^d$ with $\max_k p_k = 2$ and $\mathbf{i} = (i_1, \dots, i_d) \in I_0$ with $i_k \ge p_k$ for $k = 1, \dots, d$.
Furthermore, using the interaction restriction condition \eqref{int-rest-cond} that $f_0$ satisfies for every subset $S \subseteq [d]$ with $|S| > s$, we can show that $(D^{(\mathbf{p})} \theta)_{\mathbf{i}} = 0$ for every $\mathbf{p} \in \{0, 1\}^d$ with $\sum_{k = 1}^d p_k > s$ and $\mathbf{i} = (i_1, \dots, i_d) \in I_0$ with $i_k \ge p_k$ for $k = 1, \dots, d$.

Additionally, for a vector of positive integers $\mathbf{m} = (m_1, \dots, m_d)$, define 
\begin{equation}\label{general-index-set}
    I(\mathbf{m}) = \prod_{k = 1}^{d} \{0, 1, \dots, m_k - 1\}.
\end{equation}
Note that $I_0 = I(n_1, \dots, n_d)$.
For $\theta \in \R^{|I(\mathbf{m})|}$ and a vector of nonnegative integers $\mathbf{p} = (p_1, \dots, p_d)$, the discrete difference of $\theta$ of order $\mathbf{p}$ at $(i_1, \dots, i_d)$, where $(i_1, \dots, i_d) \in I(\mathbf{m})$ and $i_k \ge p_k$ for $k = 1, \dots, d$, is defined analogously to \eqref{discrete-differences}.

\begin{proof}[Proof of Theorem \ref{thm:risk-bound-fixed}]
Let $h$ be the function defined in \eqref{eq:chatterjee-function} with $K = T$.
We aim to find $t_2 > 0$ such that $h(t_2) \le 0$. 
After that, we will apply Theorem \ref{thm:chatterjee-risk-bound} and Proposition \ref{prop:chatterjee-maximizer} (with $K = T$) to derive an upper bound for the expected squared error $\E[\|\hat{\theta}_T - \theta^*\|_2^2]$. 

To focus on the first term of $h(\cdot)$, we define
\begin{equation*}
    H(t) = h(t) + \frac{t^2}{2} = \E\Big[\sup_{\theta \in T: \|\theta - \theta^*\|_2 \le t} \langle \xi, \theta - \theta^* \rangle \Big]
\end{equation*}
for $t \ge 0$.
For technical reasons, we decompose the index set $I_0$ into $2^d$ subsets with approximately equal size.
For each $k = 1, \dots, d$, define 
\begin{equation*}
    J_0^{(k)} = \Big\{i_k \in \mathbb{Z}: 0 \le i_k \le \frac{n_k}{2} - 1 \Big\}
    \ \text{ and } \
    J_1^{(k)} = \Big\{i_k \in \mathbb{Z}: \frac{n_k - 1}{2} \le i_k \le n_k - 1 \Big\}.
\end{equation*}
Also, for each $\delta \in \{0, 1\}^d$, let
\begin{equation*}
    J_{\delta} = \prod_{k = 1}^{d} J_{\delta_k}^{(k)}.
\end{equation*}
It is clear that $I_0 = \biguplus_{\delta \in \{0, 1\}^d} J_{\delta}$, where $\biguplus$ denotes the disjoint union. 
It thus follows that 
\begin{equation*}
    H(t) = \E\bigg[\sup_{\theta \in T: \|\theta - \theta^*\|_2 \le t} \sum_{\delta \in \{0, 1\}^d} \sum_{\mathbf{i} \in J_{\delta}} \xi_{\mathbf{i}} (\theta_{\mathbf{i}} - \theta_{\mathbf{i}}^*) \bigg] \le \sum_{\delta \in \{0, 1\}^d} \E\bigg[\sup_{\theta \in T: \|\theta - \theta^*\|_2 \le t} \sum_{\mathbf{i} \in J_{\delta}} \xi_{\mathbf{i}} (\theta_{\mathbf{i}} - \theta_{\mathbf{i}}^*) \bigg] 
    = \sum_{\delta \in \{0, 1\}^d} H_{\delta}(t),
\end{equation*}
where
\begin{equation*}
    H_{\delta}(t) = \E\bigg[\sup_{\theta \in T: \|\theta - \theta^*\|_2 \le t} \sum_{\mathbf{i} \in J_{\delta}} \xi_{\mathbf{i}} (\theta_{\mathbf{i}} - \theta_{\mathbf{i}}^*) \bigg]
\end{equation*}
for each $\delta \in \{0, 1\}^d$.

Here, we focus on bounding $H_{\zerovec}(\cdot)$, as the bounds for the other $H_{\delta}(\cdot)$ terms can be obtained analogously.
We further split the index set $J_{\zerovec}$ into dyadic pieces.
For each $k = 1, \dots, d$, let $R_k$ denote the largest integer $r_k$ for which the set 
\begin{equation*}
    L_{r_k}^{(k)} := \Big\{i_k \in \mathbb{Z}:  \frac{n_k}{2^{r_k + 1}} - 1 < i_k \le \frac{n_k}{2^{r_k}} - 1 \Big\} := \big\{a_{r_k}^{(k)}, a_{r_k}^{(k)} + 1, \dots, b_{r_k}^{(k)} \big\}
\end{equation*}
is nonempty.
Observe that $L_{R_k}^{(k)} = \{0\}$ and $2^{R_k} \le n_k < 2^{R_k + 1}$ for all $k = 1, \dots, d$.
Define $\mathcal{R} = \prod_{k = 1}^{d} \{1, \dots, R_k\}$.
Then, $J_{\zerovec}$ can be decomposed as follows:
\begin{equation*}
    J_{\zerovec} = \biguplus_{\mathbf{r} \in \mathcal{R}} L_{\mathbf{r}}, \quad \text{ where } \quad L_{\mathbf{r}} := \prod_{k = 1}^{d} L_{r_k}^{(k)}.
\end{equation*}
For $\theta = (\theta_{\mathbf{i}}, \mathbf{i} \in I_0) \in \R^{|I_0|}$ and $\mathbf{r} = (r_1, \dots, r_d) \in \prod_{k = 1}^{d}\{0, 1, \dots, R_k\}$, let $\theta^{(\mathbf{r})}$ denote the subvector of $\theta$ consisting of the components whose indices belong to $L_{\mathbf{r}}$.

Now, define
\begin{equation*}
    W = \bigg\{(w_{\mathbf{r}}, \mathbf{r} \in \mathcal{R}): w_{\mathbf{r}} \in \{1, \dots, |\mathcal{R}|\} \text{ for each } \mathbf{r} \in \mathcal{R} \text{ and } \sum_{\mathbf{r} \in \mathcal{R}} w_{\mathbf{r}} \le 2 |\mathcal{R}|\bigg\}.
\end{equation*}
Also, for each $\mathbf{w} = (w_{\mathbf{r}}, \mathbf{r} \in \mathcal{R}) \in W$, define
\begin{equation*}
    \Lambda_{\mathbf{w}}(t) = \Big\{\theta \in T: \|\theta - \theta^*\|_2 \le t \text{ and } \|\theta^{(\mathbf{r})} - {\theta^*}^{(\mathbf{r})}\|_2^2 \le \frac{w_{\mathbf{r}} t^2}{|\mathcal{R}|} \text{ for each } \mathbf{r} \in \mathcal{R} \Big\}.
\end{equation*}
We then claim that
\begin{equation}\label{eq:set-T-covering}
    \{\theta \in T: \|\theta - \theta^*\|_2 \le t\} 
    \subseteq \bigcup_{\mathbf{w} \in W} \Lambda_{\mathbf{w}}(t). 
\end{equation}
This can be easily verified as follows. 
Fix $\theta \in T$ and assume that $\|\theta -\theta^*\|_2 \le t$.
For each $\mathbf{r} \in \mathcal{R}$, note that
\begin{equation*}
    \|\theta^{(\mathbf{r})} - {\theta^*}^{(\mathbf{r})}\|_2 \le \|\theta - \theta^*\|_2 \le t. 
\end{equation*}
Thus, there exists $w_{\mathbf{r}} \in \{1, \dots, |\mathcal{R}|\}$ such that
\begin{equation*}
    \frac{(w_{\mathbf{r}} - 1) t^2}{|\mathcal{R}|} \le \|\theta^{(\mathbf{r})} - {\theta^*}^{(\mathbf{r})}\|_2^2 \le \frac{w_{\mathbf{r}} t^2}{|\mathcal{R}|}.
\end{equation*}
Furthermore, since 
\begin{equation*}
    t^2 \ge \|\theta - \theta^*\|_2^2 \ge \sum_{\mathbf{r} \in \mathcal{R}} \|\theta^{(\mathbf{r})} - {\theta^*}^{(\mathbf{r})}\|_2^2 \ge \frac{t^2}{|\mathcal{R}|} \sum_{\mathbf{r} \in \mathcal{R}}(w_{\mathbf{r}} - 1) = t^2\bigg(\frac{1}{|\mathcal{R}|} \sum_{\mathbf{r} \in \mathcal{R}} w_{\mathbf{r}} - 1 \bigg),
\end{equation*}
it follows that $\sum_{\mathbf{r} \in \mathcal{R}} w_{\mathbf{r}} \le 2 |\mathcal{R}|$. 
Consequently, $\theta \in \Lambda_{\mathbf{w}}(t)$, completing the verification.

By \eqref{eq:set-T-covering}, we have
\begin{equation*}
    H_{\zerovec}(t) \le \E\bigg[\max_{\mathbf{w} \in W} \sup_{\theta \in \Lambda_{\mathbf{w}}(t)} \sum_{\mathbf{i} \in J_{\zerovec}} \xi_{\mathbf{i}} (\theta_{\mathbf{i}} - \theta_{\mathbf{i}}^*) \bigg].
\end{equation*}
Using the following lemma, which is based on the concentration inequality for Gaussian random variables, we can switch the order of the expectation and maximum (at the cost of additional terms):
\begin{equation*}
    H_{\zerovec}(t) \le \max_{\mathbf{w} \in W} \E\bigg[\sup_{\theta \in \Lambda_{\mathbf{w}}(t)} \sum_{\mathbf{i} \in J_{\zerovec}} \xi_{\mathbf{i}} (\theta_{\mathbf{i}} - \theta_{\mathbf{i}}^*) \bigg] + \sigma t \sqrt{2 \log|W|} + \sigma t \sqrt{\frac{\pi}{2}}.
\end{equation*}

\begin{lemma}[Lemma D.1 of \citet{guntuboyina2020adaptive}]
    Let $\Theta_1, \dots, \Theta_m$ be subsets of $\R^n$ containing the origin. 
    Then, we have
    \begin{equation*}
        \E\Big[\max_{i = 1, \dots, m} \sup_{\theta \in \Theta_i} \langle \xi, \theta \rangle \Big]
        \le \max_{i = 1, \dots, m} \E \Big[\sup_{\theta \in \Theta_i} \langle \xi, \theta \rangle \Big] 
        + \sigma \Big(\sqrt{2 \log m} + \sqrt{\frac{\pi}{2}}\Big) \cdot \Big(\max_{i = 1, \dots, m} \sup_{\theta \in \Theta_i} \|\theta\|_2\Big),
    \end{equation*}
    where $\xi = (\xi_1, \dots, \xi_n)$ is a vector of Gaussian errors with $\xi_i \overset{\text{i.i.d.}}{\sim} N(0, \sigma^2)$.
\end{lemma}

Since
\begin{equation*}
    |W| \le \sum_{l = |\mathcal{R}|}^{2|\mathcal{R}|} \binom{l - 1}{l - |\mathcal{R}|} \le \sum_{l = |\mathcal{R}|}^{2|\mathcal{R}|} \binom{2|\mathcal{R}| - 1}{l - |\mathcal{R}|} \le 2^{2|\mathcal{R}| - 1},
\end{equation*}
it thus follows that
\begin{equation}\label{eq:bdd-for-H-zerovec}
    H_{\zerovec}(t) \le \max_{\mathbf{w} \in W} \E\bigg[\sup_{\theta \in \Lambda_{\mathbf{w}}(t)} \sum_{\mathbf{i} \in J_{\zerovec}} \xi_{\mathbf{i}} (\theta_{\mathbf{i}} - \theta_{\mathbf{i}}^*) \bigg] + 2 \sigma t \sqrt{|\mathcal{R}|} + \sigma t \sqrt{\frac{\pi}{2}}.
\end{equation}

Fix $\mathbf{w} = (w_{\mathbf{r}}, \mathbf{r} \in \mathcal{R}) \in W$.
Since $J_{\zerovec} = \biguplus_{\mathbf{r} \in \mathcal{R}} L_{\mathbf{r}}$, we have
\begin{align*}
    &\E\bigg[\sup_{\theta \in \Lambda_{\mathbf{w}}(t)} \sum_{\mathbf{i} \in J_{\zerovec}} \xi_{\mathbf{i}} (\theta_{\mathbf{i}} - \theta_{\mathbf{i}}^*) \bigg] 
    = \E\bigg[\sup_{\theta \in \Lambda_{\mathbf{w}}(t)} \sum_{\mathbf{r} \in \mathcal{R}} \langle \xi^{(\mathbf{r})}, \theta^{(\mathbf{r})} - {\theta^*}^{(\mathbf{r})} \rangle \bigg] \\
    &\qquad \le \sum_{\mathbf{r} \in \mathcal{R}} \E\bigg[\sup_{\theta \in \Lambda_{\mathbf{w}}(t)} \langle \xi^{(\mathbf{r})}, \theta^{(\mathbf{r})} - {\theta^*}^{(\mathbf{r})} \rangle \bigg] 
    \le \sum_{\mathbf{r} \in \mathcal{R}} \E\bigg[\sup_{\substack{\theta \in T: \|\theta - \theta^*\|_2 \le t \\ \|\theta^{(\mathbf{r})} - {\theta^*}^{(\mathbf{r})}\|_2 \le t(w_{\mathbf{r}}/|\mathcal{R}|)^{1/2}}} \langle \xi^{(\mathbf{r})}, \theta^{(\mathbf{r})} - {\theta^*}^{(\mathbf{r})} \rangle \bigg].
\end{align*}
For each $\mathbf{r} \in \mathcal{R}$, we now bound 
\begin{equation*}
    \E\bigg[\sup_{\substack{\theta \in T: \|\theta - \theta^*\|_2 \le t \\ \|\theta^{(\mathbf{r})} - {\theta^*}^{(\mathbf{r})}\|_2 \le t(w_{\mathbf{r}}/|\mathcal{R}|)^{1/2}}} \langle \xi^{(\mathbf{r})}, \theta^{(\mathbf{r})} - {\theta^*}^{(\mathbf{r})} \rangle \bigg]
\end{equation*}
using the following result (Theorem \ref{thm:emp-bound-ki}) established in the proof of Theorem 5.1 of \citet{ki2021mars} (an earlier version of \cite{ki2024mars}).

For a vector of positive integers $\mathbf{m} = (m_1, \dots, m_d)$ and $\theta = (\theta_{\mathbf{i}}, \mathbf{i} \in I(\mathbf{m}))$, define 
\begin{equation*}
    V(\theta) = \sum_{\substack{\mathbf{p} \in \{0, 1, 2\}^d \\ \max_j p_j = 2}} \bigg[\bigg(\prod_{k = 1}^{d} m_k^{p_k - \ind\{p_k = 2\}}\bigg) \cdot \sum_{\mathbf{i} \in I^{(\mathbf{p})}(\mathbf{m})} |(D^{(\mathbf{p})} \theta)_{\mathbf{i}}|\bigg],
\end{equation*}
where $\ind\{\cdot\}$ denotes the indicator function.
Here, for each $\mathbf{p} \in \{0, 1, 2\}^d$ with $\max_j p_j = 2$, the index set $I^{(\mathbf{p})}(\mathbf{m})$ is given by $I^{(\mathbf{p})}(\mathbf{m}) = I^{(\mathbf{p})}_1 \times \cdots \times I^{(\mathbf{p})}_d$, where 
\begin{equation*}
    I^{(\mathbf{p})}_k = 
    \begin{cases}
        \{p_k\} & \text{if } p_k \neq 2, \\
        \{2, 3, \dots, m_k - 1\} & \text{if } p_k = 2
    \end{cases}
    \qt{for } k = 1, \dots, d.
\end{equation*}
This $V(\cdot)$ can be viewed as a discrete analogue of $V(\cdot)$ defined in \eqref{complexity-measure} 
(see Section 9 of \citet{ki2024mars}).
Also, it is straightforward to verify that if $(D^{(\mathbf{p})} \theta)_{\mathbf{i}} \le 0$ for all $\mathbf{i} \in I^{(\mathbf{p})}(\mathbf{m})$ and $\mathbf{p} \in \{0, 1, 2\}^d$ with $\max_j p_j = 2$ (which holds when $\theta \in T$ and $\mathbf{m} = (n_1, \dots, n_d)$), then $V(\theta)$ can be expressed as
\begin{equation}\label{discrete-complexity-measure-tc}
    V(\theta) = \sum_{\mathbf{q} \in \{0, 1\}^d \setminus \{\zerovec\}} \bigg[\bigg(\prod_{k \in S_{\mathbf{q}}} m_k \bigg) \cdot \Big( (D^{(\mathbf{q})}\theta)_{(1, k \in S_{\mathbf{q}}) \times (0, k \notin S_{\mathbf{q}})} - (D^{(\mathbf{q})}\theta)_{(m_k - 1, k \in S_{\mathbf{q}}) \times (0, k \notin S_{\mathbf{q}})} \Big)\bigg],
\end{equation}
where $S_{\mathbf{q}} := \{k \in \{1, \dots, d\}: q_k = 1\}$.
Here, for each $\mathbf{q}$, $(1, k \in S_{\mathbf{q}}) \times (0, k \notin S_{\mathbf{q}})$ represents the $d$-dimensional vector whose $k^{\text{th}}$ component is $1$ if $k \in S_{\mathbf{q}}$, 0 otherwise; $(m_k - 1, k \in S_{\mathbf{q}}) \times (0, k \notin S_{\mathbf{q}})$ is similarly defined.
Moreover, for each $V > 0$, define
\begin{align*}
    C_{\mathbf{m}}^{d, s}(V) &= \bigg\{\theta \in \R^{|I(\mathbf{m})|}: V(\theta) \le V \text{ and } (D^{(\mathbf{p})}\theta)_{\mathbf{i}} = 0 \text{ for every } \mathbf{p} \in \{0, 1\}^d \text{ with } \sum_{k = 1}^{d} p_k > s \\
    &\qquad \qquad \qquad \qquad \qquad \qquad \qquad \qquad \qquad \text{ and } \mathbf{i} \in I(\mathbf{m}) \text{ with } i_k \ge p_k \text{ for } k = 1, \dots, d \bigg\}. 
\end{align*}

\begin{theorem}\label{thm:emp-bound-ki}
    There exists some constant $C_d$ depending only on $d$ such that, for every $t > 0$,
    \begingroup
    \allowdisplaybreaks
    \begin{align*}
        \E\Big[\sup_{\theta \in C_{\mathbf{m}}^{d, s}(V): \|\theta - \theta^*\|_2 \le t}\langle \xi, \theta - \theta^* \rangle \Big] 
        &\le C_d \sigma t \log (m_1 \cdots m_d) 
        + C_d \sigma t \log\Big(1 + \frac{V}{t}\Big) \\
        &\quad + C_d \sigma t \bigg[\log\Big(2 + \frac{V(m_1 \cdots m_d)^{1/2}}{t}\Big)\bigg]^{3(2s - 1)/8} \\
        &\quad + C_d \sigma V^{1/4} t^{3/4} (m_1 \cdots m_d)^{1/8} \bigg[\log\Big(2 + \frac{V(m_1 \cdots m_d)^{1/2}}{t}\Big)\bigg]^{3(2s - 1)/8}
    \end{align*}
    \endgroup
    when $s \ge 2$, and 
    \begin{equation*}
        \E\Big[\sup_{\theta \in C_{\mathbf{m}}^{d, 1}(V): \|\theta - \theta^*\|_2 \le t}\langle \xi, \theta - \theta^* \rangle \Big] 
        \le C_d \sigma t \log (m_1 \cdots m_d) 
        + C_d \sigma t \log\Big(1 + \frac{V}{t}\Big) + C_d \sigma V^{1/4} t^{3/4} (m_1 \cdots m_d)^{1/8}
    \end{equation*}
    when $s = 1$.
\end{theorem}

We will apply this theorem to the subvectors $\theta^{(\mathbf{r})}$. 
For each $\mathbf{r} \in \mathcal{R}$, define
\begin{equation*}
    m_{r_k}^{(k)} := |L_{r_k}^{(k)}| = b_{r_k}^{(k)} - a_{r_k}^{(k)} + 1 \qt{ for } k = 1, \dots, d,
\end{equation*}
and let $\mathbf{m}_{\mathbf{r}} = (m_{r_1}^{(1)}, \dots, m_{r_d}^{(d)})$.
We re-index each $\theta^{(\mathbf{\mathbf{r}})}$ with $I(\mathbf{m}_{\mathbf{r}})$ by shifting indices, ensuring that the definition of $V(\cdot)$ is applicable to $\theta^{(\mathbf{r})}$.
We will show that, for each $\mathbf{r} \in \mathcal{R}$, there exists some $V_{\mathbf{r}} > 0$ such that $\theta^{(\mathbf{r})} \in C_{\mathbf{m}_{\mathbf{r}}}^{d, s}(V_{\mathbf{r}})$ for every $\theta \in T$ satisfying $\|\theta - \theta^*\|_2 \le t$.

Fix $\theta \in T$ and assume that $\|\theta - \theta^*\|_2 \le t$.
Since $\theta \in T$, it follows that $(D^{(\mathbf{p})} \theta)_{\mathbf{i}} = 0$ for every $\mathbf{p} \in \{0, 1\}^d$ with $\sum_{k = 1}^d p_k > s$ and $\mathbf{i} = (i_1, \dots, i_d) \in I_0$ with $i_k \ge p_k$ for $k = 1, \dots, d$.
Thus, for every $\mathbf{p} \in \{0, 1\}^d$ with $\sum_{k = 1}^d p_k > s$ and $\mathbf{i} = (i_1, \dots, i_d) \in I(\mathbf{m}_{\mathbf{r}})$ with $i_k \ge p_k$ for $k = 1, \dots, d$, we have 
\begin{equation*}
    (D^{(\mathbf{p})}\theta^{(\mathbf{r})})_{\mathbf{i}} = (D^{(\mathbf{p})} \theta)_{i_1 + a_{r_1}^{(1)}, \dots, i_d + a_{r_d}^{(d)}} = 0.
\end{equation*}

Next, we bound $V(\theta^{(\mathbf{r})})$ for each $\mathbf{r} \in \mathcal{R}$.
To this end, we use the following lemma, which will be proved in Appendix \ref{pf:bounds-finite-derivatives}.
Choose any function $g^* \in \tcds$ that agrees with $f^*$ at the design points $\mathbf{x}^{(\mathbf{i})}$, 
and define
\begin{equation}\label{sup-of-first-derivatives}
\begin{split}
    M:= M(g^*) := \max_{S: 1 \le |S| \le s} \max \bigg(
    &\sup \bigg\{
    \begin{bmatrix}
        0, t_k, & k \in S    & \multirow{2}{*}{; \ $g^*$ } \\ 
        0,      & k \notin S &
    \end{bmatrix}
    : t_k > 0 \ \text{ for } k \in S \bigg\}, \\
    &\quad -\inf \bigg\{
    \begin{bmatrix}
        t_k, 1, & k \in S    & \multirow{2}{*}{; \ $g^*$ } \\ 
        0,      & k \notin S &
    \end{bmatrix}
    : t_k < 1 \ \text{ for } k \in S \bigg\}
    \bigg).
\end{split}
\end{equation}
Since $g^* \in \tcds$, it holds that $M(g^*) < +\infty$.
Also, note that $\theta^* = (f^*({\mathbf{x}}^{(\mathbf{i})}), \mathbf{i} \in I_0) = (g^*({\mathbf{x}}^{(\mathbf{i})}), \mathbf{i} \in I_0)$.

\begin{lemma}\label{lem:bounds-finite-derivatives}
    Assume that $\theta \in T$ and $\|\theta - \theta^*\|_2 \le t$.
    Then, for every $\mathbf{p} = (p_1, \dots, p_d) \in \{0, 1\}^d \setminus \{\zerovec\}$ and $\mathbf{i} = (i_1, \dots, i_d) \in L_{r}$ with $i_k \ge p_k$ for $k = 1, \dots, d$, we have 
    \begin{equation}\label{eq:finite-derivative-lower-bound}
        (D^{(\mathbf{p})}\theta)_{\mathbf{i}} \ge \frac{C_d}{\prod_{k = 1}^{d} n_k^{p_k}} \cdot \Big[- M - t \Big(\frac{2^{r_+}}{n}\Big)^{1/2} \cdot 2^{\sum_{k = 1}^{d} p_k r_k}\Big]
    \end{equation}
    if $\mathbf{r} = (r_1, \dots, r_d) \in \prod_{k = 1}^{d} \{1, \dots, R_k\} = \mathcal{R}$, and
    \begin{equation}\label{eq:finite-derivative-upper-bound}
        (D^{(\mathbf{p})}\theta)_{\mathbf{i}} \le \frac{C_d}{\prod_{k = 1}^{d} n_k^{p_k}} \cdot \Big[M + t \Big(\frac{2^{r_+}}{n}\Big)^{1/2} \cdot 2^{\sum_{k = 1}^{d} p_k r_k}\Big]
    \end{equation}
    if $\mathbf{r} = (r_1, \dots, r_d) \in \prod_{k = 1}^{d} \{0, 1, \dots, R_k\}$, where $r_+ = \sum_{k = 1}^{d} r_k$.
\end{lemma}

Recall that $S_{\mathbf{p}} = \{k \in \{1, \dots, d\}: p_k = 1\}$ for each $\mathbf{p} \in \{0, 1\}^d$.
By Lemma \ref{lem:bounds-finite-derivatives} and \eqref{discrete-complexity-measure-tc}, for each $\mathbf{r} = (r_1, \dots, r_d) \in \mathcal{R}$, we have
\begingroup
\allowdisplaybreaks
\begin{align*}
    V(\theta^{(\mathbf{r})}) &= \sum_{\mathbf{p} \in \{0, 1\}^d \setminus \{\zerovec\}} \prod_{k \in S_{\mathbf{p}}} m_{r_k}^{(k)} \cdot \Big[(D^{(\mathbf{p})} \theta^{(\mathbf{r})})_{(1, k \in S_{\mathbf{p}}) \times (0, k \notin S_{\mathbf{p}})} - (D^{(\mathbf{p})} \theta^{(\mathbf{r})})_{(m^{(k)}_{r_k} - 1, k \in S_{\mathbf{p}}) \times (0, k \notin S_{\mathbf{p}})}\Big] \\
    &= \sum_{\mathbf{p} \in \{0, 1\}^d \setminus \{\zerovec\}} \prod_{k \in S_{\mathbf{p}}} m_{r_k}^{(k)} \cdot \Big[(D^{(\mathbf{p})} \theta)_{(a_{r_k}^{(k)} + 1, k \in S_{\mathbf{p}}) \times (a_{r_k}^{(k)}, k \notin S_{\mathbf{p}})} - (D^{(\mathbf{p})} \theta)_{(b_{r_k}^{(k)}, k \in S_{\mathbf{p}}) \times (a_{r_k}^{(k)}, k \notin S_{\mathbf{p}})}\Big] \\
    &\le \sum_{\mathbf{p} \in \{0, 1\}^d \setminus \{\zerovec\}} \prod_{k = 1}^{d} \Big(\frac{n_k}{2^{r_k}}\Big)^{p_k} \cdot \frac{C_d}{\prod_{k = 1}^{d} n_k^{p_k}} \cdot \Big[M + t \Big(\frac{2^{r_+}}{n}\Big)^{1/2} \cdot 2^{\sum_{k = 1}^{d} p_k r_k}\Big] \\
    &\le \sum_{\mathbf{p} \in \{0, 1\}^d \setminus \{\zerovec\}} C_d \cdot \Big[\frac{M}{2^{\sum_{k = 1}^{d} p_k r_k}} + t \Big(\frac{2^{r_+}}{n}\Big)^{1/2}\Big] 
    \le C_d \cdot \Big[M + t \Big(\frac{2^{r_+}}{n}\Big)^{1/2}\Big] := V_{\mathbf{r}}.
\end{align*}
\endgroup
Here, for the first inequality, we use 
\begin{equation*}
    m_{r_k}^{(k)} = b_{r_k}^{(k)} - a_{r_k}^{(k)} + 1 \le \frac{n_k}{2^{r_k}} - 1 + 1 = \frac{n_k}{2^{r_k}}
\end{equation*}
for $k = 1, \dots, d$.
Thus, we have $\theta^{(\mathbf{r})} \in C_{\mathbf{m}_{\mathbf{r}}}^{d, s}(V_{\mathbf{r}})$ for each $\mathbf{r} \in \mathcal{R}$.

From now on, we handle the two cases $s \ge 2$ and $s = 1$ separately.
Let us first consider the case $s \ge 2$.
By Theorem \ref{thm:emp-bound-ki}, 
for each $\mathbf{r} = (r_1, \dots, r_d) \in \mathcal{R}$, we have
\begin{align}\label{eq:bdd-of-emp-process-subvector}
    &\E\bigg[\sup_{\substack{\theta \in T: \|\theta - \theta^*\|_2 \le t \\ \|\theta^{(\mathbf{r})} - {\theta^*}^{(\mathbf{r})}\|_2 \le t(w_{\mathbf{r}}/|\mathcal{R}|)^{1/2}}} \langle \xi^{(\mathbf{r})}, \theta^{(\mathbf{r})} - {\theta^*}^{(\mathbf{r})} \rangle \bigg] 
    \le \E\bigg[\sup_{\substack{\theta^{(\mathbf{r})} \in C_{\mathbf{m}}^{d, s}(V_{\mathbf{r}}) \\ \|\theta^{(\mathbf{r})} - {\theta^*}^{(\mathbf{r})}\|_2 \le t_{\mathbf{r}}}} \langle \xi^{(\mathbf{r})}, \theta^{(\mathbf{r})} - {\theta^*}^{(\mathbf{r})} \rangle \bigg] \nonumber \\
    \begin{split}
        &\qquad \le C_d \sigma t_{\mathbf{r}} \log |L_{\mathbf{r}}|
        + C_d \sigma t_{\mathbf{r}} \log\Big(1 + \frac{V_{\mathbf{r}}}{t_{\mathbf{r}}}\Big) + C_d \sigma t_{\mathbf{r}} \bigg[\log\Big(2 + \frac{V_{\mathbf{r}} |L_{\mathbf{r}}|^{1/2}}{t_{\mathbf{r}}}\Big)\bigg]^{3(2s - 1)/8} \\
        &\qquad \qquad + C_d \sigma V_{\mathbf{r}}^{1/4} t_{\mathbf{r}}^{3/4} |L_{\mathbf{r}}|^{1/8} \bigg[\log\Big(2 + \frac{V_{\mathbf{r}} |L_{\mathbf{r}}|^{1/2}}{t_{\mathbf{r}}}\Big)\bigg]^{3(2s - 1)/8},
    \end{split}
\end{align}
where $t_{\mathbf{r}} = t (w_{\mathbf{r}}/|\mathcal{R}|)^{1/2}$.
Recall that
\begin{equation*}
    \E\bigg[\sup_{\theta \in \Lambda_{\mathbf{w}}(t)} \sum_{\mathbf{i} \in J_{\zerovec}} \xi_{\mathbf{i}} (\theta_{\mathbf{i}} - \theta_{\mathbf{i}}^*) \bigg] 
    \le \sum_{\mathbf{r} \in \mathcal{R}} \E\bigg[\sup_{\substack{\theta \in T: \|\theta - \theta^*\|_2 \le t \\ \|\theta^{(\mathbf{r})} - {\theta^*}^{(\mathbf{r})}\|_2 \le t(w_{\mathbf{r}}/|\mathcal{R}|)^{1/2}}} \langle \xi^{(\mathbf{r})}, \theta^{(\mathbf{r})} - {\theta^*}^{(\mathbf{r})} \rangle \bigg].
\end{equation*}
Hence, we need to bound the summation of each term in $\eqref{eq:bdd-of-emp-process-subvector}$ over $\mathcal{R}$.
The summation of the first term over $\mathbf{r} \in \mathcal{R}$ can be bounded as
\begin{align*}
    \sum_{\mathbf{r} \in \mathcal{R}} t_{\mathbf{r}} \log |L_{\mathbf{r}}| 
    &\le \sum_{\mathbf{r} \in \mathcal{R}} t \cdot \Big(\frac{w_{\mathbf{r}}}{|\mathcal{R}|}\Big)^{1/2} \log\Big(\prod_{k = 1}^{d} \frac{n_k}{2^{r_k}}\Big) 
    \le \frac{t \log n}{|\mathcal{R}|^{1/2}} \cdot \sum_{\mathbf{r} \in \mathcal{R}} w_{\mathbf{r}}^{1/2} 
    \le t \log n \cdot \bigg(\sum_{\mathbf{r} \in \mathcal{R}} w_{\mathbf{r}}\bigg)^{1/2} \\
    &\le t \log n \cdot (2 |\mathcal{R}|)^{1/2}
    \le C_d t (\log n)^{(d + 2)/2}.
\end{align*}
Here, the first inequality follows from the fact that
\begin{equation*}
    |L_{\mathbf{r}}| = \prod_{k = 1}^{d} |L_{r_k}^{(k)}| = \prod_{k = 1}^{d} m_{r_k}^{(k)} \le \prod_{k = 1}^{d} \frac{n_k}{2^{r_k}}.
\end{equation*}
Also, the second-to-last inequality is due to the condition $\sum_{\mathbf{r} \in \mathcal{R}} w_{\mathbf{r}} \le 2 |\mathcal{R}|$, and the last one follows from
\begin{equation*}
    |\mathcal{R}| = \prod_{k = 1}^{d} R_k \le \prod_{k = 1}^{d} (1 + C \log n_k) 
    \le \bigg[\frac{1}{d} \sum_{k = 1}^{d} (1 + C \log n_k)\bigg]^d 
    = \Big(1 + \frac{C}{d} \cdot \log n\Big)^d 
    \le C_d (\log n)^d.
\end{equation*}
We can bound the the summation of the second term over $\mathbf{r} \in \mathcal{R}$ as follows:
\begingroup
\allowdisplaybreaks
\begin{align*}
    &\sum_{\mathbf{r} \in \mathcal{R}} t_{\mathbf{r}} \log\Big(1 + \frac{V_{\mathbf{r}}}{t_{\mathbf{r}}}\Big) 
    = \sum_{\mathbf{r} \in \mathcal{R}} t \Big(\frac{w_{\mathbf{r}}}{|\mathcal{R}|}\Big)^{1/2} \cdot \log\Big(1 + C_d \cdot \frac{M + t (2^{r_+}/n)^{1/2}}{t \cdot (w_{\mathbf{r}}/|\mathcal{R}|)^{1/2}} \Big) \\
    &\qquad \le \sum_{\mathbf{r} \in \mathcal{R}} t \Big(\frac{w_{\mathbf{r}}}{|\mathcal{R}|}\Big)^{1/2} \cdot \log\Big(1 + C_d \cdot \frac{(M + t) \cdot |\mathcal{R}|^{1/2}}{t} \Big) 
    \le C_d t (\log n)^{d/2} \cdot  \log\Big(1 + C_d \cdot \frac{(M + t) \cdot (\log n)^{d/2}}{t} \Big) \\
    &\qquad \le C_d t (\log n)^{d/2} \cdot \Big[C_d + \log\Big(1 + \frac{M}{t}\Big) + \frac{d}{2} \log\log n \Big] \\
    &\qquad \le C_d t (\log n)^{d/2} + C_d t (\log n)^{d/2} \log\Big(1 + \frac{M}{t}\Big) + C_d t (\log n)^{d/2} \cdot \log \log n,
\end{align*}
\endgroup
where the first inequality follows from that $2^{r_+} \le n$ and $w_{\mathbf{r}} \ge 1$.
Similarly, the summation of the third term over $\mathbf{r} \in \mathcal{R}$ can be bounded as  
\begingroup
\allowdisplaybreaks
\begin{align*}
    \sum_{\mathbf{r} \in \mathcal{R}} t_{\mathbf{r}} \bigg[\log\Big(2 + \frac{V_{\mathbf{r}} |L_{\mathbf{r}}|^{1/2}}{t_{\mathbf{r}}}\Big)\bigg]^{3(2s - 1)/8} &\le \sum_{\mathbf{r} \in \mathcal{R}} t \cdot \Big(\frac{w_{\mathbf{r}}}{|\mathcal{R}|}\Big)^{1/2} \bigg[\log\Big(2 + C_d \cdot \frac{M + t (2^{r_+}/n)^{1/2}}{t \cdot (w_{\mathbf{r}}/|\mathcal{R}|)^{1/2}} \cdot n^{1/2}\Big)\bigg]^{3(2s - 1)/8} \\
    &\le C_d t (\log n)^{d/2} \cdot \bigg[\log\Big(2 + C_d \cdot \frac{(M + t) \cdot n^{1/2}(\log n)^{d/2}}{t} \Big)\bigg]^{3(2s - 1)/8} \\
    &\le C_d t (\log n)^{d/2} \cdot \Big[C_d + \log \Big(1 + \frac{M}{t}\Big) + \frac{1}{2} \log n + \frac{d}{2} \log \log n \Big]^{3(2s - 1)/8} \\
    &\le C_d t (\log n)^{d/2} \Big[\log\Big(1 + \frac{M}{t}\Big)\Big]^{3(2s - 1)/8} + C_d t (\log n)^{(4d + 6s - 3)/8},
\end{align*}
\endgroup
where the first inequality uses the fact that $|L_{\mathbf{r}}| \le n$.
Lastly, the summation of the fourth term over $\mathbf{r} \in \mathcal{R}$ has the following bound:
\begingroup
\allowdisplaybreaks
\begin{align*}
    &\sum_{\mathbf{r} \in \mathcal{R}} V_{\mathbf{r}}^{1/4} t_{\mathbf{r}}^{3/4} |L_{\mathbf{r}}|^{1/8} \bigg[\log\Big(2 + \frac{V_{\mathbf{r}} |L_{\mathbf{r}}|^{1/2}}{t_{\mathbf{r}}}\Big)\bigg]^{3(2s - 1)/8} \\
    &\qquad \le \sum_{\mathbf{r} \in \mathcal{R}} C_d \Big[M + t \Big(\frac{2^{r_+}}{n}\Big)^{1/2}\Big]^{1/4} \cdot \Big(t \Big(\frac{w_{\mathbf{r}}}{|\mathcal{R}|}\Big)^{1/2}\Big)^{3/4} \cdot \bigg(\prod_{k = 1}^{d} \frac{n_k}{2^{r_k - 1}}\bigg)^{1/8} \Big[\log\Big(1 + \frac{M}{t}\Big) + \log n \Big]^{3(2s - 1)/8} \\
    &\qquad \le C_d \sum_{\mathbf{r} \in \mathcal{R}} \Big[M^{1/4} + t^{1/4} \Big(\frac{2^{r_+}}{n}\Big)^{1/8}\Big] \cdot t^{3/4} \Big(\frac{w_{\mathbf{r}}}{|\mathcal{R}|}\Big)^{3/8} \cdot \Big(\frac{n^{1/8}}{2^{r_+/8}}\Big) \cdot \Big[\log\Big(1 + \frac{M}{t}\Big) + \log n \Big]^{3(2s - 1)/8} \\
    &\qquad \le C_d \bigg[M^{1/4} t^{3/4} n^{1/8} \sum_{\mathbf{r} \in \mathcal{R}} \frac{1}{2^{r_+/8}} + t \sum_{\mathbf{r} \in \mathcal{R}} \Big(\frac{w_{\mathbf{r}}}{|\mathcal{R}|}\Big)^{3/8} \bigg] \cdot \Big[\log\Big(1 + \frac{M}{t}\Big) + \log n \Big]^{3(2s - 1)/8}.
\end{align*}
\endgroup
Here, the second inequality is based on that $(x + y)^{1/4} \le x^{1/4} + y^{1/4}$ for $x, y \ge 0$.
Observe that
\begin{equation*}
    \sum_{\mathbf{r} \in \mathcal{R}} \frac{1}{2^{r_+/8}} = \prod_{k = 1}^{d} \sum_{r_k = 1}^{R_k} \frac{1}{2^{r_k/8}} \le C.
\end{equation*}
Also, by H\"older inequality, we have
\begingroup
\allowdisplaybreaks
\begin{equation*}
    \sum_{\mathbf{r} \in \mathcal{R}} \Big(\frac{w_{\mathbf{r}}}{|\mathcal{R}|}\Big)^{3/8} 
    \le \frac{1}{|\mathcal{R}|^{3/8}} \cdot \bigg(\sum_{\mathbf{r} \in \mathcal{R}} w_{\mathbf{r}}\bigg)^{3/8} \cdot \bigg( \sum_{\mathbf{r} \in \mathcal{R}} 1 \bigg)^{5/8} 
    \le C |\mathcal{R}|^{5/8} 
    \le C_d (\log n)^{5d/8}.
\end{equation*}
\endgroup
It thus follows that
\begingroup
\allowdisplaybreaks
\begin{align*}
    &\sum_{\mathbf{r} \in \mathcal{R}} V_{\mathbf{r}}^{1/4} t_{\mathbf{r}}^{3/4} |L_{\mathbf{r}}|^{1/8} \bigg[\log\Big(2 + \frac{V_{\mathbf{r}} |L_{\mathbf{r}}|^{1/2}}{t_{\mathbf{r}}}\Big)\bigg]^{3(2s - 1)/8} 
    \le C_d M^{1/4} t^{3/4} n^{1/8} \Big[\log\Big(1 + \frac{M}{t}\Big)\Big]^{3(2s - 1)/8} \\
    &\qquad \qquad + C_d M^{1/4} t^{3/4} n^{1/8} (\log n)^{3(2s - 1)/8} 
    + C_d t (\log n)^{5d/8} \Big[\log\Big(1 + \frac{M}{t}\Big)\Big]^{3(2s - 1)/8} + C_d t (\log n)^{(5d + 6s - 3)/8}.
\end{align*}
\endgroup
Combining these results, we can derive that
\begingroup
\allowdisplaybreaks
\begin{align*}
    &\E\bigg[\sup_{\theta \in \Lambda_{\mathbf{w}}(t)} \sum_{\mathbf{i} \in J_{\zerovec}} \xi_{\mathbf{i}} (\theta_{\mathbf{i}} - \theta_{\mathbf{i}}^*) \bigg] \le C_d \sigma t (\log n)^{\max((d + 2)/2, (5d + 6s - 3)/8)}
    + C_d \sigma t (\log n)^{d/2} \log\Big(1 + \frac{M}{t}\Big) \\
    &\qquad \qquad \qquad \quad + C_d \sigma t (\log n)^{5d/8} \Big[\log\Big(1 + \frac{M}{t}\Big) \Big]^{3(2s - 1)/8} 
    + C_d \sigma M^{1/4} t^{3/4} n^{1/8} \Big[\log\Big(1 + \frac{M}{t}\Big) \Big]^{3(2s - 1)/8} \\
    &\qquad \qquad \qquad \quad + C_d \sigma M^{1/4} t^{3/4} n^{1/8} (\log n)^{3(2s - 1)/8}.
\end{align*}
\endgroup

Using \eqref{eq:bdd-for-H-zerovec}, we can show that $H_{\zerovec}(t)$ has the same upper bound (with a slightly larger $C_d$). 
Similarly, by analogous arguments, we can derive the same upper bound for $H_{\delta}(t)$ for $\delta \in \{0, 1\}^d \setminus \{\zerovec\}$.
Therefore, we can see that if we define
\begingroup
\allowdisplaybreaks
\begin{align*}
    t_2 &= \max \bigg(\sigma, C_d \sigma (\log n)^{\max((d + 2)/2, (5d + 6s - 3)/8)}, 
    C_d \sigma (\log n)^{d/2} \cdot  \log\Big(1 + \frac{M}{\sigma}\Big), \\
    &\qquad \qquad \quad C_d \sigma (\log n)^{5d/8} \Big[\log\Big(1 + \frac{M}{\sigma}\Big) \Big]^{3(2s - 1)/8}, 
    C_d \sigma^{4/5} M^{1/5} n^{1/10} \Big[\log\Big(1 + \frac{M}{\sigma}\Big) \Big]^{3(2s - 1)/10}, \\
    &\qquad \qquad \quad C_d \sigma^{4/5} M^{1/5} n^{1/10} (\log n)^{3(2s - 1)/10} \bigg),
\end{align*}
\endgroup
then $H(t_2) \le t_2^2/2$, which implies $h(t_2) \le 0$.
As a result, by Theorem \ref{thm:chatterjee-risk-bound} and Proposition \ref{prop:chatterjee-maximizer}, we have
\begingroup
\allowdisplaybreaks
\begin{align*}
    R_F(\hat{f}_{\text{TC}}^{d, s}, f^*) &= \frac{1}{n} \cdot \E\big[\|\hat{\theta}_T - \theta^*\|_2^2 \big] \le \frac{C}{n} \cdot (\sigma^2 + t_2^2) \\
    &\le C_d \Big( \frac{\sigma^2 M^{1/2}}{n} \Big)^{4/5} (\log n)^{3(2s - 1)/5} 
    + C_d \Big( \frac{\sigma^2 M^{1/2}}{n} \Big)^{4/5} \Big[\log\Big(1 + \frac{M}{\sigma}\Big) \Big]^{3(2s - 1)/5} \\
    &\qquad + C_d \cdot \frac{\sigma^2}{n} (\log n)^{5d/4} \Big[\log\Big(1 + \frac{M}{\sigma}\Big) \Big]^{3(2s - 1)/4} 
    + C_d \cdot \frac{\sigma^2}{n} (\log n)^{d} \Big[\log\Big(1 + \frac{M}{\sigma}\Big) \Big]^{2} \\
    &\qquad + C_d \cdot \frac{\sigma^2}{n} (\log n)^{\max(d + 2, (5d + 6s - 3)/4)}.  
\end{align*}
\endgroup
This can be simplified as 
\begin{equation}\label{eq:risk-with-M-s-ge-2}
    R_F(\hat{f}_{\text{TC}}^{d, s}, f^*) \le C_d \Big( \frac{\sigma^2 M^{1/2}}{n} \Big)^{4/5} \bigg[ \log\Big(\Big(1 + \frac{M}{\sigma}\Big) n \Big)\bigg]^{3(2s - 1)/5} 
    + C_d \cdot \frac{\sigma^2}{n} (\log n)^{5d/4} \bigg[ \log\Big(\Big(1 + \frac{M}{\sigma}\Big) n \Big)\bigg]^{3(2s - 1)/4},
\end{equation}
as desired, since $s \ge 2$.

We now consider the case $s = 1$.
In this case, again, by Theorem \ref{thm:emp-bound-ki}, we have
\begingroup
\allowdisplaybreaks
\begin{align*}
    &\E\bigg[\sup_{\substack{\theta \in T: \|\theta - \theta^*\|_2 \le t \\ \|\theta^{(\mathbf{r})} - {\theta^*}^{(\mathbf{r})}\|_2 \le t(w_{\mathbf{r}}/|\mathcal{R}|)^{1/2}}} \langle \xi^{(\mathbf{r})}, \theta^{(\mathbf{r})} - {\theta^*}^{(\mathbf{r})} \rangle \bigg] 
    \le \E\bigg[\sup_{\substack{\theta^{(\mathbf{r})} \in C_{\mathbf{m}}^{d, s}(V_{\mathbf{r}}) \\ \|\theta^{(\mathbf{r})} - {\theta^*}^{(\mathbf{r})}\|_2 \le t_{\mathbf{r}}}} \langle \xi^{(\mathbf{r})}, \theta^{(\mathbf{r})} - {\theta^*}^{(\mathbf{r})} \rangle \bigg] \\
    &\qquad \qquad \le C_d \sigma t_{\mathbf{r}} \log |L_{\mathbf{r}}|
    + C_d \sigma t_{\mathbf{r}} \log\Big(1 + \frac{V_{\mathbf{r}}}{t_{\mathbf{r}}}\Big) + C_d \sigma V_{\mathbf{r}}^{1/4} t_{\mathbf{r}}^{3/4} |L_{\mathbf{r}}|^{1/8}.
\end{align*}
\endgroup
Through similar computations as above, we can show that 
\begin{align*}
    &H_{\zerovec}(t) \le C_d \sigma t (\log n)^{(d + 2)/2}
    + C_d \sigma t (\log n)^{d/2} \log\Big(1 + \frac{M}{t}\Big) + C_d \sigma t (\log n)^{5d/8} + C_d \sigma M^{1/4} t^{3/4} n^{1/8}.
\end{align*}
Also, by analogous arguments, the same upper bound holds for $H_{\delta}(t)$ for $\delta \in \{0, 1\}^d \setminus \{\zerovec\}$.
Thus, in this case, the inequality $h(t_2) \le 0$ holds for
\begin{equation*}
    t_2 = \max\Big( 
    \sigma, C_d \sigma (\log n)^{(d + 2)/2}, 
    C_d \sigma (\log n)^{d/2} \cdot \log\Big(1 + \frac{M}{\sigma}\Big), 
    C_d \sigma (\log n)^{5d/8},
    C_d \sigma^{4/5} M^{1/5} n^{1/10}
    \Big).
\end{equation*}
Consequently, using Theorem \ref{thm:chatterjee-risk-bound} and Proposition \ref{prop:chatterjee-maximizer}, we can derive that
\begin{equation}\label{eq:risk-with-M-s-eq-1}
\begin{split}
    R_F(\hat{f}_{\text{TC}}^{d, 1}, f^*) &= \frac{1}{n} \cdot \E\big[\|\hat{\theta}_T - \theta^*\|_2^2 \big] \le \frac{C}{n} \cdot (\sigma^2 + t_2^2) \\
    &\le C_d \Big( \frac{\sigma^2 M^{1/2}}{n} \Big)^{4/5} 
    + C_d \cdot \frac{\sigma^2}{n} (\log n)^{d} \bigg[ \log\Big(\Big(1 + \frac{M}{\sigma}\Big) n \Big)\bigg]^{2} 
    + C_d \cdot \frac{\sigma^2}{n} (\log n)^{5d/4}.  
\end{split}
\end{equation}

We now show that $M = M(g^*)$ in the bounds \eqref{eq:risk-with-M-s-ge-2} and \eqref{eq:risk-with-M-s-eq-1} can be replaced with $V_{\text{design}}(f^*)$.
Fix $c_S \in \R$ for each subset $S \subseteq [d]$ with $1 \le |S|\le s$, 
and let $A: [0, 1]^d \rightarrow \R$ denote the multi-affine function defined as
\begin{equation*}
    A(x_1, \dots, x_d) = g^*(x_1, \dots, x_d) - \sum_{S: 1 \le |S| \le s} c_S \prod_{k \in S} x_k \qt{for $(x_1, \dots, x_d) \in [0, 1]^d$.}
\end{equation*}
Next, define $\widebar{g} = g^* - A \in \tcds$,
and let $\widebar{y}_{\mathbf{i}} = y_{\mathbf{i}} - A(\mathbf{x}^{(\mathbf{i})})$ for $\mathbf{i} \in I_0$.
It then follows that
\begin{equation}\label{eq:shifted-model}
    \widebar{y}_{\mathbf{i}} = y_{\mathbf{i}} - A(\mathbf{x}^{(\mathbf{i})}) = f^*(\mathbf{x}^{(\mathbf{i})}) - A(\mathbf{x}^{(\mathbf{i})}) + \xi_{\mathbf{i}} = g^*(\mathbf{x}^{(\mathbf{i})}) - A(\mathbf{x}^{(\mathbf{i})}) + \xi_{\mathbf{i}}  = \widebar{g}(\mathbf{x}^{(\mathbf{i})}) + \xi_{\mathbf{i}}
\end{equation}
for $\mathbf{i} \in I_0$.
We can thus view $(\mathbf{x}^{(\mathbf{i})}, \widebar{y}_{\mathbf{i}})$ for $\mathbf{i} \in I_0$ as observations from the model \eqref{eq:shifted-model}, where the true regression function is $\widebar{g}$.
Moreover, define $\widebar{f}_{\text{TC}}^{d, s} = \hat{f}_{\text{TC}}^{d, s} - A$.
Since $\tcds$ is invariant under subtraction by $A$, i.e., 
\begin{equation*}
    \tcds = \big\{f - A: f \in \tcds\big\},
\end{equation*}
it is clear that $\widebar{f}_{\text{TC}}^{d, s}$ is a least squares estimator over $\tcds$, based on observations $(\mathbf{x}^{(\mathbf{i})}, \widebar{y}_{\mathbf{i}})$ for $\mathbf{i} \in I_0$.
That is,
\begin{equation*}
  \widebar{f}_{\text{TC}}^{d, s} \in \argmin_{f \in \tcds} \sum_{\mathbf{i} \in I_0}
  \big(\widebar{y}_{\mathbf{i}} - f(\mathbf{x}^{(\mathbf{i})}) \big)^2. 
\end{equation*}
Hence, the bounds \eqref{eq:risk-with-M-s-ge-2} and \eqref{eq:risk-with-M-s-eq-1} hold for $R_F(\widebar{f}_{\text{TC}}^{d, s}, \widebar{g})$ with $M = M(\widebar{g})$.
Since
\begin{equation*}
    R_F(\hat{f}_{\text{TC}}^{d, s}, f^*) = R_F(\hat{f}_{\text{TC}}^{d, s}, g^*) = R_F(\widebar{f}_{\text{TC}}^{d, s}, \widebar{g}),
\end{equation*}
this implies that the bounds \eqref{eq:risk-with-M-s-ge-2} and \eqref{eq:risk-with-M-s-eq-1} also hold for $R_F(\hat{f}_{\text{TC}}^{d, s}, f^*)$ with $M = M(\widebar{g})$.

Observe that, for each subset $S \subseteq [d]$ with $1 \le |S| \le s$,
\begin{equation*}
    \begin{bmatrix}
        0, t_k, & k \in S    & \multirow{2}{*}{; \ $\widebar{g}$ } \\ 
        0,      & k \notin S &
    \end{bmatrix} 
    = 
    \begin{bmatrix}
        0, t_k, & k \in S    & \multirow{2}{*}{; \ $g^*$ } \\ 
        0,      & k \notin S &
    \end{bmatrix}
    - c_S
    \qt{ for every } t_k > 0 \text{ for } k \in S 
\end{equation*}
and
\begin{equation*}
    \begin{bmatrix}
        t_k, 1, & k \in S    & \multirow{2}{*}{; \ $\widebar{g}$ } \\ 
        0,      & k \notin S &
    \end{bmatrix}
    = 
    \begin{bmatrix}
        t_k, 1, & k \in S    & \multirow{2}{*}{; \ $g^*$ } \\ 
        0,      & k \notin S &
    \end{bmatrix} 
    - c_S
    \qt{ for every } t_k < 1 \text{ for } k \in S.
\end{equation*}
Hence, if we choose $c_S$ as
\begin{equation*}
    c_S = \inf \bigg\{
    \begin{bmatrix}
        t_k, 1, & k \in S    & \multirow{2}{*}{; \ $g^*$ } \\ 
        0,      & k \notin S &
    \end{bmatrix}
    : t_k < 1 \ \text{ for } k \in S \bigg\}
\end{equation*}
for each subset $S$, then we have
\begin{align*}
    M(\widebar{g}) = \max_{S: 1 \le |S| \le s} \bigg(
    &\sup \bigg\{
    \begin{bmatrix}
        0, t_k, & k \in S    & \multirow{2}{*}{; \ $g^*$ } \\ 
        0,      & k \notin S &
    \end{bmatrix}
    : t_k > 0 \ \text{ for } k \in S \bigg\} \\
    &\qquad - \inf \bigg\{
    \begin{bmatrix}
        t_k, 1, & k \in S    & \multirow{2}{*}{; \ $g^*$ } \\ 
        0,      & k \notin S &
    \end{bmatrix}
    : t_k < 1 \ \text{ for } k \in S \bigg\}
    \bigg)
\end{align*}
(recall the definition of $M(\cdot)$ from \eqref{sup-of-first-derivatives}).

Since $g^* \in \tcds$, there exist real numbers $\beta_0$ and $\beta_S$ (for $1 \le |S| \le s$) and finite measures $\nu_S$ (for $1 \le |S| \le s$) on $[0, 1)^{|S|} \setminus \{\zerovec\}$ such that
\begin{equation*}
  g^*(x_1, \dots, x_d) = \beta_0 + \sum_{S : 1 \le |S| \le s} \beta_S \bigg[\prod_{k \in S} x_k \bigg] - \sum_{S : 1 \le |S| \le s} \int_{[0, 1)^{|S|} \setminus \{\zerovec\}} 
  \bigg[\prod_{k \in S} (x_k - t_k)_+ \bigg] \, d\nu_S(t_k, k \in S)
\end{equation*}
for $(x_1, \dots, x_d) \in [0, 1]^d$.
Observe that, for each $S$, we have
\begingroup
\allowdisplaybreaks
\begin{align*}
    \int_{[0, 1)^{|S|} \setminus \{\zerovec\}} \prod_{k \in S} (x_k - t_k)_+  \, d\nu_S(t_k, k \in S) &= \int_{\prod_{k \in S}[0, x_k) \setminus \{\zerovec\}} \int_{\prod_{k \in S}[t_k, x_k)} 1 \, d(s_k, k \in S) \, d\nu_S(t_k, k \in S) \\
    &= \int_{\prod_{k \in S}[0, x_k)} \int_{\prod_{k \in S}[0, s_k] \setminus \{\zerovec\}} 1 \, d\nu_S(t_k, k \in S) \, d(s_k, k \in S) \\
    &= \int_{\prod_{k \in S}[0, x_k)} \nu_S\bigg(\prod_{k \in S}[0, s_k] \setminus \{\zerovec\}\bigg) \, d(s_k, k \in S).
\end{align*}
\endgroup
Thus, it can be readily verified that, for each $S$,
\begin{equation*}
    \begin{bmatrix}
        0, t_k, & k \in S    & \multirow{2}{*}{; \ $g^*$ } \\ 
        0,      & k \notin S &
    \end{bmatrix} 
    = \beta_S - \frac{1}{\prod_{k \in S} t_k} \int_{\prod_{k \in S}[0, t_k)} \nu_S\bigg(\prod_{k \in S}[0, s_k] \setminus \{\zerovec\}\bigg) \, d(s_k, k \in S)
\end{equation*}
for every $t_k > 0$ for $k \in S$, and
\begin{equation*}
    \begin{bmatrix}
        t_k, 1, & k \in S    & \multirow{2}{*}{; \ $g^*$ } \\ 
        0,      & k \notin S &
    \end{bmatrix}
    = \beta_S - \frac{1}{\prod_{k \in S} (1 - t_k)} \int_{\prod_{k \in S}[t_k, 1)} \nu_S\bigg(\prod_{k \in S}[0, s_k] \setminus \{\zerovec\}\bigg) \, d(s_k, k \in S)
\end{equation*}
for every $t_k < 1$ for $k \in S$.
Since $\nu_S$ is nonnegative, this implies that
\begin{equation*}
    \sup \bigg\{
    \begin{bmatrix}
        0, t_k, & k \in S    & \multirow{2}{*}{; \ $g^*$ } \\ 
        0,      & k \notin S &
    \end{bmatrix}
    : t_k > 0 \ \text{ for } k \in S \bigg\}
    \le \beta_S
\end{equation*}
and 
\begin{equation*}
    \inf \bigg\{
    \begin{bmatrix}
        t_k, 1, & k \in S    & \multirow{2}{*}{; \ $g^*$ } \\ 
        0,      & k \notin S &
    \end{bmatrix}
    : t_k < 1 \ \text{ for } k \in S \bigg\} 
    \ge \beta_S - \nu_S\big([0, 1)^{|S|} \setminus \{\zerovec\}\big).
\end{equation*}
As a result, we have
\begin{equation*}
    M(\widebar{g}) \le \max_{S: 1 \le |S| \le s} \nu_{S}\big([0, 1)^{|S|} \setminus \{\zerovec\}\big) 
    \le \sum_{S: 1 \le |S| \le s} \nu_{S}\big([0, 1)^{|S|} \setminus \{\zerovec\}\big) 
    = V(g^*).
\end{equation*}

Recall that the bounds \eqref{eq:risk-with-M-s-ge-2} and \eqref{eq:risk-with-M-s-eq-1} hold for $R_F(\hat{f}_{\text{TC}}^{d, s}, f^*)$ with $M = M(\widebar{g})$. 
Also, recall that $g^*$ is any function in $\tcds$ that agrees with $f^*$ at the design points $\mathbf{x}^{(\mathbf{i})}$. 
Hence, the bounds \eqref{eq:risk-with-M-s-ge-2} and \eqref{eq:risk-with-M-s-eq-1} still remain valid when $M$ is replaced with $V_{\text{design}}(f^*)$.
\end{proof}

\subsubsection{Proof of Theorem \ref{thm:sharp-oracle}}
The following theorem, from \citet{bellec2018sharp}, is a key tool for our proof of Theorem \ref{thm:sharp-oracle}.
Recall the model \eqref{gaussian-random-vector-model} and the definition \eqref{discrete-convex-lse} of the least squares estimator $\hat{\theta}_K$ over a closed convex subset $K$ of $\R^n$.
Theorem \ref{thm:bellec} provides an alternative bound for the expected squared error $\E[\|\hat{\theta}_K - \theta^*\|_2^2]$ of $\hat{\theta}_K$.
The bound in Theorem \ref{thm:bellec} is expressed in terms of the \textit{statistical dimension} of the \textit{tangent cone} of $K$, which are defined as follows.

For a closed convex subset $K$ of $\R^n$, the tangent cone of $K$ at $\theta \in K$ is defined as 
\begin{equation*}
    \T_{K}(\theta) = \overline{\{t(\eta - \theta): t \ge 0 \text{ and } \eta \in K\}},
\end{equation*}
where $\overline{\{\cdot\}}$ denotes the closure.
It is clear that $\T_{K}(\theta)$ is closed and convex. 
Also, it can be readily checked that $\T_{K}(\theta)$ is a cone, meaning that if $\zeta \in \T_{K}(\theta)$ and $t \ge 0$, then $t \zeta \in \T_{K}(\theta)$.
For a closed convex cone $\T \subseteq \R^n$, the statistical dimension of $\T$ is defined by
\begin{equation*}
    \delta(\T) = \E \|\Pi_{\T} (\mathbf{Z})\|_2^2,
\end{equation*}
where $\mathbf{Z} = (Z_1, \dots, Z_n)$ is a Gaussian random vector with $Z_i
\overset{\text{i.i.d.}}{\sim} N(0, 1^2)$, $\Pi_{\T} (\mathbf{Z})$ denotes the projection of $\mathbf{Z}$ onto $\T$, and $\|\cdot\|_2$ is the Euclidean norm.
It is well-known that the statistical dimension $\delta(\cdot)$ is
monotone: for closed convex cones $\T \subseteq \T'$, we have
$\delta(\T) \le \delta(\T')$ (see, e.g., \citet[Proposition
3.1]{amelunxen2014living}).

\begin{theorem}[Corollary 2.2 of \citet{bellec2018sharp}]\label{thm:bellec}
Under the model \eqref{gaussian-random-vector-model}, $\hat{\theta}_K$ (defined in \eqref{discrete-convex-lse}) satisfies that
    \begin{equation*}
        \E\big[\|\hat{\theta}_K - \theta^*\|_2^2\big] \le \inf_{\theta \in K} \big\{ \|\theta - \theta^*\|_2^2 + \sigma^2 \cdot  \delta(\T_K(\theta))\big\}.
    \end{equation*}
\end{theorem}

We also use the following notation, which extends $T$ defined in \eqref{discrete-tc-set}, in our proof of Theorem \ref{thm:sharp-oracle}.
For positive integers $m_1, \dots, m_d$, we define
\begin{equation*}
    T(m_1, \dots, m_d) = \bigg\{\Big(f\Big(\frac{i_1}{m_1}, \dots, \frac{i_d}{m_d}\Big), (i_1, \dots, i_d) \in I(m_1, \dots, m_d)\Big): f \in \tcds \bigg\},
\end{equation*}
where $I(m_1, \dots, m_d)$ is the index set defined as in \eqref{general-index-set}.
Note that $T = T(n_1, \dots, n_d)$.

Furthermore, as in the proof of Theorem \ref{thm:risk-bound-fixed}, we re-index $(\mathbf{x}^{(1)}, y_1), \dots, (\mathbf{x}^{(n)}, y_n)$ as $(\mathbf{x}^{(\mathbf{i})}, y_{\mathbf{i}})$ for $\mathbf{i} \in I_0$ in our proof of Theorem \ref{thm:sharp-oracle}.
For each real-valued function $f$ on $[0, 1]^d$, we also denote by $\theta_f$ the vector of evaluations of $f$ at the design points: 
\begin{equation*}
    \theta_f = (f(\mathbf{x}^{(\mathbf{i})}), \mathbf{i} \in I_0).
\end{equation*}

\begin{proof}[Proof of Theorem \ref{thm:sharp-oracle}]
Recall that $\theta^* = (f^*(\mathbf{x}^{(\mathbf{i})}), \mathbf{i} \in I_0) = \theta_{f^*}$, and $\hat{\theta}_T$ is defined as the least squares estimator over $T$:
\begin{equation*}
    \hat{\theta}_T = \argmin_{\theta \in T} \|\mathbf{y} - \theta\|_2,
\end{equation*}
where $\mathbf{y} = (y_{\mathbf{i}}, \mathbf{i} \in I_0)$.
Also, recall that we have
\begin{equation*}
    R_F(\hat{f}_{\text{TC}}^{d, s}, f^*) = \frac{1}{n} \cdot \E\big[\|\hat{\theta}_T - \theta^*\|_2^2 \big].
\end{equation*}

Since $T = \{\theta_f: f \in \tcds\}$, applying Theorem \ref{thm:bellec} with $K = T$, we obtain
\begin{align*}
    R_F(\hat{f}_{\text{TC}}^{d, s}, f^*) &= \frac{1}{n} \cdot \E\big[\|\hat{\theta}_T - \theta^*\|_2^2 \big] 
    \le \inf_{\theta \in T} \Big\{ \frac{1}{n} \cdot \|\theta - \theta^*\|_2^2 + \frac{\sigma^2}{n} \cdot  \delta(\T_T(\theta))\Big\} \\
    &= \inf_{f \in \tcds} \Big\{\|f - f^*\|_n^2 + \frac{\sigma^2}{n} \cdot  \delta(\T_T(\theta_f))\Big\} \le \inf_{f \in \tcds \cap \mathcal{R}^{d, s}} \Big\{\|f - f^*\|_n^2 + \frac{\sigma^2}{n} \cdot  \delta(\T_T(\theta_f))\Big\}.
\end{align*}
Hence, to prove Theorem \ref{thm:sharp-oracle}, it suffices to show that 
\begin{equation*}
    \delta(\T_T(\theta_{f})) \le 
    \begin{cases}
        C_d \cdot m(f) \cdot (\log n)^{(5d + 6s - 3)/4} & \text{if } s \ge 2, \\
        C_d \cdot m(f) \cdot (\log n)^{\max(d + 2, 5d/4)} & \text{if } s = 1
    \end{cases}
\end{equation*}
for every $f \in \tcds \cap \mathcal{R}^{d, s}$.

Fix $f \in \tcds \cap \mathcal{R}^{d, s}$, and let 
\begin{equation*}
    P = \bigg\{\prod_{k = 1}^{d} \big[u^{(k)}_{j_k - 1}, u^{(k)}_{j_k}\big]: j_k = 1, \dots, m_k\bigg\}
\end{equation*}
be an axis-aligned split of $[0, 1]^d$ with $|P| = m(f)$ such that $0 = u_0^{(k)} < \dots < u_{m_k}^{(k)} = 1$ for $k = 1, \dots, d$, and $f$ is a multi-affine function of interaction order $s$ on each rectangle of $P$.
For each $\theta \in \R^{|I_0|}$ and $\mathbf{j} = (j_1, \dots, j_d) \in \prod_{k = 1}^{d} [m_k]$, let $\theta^{(\mathbf{j})}$ be the subvector of $\theta$ defined as
\begin{equation*}
    \theta^{(\mathbf{j})} = \Big(\theta_{i_1, \dots, i_d}, \Big(\frac{i_1}{n_1}, \dots, \frac{i_d}{n_d}\Big) \in \prod_{k = 1}^{d} \big[u^{(k)}_{j_k - 1}, u^{(k)}_{j_k}\big) \Big).
\end{equation*}
Also, for each $\mathbf{j}$, by shifting indices, we re-index $\theta^{(\mathbf{j})}$ with $I(n_1^{(\mathbf{j})}, \dots, n_d^{(\mathbf{j})})$, where $n_k^{(\mathbf{j})}$ is the number of $i_k$'s for which $i_k/n_k \in [u^{(k)}_{j_k - 1}, u^{(k)}_{j_k})$.

Now, we claim that 
\begin{equation*}
    \T_{T}(\theta_{f}) \subseteq \bigg\{\theta \in \R^{|I_0|}: \theta^{(\mathbf{j})} \in T\big(n_1^{(\mathbf{j})}, \dots, n_d^{(\mathbf{j})}\big) \text{ for every } \mathbf{j} = (j_1, \dots, j_d) \in \prod_{k = 1}^{d} [m_k] \bigg\} := \T_P.
\end{equation*}
It is clear that $\T_P$ is closed, and if $\theta \in \T_P$ and $t \ge 0$, then $t \theta \in \T_P$.
For the claim, it is thus enough to show that, for every $\theta \in T$, we have $\theta - \theta_{f} \in \T_P$.
Fix $\theta \in T$, and choose any function $g \in \tcds$ with $\theta_g = \theta$.
Also, fix $\mathbf{j} = (j_1, \dots, j_d) \in \prod_{k = 1}^{d} [m_k]$, and, for each $k = 1, \dots, d$, let $l_k^{(\mathbf{j})}$ be the smallest integer for which $l_k^{(\mathbf{j})}/n_k \in [u^{(k)}_{j_k - 1}, u^{(k)}_{j_k})$.
For these fixed $\theta$ and $\mathbf{j}$, we show that $(\theta - \theta_{f})^{(\mathbf{j})} = \theta_{g - f}^{(\mathbf{j})} \in T(n_1^{(\mathbf{j})}, \dots, n_d^{(\mathbf{j})})$, from which the claim follows directly.

Suppose $f$ is of the form 
\begin{equation}\label{eq:multi-affine-function}
    f(x_1, \dots, x_d) = \beta_0^{(\mathbf{j})} + \sum_{S: 1 \le |S| \le s} \beta_S^{(\mathbf{j})} \bigg[\prod_{k \in S} x_k\bigg]
\end{equation}
on $\prod_{k = 1}^{d} [u^{(k)}_{j_k - 1}, u^{(k)}_{j_k}]$, for some constants $\beta_0^{(\mathbf{j})}$ and $\beta_S^{(\mathbf{j})}$, $1 \le |S| \le s$.
First, let $f_0: [0, 1]^d \rightarrow \R$ be the function that is of the form \eqref{eq:multi-affine-function} over the entire domain $[0, 1]^d$.
Note that while $f$ may not be a multi-affine function on $\prod_{k = 1}^{d} [l_k^{(\mathbf{j})}/n_k, (l_k^{(\mathbf{j})} + n_k^{(\mathbf{j})})/n_k]$, the function $f_0$ is multi-affine on $\prod_{k = 1}^{d} [l_k^{(\mathbf{j})}/n_k, (l_k^{(\mathbf{j})} + n_k^{(\mathbf{j})})/n_k]$ by definition.
Next, define $h$ as
\begin{equation*}
    h(x_1, \dots, x_d) = (g - f_0)\Big(\frac{l_1^{(\mathbf{j})} + n_1^{(\mathbf{j})} \cdot x_1}{n_1}, \dots, \frac{l_d^{(\mathbf{j})} + n_d^{(\mathbf{j})} \cdot x_d}{n_d}\Big)
\end{equation*}
for $(x_1, \dots, x_d) \in [0, 1]^d$.
Observe that 
\begin{equation*}
    \Big(\frac{l_1^{(\mathbf{j})} + n_1^{(\mathbf{j})} \cdot x_1}{n_1}, \dots, \frac{l_d^{(\mathbf{j})} + n_d^{(\mathbf{j})} \cdot x_d}{n_d}\Big) \in \prod_{k = 1}^{d} \big[u^{(k)}_{j_k - 1}, u^{(k)}_{j_k}\big)
\end{equation*}
provided that $x_k \le (n_k^{(\mathbf{j})} - 1)/n_k^{(\mathbf{j})}$ for $k = 1, \dots, d$. 
It thus follows that
\begin{equation*}
    \theta_{g - f}^{(\mathbf{j})} = \theta_{g - f_0}^{(\mathbf{j})} 
    = \bigg(h\Big(\frac{i_1}{n_1^{(\mathbf{j})}}, \dots, \frac{i_d}{n_d^{(\mathbf{j})}}\Big), (i_1, \dots, i_d) \in I\big(n_1^{(\mathbf{j})}, \dots, n_d^{(\mathbf{j})}\big)\bigg).
\end{equation*}

We now verify that $h \in \tcp$ and satisfies the interaction restriction condition \eqref{int-rest-cond} for all subsets $S \subseteq [d]$ with $|S| > s$.
By the definition of $T(n_1^{(\mathbf{j})}, \dots, n_d^{(\mathbf{j})})$ and (the proof of) Proposition \ref{lem:reduction-from-popoviciu}, this implies that $(\theta - \theta_{f})^{(\mathbf{j})} = \theta_{g - f}^{(\mathbf{j})} \in T(n_1^{(\mathbf{j})}, \dots, n_d^{(\mathbf{j})})$, which proves the claim.
To see this, first, note that the divided differences of $f_0$ of order $\mathbf{p} = (p_1, \dots, p_d)$ with $\max_j p_j = 2$ are all zero because $f_0$ is a multi-affine function.
Combining this with the fact that $g \in \tcds \subseteq \tcp$, it follows that $h \in \tcp$.
Also, it is straightforward to see that the interaction restriction condition \eqref{int-rest-cond} for $h$ can be derived from those for $g$ and $f_0$, which are from the fact that $g, f_0 \in \tcds$.

Due to the monotonicity of the statistical dimension $\delta(\cdot)$, we have
\begin{equation*}
    \delta(\T_{T}(\theta_{f})) \le \delta(\T_P) 
    = \E \|\Pi_{\T_P} (\mathbf{Z})\|_2^2 = \sum_{\mathbf{j} \in \prod_{k = 1}^{d} [m_k]} \E \|\Pi_{T(n_1^{(\mathbf{j})}, \dots, n_d^{(\mathbf{j})})} (\mathbf{Z}^{(\mathbf{j})})\|_2^2.
\end{equation*}
Note that $\E \|\Pi_{T(n_1^{(\mathbf{j})}, \dots, n_d^{(\mathbf{j})})} (\mathbf{Z}^{(\mathbf{j})})\|_2^2 / (n_1^{(\mathbf{j})} \cdots n_d^{(\mathbf{j})})$ is equal to the risk $R_F(\hat{f}_{\text{TC}}^{d, s}, f^*)$ of the estimator $\hat{f}_{\text{TC}}^{d, s}$, 
when $f^* = 0$, $\sigma = 1$, and the design points are $(i_1/n_1^{(\mathbf{j})}, \dots, i_d/n_d^{(\mathbf{j})})$ for $(i_1, \dots, i_d) \in I(n_1^{(\mathbf{j})}, \dots, n_d^{(\mathbf{j})})$.
Hence, we can reuse the results established in Theorem \ref{thm:risk-bound-fixed} to derive that 
\begin{equation*}
    \delta(\T_{T}(\theta_{f})) \le \sum_{\mathbf{j} \in \prod_{k = 1}^{d} [m_k]} C_d \cdot \Big(\log \big(n_1^{(\mathbf{j})} \cdots n_d^{(\mathbf{j})}\big) \Big)^{(5d + 6s - 3)/4} 
    \le C_d \cdot m(f) \cdot (\log n)^{(5d + 6s - 3)/4} 
\end{equation*}
when $s \ge 2$, 
\begin{equation*}
    \delta(\T_{T}(\theta_{f})) \le \sum_{\mathbf{j} \in \prod_{k = 1}^{d} [m_k]} C_d \cdot \Big(\log \big(n_1^{(\mathbf{j})} \cdots n_d^{(\mathbf{j})}\big) \Big)^{\max(d + 2, 5d/4)} 
    \le C_d \cdot m(f) \cdot (\log n)^{\max(d + 2, 5d/4)}
\end{equation*}
when $s = 1$.
This concludes the proof.
\end{proof}

\subsubsection{Proof of Theorem \ref{thm:risk-bound-random}}
Our proof of Theorem \ref{thm:risk-bound-random} almost follows the proof of Theorem 3.5 of \citet{ki2024mars}.
As in that proof, the following theorem, which is a version of \citet[Proposition 2]{han2019convergence}, plays a central role in our proof of Theorem \ref{thm:risk-bound-random}. 

\begin{theorem}[Theorem 11.12 of \citet{ki2024mars}]\label{thm:ki-random-bound}
    Suppose we are given $(\mathbf{x}^{(1)}, y_1), \dots, (\mathbf{x}^{(n)}, y_n)$ generated according to the model 
    \begin{equation*}
        y_i = f^{*}(\mathbf{x}^{(i)}) + \xi_i, 
    \end{equation*}
    where $\mathbf{x}^{(i)}$ are i.i.d. random variables with law $P$ on $[0, 1]^d$, and $\xi_i$ are i.i.d. error terms independent of $\mathbf{x}^{(1)}, \dots, \mathbf{x}^{(n)}$, with mean zero and finite $L^q$ norm for some $q \ge 3$. 
    The function $f^{*}: [0, 1]^d \rightarrow \R$ is an unknown regression function to be estimated.
    Let $\F$ denote a collection of continuous real-valued functions on $[0, 1]^d$, and assume $f^* \in \F$.
    Also, assume that $\|f - f^*\|_{\infty} \le M$ for all $f \in \F$. 
    Let $\hat{f}$ be an estimator of $f^{*}$ satisfying
    \begin{equation*}
        \hat{f} \in \argmin_{f \in \F} \bigg\{\sum_{i = 1}^{n} \big(y_i - f(\mathbf{x}^{(i)})\big)^2 \bigg\} 
    \end{equation*}
    with probability at least $1 - \epsilon$ for some $\epsilon \ge 0$.
    
    Suppose there exists $t_n > 0$ such that the following inequalities hold for every $r \ge 1$:
    \begin{equation*}
        \E\bigg[\sup_{\substack{f \in \F - \{f^*\}\\ \|f\|_{P, 2} \le rt_n} } \Big| \frac{1}{\sqrt{n}} \sum_{i=1}^{n} \xi_i f(\mathbf{x}^{(i)}) \Big| \bigg] \le r \sqrt{n} t_n^2 \ \text{ and } \
        \E\bigg[\sup_{\substack{f \in \F - \{f^*\}\\ \|f \|_{P, 2} \le rt_n} } \Big| \frac{1}{\sqrt{n}} \sum_{i=1}^{n} \epsilon_i f(\mathbf{x}^{(i)}) \Big| \bigg] \le r \sqrt{n} t_n^2. 
    \end{equation*}
    Here, $\|\cdot\|_{P, 2}$ is the norm defined by 
    \begin{equation*}
        \| f \|_{P, 2} = \big(\E_{\mathbf{X} \sim P}[f(\mathbf{X})^2 ]\big)^{1/2},
    \end{equation*}
    and $\epsilon_i$ are i.i.d. Rademacher variables independent of $\mathbf{x}^{(1)}, \dots, \mathbf{x}^{(n)}$.
    Under these assumptions, there exists a universal positive constant $C$ such that, for every $t \ge t_n$,
    \begin{equation*}
        \P (\|\hat{f} - f^* \|_{P, 2} > t) \le \epsilon + C \cdot \frac{1 + M^3}{t^3} \cdot t_n^3 + C \cdot \frac{\|\xi_1\|_2^3 + M^3}{n^{3/2} t^3} + C \cdot \frac{M^3(\|\xi_1\|_q^3 n^{3/q} + M^3)}{n^{3} t^6}.
    \end{equation*}
\end{theorem}

\begin{proof}[Proof of Theorem \ref{thm:risk-bound-random}]
Let $\infmars$ denote the collection of all functions of the form
\eqref{intermodds} where $\nu_S$ is a finite signed (Borel) measure on
$[0, 1)^{|S|} \setminus \{\zerovec\}$ for each subset $S \subseteq
[d]$ with $1 \le |S| \le s$.   
As discussed in Section \ref{comparison-ki}, this is the function class studied in \citet{ki2024mars}. 
For this proof, we use the results established for $\infmars$ in \cite{ki2024mars}.

First, using the same argument as in the proof of Lemma 11.15 of \citet{ki2024mars}, we can prove the following:
\begin{lemma}\label{lem:mod-est-bdd-in-prob}
    The regularized variant $\hat{f}_{\text{TC}, V}^{d, s}$ satisfies that 
    \begin{equation*}
        \|\hat{f}_{\text{TC}, V}^{d, s} - f^*\|_{\infty} = O_p(1).
    \end{equation*}
\end{lemma}

Fix $\epsilon > 0$. 
This lemma guarantees that, there exists $M > 0$ such that
\begin{equation}\label{eq:mod-est-prob-bdd}
    \P\big(\|\hat{f}_{\text{TC}, V}^{d, s} - f^*\|_{\infty} \le M\big) \ge 1 - \epsilon
\end{equation}
for sufficiently large $n$.
Define
\begin{equation*}
    \F_{\text{TC}}^{M}(V) = \big\{f \in \tcds: V(f) \le V \text{ and } \|f - f^*\|_{\infty} \le M\big\},
\end{equation*}
and
\begin{equation*}
    \F_{\infty-\text{mars}}^{M}(V) = \big\{f \in \infmars: V(f) \le V \text{ and } \|f - f^*\|_{\infty} \le M\big\}.
\end{equation*}
It is clear that $\F_{\text{TC}}^{M}(V)$ is a collection of continuous functions on $[0, 1]^d$ and contains $f^*$.
Also, from \eqref{eq:mod-est-prob-bdd} and the definition of $\hat{f}_{\text{TC}, V}^{d, s}$, it follows that, for sufficiently large $n$, 
\begin{equation*}
    \hat{f}_{\text{TC}, V}^{d, s} \in \argmin_f \bigg\{\sum_{i = 1}^{n} \big(y_i - f(\mathbf{x}^{(i)})\big)^2: f \in \F_{\text{TC}}^{M}(V)\bigg\}
\end{equation*}
with probability at least $1 - \epsilon$.

Since $\tcds \subseteq \infmars$, we have $\F_{\text{TC}}^{M}(V) \subseteq \F_{\infty-\text{mars}}^{M}(V)$.
It follows that, for every $t > 0$, 
\begin{equation*}
    \E\bigg[\sup_{\substack{f \in \F_{\text{TC}}^{M}(V) - \{f^*\}\\ \|f\|_{p_0, 2} \le t}} \Big|\frac{1}{\sqrt{n}} \sum_{i=1}^{n} \xi_i f(\mathbf{x}^{(i)}) \Big| \bigg] \le \E\bigg[\sup_{\substack{f \in \F_{\infty-\text{mars}}^{M}(V) - \{f^*\}\\ \|f\|_{p_0, 2} \le t}} \Big|\frac{1}{\sqrt{n}} \sum_{i=1}^{n} \xi_i f(\mathbf{x}^{(i)}) \Big| \bigg]
\end{equation*}
and
\begin{equation*}
    \E\bigg[\sup_{\substack{f \in \F_{\text{TC}}^{M}(V) - \{f^*\}\\ \|f\|_{p_0, 2} \le t}} \Big|\frac{1}{\sqrt{n}} \sum_{i=1}^{n} \epsilon_i f(\mathbf{x}^{(i)}) \Big| \bigg] \le \E\bigg[\sup_{\substack{f \in \F_{\infty-\text{mars}}^{M}(V) - \{f^*\}\\ \|f\|_{p_0, 2} \le t}} \Big|\frac{1}{\sqrt{n}} \sum_{i=1}^{n} \epsilon_i f(\mathbf{x}^{(i)}) \Big| \bigg],
\end{equation*}
where $\epsilon_i$ are i.i.d. Rademacher variables independent of $\mathbf{x}^{(i)}$.
Thus, by repeating the arguments in the proof of Theorem 3.5 of \citet{ki2024mars}, we can obtain
\begin{equation*}
    \E\bigg[\sup_{\substack{f \in \F_{\text{TC}}^{M}(V) - \{f^*\}\\ \|f\|_{p_0, 2} \le rt_n}} \Big|\frac{1}{\sqrt{n}} \sum_{i=1}^{n} \xi_i f(\mathbf{x}^{(i)}) \Big| \bigg] \le r\sqrt{n} t_n^2 
    \ \ \text{ and } \ \
    \E\bigg[\sup_{\substack{f \in \F_{\text{TC}}^{M}(V) - \{f^*\}\\ \|f\|_{p_0, 2} \le rt_n}} \Big|\frac{1}{\sqrt{n}} \sum_{i=1}^{n} \epsilon_i f(\mathbf{x}^{(i)}) \Big| \bigg] \le r\sqrt{n} t_n^2
\end{equation*}
for every $r \ge 1$, provided that
\begingroup
\allowdisplaybreaks
\begin{align*}
    &t_n = (1 + 4\|\xi_1\|_{5, 1}) \bigg[C_{B, d} (1 + M^{1/2})^{4/5} \big[ (1 + M^{1/2})^{1/5} + (M + V)^{1/5} \big] + C_{B, d} M^{4/9} V^{1/9} \bigg[ 1 + \Big(\frac{M + V}{1 + M^{1/2}}\Big)^{4/45}\bigg] \\
    &\qquad \qquad \qquad \qquad \quad + C_{B, d} (1 + M^{1/2})^{4/5} V^{1/5} \bigg[\log\Big( 2 + \frac{V n^{1/2}}{1 + M^{1/2}}\Big)\bigg]^{4(s - 1)/5} \bigg] \cdot n^{-2/5}.
\end{align*}
\endgroup
By using Theorem \ref{thm:ki-random-bound} with the function class $\F_{\text{TC}}^{M}(V)$, we can show that 
\begin{equation*}
    \P \big(\|\hat{f}_{\text{TC}, V}^{d, s} - f^* \|_{p_0, 2} > t\big) \le \epsilon + C \cdot \frac{1 + M^3}{t^3} \cdot t_n^3 + C \cdot \frac{\|\xi_1\|_2^3 + M^3}{n^{3/2} t^3} + C \cdot \frac{M^3(\|\xi_1\|_5^3 n^{3/5} + M^3)}{n^{3} t^6}
\end{equation*}
for sufficiently large $n$ and $t \ge t_n$.
As a result, again, by repeating the arguments in the proof of Theorem 3.5 of \citet{ki2024mars}, we can find $K > 0$ for which
\begin{equation*}
    \P \big(\|\hat{f}_{\text{TC}, V}^{d, s} - f^* \|_{p_0, 2} > K n^{-2/5} (\log n)^{4(s - 1)/5}\big) < 2 \epsilon
\end{equation*}
for sufficiently large $n$. 
This proves that
\begin{equation*}
    \|\hat{f}_{\text{TC}, V}^{d, s} - f^*\|_{p_0, 2}^2 = O_p \big(n^{-4/5} (\log n)^{8(s - 1)/5}\big).
\end{equation*}
\end{proof}

\subsubsection{Proof of Theorem \ref{thm:adaptive-random}}
We use the following variant of Theorem \ref{thm:ki-random-bound} 
(\cite[Theorem 11.12]{ki2024mars}) to prove Theorem
\ref{thm:adaptive-random}. The only differences are that the errors are 
now assumed to be sub-exponential, and the term $\|\xi_1\|_q^3 n^{3/q}$ in the 
probability bound is replaced by $K^3 [ \log( 2 + n\sigma^2 / K^2)]^3$ as a 
result of that stronger assumption. Below, we describe the modifications 
needed in the proof of Theorem \ref{thm:ki-random-bound} to establish Theorem 
\ref{thm:ki-random-bound-sub-exp}.

\begin{theorem}\label{thm:ki-random-bound-sub-exp}
    Suppose we are given 
    $(\mathbf{x}^{(1)}, y_1), \dots, (\mathbf{x}^{(n)}, y_n)$ generated 
    according to the model 
    \begin{equation*}
        y_i = f^{*}(\mathbf{x}^{(i)}) + \xi_i, 
    \end{equation*}
    where $\mathbf{x}^{(i)}$ are i.i.d. random variables with law $P$ on 
    $[0, 1]^d$, and $\xi_i$ are i.i.d. mean-zero sub-exponential errors,
    independent of $\mathbf{x}^{(1)}, \dots, \mathbf{x}^{(n)}$ and
    satisfying \eqref{sub-exponential} for some $K, \sigma > 0$. The function 
    $f^{*}: [0, 1]^d \rightarrow \R$ is an unknown regression function to be 
    estimated. Let $\F$ denote a collection of continuous real-valued functions 
    on $[0, 1]^d$, and assume $f^* \in \F$. Also, assume that 
    $\|f - f^*\|_{\infty} \le M$ for all $f \in \F$. Let $\hat{f}$ be an 
    estimator of $f^{*}$ satisfying
    \begin{equation*}
        \hat{f} \in \argmin_{f \in \F} \bigg\{\sum_{i = 1}^{n} 
        \big(y_i - f(\mathbf{x}^{(i)})\big)^2 \bigg\} 
    \end{equation*}
    with probability at least $1 - \epsilon$ for some $\epsilon \ge 0$.
    
    Suppose there exists $t_n > 0$ such that the following inequalities hold 
    for every $r \ge 1$:
    \begin{equation*}
        \E\bigg[\sup_{\substack{f \in \F - \{f^*\}\\ \|f\|_{P, 2} \le rt_n} } 
        \Big| \frac{1}{\sqrt{n}} \sum_{i=1}^{n} \xi_i f(\mathbf{x}^{(i)}) 
        \Big| \bigg] 
        \le r \sqrt{n} t_n^2 \ \text{ and } \ 
        \E\bigg[\sup_{\substack{f \in \F - \{f^*\}\\ \|f \|_{P, 2} \le rt_n} } 
        \Big| \frac{1}{\sqrt{n}} \sum_{i=1}^{n} \epsilon_i f(\mathbf{x}^{(i)}) 
        \Big| \bigg] 
        \le r \sqrt{n} t_n^2. 
    \end{equation*}
    Here, $\|\cdot\|_{P, 2}$ is the norm defined by 
    \begin{equation*}
        \| f \|_{P, 2} 
        = \big(\E_{\mathbf{X} \sim P}[f(\mathbf{X})^2 ]\big)^{1/2},
    \end{equation*}
    and $\epsilon_i$ are i.i.d. Rademacher variables independent of 
    $\mathbf{x}^{(1)}, \dots, \mathbf{x}^{(n)}$.
    Under these assumptions, there exists a universal positive constant $C$ 
    such that, for every $t \ge t_n$,
    \begin{equation*}
        \P (\|\hat{f} - f^* \|_{P, 2} > t) 
        \le \epsilon + C \cdot \frac{1 + M^3}{t^3} \cdot t_n^3 
        + C \cdot \frac{\|\xi_1\|_2^3 + M^3}{n^{3/2} t^3} 
        + C \cdot \frac{M^3\{K^3 (\log( 2 + n\sigma^2 / K^2))^3 
        + M^3\}}{n^{3} t^6}.
    \end{equation*}
\end{theorem} 

\begin{remark}
    In the proof of Theorem \ref{thm:ki-random-bound}, the term 
    $\|\xi_1\|_q^3 n^{3/q}$ in the probability bound arises as an upper bound 
    (up to a constant) on $\E[\max_i |\xi_i|^3]$ under the moment assumption 
    $\|\xi_i\|_q < \infty$ for some $q \ge 3$. Under the stronger assumption 
    that $\xi_i$ are sub-exponential random variables satisfying \eqref{sub-exponential} 
    for some $K, \sigma > 0$, we can obtain a sharper bound for 
    $\E[\max_i |\xi_i|^3]$ as follows.

    Define $\psi: \R_{\ge 0} \to \R_{\ge 0}$ by 
    \begin{equation*}
        \psi(x) = \exp(x^{1/3}) - 1 - x^{1/3} - \frac{1}{2} \cdot x^{2/3}.
    \end{equation*}
    It is straightforward to check that $\psi$ is strictly increasing, convex, and 
    invertible. Hence,
    \begin{align*}
        \E\big[\max_i |\xi_i|^3\big] 
        &= K^3 \E\Big[\max_i \Big|\frac{\xi_i}{K}\Big|^3\Big]  
        = K^3 \E\Big[\psi^{-1}\Big(\max_i \psi \Big( 
        \Big|\frac{\xi_i}{K}\Big|^3\Big)\Big)\Big]
        \le K^3 \psi^{-1}\bigg(\E\Big[\max_i \psi \Big( 
        \Big|\frac{\xi_i}{K}\Big|^3\Big)\Big]\bigg) \\
        &\le K^3 \psi^{-1}\bigg(\sum_{i = 1}^{n} \E\Big[ \psi \Big( 
        \Big|\frac{\xi_i}{K}\Big|^3\Big)\Big]\bigg)
        = K^3 \psi^{-1}\bigg(n \E\Big[ \psi \Big( 
        \Big|\frac{\xi_1}{K}\Big|^3\Big)\Big]\bigg) 
        \le K^3 \psi^{-1}\Big(\frac{n\sigma^2}{2K^2}\Big),
    \end{align*}
    where the first inequality uses the concavity of $\psi^{-1}$, and the 
    last inequality follows from  
    \begin{equation*}
        \E\Big[ \psi \Big(\Big|\frac{\xi_1}{K}\Big|^3\Big)\Big]
        \le \E\bigg[ \exp\Big(\Big|\frac{\xi_1}{K}\Big|\Big) - 1 - 
        \Big|\frac{\xi_1}{K} \Big| \bigg] 
        \le \frac{\sigma^2}{2K^2}.
    \end{equation*} 
    Let $C > 0$ be a constant such that
    \begin{equation*}
        \psi(x) \ge \frac{1}{2} \cdot \exp(x^{1/3}) \ \text{ for } x \ge C.
    \end{equation*}
    Then, for $x \ge \psi(C)$, we have
    \begin{equation*}
        \psi^{-1}(x) \le [\log(2x)]^3.
    \end{equation*}    
    Consequently,
    \begin{equation*}
        \psi^{-1}\Big(\frac{n\sigma^2}{2K^2}\Big)
        \le C + \Big[\log\Big(1 + \frac{n\sigma^2}{K^2}\Big)\Big]^3 
        \le C\Big[\log\Big(2 + \frac{n\sigma^2}{K^2}\Big)\Big]^3,
    \end{equation*}
    from which it follows that
    \begin{equation*}
        \E\big[\max_i |\xi_i|^3\big] 
        \le CK^3 \bigg[ \log\Big( 2 + \frac{n\sigma^2}{K^2}
        \Big)\bigg]^3.
    \end{equation*}
    Recall that in this section, the value of $C$ may vary from line 
    to line.

    By using this sharper bound for $\E\big[\max_i |\xi_i|^3\big]$, we 
    can replace the term $\|\xi_1\|_q^3 n^{3/q}$ in the probability bound of 
    Theorem \ref{thm:ki-random-bound} with $K^3 [ \log( 2 + n\sigma^2 / K^2)]^3$ 
    and establish Theorem \ref{thm:ki-random-bound-sub-exp}.
\end{remark}

\begin{proof}[Proof of Theorem \ref{thm:adaptive-random}]
Fix $\epsilon_0 > 0$. By the same argument as in the proof of Theorem 
\ref{thm:risk-bound-random}, there exists $M > 0$ such that
\begin{equation*}
    \P\big(\|\hat{f}^{d, s}_{\text{TC}, V} - f^*\|_{\infty} \le M\big) 
    \ge 1 - \epsilon_0
\end{equation*}
for sufficiently large $n$.
Define
\begin{equation*}
    \F_{\text{TC}}^{M}(V) = \big\{f \in \F_{\text{TC}}^{d, s}: V(f) \le V 
    \text{ and } \|f - f^*\|_{\infty} \le M\big\}
\end{equation*}
and
\begin{equation*}
    B_{M, p_0}(t) = \big\{f \in \F_{\text{TC}}^{d, s}: \|f - f^*\|_{\infty} 
    \le M \text{ and } \|f - f^*\|_{p_0, 2} \le t\big\}.
\end{equation*}
Then, $f^* \in \F_{\text{TC}}^{M}(V)$, and for sufficiently large $n$, we have
\begin{equation*}
    \hat{f}_{\text{TC}, V}^{d, s} 
    \in \argmin_f \bigg\{\sum_{i = 1}^{n} 
    \big(y_i - f(\mathbf{x}^{(i)})\big)^2: f \in \F_{\text{TC}}^{M}(V)\bigg\}
\end{equation*}
with probability at least $1 - \epsilon_0$.

We first consider the special case $f^* = 0$.
In this case,
\begin{equation*}
    \F_{\text{TC}}^{M}(V) = \big\{f \in \F_{\text{TC}}^{d, s}: V(f) \le V 
    \text{ and } \|f\|_{\infty} \le M\big\}
\end{equation*}
and
\begin{equation*}
    B_{M, p_0}(t) = \big\{f \in \F_{\text{TC}}^{d, s}: \|f\|_{\infty} \le M 
    \text{ and } \|f\|_{p_0, 2} \le t\big\}.
\end{equation*}
Our goal is to apply Theorem \ref{thm:ki-random-bound-sub-exp}. 
To this end, we need to bound
\begin{equation}\label{eq:exp-sup-FMTCV}
    \E\bigg[\sup_{\substack{f \in \F_{\text{TC}}^{M}(V)\\ \|f\|_{p_0, 2} \le t}}
    \Big|\frac{1}{\sqrt{n}} \sum_{i=1}^{n} \xi_i f(\mathbf{x}^{(i)}) \Big| 
    \bigg].
\end{equation}
Since the Rademacher variables $\epsilon_i$ are also sub-exponential, the same 
bound applies to
\begin{equation*}
    \E\bigg[\sup_{\substack{f \in \F_{\text{TC}}^{M}(V)\\ \|f\|_{p_0, 2} \le t}}
    \Big|\frac{1}{\sqrt{n}} \sum_{i=1}^{n} \epsilon_i f(\mathbf{x}^{(i)}) \Big| 
    \bigg].
\end{equation*}

We first enlarge the function class over which the supremum is taken:
\begin{equation}\label{eq:bound-by-BMt}
    \E\bigg[\sup_{\substack{f \in \F_{\text{TC}}^{M}(V)\\ \|f\|_{p_0, 2} \le t}}
    \Big|\frac{1}{\sqrt{n}} \sum_{i=1}^{n} \xi_i f(\mathbf{x}^{(i)}) \Big| 
    \bigg] 
    \le \E\bigg[\sup_{f \in B_{M, p_0}(t)}
    \Big|\frac{1}{\sqrt{n}} \sum_{i=1}^{n} \xi_i f(\mathbf{x}^{(i)}) \Big| 
    \bigg].
\end{equation}
By the same argument as in the proof of Lemma 11.17 of \cite{ki2024mars}, 
the expectation on the right-hand side of \eqref{eq:bound-by-BMt} can be 
bounded in terms of the bracketing entropy integral:
\begin{equation}\label{eq:bound-by-bracket-int}
    \E\bigg[\sup_{f \in B_{M, p_0}(t)}
    \Big|\frac{1}{\sqrt{n}} \sum_{i=1}^{n} \xi_i f(\mathbf{x}^{(i)}) \Big| 
    \bigg] 
    \le C J_{[ \ ]}(t, B_{M, p_0}(t), \|\cdot\|_{p_0, 2}) 
    \cdot \Big(\sigma + KM \cdot 
    \frac{J_{[ \ ]}(t, B_{M, p_0}(t), \|\cdot\|_{p_0, 2})}{t^2 \sqrt{n}}\Big),
\end{equation}
where
\begin{equation*}
    J_{[ \ ]}(t, B_{M, p_0}(t), \|\cdot\|_{p_0, 2}) 
    = \int_{0}^{t} \sqrt{1 + \log N_{[ \ ]}(\epsilon, B_{M, p_0}(t), \|\cdot\|_{p_0, 2})} \, d\epsilon,
\end{equation*}
and $N_{[ \ ]}(\epsilon, \mathcal{G}, \| \cdot \|)$ denotes the 
$\epsilon$-bracketing number of $\mathcal{G}$ with respect to $\| \cdot \|$.
Since $b \le p_0(x) \le B$ for all $x \in [0, 1]^d$, we also have
\begin{equation}\label{eq:bound-norm-change}
    N_{[ \ ]}(\epsilon, B_{M, p_0}(t), \|\cdot\|_{p_0, 2})
    \le N_{[ \ ]}(\epsilon/B^{1/2}, B_M(t/b^{1/2}), \|\cdot\|_{2}),
\end{equation}
where 
\begin{equation*}
    B_M(t) = \big\{f \in \F_{\text{TC}}^{d, s}: \|f\|_{\infty} \le M 
    \text{ and } \|f\|_{2} \le t\big\}.
\end{equation*}
For $\delta \in \{0, 1\}^d$, define 
\begin{equation*}
    P^{(\delta)} = \prod_{k = 1}^{d} P_{\delta_k},
\end{equation*}
where $P_0 = [0, 1/2]$ and $P_1 = [1/2, 1]$.
Also, let $B_M^{(\delta)}(t)$ denote the collection of functions 
$g: P^{(\delta)} \to \R$ such that $g = f|_{P^{(\delta)}}$ for some
$f \in B_M(t)$.
It then follows that
\begin{equation}\label{eq:bound-by-Bdelta}
    \log N_{[ \ ]}(\epsilon, B_M(t), \|\cdot\|_{2}) 
    \le \sum_{\delta \in \{0, 1\}^d} 
    \log N_{[ \ ]}\Big(\frac{\epsilon}{2^{d/2}}, 
    B_M^{(\delta)}(t), \|\cdot\|_{2}\Big).
\end{equation}

From this point forward, we focus on bounding
$\log N_{[ \ ]}(\epsilon, B_M^{(\zerovec)}(t), \|\cdot\|_{2})$.
An entirely analogous argument applies to
$\log N_{[ \ ]}(\epsilon, B_M^{(\delta)}(t), \|\cdot\|_{2})$ for 
$\delta \in \{0, 1\}^d \setminus \{\zerovec\}$.
Fix $\epsilon > 0$, and define 
\begin{equation*}
    R = \max\bigg(\bigg\lceil \log\Big(\frac{dM^2}{\epsilon^2}\Big) 
    / \log2 \bigg\rceil - (d - 3), 1\bigg).
\end{equation*}
This choice of $R$ ensures that
\begin{equation*}
    \frac{M^2}{2^{d - 2}} \cdot \frac{d}{2^R} \le \frac{\epsilon^2}{2}.
\end{equation*}
Now, consider the set
\begin{equation*}
    P^{(\zerovec)}_{\text{int}} = \Big[\frac{1}{2^{R + 1}}, \frac{1}{2}\Big]^d,
\end{equation*}
and let $B_{\text{int}}^{(\zerovec)}(t)$ denote the collection of functions 
$g: P_{\text{int}}^{(\zerovec)} \to \R$ such that 
$g = f|_{P_{\text{int}}^{(\zerovec)}}$ for some $f \in \F_{\text{TC}}^{d, s}$ 
with $\|f\|_2 \le t$.
Intuitively, $P_{\text{int}}^{(\zerovec)}$ can be thought of as an interior of 
$P^{(\zerovec)} = [0, 1/2]^d$ obtained by peeling off its boundary.
We claim that
\begin{equation}\label{eq:bound-by-Bint}
    \log N_{[ \ ]}\big(\epsilon, B_M^{(\zerovec)}(t), \|\cdot\|_2\big) 
    \le \log N_{[ \ ]}\Big(\frac{\epsilon}{\sqrt{2}}, 
    B_{\text{int}}^{(\zerovec)}(t), \|\cdot\|_2\Big).
\end{equation}
To prove this, let $g \in B_M^{(\zerovec)}(t)$. By definition, 
$g = f|_{P^{(\zerovec)}}$ for some $f \in B_M(t)$, and hence,
$g|_{P_{\text{int}}^{(\zerovec)}} \in B_{\text{int}}^{(\zerovec)}(t)$.
Let $[h_1, h_2]$ be a bracket containing $g|_{P_{\text{int}}^{(\zerovec)}}$ with 
$\|h_1 - h_2\|_2 \le \epsilon / \sqrt{2}$.
Consider extensions $h_1^{\text{ext}}$ and $h_2^{\text{ext}}$ of $h_1$ and $h_2$ 
to $P^{(\zerovec)}$ for which $h_1^{\text{ext}}(x) = - M$ for 
$x \notin P_{\text{int}}^{(\zerovec)}$ and $h_2^{\text{ext}}(x) = M$ for 
$x \notin P_{\text{int}}^{(\zerovec)}$.
Then, $g \in [h_1^{\text{ext}}, h_2^{\text{ext}}]$ and 
\begin{align*}
    \|h_1^{\text{ext}} - h_2^{\text{ext}}\|_2^2 
    &\le \|h_1 - h_2\|_2^2 
    + 4M^2 \cdot \bigg(\frac{1}{2^d} 
    - \Big(\frac{1}{2} - \frac{1}{2^{R + 1}}\Big)^d\bigg) 
    \le \|h_1 - h_2\|_2^2 
    + \frac{M^2}{2^{d - 2}} \cdot \Big(1
    - \Big(1 - \frac{1}{2^{R}}\Big)^d\Big) \\
    &\le \|h_1 - h_2\|_2^2 
    + \frac{M^2}{2^{d - 2}} \cdot \frac{d}{2^{R}} 
    \le \|h_1 - h_2\|_2^2 + \frac{\epsilon^2}{2} 
    \le \epsilon^2.
\end{align*}
This proves the claim \eqref{eq:bound-by-Bint}.

Next, we split $P_{\text{int}}^{(\zerovec)}$ into dyadic pieces. 
Let $\mathcal{R} = \{1, \dots, R\}^d$, and for each 
$\mathbf{r} = (r_1, \dots, r_d) \in \mathcal{R}$, define
\begin{equation*}
    Q^{(\mathbf{r})} = \prod_{k = 1}^{d} Q_{r_k},
\end{equation*}
where $Q_{l} = [1/2^{l + 1}, 1/2^l]$ for $l = 1, \dots, R$.
Also, for each $\mathbf{r} \in \mathcal{R}$, let 
$B_{\mathbf{r}}^{(\zerovec)}(t)$ denote the collection of functions 
$g: Q^{(\mathbf{r})} \to \R$ such that $g = f|_{Q^{(\mathbf{r})}}$ for some 
$f \in \F_{\text{TC}}^{d, s}$ with $\|f\|_2 \le t$.
Since
\begin{equation*}
    P_{\text{int}}^{(\zerovec)} 
    = \bigcup_{\mathbf{r} \in \mathcal{R}} Q^{(\mathbf{r})},
\end{equation*}
we have
\begin{equation}\label{eq:bound-by-Br}
    \log N_{[ \ ]}\Big(\frac{\epsilon}{\sqrt{2}}, 
    B_{\text{int}}^{(\zerovec)}(t), \|\cdot\|_2\Big)
    \le \sum_{\mathbf{r} \in \mathcal{R}} 
    \log N_{[ \ ]}\Big(\frac{\epsilon}{\sqrt{2|\mathcal{R}|}}, 
    B_{\mathbf{r}}^{(\zerovec)}(t), \|\cdot\|_2\Big).
\end{equation}
We now apply the following lemma, proved in Appendix 
\ref{pf:bounds-finite-derivatives-cont}, to bound each term on the right-hand
side of \eqref{eq:bound-by-Br}.

\begin{lemma}\label{lem:bounds-finite-derivatives-cont}
    Assume that $f \in \F_{\text{TC}}^{d, s}$ and $\|f\|_2 \le t$. 
    Then, for every $\emptyset \neq S \subseteq [d]$,
    $1/2^{r_k + 1} \le s_k < t_k \le 1/2^{r_k}$ for $k \in S$, and 
    $t_k \in [1/2^{r_k + 1}, 1/2^{r_k}]$ for $k \notin S$, we have
    \begin{equation}\label{eq:bounds-finite-derivatives-cont-lower}
        \begin{bmatrix}
            s_k, t_k, & k \in S    & \multirow{2}{*}{; \ $f$ } \\ 
            t_k,      & k \notin S &
        \end{bmatrix}
        \ge - C_{d} t \cdot 2^{r_+/2} \cdot 2^{\sum_{k \in S} r_k}
    \end{equation}
    if $\mathbf{r} = (r_1, \dots, r_d) \in \mathbb{N}^d$, and
    \begin{equation}\label{eq:bounds-finite-derivatives-cont-upper}
        \begin{bmatrix}
            s_k, t_k, & k \in S    & \multirow{2}{*}{; \ $f$ } \\ 
            t_k,      & k \notin S &
        \end{bmatrix}
        \le C_{d} t \cdot 2^{r_+/2} \cdot 2^{\sum_{k \in S} r_k}
    \end{equation}
    if $\mathbf{r} = (r_1, \dots, r_d) \in \mathbb{Z}_{\ge 0}^d$, where 
    $r_+ = \sum_{k = 1}^d r_k$.
\end{lemma}

Fix $\mathbf{r} = (r_1, \dots, r_d) \in \mathcal{R}$, and
suppose $g \in B_{\mathbf{r}}^{(\zerovec)}(t)$. Define 
$h: [0, 1]^d \to \R$ by 
\begin{equation*}
    h(x_1, \dots, x_d) = g\Big(\frac{x_1 + 1}{2^{r_1 + 1}}, 
    \dots, \frac{x_d + 1}{2^{r_d + 1}}\Big),
\end{equation*}
so that $h$ is a domain-scaled version of $g$.
By the definition of $B_{\mathbf{r}}^{(\zerovec)}(t)$, there exists 
$f \in \F_{\text{TC}}^{d, s}$ with $\|f\|_2 \le t$ such that 
$g = f|_{Q^{(\mathbf{r})}}$. Since $f \in \F_{\text{TC}}^{d, s}$, it follows 
that $h \in \tcp$, and $h$ satisfies the interaction restriction condition 
\eqref{int-rest-cond} (see also \eqref{int-rest-cond-restated} and 
\eqref{int-rest-cond-restated-2}) for every subset $S$ of $[d]$ with $|S| > s$.
By Lemma \ref{lem:bounds-finite-derivatives-cont}, for every nonempty subset 
$S$ of $[d]$, 
\begin{align*}
    &\sup \bigg\{
    \begin{bmatrix}
        0, t_k, & k \in S    & \multirow{2}{*}{; \ $h$ } \\ 
        0,      & k \notin S & \\
    \end{bmatrix}
    : t_k > 0 \ \text{ for } k \in S \bigg\} \\
    &\qquad = \prod_{k \in S} 
    \frac{1}{2^{r_k + 1}} \cdot \sup \bigg\{
    \begin{bmatrix}
        1/2^{r_k + 1}, t_k, & k \in S    & \multirow{2}{*}{; \ $f$ } \\ 
        1/2^{r_k + 1},      & k \notin S & \\
    \end{bmatrix}
    : \frac{1}{2^{r_k + 1}} < t_k < \frac{1}{2^{r_k}} \ \text{ for } 
    k \in S \bigg\} 
    \le \frac{1}{2^{|S|}} \cdot C_d t \cdot 2^{r_+/2}
\end{align*}
and
\begin{align*}
    &\inf \bigg\{
    \begin{bmatrix}
        t_k, 1, & k \in S    & \multirow{2}{*}{; \ $h$ } \\ 
        0,      & k \notin S &
    \end{bmatrix}
    : t_k < 1 \ \text{ for } k \in S \bigg\} \\
    &\qquad= \prod_{k \in S} 
    \frac{1}{2^{r_k + 1}} \cdot \inf \bigg\{
    \begin{bmatrix}
        t_k, 1/2^{r_k}, & k \in S    & \multirow{2}{*}{; \ $f$ } \\ 
        1/2^{r_k + 1},      & k \notin S &
    \end{bmatrix}
    : \frac{1}{2^{r_k + 1}} < t_k < \frac{1}{2^{r_k}} \ \text{ for }
    k \in S \bigg\} 
    \ge -\frac{1}{2^{|S|}} \cdot C_d t \cdot 2^{r_+/2}.
\end{align*}
Hence, Proposition \ref{prop:tc-equiv-restricted-interaction} ensures 
$h \in \tcds$.
Also, $V(h)$ of $h$ can be bounded as
\begingroup
\allowdisplaybreaks
\begin{align*}
    V(h) &= \sum_{\emptyset \neq S \subseteq \{1, \dots, d\}} 
    \bigg( \sup \bigg\{
    \begin{bmatrix}
        0, t_k, & k \in S    & \multirow{2}{*}{; \ $h$ } \\ 
        0,      & k \notin S & \\
    \end{bmatrix}
    : t_k > 0 \ \text{ for } k \in S \bigg\}, \\
    &\qquad \qquad \qquad \qquad -\inf \bigg\{
    \begin{bmatrix}
        t_k, 1, & k \in S    & \multirow{2}{*}{; \ $h$ } \\ 
        0,      & k \notin S &
    \end{bmatrix}
    : t_k < 1 \ \text{ for } k \in S \bigg\} \bigg) \\
    &\le \sum_{\emptyset \neq S \subseteq \{1, \dots, d\}} \frac{1}{2^{|S|}} 
    \cdot C_d t \cdot 2^{r_+/2} 
    \le C_d t \cdot 2^{r_+/2}.
\end{align*}
\endgroup
Moreover, 
\begin{equation*}
    \|g\|_2^2 = \int_{Q^{(\mathbf{r})}} \big(g(x)\big)^2 \, dx
    = \prod_{k = 1}^{d} \frac{1}{2^{r_k + 1}} \cdot \int_{[0, 1]^d} 
    \big(h(x)\big)^2 \, dx
    = \frac{1}{2^{r_+ + d}} \cdot \|h\|_2^2.
\end{equation*}
Consequently, it follows that
\begin{equation*}
    \log N_{[ \ ]}\Big(\frac{\epsilon}{\sqrt{2 |\mathcal{R}|}},
    B_{\mathbf{r}}^{(\zerovec)}(t), \|\cdot\|_2\Big) 
    \le \log N_{[ \ ]} 
    \Big(\frac{2^{(r_+ + d - 1)/2} \epsilon}{|\mathcal{R}|^{1/2}},
    D(C_d t \cdot 2^{r_+/2}, t \cdot 2^{(r_+ + d)/2}), 
    \|\cdot\|_2 \Big),
\end{equation*}
where
\begin{equation*}
    D(V, t) := \big\{f \in \F_{\text{TC}}^{d, s}: V(f) \le V 
    \text{ and } \|f\|_2 \le t \big\}.
\end{equation*}

The following lemma, proved in Appendix \ref{pf:fvm-bracketing-entropy}, 
provides an upper bound on the bracketing number of $D(V, t)$ for $V, t > 0$.
\begin{lemma}\label{lem:fvm-bracketing-entropy}
    There exists a positive constant $C_{d}$ depending on $d$ such that
    \begin{align*}
        &\log N_{[ \ ]} (\epsilon, D(V, t), \|\cdot\|_2) 
        \le C_{d} \log \Big( 2 + \frac{t + V}{\epsilon} \Big) 
        + C_{d} \Big(2 + \frac{V}{\epsilon} \Big)^{1/2}
        \bigg[\log \Big(2 + \frac{V}{\epsilon}\Big)\bigg]^{2(s - 1)}
    \end{align*}
    for every $V, t > 0$ and $\epsilon > 0$.
\end{lemma}

Applying Lemma \ref{lem:fvm-bracketing-entropy}, we obtain
\begin{align*}
    &\log N_{[ \ ]}\Big(\frac{\epsilon}{\sqrt{2 |\mathcal{R}|}},
    B_{\mathbf{r}}^{(\zerovec)}(t), \|\cdot\|_2\Big) \\
    &\qquad \le C_d \log \Big( 2 
    + \frac{C_d t |\mathcal{R}|^{1/2}}{\epsilon}\Big)
    + C_d \Big(2 + \frac{C_d t |\mathcal{R}|^{1/2}}{\epsilon} \Big)^{1/2} 
    \bigg[\log \Big(2 + 
    \frac{C_d t |\mathcal{R}|^{1/2}}{\epsilon}\Big)\bigg]^{2(s - 1)}.
\end{align*}
Combining this with \eqref{eq:bound-by-Bint} and \eqref{eq:bound-by-Br}, we 
deduce
\begin{align*}
    &\log N_{[ \ ]}\big(\epsilon, B_M^{(\zerovec)}(t), \|\cdot\|_2\big)
    \le \log N_{[ \ ]}\Big(\frac{\epsilon}{\sqrt{2}}, 
    B_{\text{int}}^{(\zerovec)}(t), \|\cdot\|_2\Big)
    \le \sum_{\mathbf{r} \in \mathcal{R}} 
    \log N_{[ \ ]}\Big(\frac{\epsilon}{\sqrt{2|\mathcal{R}|}}, 
    B_{\mathbf{r}}^{(\zerovec)}(t), \|\cdot\|_2\Big) \\
    &\qquad \le C_d |\mathcal{R}| \cdot \log \Big( 2 
    + \frac{C_d t |\mathcal{R}|^{1/2}}{\epsilon}\Big)
    + C_d |\mathcal{R}| \cdot \Big(2 
    + \frac{C_d t |\mathcal{R}|^{1/2}}{\epsilon} \Big)^{1/2} 
    \bigg[\log \Big(2 + 
    \frac{C_d t |\mathcal{R}|^{1/2}}{\epsilon}\Big)\bigg]^{2(s - 1)}.
\end{align*}
Next, by \eqref{eq:bound-by-Bdelta},
\begingroup
\allowdisplaybreaks
\begin{align*}
    &\log N_{[ \ ]}(\epsilon, B_M(t), \|\cdot\|_2) 
    \le \sum_{\delta \in \{0, 1\}^d} \log N_{[ \ ]}\Big(\frac{\epsilon}{2^{d/2}}, 
    B_M^{(\delta)}(t), \|\cdot\|_2\Big) \\
    &\qquad \le C_{d} |\mathcal{R}| \cdot \log \Big( 2 
    + \frac{C_d t |\mathcal{R}|^{1/2}}{\epsilon}\Big)
    + C_d |\mathcal{R}| \cdot \Big(2 + \frac{C_d 
    t |\mathcal{R}|^{1/2}}{\epsilon} \Big)^{1/2} \bigg[\log \Big(2 + 
    \frac{C_d t |\mathcal{R}|^{1/2}}{\epsilon}\Big)\bigg]^{2(s - 1)} \\
    &\qquad \le C_{d, M} \log \Big(2 + \frac{C_{d, M} t}{\epsilon} 
    \cdot \Big(\log \Big( 2 + \frac{1}{\epsilon}\Big)\Big)^{d/2} \Big) \cdot
    \bigg(\log \Big( 2 + \frac{1}{\epsilon}\Big)\bigg)^d \\
    &\qquad \qquad + C_{d, M} \Big(2 + \frac{C_{d, M} t}{\epsilon} \cdot 
    \Big(\log \Big( 2 + \frac{1}{\epsilon}\Big)\Big)^{d/2}\Big)^{1/2} \\
    &\qquad \qquad \qquad \qquad \quad \cdot \bigg(\log \Big( 2 
    + \frac{1}{\epsilon}\Big)\bigg)^d \bigg[\log \Big(2 
    + \frac{C_{d, M} t}{\epsilon} \cdot \Big(\log \Big( 2 
    + \frac{1}{\epsilon}\Big)\Big)^{d/2}\Big)\bigg]^{2(s - 1)} \\
    &\qquad \le C_{d, M} \log \Big(3 + \frac{t}{\epsilon}\Big) \cdot
    \bigg(\log \Big( 3 + \frac{1}{\epsilon}\Big)\bigg)^d  
    + C_{d, M} \log \log \Big(3 + \frac{1}{\epsilon}\Big) \cdot
    \bigg(\log \Big( 3 + \frac{1}{\epsilon}\Big)\bigg)^d \\
    &\qquad \qquad + C_{d, M} \Big(3 + \frac{t}{\epsilon}\Big)^{1/2} 
    \bigg(\log \Big( 3 + \frac{1}{\epsilon}\Big)\bigg)^{5d/4} 
    \bigg(\log \Big(3 + \frac{t}{\epsilon}\Big)\bigg)^{2(s - 1)} \\
    &\qquad \qquad + C_{d, M} \Big(3 + \frac{t}{\epsilon}\Big)^{1/2}
    \bigg(\log \Big( 3 + \frac{1}{\epsilon}\Big)\bigg)^{5d/4} 
    \bigg(\log \log \Big(3 + \frac{1}{\epsilon}\Big)\bigg)^{2(s - 1)},
\end{align*}
\endgroup
where the third inequality uses
\begin{equation*}
    |\mathcal{R}|  = R^d 
    \le C_{d, M} \bigg(\log \Big( 2 + \frac{1}{\epsilon}\Big)\bigg)^d,
\end{equation*}
which follows from
\begin{equation*}
    R \le \max\Big(\log \Big(\frac{dM^2}{\epsilon^2}\Big) 
    / \log 2 - (d - 2), 1\Big)
    \le C_{d, M} \log \Big( 2 + \frac{1}{\epsilon}\Big).
\end{equation*}
If $t \ge 1$, then
\begin{equation*}
    \log \Big( 3 + \frac{1}{\epsilon}\Big) \le \log \Big( 3 
    + \frac{t}{\epsilon}\Big).
\end{equation*}
If instead $t < 1$, then
\begin{equation*}
    \log \Big( 3 + \frac{1}{\epsilon}\Big) = \log \Big( 3t 
    + \frac{t}{\epsilon}\Big) - \log t 
    \le \log \Big( 3 + \frac{t}{\epsilon}\Big) + \log \frac{1}{t} 
    \le \log \Big( 3 + \frac{t}{\epsilon}\Big) + \log \Big(3 + \frac{1}{t}\Big).
\end{equation*}
Thus, in all cases, 
\begin{equation*}
    \log \Big( 3 + \frac{1}{\epsilon}\Big) 
    \le \log \Big( 3 + \frac{t}{\epsilon}\Big)
    + \log \Big(3 + \frac{1}{t}\Big).
\end{equation*}
As a result,
\begingroup
\allowdisplaybreaks
\begin{align}\label{bound-on-BMt}
    \log N_{[ \ ]}(\epsilon, B_M(t), \|\cdot\|_2) 
    &\le C_{d, M} \log \Big(3 + \frac{t}{\epsilon}\Big) \cdot
    \bigg(\log \Big( 3 + \frac{t}{\epsilon}\Big)
    + \log \Big(3 + \frac{1}{t}\Big)\bigg)^{d + 1/e} \nonumber \\
    &\qquad + C_{d, M} \Big(3 + \frac{t}{\epsilon}\Big)^{1/2} 
    \bigg(\log \Big( 3 + \frac{t}{\epsilon}\Big)
    + \log \Big(3 + \frac{1}{t}\Big)\bigg)^{5d/4}
    \bigg(\log \Big(3 + \frac{t}{\epsilon}\Big)\bigg)^{2(s - 1)} \nonumber \\
    &\qquad + C_{d, M} \Big(3 + \frac{t}{\epsilon}\Big)^{1/2} 
    \bigg(\log \Big( 3 + \frac{t}{\epsilon}\Big)
    + \log \Big(3 + \frac{1}{t}\Big)\bigg)^{5d/4 + 2(s - 1)/e} \nonumber \\ 
    &\le C_{d, M} \log \Big(3 + \frac{t}{\epsilon}\Big) \cdot
    \bigg(\log \Big( 3 + \frac{t}{\epsilon}\Big)\bigg)^{d + 1/e} 
    + C_{d, M} \log \Big(3 + \frac{t}{\epsilon}\Big) \cdot
    \bigg(\log \Big( 3 + \frac{1}{t}\Big)\bigg)^{d + 1/e} \nonumber \\
    &\qquad + C_{d, M}  \Big(3 + \frac{t}{\epsilon}\Big)^{1/2}
    \bigg(\log \Big(3 + \frac{t}{\epsilon}\Big)\bigg)^{5d/4 + 2(s - 1)} \\
    &\qquad + C_{d, M} \Big(3 + \frac{t}{\epsilon}\Big)^{1/2} 
    \bigg(\log \Big(3 + \frac{t}{\epsilon}\Big)\bigg)^{2(s - 1)} 
    \bigg(\log \Big(3 + \frac{1}{t}\Big)\bigg)^{5d/4 + 2(s - 1)/e} 
    \nonumber.
\end{align}
\endgroup
For the first inequality, we also used $\log \log x \le (\log x)^{1/e}$ 
for all $x > 1$.
By \eqref{eq:bound-norm-change}, it follows that
\begin{align*}
    &\log N_{[ \ ]}(\epsilon, B_{M, p_0}(t), \|\cdot\|_{p_0, 2}) 
    \le C_{b, B, d, M} \log \Big(3 + \frac{t}{\epsilon}\Big) \cdot
    \bigg(\log \Big( 3 + \frac{t}{\epsilon}\Big)\bigg)^{d + 1/e} \\
    &\qquad \qquad + C_{b, B, d, M} \log \Big(3 + \frac{t}{\epsilon}\Big) \cdot
    \bigg(\log \Big( 3 + \frac{1}{t}\Big)\bigg)^{d + 1/e} 
    + C_{b, B, d, M}  \Big(3 + \frac{t}{\epsilon}\Big)^{1/2}
    \bigg(\log \Big(3 + \frac{t}{\epsilon}\Big)\bigg)^{5d/4 + 2(s - 1)} \\
    &\qquad \qquad + C_{b, B, d, M} \Big(3 + \frac{t}{\epsilon}\Big)^{1/2} 
    \bigg(\log \Big(3 + \frac{t}{\epsilon}\Big)\bigg)^{2(s - 1)} 
    \bigg(\log \Big(3 + \frac{1}{t}\Big)\bigg)^{5d/4 + 2(s - 1)/e}.
\end{align*}

Next, we bound the integrals appearing in the bracketing entropy integral. 
Observe that
\begin{equation*}
    \int_{0}^{t} \bigg(\log \Big(3 + \frac{t}{\epsilon}\Big)\bigg)^{1/2} 
    \bigg(\log \Big( 3 + \frac{t}{\epsilon}\Big)\bigg)^{d/2 + 1/2e} \, d\epsilon
    \le Ct,
\end{equation*}
\begin{equation*}
    \int_{0}^{t} \bigg(\log \Big(3 + \frac{t}{\epsilon}\Big)\bigg)^{1/2}
    \bigg(\log \Big( 3 + \frac{1}{t}\Big)\bigg)^{d/2 + 1/2e} \, d\epsilon
    \le Ct \bigg(\log \Big( 3 + \frac{1}{t}\Big)\bigg)^{d/2 + 1/2e},
\end{equation*}
\begin{equation*}
    \int_{0}^{t} \Big(3 + \frac{t}{\epsilon}\Big)^{1/4} \bigg(\log \Big(3 
    + \frac{t}{\epsilon}\Big)\bigg)^{5d/8 + s - 1} \, d\epsilon
    \le Ct,
\end{equation*}
and
\begin{equation*}
    \int_{0}^{t} \Big(3 + \frac{t}{\epsilon}\Big)^{1/4} 
    \bigg(\log \Big(3 + \frac{t}{\epsilon}\Big)\bigg)^{s - 1}
    \bigg(\log \Big(3 + \frac{1}{t}\Big)\bigg)^{5d/8 + (s- 1)/e} \, d\epsilon
    \le Ct \bigg(\log \Big(3 + \frac{1}{t}\Big)\bigg)^{5d/8 + (s - 1)/e}.
\end{equation*}
Hence, the bracketing entropy integral satisfies
\begin{equation*}
    J_{[ \ ]}(t, B_{M, p_0}(t), \|\cdot\|_{p_0, 2}) 
    = \int_{0}^{t} \sqrt{1 + \log
    N_{[ \ ]}(\epsilon, B_{M, p_0}(t), \|\cdot\|_{p_0, 2})} \, d\epsilon 
    \le C_{b, B, d, M} t \bigg(\log \Big(3 + \frac{1}{t}\Big)
    \bigg)^{a(d, s)},
\end{equation*}
where
\begin{equation*}
    a(d, s) := \max\Big(\frac{5d}{8} + \frac{s - 1}{e},
    \frac{d}{2} + \frac{1}{2e}\Big).
\end{equation*}
Combining this with \eqref{eq:bound-by-BMt} and \eqref{eq:bound-by-bracket-int} 
gives
\begin{align*}
    &\E\bigg[\sup_{\substack{f \in \F_{\text{TC}}^{M}(V)\\ \|f\|_{p_0, 2} \le t}}
    \Big|\frac{1}{\sqrt{n}} \sum_{i=1}^{n} \xi_i f(\mathbf{x}^{(i)}) \Big| 
    \bigg] 
    \le \E\bigg[\sup_{f \in B_M(t)} \Big|\frac{1}{\sqrt{n}} 
    \sum_{i=1}^{n} \xi_i f(\mathbf{x}^{(i)}) \Big| \bigg] \\
    &\qquad \quad \le C J_{[ \ ]}(t, B_M(t), \|\cdot\|_{p_0, 2}) 
    \cdot \Big(\sigma + KM \cdot 
    \frac{J_{[ \ ]}(t, B_M(t), \|\cdot\|_{p_0, 2})}{t^2 \sqrt{n}}\Big) \\
    &\qquad \quad \le C_{b, B, d, M} \sigma t \bigg(\log \Big(3 
    + \frac{1}{t}\Big) \bigg)^{a(d, s)} 
    + C_{b, B, d, M} K n^{-1/2} \bigg(\log \Big(3 + \frac{1}{t}\Big)
    \bigg)^{2a(d, s)}
    := \Psi(t).
\end{align*}
Let
\begin{equation*}
    t_n = C_{b, B, d, M} (1 + \sigma + K^{1/2}) \cdot n^{-1/2} \big(\log (3 + n^{1/2})
    \big)^{a(d, s)}.
\end{equation*}
Then, we have
\begin{equation*}
    \Psi(t_n) \le \sqrt{n} t_n^2.
\end{equation*}
Since $t \mapsto \Psi(t)/t$ is decreasing, for every $r \ge 1$,
\begin{equation*}
    \E\bigg[\sup_{\substack{f \in \F_{\text{TC}}^{M}(V)\\ \|f\|_{p_0, 2} \le 
    rt_n}} \Big|\frac{1}{\sqrt{n}} \sum_{i=1}^{n} \xi_i 
    f(\mathbf{x}^{(i)}) \Big| \bigg] 
    \le \Psi(rt_n) \le r \Psi(t_n)
    \le r \sqrt{n} t_n^2.
\end{equation*}
Repeating all of the preceding arguments with the Rademacher variables 
$\epsilon_i$ in place of $\xi_i$ also yields
\begin{equation*}
    \E\bigg[\sup_{\substack{f \in \F_{\text{TC}}^{M}(V)\\ \|f\|_{p_0, 2} \le 
    rt_n}} \Big|\frac{1}{\sqrt{n}} \sum_{i=1}^{n} \epsilon_i 
    f(\mathbf{x}^{(i)}) \Big| \bigg] 
    \le r\sqrt{n} t_n^2
\end{equation*}
for every $r \ge 1$.

Applying Theorem \ref{thm:ki-random-bound-sub-exp} therefore gives
\begin{equation*}
    \P \big(\|\hat{f}^{d, s}_{n, V} - f^* \|_{p_0, 2} > t\big) 
    \le \epsilon_0 + C \cdot \frac{1 + M^3}{t^3} \cdot t_n^3 
    + C \cdot \frac{\|\xi_1\|_2^3 + M^3}{n^{3/2} t^3} 
    + C \cdot \frac{M^3K^3 [ \log( 2 + n\sigma^2 / K^2)]^3 
    + M^6}{n^{3} t^6}.
\end{equation*}
It follows that for each $L > 0$,
\begingroup
\allowdisplaybreaks
\begin{align*}
    &\limsup_{n \rightarrow \infty} \P \big(\|\hat{f}^{d, s}_{n, V} 
    - f^* \|_{p_0, 2} > L n^{-1/2} (\log n)^{a(d, s)}\big) \\
    &\qquad \le \epsilon_0 + \limsup_{n \rightarrow \infty} 
    \bigg[C \cdot \frac{(1 + M^3) t_n^3}{L^3 n^{-3/2} 
    (\log n)^{3a(d, s)}} 
    + C \cdot \frac{\|\xi_1\|_2^3 + M^3}{L^3 (\log n)^{3a(d, s)}} 
    + C \cdot \frac{M^3K^3 [ \log( 2 + n\sigma^2 / K^2)]^3 
    + M^6}{L^6 (\log n)^{6a(d, s)}} \bigg] \\
    &\qquad \le \epsilon_0 + C_{b, B, d, M} \cdot 
    \frac{(1 + \sigma + K^{1/2})^3}{L^3},
\end{align*}
\endgroup
where the second inequality uses that $6a(d, s) > 3d \ge 3$.
Hence, for sufficiently large $L$, 
\begin{equation*}
    \limsup_{n \rightarrow \infty} \P \big(\|\hat{f}^{d, s}_{n, V} 
    - f^* \|_{p_0, 2} > L n^{-1/2} (\log n)^{a(d, s)}\big) < 2\epsilon_0,
\end{equation*}
which proves that
\begin{equation*}
    \|\hat{f}^{d, s}_{n, V} - f^* \|_{p_0, 2} 
    = O_p\big(n^{-1/2} (\log n)^{a(d, s)}\big).
\end{equation*}
Note that the constants underlying $O_p$ depend on $b, B, d, V, K$, and 
$\sigma$, with the dependence on $V$ entering through $M$.

We now prove the theorem for the general case.
Suppose that $f^*$ is multi-affine of the form \eqref{piecewise-multi-affine} on 
each rectangle
\begin{equation*}
    \mathcal{U}^{(\mathbf{j})} 
    := \prod_{k = 1}^d \big[u^{(k)}_{j_k - 1}, u^{(k)}_{j_k}\big], \quad
    \mathbf{j} = (j_1, \dots, j_d) \in \prod_{k = 1}^d \{1, \dots, n_k\}, 
\end{equation*}
where $0 = u^{(k)}_0 < \dots < u^{(k)}_{n_k} = 1$ for $k = 1, \dots, d$.
Also, assume that $\prod_{k = 1}^d n_k = m(f^*) =: m$.
As in the special case $f^* = 0$, our goal is to apply Theorem \ref{thm:ki-random-bound-sub-exp}. 
To this end, we aim to bound
\begin{equation*}
    \E\bigg[\sup_{\substack{f \in \F_{\text{TC}}^{M}(V) - \{f^*\} \\ 
    \|f\|_{p_0, 2} \le t}}
    \Big|\frac{1}{\sqrt{n}} \sum_{i=1}^{n} \xi_i f(\mathbf{x}^{(i)}) \Big| 
    \bigg]. 
\end{equation*}

As before, we first enlarge the function class:
\begin{equation*}
    \E\bigg[\sup_{\substack{f \in \F_{\text{TC}}^{M}(V) - \{f^*\} \\ 
    \|f\|_{p_0, 2} \le t}}
    \Big|\frac{1}{\sqrt{n}} \sum_{i=1}^{n} \xi_i f(\mathbf{x}^{(i)}) \Big| 
    \bigg] 
    \le \E\bigg[\sup_{f \in B_{M, p_0}(t) - \{f^*\}}
    \Big|\frac{1}{\sqrt{n}} \sum_{i=1}^{n} \xi_i f(\mathbf{x}^{(i)}) \Big| 
    \bigg]. 
\end{equation*}
Analogous to \eqref{eq:bound-by-bracket-int}, this expectation can bounded 
using the bracketing entropy integral as follows:
\begin{align*}
    \E\bigg[\sup_{f \in B_{M, p_0}(t) - \{f^*\}}
    \Big|\frac{1}{\sqrt{n}} \sum_{i=1}^{n} \xi_i f(\mathbf{x}^{(i)}) \Big| 
    \bigg] 
    &\le C J_{[ \ ]}(t, B_{M, p_0}(t) - \{f^*\}, \|\cdot\|_{p_0, 2}) \\
    &\qquad \cdot 
    \Big(\sigma + KM \cdot \frac{J_{[ \ ]}(t, B_{M, p_0}(t) - \{f^*\}, 
    \|\cdot\|_{p_0, 2})}{t^2 \sqrt{n}}\Big).
\end{align*}
Since $b \le p_0(x) \le B$ for all $x \in [0, 1]^d$, we also have
\begin{equation}\label{bound-norm-change-general}
    N_{[ \ ]}(\epsilon, B_{M, p_0}(t) - \{f^*\}, \|\cdot\|_{p_0, 2})
    \le N_{[ \ ]}(\epsilon/B^{1/2}, \widebar{B}_M(t/b^{1/2}) - \{f^*\}, 
    \|\cdot\|_{2}),
\end{equation}
where 
\begin{equation*}
    \widebar{B}_M(t) = \big\{f \in \F_{\text{TC}}^{d, s}: \|f - f^*\|_{\infty} 
    \le M \text{ and } \|f - f^*\|_2 \le t\big\}.
\end{equation*}

Next, we decompose $[0, 1]^d$ into $m$ rectangles $\mathcal{U}^{(\mathbf{j})}$, 
on which $f^*$ is multi-affine.
For each $\mathbf{j} = (j_1, \dots, j_d) \in \prod_{k = 1}^d \{1, \dots, n_k\}$,
let $A_{\mathbf{j}}$ denote the area of $\mathcal{U}^{(\mathbf{j})}$:
\begin{equation*}
    A_{\mathbf{j}} = \prod_{k = 1}^{d} \big(u^{(k)}_{j_k} - u^{(k)}_{j_k - 1}\big),
\end{equation*}
and let $\widebar{B}^{(\mathbf{j})}_M(t)$ be the collection of functions 
$g: \mathcal{U}^{(\mathbf{j})} \to \R$ such that 
$g = f|_{\mathcal{U}^{(\mathbf{j})}}$ for some $f \in \widebar{B}_M(t)$.
It then follows that
\begin{equation}\label{bound-by-BjMt}
    \log N_{[ \ ]}(\epsilon, \widebar{B}_M(t) - \{f^*\}, \|\cdot\|_2) 
    \le \sum_{\mathbf{j} \in \prod_{k = 1}^d \{1, \dots, n_k\}} 
    \log N_{[ \ ]}\big(\epsilon / \sqrt{m}, 
    \widebar{B}^{(\mathbf{j})}_M(t) - \{f^*\}, \|\cdot\|_2\big).
\end{equation}
Recall that $m = m(f^*) = \prod_{k = 1}^d n_k$.

Suppose $g \in \widebar{B}^{(\mathbf{j})}_M(t)$.
By the definition of $\widebar{B}^{(\mathbf{j})}_M(t)$, there exists 
$f \in \widebar{B}_M(t)$ with $\|f - f^*\|_{\infty} \le M$ and 
$\|f - f^*\|_2 \le t$ such that $g = f|_{\mathcal{U}^{(\mathbf{j})}}$.
Define $h: [0, 1]^d \to \R$ by 
\begin{equation*}
    h(x_1, \dots, x_d) = (g - f^*)\Big(u^{(1)}_{j_1 - 1} + 
    \big(u^{(1)}_{j_1} - u^{(1)}_{j_1 - 1}\big) x_1, \dots, 
    u^{(d)}_{j_d - 1} + \big(u^{(d)}_{j_d} - u^{(d)}_{j_d - 1}\big) x_d\Big).
\end{equation*}
Since $f, f^* \in \tcds$ and $f^*$ is multi-affine on 
$\mathcal{U}^{(\mathbf{j})}$, it follows that $h \in \tcp$, and $h$ satisfies 
the interaction restriction condition \eqref{int-rest-cond} (see also 
\eqref{int-rest-cond-restated} and \eqref{int-rest-cond-restated-2}) for every 
subset $S$ of $[d]$ with $|S| > s$. By induction, it can also be readily 
verified that for every nonempty subset 
$S \subset [d]$,
\begin{align*}
    &\sup \bigg\{
    \begin{bmatrix}
        0, t_k, & k \in S    & \multirow{2}{*}{; \ $h$ } \\ 
        0,      & k \notin S & \\
    \end{bmatrix}
    : t_k > 0 \ \text{ for } k \in S \bigg\} \\
    &\quad= \prod_{k \in S} 
    \big(u^{(k)}_{j_k} - u^{(k)}_{j_k - 1}\big) \cdot \sup \bigg\{
    \begin{bmatrix}
        u^{(k)}_{j_k - 1}, t_k, & k \in S    
        & \multirow{2}{*}{; \ $f - f^*$ } \\ 
        u^{(k)}_{j_k - 1},      & k \notin S & \\
    \end{bmatrix}
    : u^{(k)}_{j_k - 1} < t_k < u^{(k)}_{j_k} \ \text{ for } 
    k \in S \bigg\} < \infty
\end{align*}
and
\begin{align*}
    &\inf \bigg\{
    \begin{bmatrix}
        t_k, 1, & k \in S    & \multirow{2}{*}{; \ $h$ } \\ 
        0,      & k \notin S &
    \end{bmatrix}
    : t_k < 1 \ \text{ for } k \in S \bigg\} \\
    &\quad = \prod_{k \in S} 
    \big(u^{(k)}_{j_k} - u^{(k)}_{j_k - 1}\big) \cdot \inf \bigg\{
    \begin{bmatrix}
        t_k, u^{(k)}_{j_k}, & k \in S    
        & \multirow{2}{*}{; \ $f - f^*$ } \\ 
        u^{(k)}_{j_k - 1},      & k \notin S &
    \end{bmatrix}   
    : u^{(k)}_{j_k - 1} < t_k < u^{(k)}_{j_k} \ \text{ for }
    k \in S \bigg\} > -\infty.
\end{align*}
By Proposition \ref{prop:tc-equiv-restricted-interaction}, these facts ensure 
that $h \in \tcds$. Furthermore, since
\begin{equation*}
    \|g - f^*\|_2^2 = A_{\mathbf{j}} \cdot \|h\|_2^2,
\end{equation*}
we obtain
\begin{equation*}
    \log N_{[ \ ]}\big(\epsilon / \sqrt{m}, 
    \widebar{B}^{(\mathbf{j})}_M(t) - \{f^*\}, \|\cdot\|_2\big) 
    \le \log N_{[ \ ]} \bigg(\frac{\epsilon}{m^{1/2} A_{\mathbf{j}}^{1/2}}, 
    B_M\Big(\frac{t}{A_{\mathbf{j}}^{1/2}}\Big), 
    \|\cdot\|_2\bigg),
\end{equation*}
where
\begin{equation*}
    B_M(t) = \{f \in \F_{\text{TC}}^{d, s}: \|f\|_{\infty} \le M 
    \text{ and } \|f\|_2 \le t \}.
\end{equation*}

Using the bound \eqref{bound-on-BMt} established in the special case $f^* = 0$, 
we obtain
\begingroup
\allowdisplaybreaks
\begin{align*}
    &\log N_{[ \ ]}\big(\epsilon / \sqrt{m}, 
    \widebar{B}^{(\mathbf{j})}_M(t) - \{f^*\}, \|\cdot\|_2\big) 
    \le C_{d, M} \log \Big(3 
    + \frac{m^{1/2} t}{\epsilon}\Big) \cdot
    \bigg(\log \Big( 3 + \frac{m^{1/2} t}{\epsilon}\Big)\bigg)^{d + 1/e} \\
    &\qquad \qquad + C_{d, M} \log \Big(3 + \frac{m^{1/2} t}{\epsilon}\Big) 
    \cdot \bigg(\log \Big( 3 
    + \frac{A_{\mathbf{j}}^{1/2}}{t}\Big)\bigg)^{d + 1/e}
    + C_{d, M} \bigg(\log \Big(3 
    + \frac{m^{1/2} t}{\epsilon}\Big)\bigg)^{5d/4 + 2(s - 1)} \\
    &\qquad \qquad + C_{d, M} \Big(3 + \frac{m^{1/2} t}{\epsilon}\Big)^{1/2}
    \bigg(\log \Big(3 + \frac{m^{1/2} t}{\epsilon}\Big)\bigg)^{2(s - 1)} 
    \bigg(\log \Big(3 
    + \frac{A_{\mathbf{j}}^{1/2}}{t}\Big)\bigg)^{5d/4 + 2(s - 1)/e} \\
    &\qquad \le C_{d, M, m} \log \Big(3 
    + \frac{t}{\epsilon}\Big) \cdot
    \bigg(\log \Big( 3 + \frac{t}{\epsilon}\Big)\bigg)^{d + 1/e} 
    + C_{d, M, m} \log \Big(3  
    + \frac{t}{\epsilon}\Big) \cdot \bigg(\log \Big(3 
    + \frac{1}{t}\Big)\bigg)^{d + 1/e}  \\
    &\qquad \qquad + C_{d, M, m} \bigg(\log \Big(3 
    + \frac{t}{\epsilon}\Big)\bigg)^{5d/4 + 2(s - 1)} \\
    &\qquad \qquad + C_{d, M, m} \Big(3 + \frac{t}{\epsilon}\Big)^{1/2} 
    \bigg(\log \Big(3 + \frac{t}{\epsilon}\Big)\bigg)^{2(s - 1)} 
    \bigg(\log \Big(3 + \frac{1}{t}\Big)\bigg)^{5d/4 + 2(s - 1)/e},
\end{align*}
\endgroup
where the second inequality also uses the fact that $A_{\mathbf{j}} \le 1$. 
Combining this bound with \eqref{bound-norm-change-general} and 
\eqref{bound-by-BjMt} yields
\begingroup
\allowdisplaybreaks
\begin{align*}
    &\log N_{[ \ ]}(\epsilon, B_{M, p_0}(t) - \{f^*\}, \|\cdot\|_{p_0, 2}) 
    \le C_{b, B, d, M, m} \log \Big(3 
    + \frac{t}{\epsilon}\Big) \cdot
    \bigg(\log \Big( 3 + \frac{t}{\epsilon}\Big)\bigg)^{d + 1/e} \\
    &\qquad \qquad \qquad + C_{b, B, d, M, m} \log \Big(3  
    + \frac{t}{\epsilon}\Big) \cdot \bigg(\log \Big(3 
    + \frac{1}{t}\Big)\bigg)^{d + 1/e} 
    + C_{b, B, d, M, m} \bigg(\log \Big(3 
    + \frac{t}{\epsilon}\Big)\bigg)^{5d/4 + 2(s - 1)} \\
    &\qquad \qquad \qquad + C_{b, B, d, M, m} \Big(3 + \frac{t}{\epsilon}\Big)^{1/2} 
    \bigg(\log \Big(3 + \frac{t}{\epsilon}\Big)\bigg)^{2(s - 1)} 
    \bigg(\log \Big(3 + \frac{1}{t}\Big)\bigg)^{5d/4 + 2(s - 1)/e}.
\end{align*}
\endgroup

From here on, we can repeat the same argument and computations as in the case 
$f^* = 0$. The only difference is that the constants now also depend on 
$m = m(f^*)$.
As a result, we obtain
\begin{equation*}
    \|\hat{f}^{d, s}_{n, V} - f^* \|_{p_0, 2} 
    = O_p\big(n^{-1/2} (\log n)^{a(d, s)}\big),
\end{equation*}
where the constants underlying $O_p$ depend on $b, B, d, V, K, \sigma$, and 
$m(f^*)$.
\end{proof}

\subsection{Proofs of Lemmas in Appendix \ref{pf:totconcdef}}
\subsubsection{Proof of Lemma \ref{lem:alt-characterization}}\label{pf:alt-characterization}
Our proof of Lemma \ref{lem:alt-characterization} is based on the following lemma. 
This lemma is a variant of \citet[Theorem 3]{aistleitner2015functions}, which we restate here as Theorem \ref{thm:aist-dick}.
\citet[Theorem 3]{aistleitner2015functions} offers a connection between functions on $[0, 1]^m$ that are entirely monotone and coordinate-wise right-continuous, and the cumulative distribution functions of finite measures on $[0, 1]^m \setminus \{\zerovec\}$.
The proof of Lemma \ref{lem:version-aist-dick} will be provided right after the proof of Lemma \ref{lem:alt-characterization}.

\begin{lemma}\label{lem:version-aist-dick}
    Suppose a real-valued function $g$ on $[0, 1]^m$ is entirely monotone and coordinate-wise right-continuous. 
    Also, assume that $g$ is coordinate-wise left-continuous at each point $\mathbf{x} \in [0, 1]^m \setminus [0, 1)^m$ with respect to all the $j^{\text{th}}$ coordinates where $x_j = 1$.
    Then, there exists a finite (Borel) measure $\nu$ on $[0, 1)^m \setminus \{\zerovec\}$ such that
    \begin{equation}\label{eq:version-aist-dict}
        g(x_1, \dots, x_m) - g(0, \dots, 0) = \nu\bigg(\prod_{k = 1}^{m}[0, x_k] \cap \big([0, 1)^m \setminus \{\zerovec\}\big)\bigg)
    \end{equation}
    for every $\mathbf{x} = (x_1, \dots, x_m) \in [0, 1]^m$.
\end{lemma}

\begin{theorem}[Theorem 3 of \citet{aistleitner2015functions}]\label{thm:aist-dick}
    Suppose a real-valued function $g$ on $[0, 1]^m$ is entirely monotone and coordinate-wise right-continuous. 
    Then, there exists a unique finite (Borel) measure $\nu$ on $[0, 1]^m \setminus \{\zerovec\}$ such that 
    \begin{equation*}
        g(x_1, \dots, x_m) - g(0, \dots, 0) = \nu\bigg(\prod_{k = 1}^{m}[0, x_k] \setminus \{\zerovec\}\bigg)
    \end{equation*}
    for every $\mathbf{x} = (x_1, \dots, x_m) \in [0, 1]^m$.
\end{theorem}

\begin{proof}[Proof of Lemma \ref{lem:alt-characterization}]
First, we assume that $f \in \tcds$ is of the form \eqref{intermodds} and show that such $f$ can be represented as \eqref{eq:alt-characterization}.
For each $S \subseteq [d]$ with $1 \le |S| \le s$, let $g_S$ be the function on $[0, 1]^{|S|}$ defined by
\begin{equation*}
    g_S((x_k, k \in S)) = -\beta_S + \nu_S \bigg(\prod_{k \in S}[0, x_k] \cap \big([0, 1)^m \setminus \{\zerovec\}\big)\bigg)
\end{equation*}
for $(x_k, k \in S) \in [0, 1]^{|S|}$.
It can be readily checked that $g_S$ is coordinate-wise right-continuous on $[0, 1]^{|S|}$ and coordinate-wise left-continuous at each point $(x_k, k \in S) \in [0, 1]^{|S|} \setminus [0, 1)^{|S|}$ with respect to all the $k^{\text{th}}$ coordinates where $x_k = 1$. 
Moreover, for each $S$, we have 
\begingroup
\allowdisplaybreaks
\begin{align*}
    \int_{[0, 1)^{|S|} \setminus \{\zerovec\}} \prod_{k \in S} (x_k - t_k)_+  \, d\nu_S(t_k, k \in S) 
    &= \int_{\prod_{k \in S}[0, x_k) \setminus \{\zerovec\}} \int_{\prod_{k \in S}[t_k, x_k)} 1 \, d(s_k, k \in S) \, d\nu_S(t_k, k \in S) \\
    &= \int_{\prod_{k \in S}[0, x_k)} \int_{\prod_{k \in S}[0, s_k] \setminus \{\zerovec\}} 1 \, d\nu_S(t_k, k \in S) \, d(s_k, k \in S) \\
    &= \beta_S \cdot \prod_{k \in S} x_k + \int_{\prod_{k \in S}[0, x_k]} g_S((s_k, k \in S)) \, d(s_k, k \in S),
\end{align*}
\endgroup
from which it follows that
\begin{align*}
    f(x_1, \dots, x_d) &= \beta_0 + \sum_{S : 1 \le |S| \le s} \beta_S \cdot \prod_{k \in S} x_k - \sum_{S : 1 \le |S| \le s} \int_{[0, 1)^{|S|} \setminus \{\zerovec\}} \prod_{k \in S} (x_k - t_k)_+ \, d\nu_S(t_k, k \in S) \\
    &= \beta_0 - \sum_{S : 1 \le |S| \le s} \int_{\prod_{k \in S}[0, x_k]} g_S((s_k, k \in S)) \, d(s_k, k \in S).
\end{align*}
Hence, it remains to verify that $g_S$ is entirely monotone for each $S$.
Without loss of generality, we assume that $S = [m]$.
Recall that, when $S = [m]$, $g_S$ is entirely monotone if 
\begin{equation*}
    \begin{bmatrix}
        x_1^{(1)}, \dots, x_{p_1 + 1}^{(1)} & \\ 
        \vdots                              &; \ g_S \\
        x_1^{(m)}, \dots, x_{p_m + 1}^{(m)} &
    \end{bmatrix}
    \ge 0
\end{equation*}
for every $(p_1, \dots, p_m) \in \{0, 1\}^m \setminus \{\zerovec\}$.
For each $\mathbf{p} = (p_1, \dots, p_m)$, this inequality follows from the nonnegativity of $\nu_S$ and the following equation:
\begin{equation*}
    \begin{bmatrix}
        x_1^{(1)}, \dots, x_{p_1 + 1}^{(1)} & \\ 
        \vdots                              &; \ g_S \\
        x_1^{(m)}, \dots, x_{p_m + 1}^{(m)} &
    \end{bmatrix}
    = \frac{1}{\prod_{k \in S_{\mathbf{p}}} (x_2^{(k)} - x_1^{(k)})} \cdot \nu_S\bigg(\bigg(\prod_{k \in S_{\mathbf{p}}} \big(x_1^{(k)}, x_2^{(k)}\big] \times \prod_{k \notin S_{\mathbf{p}}} \big[0, x_1^{(k)}\big]\bigg) \cap [0, 1)^m \bigg),
\end{equation*}
where $S_{\mathbf{p}} = \{k \in \{1, \dots, m\}: p_k = 1\}$.

Next, we assume that $f$ is of the form \eqref{eq:alt-characterization} and show that $f \in \tcds$.
For each $S \subseteq [d]$ with $1 \le |S| \le s$, since $g_S$ appearing in \eqref{eq:alt-characterization} satisfies the conditions of Lemma \ref{lem:version-aist-dick}, there exists a finite measure $\nu_S$ on $[0, 1)^{|S|} \setminus \{\zerovec\}$ such that
\begin{equation*}
    g_S((x_k, k \in S)) - g_S((0, k \in S)) = \nu_S\bigg(\prod_{k \in S}[0, x_k] \cap \big([0, 1)^{|S|} \setminus \{\zerovec\}\big)\bigg)
\end{equation*}
for every $(x_k, k \in S) \in [0, 1]^{|S|}$.
If we let $\beta_S = - g_S((0, k \in S))$, then we can express $f$ as 
\begingroup
\allowdisplaybreaks
\begin{align*}
    f(x_1, \dots, x_d) &= \beta_0 + \sum_{S : 1 \le |S| \le s} \beta_S \prod_{k \in S} x_k - \sum_{S : 1 \le |S| \le s} \int_{\prod_{k \in S} [0, x_k)} \nu_S\bigg(\prod_{k \in S}[0, t_k] \setminus \{\zerovec\}\bigg) \, d(t_k, k \in S) \\
    &= \beta_0 + \sum_{S : 1 \le |S| \le s} \beta_S \prod_{k \in S} x_k - \sum_{S : 1 \le |S| \le s} \int_{\prod_{k \in S}[0, x_k)} \int_{\prod_{k \in S}[0, t_k] \setminus \{\zerovec\}} 1 \, d\nu_S(s_k, k \in S) \, d(t_k, k \in S) \\
    &= \beta_0 + \sum_{S : 1 \le |S| \le s} \beta_S \prod_{k \in S} x_k - \sum_{S : 1 \le |S| \le s} \int_{\prod_{k \in S}[0, x_k) \setminus \{\zerovec\}} \int_{\prod_{k \in S}[s_k, x_k)} 1 \, d(t_k, k \in S) \, d\nu_S(s_k, k \in S) \\
    &= \beta_0 + \sum_{S : 1 \le |S| \le s} \beta_S \prod_{k \in S} x_k - \sum_{S : 1 \le |S| \le s} \int_{[0, 1)^{|S|} \setminus \{\zerovec\}} \prod_{k \in S} (x_k - s_k)_+  \, d\nu_S(s_k, k \in S),
\end{align*}
\endgroup
which proves that $f \in \tcds$.
\end{proof}

\begin{proof}[Proof of Lemma \ref{lem:version-aist-dick}]
Since $g$ satisfies the conditions of Theorem \ref{thm:aist-dick}, there exists a finite measure $\widebar{\nu}$ on $[0, 1]^m \setminus \{\zerovec\}$ such that
\begin{equation*}
    g(x_1, \dots, x_m) - g(0, \dots, 0) = \widebar{\nu}\bigg(\prod_{k = 1}^{m}[0, x_k] \setminus \{\zerovec\}\bigg)
\end{equation*}
for every $\mathbf{x} = (x_1, \dots, x_m) \in [0, 1]^m$.
Let $\nu$ be the measure obtained by restricting $\widebar{\nu}$ to $[0, 1)^m \setminus \{\zerovec\}$.
We show that $\nu$ is a desired measure, i.e., $\nu$ satisfies the equation \eqref{eq:version-aist-dict} for every $\mathbf{x} \in [0, 1]^m$.
For $\mathbf{x} \in [0, 1)^m$, \eqref{eq:version-aist-dict} is clear from the fact that
\begin{equation*}
    \widebar{\nu}\bigg(\prod_{k = 1}^{m}[0, x_k] \setminus \{\zerovec\}\bigg) = \widebar{\nu}\bigg(\prod_{k = 1}^{m}[0, x_k] \cap \big([0, 1)^m \setminus \{\zerovec\}\big)\bigg).
\end{equation*}
Hence, we only need to prove \eqref{eq:version-aist-dict} for $\mathbf{x} \in [0, 1]^m \setminus [0, 1)^m$.
Without loss of generality, let $\mathbf{x} = (1, \dots, 1, x_{l + 1}, \dots, x_m)$ where $x_{l + 1}, \dots, x_m < 1$.
Then, using the coordinate-wise left-continuity of $g$ at $\mathbf{x}$, we can show that
\begingroup
\allowdisplaybreaks
\begin{align*}
    &g(x_1, \dots, x_m) - g(0, \dots, 0) = \lim_{y_1 \uparrow 1} \dots \lim_{y_l \uparrow 1} g(y_1, \dots, y_l, x_{l + 1}, \dots, x_m) - g(0, \dots, 0) \\
    &\qquad= \lim_{y_1 \uparrow 1} \dots \lim_{y_l \uparrow 1} \widebar{\nu}\bigg(\prod_{k = 1}^{l}[0, y_k] \times \prod_{k = l + 1}^{m} [0, x_k] \setminus \{\zerovec\}\bigg) 
    = \widebar{\nu}\bigg([0, 1)^l \times \prod_{k = l + 1}^{m} [0, x_k] \setminus \{\zerovec\}\bigg) \\
    &\qquad= \widebar{\nu}\bigg(\prod_{k = 1}^{m} [0, x_k] \cap \big([0, 1)^m \setminus \{\zerovec\}\big)\bigg)
    = \nu\bigg(\prod_{k = 1}^{m} [0, x_k] \cap \big([0, 1)^m \setminus \{\zerovec\}\big)\bigg),
\end{align*}
\endgroup
as desired.
\end{proof}

\subsubsection{Proof of Lemma \ref{lem:tc-weak-cond}}\label{pf:tc-weak-cond}
\begin{proof}[Proof of Lemma \ref{lem:tc-weak-cond}]
We prove that the divided difference of $f$ of order $\mathbf{p}$ is nonpositive on $[0, 1]^d$ for every $\mathbf{p} \in \{0, 1, 2\}^d$ with $\max_k p_k = 2$ by induction on $|\mathbf{p}|:= |\{k: p_k \neq 0\}|$.
It is clear that the claim is true when $|\mathbf{p}| = d$, since $N^{(\mathbf{p})} = [0, 1]^d$ in this case.
Suppose that the claim is true when $|\mathbf{p}| = m + 1$ for some $m < d$ and fix a nonnegative integer vector $\mathbf{p} \in \{0, 1, 2\}^d$ with $\max_k p_k = 2$ and $|\mathbf{p}| = m$.
Without loss of generality, we assume that $\{k: p_k \neq 0\} = [m]$.
Then, for every $0 \le x_1^{(k)} < \cdots < x_{p_k + 1}^{(k)} \le 1$, $k = 1, \dots, d$, we have
\begingroup
\allowdisplaybreaks
\begin{align*}
    \begin{bmatrix}
        x_1^{(1)}, \dots, x_{p_1 + 1}^{(1)} & \\ 
        \vdots                              & \\
        x_1^{(m)}, \dots, x_{p_m + 1}^{(m)} & \\
        x_1^{(m + 1)}                       &; \ f \\
        \vdots                              & \\
        x_1^{(d - 1)}                       & \\
        x_1^{(d)}                           & 
    \end{bmatrix}
    &= x_1^{(d)} \cdot
    \begin{bmatrix}
        x_1^{(1)}, \dots, x_{p_1 + 1}^{(1)} & \\ 
        \vdots                              & \\
        x_1^{(m)}, \dots, x_{p_m + 1}^{(m)} & \\
        x_1^{(m + 1)}                       &; \ f \\
        \vdots                              & \\
        x_1^{(d - 1)}                       & \\
        x_1^{(d)}, 0                        & 
    \end{bmatrix}
    + 
    \begin{bmatrix}
        x_1^{(1)}, \dots, x_{p_1 + 1}^{(1)} & \\ 
        \vdots                              & \\
        x_1^{(m)}, \dots, x_{p_m + 1}^{(m)} & \\
        x_1^{(m + 1)}                       &; \ f \\
        \vdots                              & \\
        x_1^{(d - 1)}                       & \\
        0                                   & 
    \end{bmatrix} \\
    &\le 
    \begin{bmatrix}
        x_1^{(1)}, \dots, x_{p_1 + 1}^{(1)} & \\ 
        \vdots                              & \\
        x_1^{(m)}, \dots, x_{p_m + 1}^{(m)} & \\
        x_1^{(m + 1)}                       &; \ f \\
        \vdots                              & \\
        x_1^{(d - 1)}                       & \\
        0                                   & 
    \end{bmatrix},
\end{align*}
\endgroup
where the inequality is due to the fact that the divided difference of $f$ of order $(p_1, \dots, p_m, 0, \dots, 0, 1)$ is nonpositive because $|(p_1, \dots, p_m, 0, \dots, 0, 1)| = m + 1$.
If we repeat this argument with the induction hypothesis that the divided differences of $f$ of order $(p_1, \dots, p_m, 0, \dots, 0, 1, 0), \dots, (p_1, \dots, p_m, 1, 0, \dots, 0)$ are nonpositive, we can show that 
\begin{equation*}
    \begin{bmatrix}
        x_1^{(1)}, \dots, x_{p_1 + 1}^{(1)} & \multirow{6}{*}{; \ $f$ } \\ 
        \vdots                              & \\
        x_1^{(m)}, \dots, x_{p_m + 1}^{(m)} & \\
        x_1^{(m + 1)}                       & \\
        \vdots                              & \\
        x_1^{(d)}                           & 
    \end{bmatrix}
    \le
    \begin{bmatrix}
        x_1^{(1)}, \dots, x_{p_1 + 1}^{(1)} & \multirow{6}{*}{; \ $f$ } \\ 
        \vdots                              & \\
        x_1^{(m)}, \dots, x_{p_m + 1}^{(m)} & \\
        0                                   & \\
        \vdots                              & \\
        0                                   & 
    \end{bmatrix}.
\end{equation*}
Since the divided difference of $f$ of order $(p_1, \dots, p_m, 0, \dots, 0)$ is nonpositive on $N^{(p_1, \dots, p_m, 0, \dots, 0)} = [0, 1]^m \times \{0\} \times \cdots \times \{0\}$ by the assumption, the right-hand side of this inequality is nonpositive. 
This proves that the divided difference of $f$ of order $\mathbf{p}$ is nonpositive on $[0, 1]^d$ and hence concludes the proof.
\end{proof}

\subsubsection{Proof of Lemma \ref{lem:div-diff-mono}}\label{pf:div-diff-mono}
\begin{proof}[Proof of Lemma \ref{lem:div-diff-mono}]
The last part is a direct consequence of the first two.
This is because
\begin{equation*}
    \begin{bmatrix}
        x_1^{(1)}, x_2^{(1)} & \\ 
        \vdots               &; \ f \\
        x_1^{(d)}, x_2^{(d)} &
    \end{bmatrix} \ge
    \begin{bmatrix}
        x_1^{(1)}, y_2^{(1)} & \\ 
        \vdots               &; \ f \\
        x_1^{(d)}, y_2^{(d)} &
    \end{bmatrix} \ge
    \begin{bmatrix}
        y_1^{(1)}, y_2^{(1)} & \\ 
        \vdots               &; \ f \\
        y_1^{(d)}, y_2^{(d)} &
    \end{bmatrix}
\end{equation*}
since $x_1^{(k)} \le y_1^{(k)}$ and $x_2^{(k)} \le y_2^{(k)}$ for $k = 1, \dots, d$.

Suppose we are given $x_1^{(k)} < x_2^{(k)} \le x_3^{(k)}$ for $k = 1, \dots, d$.
Since the divided difference of $f$ of order $(1, \dots, 1, 2)$ is nonpositive, we have the inequality
\begingroup
\allowdisplaybreaks
\begin{align*}
    \begin{bmatrix}
        x_1^{(1)}, x_2^{(1)}         & \multirow{4}{*}{; \ $f$ } \\ 
        \vdots                       & \\
        x_1^{(d - 1)}, x_2^{(d - 1)} & \\
        x_1^{(d)}, x_2^{(d)}         &
    \end{bmatrix} 
    &= -\big(x_3^{(d)} - x_2^{(d)}\big) \cdot
    \begin{bmatrix}
        x_1^{(1)}, x_2^{(1)}            & \multirow{4}{*}{; \ $f$ } \\ 
        \vdots                          & \\
        x_1^{(d - 1)}, x_2^{(d - 1)}    & \\
        x_1^{(d)}, x_2^{(d)}, x_3^{(d)} &
    \end{bmatrix} 
    + 
    \begin{bmatrix}
        x_1^{(1)}, x_2^{(1)}         & \multirow{4}{*}{; \ $f$ } \\ 
        \vdots                       & \\
        x_1^{(d - 1)}, x_2^{(d - 1)} & \\
        x_1^{(d)}, x_3^{(d)}         &
    \end{bmatrix} \\
    &\ge 
    \begin{bmatrix}
        x_1^{(1)}, x_2^{(1)}         & \multirow{4}{*}{; \ $f$ } \\ 
        \vdots                       & \\
        x_1^{(d - 1)}, x_2^{(d - 1)} & \\
        x_1^{(d)}, x_3^{(d)}         &
    \end{bmatrix}.
\end{align*}
\endgroup
Applying the same argument to each coordinate in turn and leveraging the assumption that the divided differences of $f$ of order $(1, \dots, 1, 2, 1), \dots, (2, 1, \dots, 1)$ are nonpositive, 
we can derive that
\begin{equation*}
    \begin{bmatrix}
        x_1^{(1)}, x_2^{(1)}         & \multirow{4}{*}{; \ $f$ } \\ 
        \vdots                       & \\
        x_1^{(d - 1)}, x_2^{(d - 1)} & \\
        x_1^{(d)}, x_2^{(d)}         &
    \end{bmatrix} 
    \ge
    \begin{bmatrix}
        x_1^{(1)}, x_2^{(1)}         & \multirow{4}{*}{; \ $f$ } \\ 
        \vdots                       & \\
        x_1^{(d - 1)}, x_2^{(d - 1)} & \\
        x_1^{(d)}, x_3^{(d)}         &
    \end{bmatrix}
    \ge 
    \cdots
    \ge
    \begin{bmatrix}
        x_1^{(1)}, x_3^{(1)}         & \multirow{4}{*}{; \ $f$ } \\ 
        \vdots                       & \\
        x_1^{(d - 1)}, x_3^{(d - 1)} & \\
        x_1^{(d)}, x_3^{(d)}         &
    \end{bmatrix},
\end{equation*}
as desired. 
By the same argument, we can show that
\begin{equation*}
        \begin{bmatrix}
            x_1^{(1)}, x_3^{(1)} & \\ 
            \vdots               &; \ f \\
            x_1^{(d)}, x_3^{(d)} &
        \end{bmatrix} \ge
        \begin{bmatrix}
            x_2^{(1)}, x_3^{(1)} & \\ 
            \vdots               &; \ f \\
            x_2^{(d)}, x_3^{(d)} &
        \end{bmatrix}
    \end{equation*}
    when we are given $x_1^{(k)} \le x_2^{(k)} < x_3^{(k)}$ for $k = 1, \dots, d$.
\end{proof}

\subsubsection{Proof of Lemma \ref{lem:div-diff-first-order-ineq}}\label{pf:div-diff-first-order-ineq}
\begin{proof}[Proof of Lemma \ref{lem:div-diff-first-order-ineq}]
Without loss of generality, we assume that $\mathbf{p} = (1, \dots, 1, 0, \dots, 0)$ and $m = |\{k: p_k \neq 0\}|$.
Fix a sufficiently small $\epsilon > 0$ and, for each $k = 1, \dots, d$, let $y_1^{(k)} = z_1^{(k)} = x_1^{(k)}$, $y_2^{(k)} = z_2^{(k)} = x_1^{(k)} + \epsilon$ and, if $p_k = 1$, further let $z_3^{(k)} = x_2^{(k)} - \epsilon$ and $y_3^{(k)} = z_4^{(k)} = x_2^{(k)}$.
Since the divided difference of $f$ of order $(p_1 + 1, \dots, p_d + 1) = (2, \dots, 2, 1, \dots, 1)$ is nonpositive, we have the inequality
\begingroup
\allowdisplaybreaks
\begin{align}\label{eq:first-order-ineq-first}
    0 &\ge 
    \begin{bmatrix}
        y_1^{(1)}, y_2^{(1)}, y_3^{(1)} & \multirow{6}{*}{; \ $f$ } \\ 
        \vdots                          & \\
        y_1^{(m)}, y_2^{(m)}, y_3^{(m)} & \\ 
        y_1^{(m + 1)}, y_2^{(m + 1)}    & \\
        \vdots                          & \\
        y_1^{(d)}, y_2^{(d)}            &
    \end{bmatrix} \nonumber \\ 
    &= \sum_{i_1 = 1}^{3} \cdots \sum_{i_m = 1}^{3} \sum_{i_{m + 1} = 1}^{2} \cdots \sum_{i_d = 1}^{2} \frac{1}{\prod_{j_1 \neq i_1} (y_{i_1}^{(1)} - y_{j_1}^{(1)})} \times \cdots \times \frac{1}{\prod_{j_m \neq i_m} (y_{i_m}^{(m)} - y_{j_m}^{(m)})} \nonumber \\
    &\qquad \qquad \qquad \qquad \qquad \qquad \quad \cdot \frac{1}{\epsilon^{d - m}} \cdot (-1)^{\sum_{k = m + 1}^{d} i_k} \cdot f\big(y_{i_1}^{(1)}, \dots, y_{i_d}^{(d)}\big) \nonumber \\
    &= \frac{1}{\prod_{k = 1}^{m} (y_{3}^{(k)} - y_{1}^{(k)})} \cdot \frac{1}{\epsilon^{d - m}} \sum_{\delta \in \{0, 1\}^m} (-1)^{\sum_{k = 1}^{m} \delta_k} \cdot \frac{1}{\prod_{k = 1}^{m} (y_{3 - \delta_k}^{(k)} - y_{2 - \delta_k}^{(k)})} \nonumber \\
    &\qquad \quad \cdot \sum_{i_1 = 2 - \delta_1}^{3 - \delta_1} \cdots \sum_{i_m = 2 - \delta_m}^{3 - \delta_m} \sum_{i_{m + 1} = 1}^{2} \cdots \sum_{i_d = 1}^{2} (-1)^{\sum_{k = 1}^{m} (i_k + \delta_k - 1)}  \cdot (-1)^{\sum_{k = m + 1}^{d} i_k} \cdot f\big(y_{i_1}^{(1)}, \dots, y_{i_d}^{(d)}\big) \nonumber \\
    \begin{split}
    &= \frac{1}{\prod_{k = 1}^{m} (y_{3}^{(k)} - y_{1}^{(k)})} \cdot \frac{1}{\epsilon^{d - m}} \sum_{\delta \in \{0, 1\}^m} (-1)^{\sum_{k = 1}^{m} \delta_k} \cdot \frac{1}{\prod_{k = 1}^{m} (y_{3 - \delta_k}^{(k)} - y_{2 - \delta_k}^{(k)})} \sum_{i_1 = 2 - \delta_1}^{3 - \delta_1} \cdots \sum_{i_m = 2 - \delta_m}^{3 - \delta_m} \\
    &\qquad \qquad \qquad \sum_{i_{m + 1} = 1}^{2} \cdots \sum_{i_d = 1}^{2} (-1)^{\sum_{k = 1}^{m} (i_k + \delta_k - 1)}  \cdot (-1)^{\sum_{k = m + 1}^{d} i_k} \cdot f\big(y_{i_1}^{(1)}, \dots, y_{i_{m}}^{(m)}, z_{i_{m + 1}}^{(m + 1)}, \dots, z_{i_{d}}^{(d)}\big).
    \end{split}
\end{align}
\endgroup
Also, we have
\begin{equation}\label{eq:first-order-ineq-second}
    \begin{split}
    \eqref{eq:first-order-ineq-first} &\ge \frac{1}{\prod_{k = 1}^{m} (y_{3}^{(k)} - y_{1}^{(k)})} \cdot \frac{1}{\epsilon^{d - m}} \sum_{\delta \in \{0, 1\}^m} (-1)^{\sum_{k = 1}^{m} \delta_k} \cdot \frac{1}{\prod_{k = 1}^{m - 1} (y_{3 - \delta_k}^{(k)} - y_{2 - \delta_k}^{(k)})} \cdot \frac{1}{z_{4 - 2\delta_m}^{(m)} - z_{3 - 2\delta_m}^{(m)}}\\
    &\qquad \cdot \sum_{i_1 = 2 - \delta_1}^{3 - \delta_1} \cdots \sum_{i_{m - 1} = 2 - \delta_{m - 1}}^{3 - \delta_{m - 1}} \sum_{i_m = 3 - 2\delta_m}^{4 - 2\delta_m} \sum_{i_{m + 1} = 1}^{2} \cdots \sum_{i_d = 1}^{2} \\
    &\qquad \qquad \qquad \qquad \qquad \qquad (-1)^{\sum_{k = 1}^{m - 1} (i_k + \delta_k - 1)}  \cdot (-1)^{\sum_{k = m}^{d} i_k} \cdot f\big(y_{i_1}^{(1)}, \dots, y_{i_{m - 1}}^{(m - 1)}, z_{i_m}^{(m)}, \dots, z_{i_{d}}^{(d)}\big),
    \end{split}  
\end{equation}
which is because
\begingroup
\allowdisplaybreaks
\begin{align*}
    \eqref{eq:first-order-ineq-second} - \eqref{eq:first-order-ineq-first} 
    &= \frac{z_{3}^{(m)} - z_{2}^{(m)}}{y_{3}^{(m)} - y_{1}^{(m)}} \cdot \frac{1}{\prod_{k = 1}^{m - 1} (y_{3}^{(k)} - y_{1}^{(k)})} \cdot \frac{1}{z_{4}^{(m)} - z_{2}^{(m)}} \cdot \frac{1}{\epsilon^{d - m}} \sum_{\delta \in \{0, 1\}^m} (-1)^{\sum_{k = 1}^{m} \delta_k} \\
    &\quad \cdot \frac{1}{\prod_{k = 1}^{m - 1} (y_{3 - \delta_k}^{(k)} - y_{2 - \delta_k}^{(k)})} \cdot \frac{1}{z_{4 - \delta_m}^{(m)} - z_{3 - \delta_m}^{(m)}} \cdot \sum_{i_1 = 2 - \delta_1}^{3 - \delta_1} \cdots \sum_{i_{m - 1} = 2 - \delta_{m - 1}}^{3 - \delta_{m - 1}} \sum_{i_m = 3 - \delta_m}^{4 - \delta_m} \sum_{i_{m + 1} = 1}^{2} \cdots \sum_{i_d = 1}^{2} \\
    &\qquad \qquad (-1)^{\sum_{k = 1}^{m - 1} (i_k + \delta_k - 1)}  \cdot (-1)^{i_m + \delta_m} \cdot (-1)^{\sum_{k = m + 1}^{d} i_k} \cdot f\big(y_{i_1}^{(1)}, \dots, y_{i_{m - 1}}^{(m - 1)}, z_{i_m}^{(m)}, \dots, z_{i_{d}}^{(d)}\big) \\
    &= \frac{z_{3}^{(m)} - z_{2}^{(m)}}{y_{3}^{(m)} - y_{1}^{(m)}} \cdot
    \begin{bmatrix}
        y_1^{(1)}, y_2^{(1)}, y_3^{(1)} & \\ 
        \vdots                          & \\
        y_1^{(m - 1)}, y_2^{(m - 1)}, y_3^{(m - 1)} &\\ 
        z_2^{(m)}, z_3^{(m)}, z_4^{(m)} &; \ f \\ 
        y_1^{(m + 1)}, y_2^{(m + 1)}    & \\
        \vdots                          & \\
        y_1^{(d)}, y_2^{(d)} &
    \end{bmatrix}
    \le 0.
\end{align*}
\endgroup
Since $z_4^{(m)} - z_3^{(m)} = z_2^{(m)} - z_1^{(m)} = \epsilon$, it follows that 
\begin{align*}
    0 \ge \eqref{eq:first-order-ineq-second} &= \frac{1}{\prod_{k = 1}^{m} (y_{3}^{(k)} - y_{1}^{(k)})} \cdot \frac{1}{\epsilon^{d - m + 1}} \sum_{\delta \in \{0, 1\}^m} (-1)^{\sum_{k = 1}^{m} \delta_k} \cdot \frac{1}{\prod_{k = 1}^{m - 1} (y_{3 - \delta_k}^{(k)} - y_{2 - \delta_k}^{(k)})} \\
    &\qquad \cdot \sum_{i_1 = 2 - \delta_1}^{3 - \delta_1} \cdots \sum_{i_{m - 1} = 2 - \delta_{m - 1}}^{3 - \delta_{m - 1}} \sum_{i_m = 3 - 2\delta_m}^{4 - 2\delta_m} \sum_{i_{m + 1} = 1}^{2} \cdots \sum_{i_d = 1}^{2} \\
    &\qquad \qquad \qquad \qquad \quad (-1)^{\sum_{k = 1}^{m - 1} (i_k + \delta_k - 1)}  \cdot (-1)^{\sum_{k = m}^{d} i_k} \cdot f\big(y_{i_1}^{(1)}, \dots, y_{i_{m - 1}}^{(m - 1)}, z_{i_m}^{(m)}, \dots, z_{i_{d}}^{(d)}\big).
\end{align*}
If we apply the same argument to the $(m - 1)^{\text{th}}, \dots, 1^{\text{st}}$ coordinates sequentially, we obtain that 
\begin{align*}
    0 &\ge \frac{1}{\prod_{k = 1}^{m} (y_{3}^{(k)} - y_{1}^{(k)})} \cdot \frac{1}{\epsilon^{d}} \sum_{\delta \in \{0, 1\}^m} (-1)^{\sum_{k = 1}^{m} \delta_k} \cdot \sum_{i_1 = 3 - 2\delta_1}^{4 - 2\delta_1} \cdots \sum_{i_m = 3 - 2\delta_m}^{4 - 2\delta_m} \sum_{i_{m + 1} = 1}^{2} \cdots \sum_{i_d = 1}^{2} f\big(z_{i_1}^{(1)}, \dots, z_{i_{d}}^{(d)}\big) \\
    &= \frac{1}{\prod_{k = 1}^{m} (x_{2}^{(k)} - x_{1}^{(k)})} \cdot \frac{1}{\epsilon^{d}} \sum_{\delta \in \{0, 1\}^m} (-1)^{\sum_{k = 1}^{m} \delta_k} \triangle f \bigg(\prod_{k = 1}^{m} \big[z_{3 - 2\delta_k}^{(k)}, z_{4 - 2\delta_k}^{(k)}\big] \times \prod_{k = m + 1}^{d} \big[z_{1}^{(k)}, z_{2}^{(k)}\big] \bigg).
\end{align*}
If we take the limit $\epsilon \rightarrow 0$, this gives that 
\begingroup
\allowdisplaybreaks
\begin{align*}
    0 &\ge \frac{1}{\prod_{k = 1}^{m} (x_{2}^{(k)} - x_{1}^{(k)})} \sum_{i_1 = 1}^{2} \cdots \sum_{i_m = 1}^{2} (-1)^{\sum_{k = 1}^{m} i_k} D \triangle f \big(x_{i_1}^{(1)}, \dots, x_{i_{m}}^{(m)}, x_{1}^{(m + 1)}, \dots, x_{1}^{(d)}\big) \\
    &= \begin{bmatrix}
        x_1^{(1)}, x_{2}^{(1)} & \multirow{6}{*}{; \ $D\triangle f$ } \\ 
        \vdots                 & \\
        x_1^{(m)}, x_{2}^{(m)} & \\ 
        x_1^{(m + 1)}          & \\ 
        \vdots                 & \\
        x_1^{(d)}              &
    \end{bmatrix},
\end{align*}
\endgroup
as desired.
\end{proof}

\subsection{Proofs of Lemmas in Appendix \ref{pf:rates}}
\subsubsection{Proof of Lemma \ref{lem:bounds-finite-derivatives}}\label{pf:bounds-finite-derivatives}
\begin{proof}[Proof of Lemma \ref{lem:bounds-finite-derivatives}]
The following lemma will be repeatedly used in our proof of Lemma \ref{lem:bounds-finite-derivatives}.
This lemma can be viewed as a generalization of Lemma \ref{lem:div-diff-mono}, and it can be proved similarly to Lemma \ref{lem:div-diff-mono}.
    
\begin{lemma}\label{lem:div-diff-mono-general}
    Suppose $f \in \tcp$. 
    Then, for every nonempty subset $S \subseteq [d]$, we have 
    \begin{equation*}
        \begin{bmatrix}
            v_1^{(k)}, v_2^{(k)}, & k \in S    & \multirow{2}{*}{; \ $f$ }\\ 
            v_1^{(k)},            & k \notin S &
        \end{bmatrix} \ge
        \begin{bmatrix}
            w_1^{(k)}, w_2^{(k)}, & k \in S    & \multirow{2}{*}{; \ $f$ }\\ 
            w_1^{(k)},            & k \notin S &
        \end{bmatrix}
    \end{equation*}
    provided that $0 \le v_1^{(k)} \le w_1^{(k)} \le 1$ and $0 \le v_2^{(k)} \le w_2^{(k)} \le 1$ for $k \in S$, and $0 \le v_1^{(k)} = w_1^{(k)} \le 1$ for $k \notin S$.
\end{lemma}
    
Also, note that since $g^*$ satisfies the interaction restriction condition \eqref{int-rest-cond} (recall \eqref{int-rest-cond-restated}) for every subset $S \subseteq [d]$ with $|S| > s$, in the definition \eqref{sup-of-first-derivatives} of $M(g^*)$, we can equivalently take the maximum over all nonempty subsets $S \subseteq [d]$ as follows:
\begin{equation}\label{eq:redefine-M-g*}
\begin{split}
    M = M(g^*) = \max_{\emptyset \neq S \subseteq [d]} \max \bigg(
    &\sup \bigg\{
    \begin{bmatrix}
        0, t_k, & k \in S    & \multirow{2}{*}{; \ $g^*$ } \\ 
        0,      & k \notin S &
    \end{bmatrix}
    : t_k > 0 \ \text{ for } k \in S \bigg\}, \\
    &\quad -\inf \bigg\{
    \begin{bmatrix}
        t_k, 1, & k \in S    & \multirow{2}{*}{; \ $g^*$ } \\ 
        0,      & k \notin S &
    \end{bmatrix}
    : t_k < 1 \ \text{ for } k \in S \bigg\}
    \bigg).
\end{split}
\end{equation}
    
Before proving \eqref{eq:finite-derivative-lower-bound} and \eqref{eq:finite-derivative-upper-bound}, we show that
there exists some constant $B_d$ depending on $d$ such that
\begin{equation}\label{eq:bdd-on-first-derivative-g*}
    \bigg|
    \begin{bmatrix}
        v_1^{(1)}, \dots, v_{p_1 + 1}^{(1)} & \\ 
        \vdots                              &; g^* \\
        v_1^{(d)}, \dots, v_{p_d + 1}^{(d)} &
    \end{bmatrix}
    \bigg|
    \le B_d M
\end{equation}
for every $\mathbf{p} = (p_1, \dots, p_d) \in \{0, 1\}^d \setminus \{\zerovec\}$ and $0 \le v_1^{(k)} < \dots < v_{p_k + 1}^{(k)} \le 1$ for $k = 1, \dots, d$.
Since $\theta^* = (g^*(\mathbf{x}^{(\mathbf{i})}), \mathbf{i} \in I_0)$, this result directly implies that  
\begin{equation}\label{eq:bound-on-discrete-differences}
    \bigg|\bigg(\prod_{k = 1}^{d} n_k^{p_k}\bigg) \cdot (D^{(\mathbf{p})} \theta^*)_{\mathbf{i}} \bigg| \le B_d M
\end{equation}
for every $\mathbf{p} = (p_1, \dots, p_d) \in \{0, 1\}^d \setminus \{\zerovec\}$ and $\mathbf{i} = (i_1, \dots, i_d) \in I_0$ with $i_k \ge p_k$ for $k = 1, \dots, d$.

We prove \eqref{eq:bdd-on-first-derivative-g*} by induction on $|\mathbf{p}| := |\{k: p_k \neq 0\}|$.
When $|\mathbf{p}| = d$ ($\mathbf{p} = (1, \dots, 1)$), it directly follows from Lemma \ref{lem:div-diff-mono-general} and \eqref{eq:redefine-M-g*} that the inequality \eqref{eq:bdd-on-first-derivative-g*} holds with $B_d = 1 =: B_{d, d}$. 
Suppose that the inequality \eqref{eq:bdd-on-first-derivative-g*} holds with $B_d = B_{d, m + 1}$ for every $\mathbf{p} = (p_1, \dots, p_d) \in \{0, 1\}^d$ with $|\mathbf{p}| = m + 1$, and fix $\mathbf{p} = (p_1, \dots, p_d) \in \{0, 1\}^d$ with $|\mathbf{p}| = m$.
Without loss of generality, let $\mathbf{p} = (1, \dots, 1, 0, \dots, 0)$.
Observe that
\begingroup
\allowdisplaybreaks
\begin{align*}
    \bigg|
    \begin{bmatrix}
        v_1^{(1)}, v_2^{(1)} & \multirow{6}{*}{; \ $g^*$ }\\ 
        \vdots               & \\
        v_1^{(m)}, v_2^{(m)} & \\
        v_1^{(m + 1)}        & \\
        \vdots               & \\
        v_1^{(d)}            &
    \end{bmatrix}
    \bigg|
    &\le 
    \bigg|
    \begin{bmatrix}
        v_1^{(1)}, v_2^{(1)} & \\ 
        \vdots               & \\
        v_1^{(m)}, v_2^{(m)} & \\
        0                    &; g^* \\
        v_1^{(m + 2)}        & \\
        \vdots               & \\
        v_1^{(d)}            &
    \end{bmatrix}
    \bigg|
    + v_1^{(m + 1)} \cdot
    \bigg|
    \begin{bmatrix}
        v_1^{(1)}, v_2^{(1)} & \\ 
        \vdots               & \\
        v_1^{(m)}, v_2^{(m)} & \\
        0, v_1^{(m + 1)}     &; g^* \\
        v_1^{(m + 2)}        & \\
        \vdots               & \\
        v_1^{(d)}            &
    \end{bmatrix}
    \bigg| \\
    &\le 
    \bigg|
    \begin{bmatrix}
        v_1^{(1)}, v_2^{(1)} & \\ 
        \vdots               & \\
        v_1^{(m)}, v_2^{(m)} & \\
        0                    &; g^* \\
        v_1^{(m + 2)}        & \\
        \vdots               & \\
        v_1^{(d)}            &
    \end{bmatrix}
    \bigg| 
    + B_{d, m + 1} M.
\end{align*}
\endgroup
Repeating this argument with the $(m + 2)^{\text{th}}, \dots, d^{\text{th}}$ coordinates, we obtain 
\begingroup
\allowdisplaybreaks
\begin{align*}
    \bigg|
    \begin{bmatrix}
        v_1^{(1)}, v_2^{(1)} & \multirow{6}{*}{; \ $g^*$ } \\ 
        \vdots               & \\
        v_1^{(m)}, v_2^{(m)} & \\
        v_1^{(m + 1)}        & \\
        \vdots               & \\
        v_1^{(d)}            &
    \end{bmatrix}
    \bigg|
    &\le 
    \bigg|
    \begin{bmatrix}
        v_1^{(1)}, v_2^{(1)} & \multirow{6}{*}{; \ $g^*$ } \\ 
        \vdots               & \\
        v_1^{(m)}, v_2^{(m)} & \\
        0                    & \\
        \vdots               & \\
        0                    &
    \end{bmatrix}
    \bigg|
    + (d - m) \cdot B_{d, m + 1} M \\
    &\le M + (d - m) \cdot B_{d, m + 1} M 
    = \big(1 + (d - m) \cdot B_{d, m + 1} \big) M,
\end{align*}
\endgroup
where the second inequality is again from Lemma \ref{lem:div-diff-mono-general} and \eqref{eq:redefine-M-g*}. 
This proves that the inequality \eqref{eq:bdd-on-first-derivative-g*} holds with $B_d = 1 + (d - m) \cdot B_{d, m + 1} =: B_{d, m}$ for every $\mathbf{p} = (p_1, \dots, p_d) \in \{0, 1\}^d$ with $|\mathbf{p}| = m$.
Consequently, the inequality \eqref{eq:bdd-on-first-derivative-g*} holds for every $\mathbf{p} \in \{0, 1\}^d \setminus \{\zerovec\}$, if $B_d = \max_{m = 1, \dots, d} B_{d, m}$.

Now, we are ready to prove \eqref{eq:finite-derivative-lower-bound} and \eqref{eq:finite-derivative-upper-bound}. 
We also prove them by induction on $|\mathbf{p}|$.
Let us first consider the case $|\mathbf{p}| = d$ ($\mathbf{p} = (1, \dots, 1)$).
Suppose, for contradiction, that there exist $\mathbf{r} = (r_1, \dots, r_d) \in \prod_{k = 1}^{d} \{1, \dots, R_k\}$ and $\mathbf{i} = (i_1, \dots, i_d) \in L_{\mathbf{r}}$ such that
\begin{equation*}
    (D^{(\mathbf{p})}\theta)_{\mathbf{i}} < \frac{C_d}{\prod_{k = 1}^{d} n_k} \cdot \Big[- M - t \Big(\frac{2^{r_+}}{n}\Big)^{1/2} \cdot 2^{\sum_{k = 1}^{d} r_k}\Big],
\end{equation*}
where $C_d = \max((2\sqrt{3})^d, B_d) =: C^{L}_{d, d}$.
Then, by the definition \eqref{discrete-differences} of discrete differences and Lemma \ref{lem:div-diff-mono-general}, for every $\mathbf{j} = (j_1, \dots, j_d) \in I_0$ with $j_k \ge i_k$ for $k = 1, \dots, d$, we have
\begin{align*}
    (D^{(\mathbf{p})} \theta)_{\mathbf{j}} 
    &\le (D^{(\mathbf{p})} \theta)_{\mathbf{i}} 
    < \frac{C_d}{\prod_{k = 1}^{d} n_k} \cdot \Big[- M - t \Big(\frac{2^{r_+}}{n}\Big)^{1/2} \cdot 2^{\sum_{k = 1}^{d} r_k}\Big] \\
    &\le (D^{(\mathbf{p})} \theta^*)_{\mathbf{j}} - \frac{C_d}{\prod_{k = 1}^{d} n_k} \cdot t \Big(\frac{2^{r_+}}{n}\Big)^{1/2} \cdot 2^{\sum_{k = 1}^{d} r_k}.
\end{align*}
Here, for the last inequality, we use \eqref{eq:bound-on-discrete-differences} and the fact that $C_d \ge B_d$.
It thus follows that if $i_1 - 1 \le j'_1 < j_1, \dots, i_d - 1 \le j'_d < j_d$, then
\begingroup
\allowdisplaybreaks
\begin{align*}
    &\sum_{\delta \in \{0, 1\}^d} (\theta - \theta^*)^2_{(1 - \delta_1) j_1 + \delta_1 j'_1, \dots, (1 - \delta_d) j_d + \delta_d j'_d} \\
    &\qquad \ge \frac{1}{2^d} \cdot \bigg(\sum_{\delta \in \{0, 1\}^d} (-1)^{\delta_1 + \cdots + \delta_d} \cdot (\theta - \theta^*)_{(1 - \delta_1) j_1 + \delta_1 j'_1, \dots, (1 - \delta_d) j_d + \delta_d j'_d} \bigg)^2 \\
    &\qquad = \frac{1}{2^d} \cdot \bigg(\sum_{l_1 = j'_1 + 1}^{j_1} \cdots \sum_{l_d = j'_d + 1}^{j_d} \big(D^{(\mathbf{p})} \theta - D^{(\mathbf{p})} \theta^*\big)_{l_1, \dots, l_d} \bigg)^2 \\
    &\qquad > (j_1 - j'_1)^2 \times \cdots \times (j_d - j'_d)^2 \cdot \frac{C_d^2 / 2^d}{\prod_{k = 1}^{d} n_k^2} \cdot t^2 \Big(\frac{2^{r_+}}{n}\Big) \cdot 2^{\sum_{k = 1}^{d} 2r_k}.
\end{align*}
\endgroup
This implies that
\begingroup
\allowdisplaybreaks
\begin{align*}
    t^2 &\ge \|\theta - \theta^*\|_2^2 \ge \sum_{(i_1 + n_1 - 2)/2 < j_1 \le n_1 - 1} \cdots \sum_{(i_d + n_d - 2)/2 < j_d \le n_d - 1} \\
    &\qquad \qquad \qquad \qquad \qquad \quad \sum_{\delta \in \{0, 1\}^d} (\theta - \theta^*)^2_{(1 - \delta_1)j_1 + \delta_1 (i_1 + n_1 - 2 - j_1), \dots, (1 - \delta_d)j_d + \delta_d (i_d + n_d - 2 - j_d)} \\
    &> \sum_{(i_1 + n_1 - 2)/2 < j_1 \le n_1 - 1} \cdots \sum_{(i_d + n_d - 2)/2 < j_d \le n_d - 1} \\
    &\qquad \qquad \qquad (2j_1 - i_1 - n_1 + 2)^2 \times \cdots \times (2j_d - i_d - n_d + 2)^2 \cdot \frac{C_d^2 / 2^d}{\prod_{k = 1}^{d} n_k^2} \cdot t^2 \Big(\frac{2^{r_+}}{n}\Big) \cdot 2^{\sum_{k = 1}^{d} 2r_k} \\
    &= \prod_{k = 1}^{d} \Big[\frac{1}{6}(n_k - i_k)(n_k - i_k + 1)(n_k - i_k + 2)\Big] \cdot \frac{C_d^2 / 2^d}{\prod_{k = 1}^{d} n_k^2} \cdot t^2 \Big(\frac{2^{r_+}}{n}\Big) \cdot 2^{\sum_{k = 1}^{d} 2r_k} \\
    &> \prod_{k = 1}^{d} (n_k - i_k)^3 \cdot \frac{C_d^2 / 12^d}{\prod_{k = 1}^{d} n_k^2} \cdot t^2 \Big(\frac{2^{r_+}}{n}\Big) \cdot 2^{\sum_{k = 1}^{d} 2r_k} \\
    &> \prod_{k = 1}^{d} \Big(\frac{n_k}{2^{r_k}}\Big)^3 \cdot \frac{C_d^2 / 12^d}{\prod_{k = 1}^{d} n_k^2} \cdot t^2 \Big(\frac{2^{r_+}}{n}\Big) \cdot 2^{\sum_{k = 1}^{d} 2r_k} 
    = \frac{C_d^2}{12^d} \cdot t^2 \ge t^2, 
\end{align*}
\endgroup
which leads to contradiction.
Here, the second-to-last inequality is from the fact that, for each $k$, 
\begin{equation*}
    n_k - i_k \ge n_k - \frac{n_k}{2^{r_k}} + 1 > \frac{n_k}{2^{r_k}},
\end{equation*}
since $r_k \ge 1$.
    
Next, suppose, for contradiction, that there exist $\mathbf{r} = (r_1, \dots, r_d) \in \prod_{k = 1}^{d} \{0, 1, \dots, R_k\}$ and $\mathbf{i} = (i_1, \dots, i_d) \in L_{\mathbf{r}}$ with $i_k \ge 1 = p_k$ for $k = 1, \dots, d$ such that 
\begin{equation*}
    (D^{(\mathbf{p})}\theta)_{\mathbf{i}} > \frac{C_d}{\prod_{k = 1}^{d} n_k} \cdot \Big[M + t \Big(\frac{2^{r_+}}{n}\Big)^{1/2} \cdot 2^{\sum_{k = 1}^{d} r_k}\Big],
\end{equation*}
where $C_d = \max((8\sqrt{2})^d, B_d) =: C^{U}_{d, d}$. 
Then, for every $\mathbf{j} = (j_1, \dots, j_d) \in I_0$ with $j_k \le i_k$, we have
\begin{equation*}
    (D^{(\mathbf{p})} \theta)_{\mathbf{j}} \ge (D^{(\mathbf{p})} \theta)_{\mathbf{i}} > \frac{C_d}{\prod_{k = 1}^{d} n_k} \cdot \Big[M + t \Big(\frac{2^{r_+}}{n}\Big)^{1/2} \cdot 2^{\sum_{k = 1}^{d} r_k}\Big] \ge (D^{(\mathbf{p})} \theta^*)_{\mathbf{j}} + \frac{C_d}{\prod_{k = 1}^{d} n_k} \cdot t \Big(\frac{2^{r_+}}{n}\Big)^{1/2} \cdot 2^{\sum_{k = 1}^{d} r_k}.
\end{equation*}
By the same argument as above, we can show that  
\begin{equation*}
    \sum_{\delta \in \{0, 1\}^d} (\theta - \theta^*)^2_{(1 - \delta_1) j_1 + \delta_1 j'_1, \dots, (1 - \delta_d) j_d + \delta_d j'_d} 
    > (j_1 - j'_1)^2 \times \cdots \times (j_d - j'_d)^2 \cdot \frac{C_d^2 / 2^d}{\prod_{k = 1}^{d} n_k^2} \cdot t^2 \Big(\frac{2^{r_+}}{n}\Big) \cdot 2^{\sum_{k = 1}^{d} 2r_k} 
\end{equation*}
provided that $0 \le j'_1 < j_1 \le i_1, \dots, 0 \le j'_d < j_d \le i_d$.
It thus follows that
\begingroup
\allowdisplaybreaks
\begin{align*}
    t^2 &\ge \|\theta - \theta^*\|_2^2 \ge \sum_{i_1/2 < j_1 \le i_1} \cdots \sum_{i_d/2 < j_d \le i_d} \sum_{\delta \in \{0, 1\}^d} (\theta - \theta^*)^2_{(1 - \delta_1)j_1 + \delta_1 (i_1 - j_1), \dots, (1 - \delta_d)j_d + \delta_d (i_d - j_d)} \\
    &> \sum_{i_1/2 < j_1 \le i_1} \cdots \sum_{i_d/2 < j_d \le i_d} (2j_1 - i_1)^2 \times \cdots \times (2j_d - i_d)^2 \cdot \frac{C_d^2 / 2^d}{\prod_{k = 1}^{d} n_k^2} \cdot t^2 \Big(\frac{2^{r_+}}{n}\Big) \cdot 2^{\sum_{k = 1}^{d} 2r_k} \\
    &= \prod_{k = 1}^{d} \Big[\frac{1}{6}i_k(i_k + 1)(i_k + 2)\Big] \cdot \frac{C_d^2 / 2^d}{\prod_{k = 1}^{d} n_k^2} \cdot t^2 \Big(\frac{2^{r_+}}{n}\Big) \cdot 2^{\sum_{k = 1}^{d} 2r_k} \\
    &\ge \prod_{k = 1}^{d} (i_k + 1)^3 \cdot \frac{C_d^2 / 16^d}{\prod_{k = 1}^{d} n_k^2} \cdot t^2 \Big(\frac{2^{r_+}}{n}\Big) \cdot 2^{\sum_{k = 1}^{d} 2r_k} \\
    &> \prod_{k = 1}^{d} \Big(\frac{n_k}{2^{r_k + 1}}\Big)^3 \cdot \frac{C_d^2 / 16^d}{\prod_{k = 1}^{d} n_k^2} \cdot t^2 \Big(\frac{2^{r_+}}{n}\Big) \cdot 2^{\sum_{k = 1}^{d} 2r_k} 
    = \frac{C_d^2}{128^d} \cdot t^2 \ge t^2,
\end{align*}
\endgroup
which leads to contradiction.
Here, for the fourth inequality, we use 
\begin{equation*}
    \frac{4}{3} i_k(i_k + 1)(i_k + 2) \ge (i_k + 1)^3,
\end{equation*}
which holds since $i_k \ge 1$.
    
Now, we assume that \eqref{eq:finite-derivative-lower-bound} and \eqref{eq:finite-derivative-upper-bound} hold with $C_d = C^{L}_{d, m + 1}$ and $C_d = C^{U}_{d, m + 1}$, respectively, for every $\mathbf{p} \in \{0, 1\}^d$ with $|\mathbf{p}| = m + 1$, 
and prove that they also hold with some $C_d = C^{L}_{d, m}$ and $C_d = C^{U}_{d, m}$ for every $\mathbf{p} \in \{0, 1\}^d$ with $|\mathbf{p}| = m$.
For notational convenience, here, we only consider the case where $\mathbf{p} = (1, \dots, 1, 0, \dots, 0) := \mathbf{1}_m$, but the same argument applies to other $\mathbf{p} \in \{0, 1\}^d$ with $|\mathbf{p}| = m$.
    
Suppose, for contradiction, that there exist $\mathbf{r} = (r_1, \dots, r_d) \in \prod_{k = 1}^{d} \{1, \dots, R_k\}$ and $\mathbf{i} = (i_1, \dots, i_d) \in L_{\mathbf{r}}$ such that
\begin{equation*}
    (D^{(\mathbf{1}_m)}\theta)_{\mathbf{i}} < \frac{C_d}{\prod_{k = 1}^{m} n_k} \cdot \Big[- M - t \Big(\frac{2^{r_+}}{n}\Big)^{1/2} \cdot 2^{\sum_{k = 1}^{m} r_k}\Big],
\end{equation*}
where 
\begin{equation*}
    C_d =  \frac{3}{2} \cdot (d - m) C^{U}_{d, m + 1} + \max((2\sqrt{3})^m, B_d) =: C^{L}_{d, m}.
\end{equation*}
Then, by the induction hypothesis,  
for $\mathbf{j} = (j_1, \dots, j_d) \in I_0$ with $j_1 \ge i_1, \dots, j_m \ge i_m$, and $i_{m + 1} \le j_{m + 1} \le b^{(m + 1)}_{r_{m + 1} - 1}, \dots, i_{d} \le j_{d} \le b^{(d)}_{r_{d} - 1}$, we have
\begingroup
\allowdisplaybreaks
\begin{align*}
    (D^{(\mathbf{1}_m)} \theta)_{j_1, \dots, j_d} &= (D^{(\mathbf{1}_m)} \theta)_{j_1, \dots, j_m, i_{m + 1}, j_{m + 2}, \dots, j_{d}} + \sum_{l_{m + 1} = i_{m + 1} + 1}^{j_{m + 1}} (D^{(\mathbf{1}_{m + 1})} \theta)_{j_1, \dots, j_m, l_{m + 1}, j_{m + 2}, \dots, j_{d}} \\
    &\le (D^{(\mathbf{1}_m)} \theta)_{j_1, \dots, j_m, i_{m + 1}, j_{m + 2}, \dots, j_{d}} + (j_{m + 1} - i_{m + 1}) \cdot \frac{C^{U}_{d, m + 1}}{\prod_{k = 1}^{m + 1} n_k} \cdot \Big[M + t \Big(\frac{2^{r_+}}{n}\Big)^{1/2} \cdot 2^{\sum_{k = 1}^{m + 1} r_k}\Big] \\
    &\le (D^{(\mathbf{1}_m)} \theta)_{j_1, \dots, j_m, i_{m + 1}, j_{m + 2}, \dots, j_{d}} + \frac{3}{2} \cdot \frac{C^{U}_{d, m + 1}}{\prod_{k = 1}^{m} n_k} \cdot \bigg[\frac{M}{2^{r_{m + 1}}} + t \Big(\frac{2^{r_+}}{n}\Big)^{1/2} \cdot 2^{\sum_{k = 1}^{m} r_k}\bigg],
\end{align*}
\endgroup
where $\mathbf{1}_{m + 1} = (1, \dots, 1, 0, \dots, 0)$ with $|\mathbf{1}_{m + 1}| = m + 1$.
Here, the second inequality is because
\begin{equation*}
    j_{m + 1} - i_{m + 1} \le b^{(m + 1)}_{r_{m + 1} - 1} - i_{m + 1} < \frac{n_{m + 1}}{2^{r_{m + 1} - 1}} - 1 - \Big(\frac{n_{m + 1}}{2^{r_{m + 1} + 1}} - 1\Big) = \frac{3}{2} \cdot \frac{n_{m + 1}}{2^{r_{m + 1}}}.
\end{equation*}
Repeating this argument with the $(m + 2)^{\text{th}}, \dots, d^{\text{th}}$ coordinates in turn, we can derive that
\begingroup
\allowdisplaybreaks
\begin{align*}
    &(D^{(\mathbf{1}_m)} \theta)_{j_1, \dots, j_d} \\
    &\qquad \le (D^{(\mathbf{1}_m)} \theta)_{j_1, \dots, j_m, i_{m + 1}, \dots, i_{d}} + \frac{3}{2} \cdot \frac{C^{U}_{d, m + 1}}{\prod_{k = 1}^{m} n_k} \cdot \bigg[\bigg(\sum_{k = m + 1}^{d} \frac{1}{2^{r_k}}\bigg) \cdot M + (d - m) \cdot t \Big(\frac{2^{r_+}}{n}\Big)^{1/2} \cdot 2^{\sum_{k = 1}^{m} r_k}\bigg] \\
    &\qquad \le (D^{(\mathbf{1}_m)} \theta)_{j_1, \dots, j_m, i_{m + 1}, \dots, i_{d}} + \frac{3}{2} \cdot (d - m) \cdot \frac{C^{U}_{d, m + 1}}{\prod_{k = 1}^{m} n_k} \cdot \bigg[M + t \Big(\frac{2^{r_+}}{n}\Big)^{1/2} \cdot 2^{\sum_{k = 1}^{m} r_k}\bigg] \\
    &\qquad \le (D^{(\mathbf{1}_m)} \theta)_{i_1, \dots, i_m, i_{m + 1}, \dots, i_{d}} + \frac{3}{2} \cdot (d - m) \cdot \frac{C^{U}_{d, m + 1}}{\prod_{k = 1}^{m} n_k} \cdot \bigg[M + t \Big(\frac{2^{r_+}}{n}\Big)^{1/2} \cdot 2^{\sum_{k = 1}^{m} r_k}\bigg] \\
    &\qquad < -\Big(C_d - \frac{3}{2} \cdot (d - m) C^{U}_{d, m + 1} \Big) \cdot \frac{1}{\prod_{k = 1}^{m} n_k} \cdot \bigg[M + t \Big(\frac{2^{r_+}}{n}\Big)^{1/2} \cdot 2^{\sum_{k = 1}^{m} r_k}\bigg] \\
    & \qquad \le (D^{(\mathbf{1}_m)} \theta^*)_{j_1, \dots, j_d} - \Big(C_d - \frac{3}{2} \cdot (d - m) C^{U}_{d, m + 1} \Big) \cdot \frac{1}{\prod_{k = 1}^{m} n_k} \cdot t \Big(\frac{2^{r_+}}{n}\Big)^{1/2} \cdot 2^{\sum_{k = 1}^{m} r_k},
\end{align*}
\endgroup
where the third inequality follows from Lemma \ref{lem:div-diff-mono-general} and the definition \eqref{discrete-differences} of discrete differences, and the last inequality is due to \eqref{eq:bound-on-discrete-differences} and the fact that 
\begin{equation*}
    \widetilde{C}_d := C_d - \frac{3}{2} \cdot (d - m) C^{U}_{d, m + 1} \ge B_d.
\end{equation*}
It thus follows that if $i_1 - 1 \le j'_1 < j_1, \dots, i_m - 1 \le j'_m < j_m$, and if $i_{m + 1} \le j_{m + 1} \le b^{(m + 1)}_{r_{m + 1} - 1}, \dots, i_{d} \le j_{d} \le b^{(d)}_{r_{d} - 1}$, then
\begingroup
\allowdisplaybreaks
\begin{align*}
    &\sum_{\delta \in \{0, 1\}^m} (\theta - \theta^*)^2_{(1 - \delta_1) j_1 + \delta_1 j'_1, \dots, (1 - \delta_m) j_m + \delta_m j'_m, j_{m + 1}, \dots, j_d} \\
    &\qquad \ge \frac{1}{2^m} \cdot \bigg(\sum_{\delta \in \{0, 1\}^m} (-1)^{\delta_1 + \cdots + \delta_m} \cdot (\theta - \theta^*)_{(1 - \delta_1) j_1 + \delta_1 j'_1, \dots, (1 - \delta_m) j_m + \delta_m j'_m, j_{m + 1}, \dots, j_d} \bigg)^2 \\
    &\qquad = \frac{1}{2^m} \cdot \bigg(\sum_{l_1 = j'_1 + 1}^{j_1} \cdots \sum_{l_m = j'_m + 1}^{j_m} \big(D^{(\mathbf{1}_m)} \theta - D^{(\mathbf{1}_m)} \theta^*\big)_{l_1, \dots, l_m, j_{m + 1}, \dots, j_d} \bigg)^2 \\
    &\qquad > (j_1 - j'_1)^2 \times \cdots \times (j_m - j'_m)^2 \cdot \frac{\widetilde{C}_d^2 / 2^m}{\prod_{k = 1}^{m} n_k^2} \cdot t^2 \Big(\frac{2^{r_+}}{n}\Big) \cdot 2^{\sum_{k = 1}^{m} 2r_k}.
\end{align*}
\endgroup 
This implies that
\begingroup
\allowdisplaybreaks
\begin{align*}
    t^2 &\ge \|\theta - \theta^*\|_2^2 \ge \sum_{(i_1 + n_1 - 2)/2 < j_1 \le n_1 - 1} \cdots \sum_{(i_m + n_m - 2)/2 < j_m \le n_m - 1} \sum_{j_{m + 1} = i_{m + 1}}^{b^{(m + 1)}_{r_{m + 1} - 1}} \cdots \sum_{j_d = i_d}^{b^{(d)}_{r_{d} - 1}} \\
    &\qquad \qquad \qquad \qquad \qquad \sum_{\delta \in \{0, 1\}^m} (\theta - \theta^*)^2_{(1 - \delta_1)j_1 + \delta_1 (i_1 + n_1 - 2 - j_1), \dots, (1 - \delta_m)j_m + \delta_m (i_m + n_m - 2 - j_m), j_{m + 1}, \dots, j_d} \\
    &> \sum_{(i_1 + n_1 - 2)/2 < j_1 \le n_1 - 1} \cdots \sum_{(i_m + n_m - 2)/2 < j_m \le n_m - 1} \sum_{j_{m + 1} = i_{m + 1}}^{b^{(m + 1)}_{r_{m + 1} - 1}} \cdots \sum_{j_{d} = i_d}^{b^{(d)}_{r_{d} - 1}} \\
    &\qquad \qquad \qquad (2j_1 - i_1 - n_1 + 2)^2 \times \cdots \times (2j_m - i_m - n_m + 2)^2 \cdot \frac{\widetilde{C}_d^2 / 2^m}{\prod_{k = 1}^{m} n_k^2} \cdot t^2 \Big(\frac{2^{r_+}}{n}\Big) \cdot 2^{\sum_{k = 1}^{m} 2r_k} \\
    &= \prod_{k = 1}^{m} \Big[\frac{1}{6}(n_k - i_k)(n_k - i_k + 1)(n_k - i_k + 2)\Big] \cdot \prod_{k = m + 1}^{d} \big(b^{(k)}_{r_{k} - 1} - i_k + 1\big) \cdot \frac{\widetilde{C}_d^2 / 2^m}{\prod_{k = 1}^{m} n_k^2} \cdot t^2 \Big(\frac{2^{r_+}}{n}\Big) \cdot 2^{\sum_{k = 1}^{m} 2r_k} \\
    &> \prod_{k = 1}^{m} (n_k - i_k)^3 \cdot \prod_{k = m + 1}^{d} \big(b^{(k)}_{r_{k} - 1} - i_k + 1\big) \cdot \frac{\widetilde{C}_d^2 / 12^m}{\prod_{k = 1}^{m} n_k^2} \cdot t^2 \Big(\frac{2^{r_+}}{n}\Big) \cdot 2^{\sum_{k = 1}^{m} 2r_k} \\
    &> \prod_{k = 1}^{m} \Big(\frac{n_k}{2^{r_k}}\Big)^3 \cdot \prod_{k = m + 1}^{d} \frac{n_k}{2^{r_k}} \cdot \frac{\widetilde{C}_d^2 / 12^m}{\prod_{k = 1}^{m} n_k^2} \cdot t^2 \Big(\frac{2^{r_+}}{n}\Big) \cdot 2^{\sum_{k = 1}^{m} 2r_k} 
    = \frac{\widetilde{C}_d^2}{12^m} \cdot t^2 \ge t^2, 
\end{align*}
\endgroup
which leads to contradiction.
Here, the second-to-last inequality is because 
\begin{equation*}
    b^{(k)}_{r_{k} - 1} - i_k + 1 > \frac{n_k}{2^{r_k - 1}} - 2 - \Big( \frac{n_k}{2^{r_k}} - 1\Big) + 1 = \frac{n_k}{2^{r_k}}
\end{equation*}
for each $k = m + 1, \dots, d$. 
    
Next, suppose, for contradiction, that there exist $\mathbf{r} = (r_1, \dots, r_d) \in \prod_{k = 1}^{d} \{0, 1, \dots, R_k\}$ and $\mathbf{i} = (i_1, \dots, i_d) \in L_{\mathbf{r}}$ with $i_k \ge 1 = p_k$ for $k = 1, \dots, d$ such that 
\begin{equation*}
    (D^{(\mathbf{1}_m)}\theta)_{\mathbf{i}} > \frac{C_d}{\prod_{k = 1}^{m} n_k} \cdot \Big[M + t \Big(\frac{2^{r_+}}{n}\Big)^{1/2} \cdot 2^{\sum_{k = 1}^{m} r_k}\Big],
\end{equation*}
where 
\begin{equation*}
    C_d =  (d - m) C^{U}_{d, m + 1} + \max((4\sqrt{2})^m \cdot 2^d, B_d) =: C^{U}_{d, m}.
\end{equation*}
It then follows that, for every $\mathbf{j} = (j_1, \dots, j_d) \in I_0$ with $j_1 \le i_1, \dots, j_m \le i_m$, and $a^{(m + 1)}_{r_{m + 1} + 1} \le j_{m + 1} \le i_{m + 1}, \dots, a^{(d)}_{r_{d} + 1} \le j_{d} \le i_{d}$ (if $r_k = R_k$, let $a^{(k)}_{R_k + 1} = 0$),
\begingroup
\allowdisplaybreaks
\begin{align*}
    (D^{(\mathbf{1}_m)} \theta)_{j_1, \dots, j_d} 
    &\ge (D^{(\mathbf{1}_m)} \theta)_{i_1, \dots, i_m, j_{m + 1}, \dots, j_d} \\
    &= (D^{(\mathbf{1}_m)} \theta)_{i_1, \dots, i_m, i_{m + 1}, j_{m + 2}, \dots, j_{d}} - \sum_{l_{m + 1} = j_{m + 1} + 1}^{i_{m + 1}} (D^{(\mathbf{1}_{m + 1})} \theta)_{i_1, \dots, i_m, l_{m + 1}, j_{m + 2}, \dots, j_{d}} \\
    &\ge (D^{(\mathbf{1}_m)} \theta)_{i_1, \dots, i_m, i_{m + 1}, j_{m + 2}, \dots, j_{d}} - (i_{m + 1} - j_{m + 1}) \cdot \frac{C^{U}_{d, m + 1}}{\prod_{k = 1}^{m + 1} n_k} \cdot \Big[M + t \Big(\frac{2^{r_+}}{n}\Big)^{1/2} \cdot 2^{\sum_{k = 1}^{m + 1} r_k}\Big] \\
    &\ge (D^{(\mathbf{1}_m)} \theta)_{i_1, \dots, i_m, i_{m + 1}, j_{m + 2}, \dots, j_{d}} - \frac{C^{U}_{d, m + 1}}{\prod_{k = 1}^{m} n_k} \cdot \bigg[\frac{M}{2^{r_{m + 1}}} + t \Big(\frac{2^{r_+}}{n}\Big)^{1/2} \cdot 2^{\sum_{k = 1}^{m} r_k}\bigg],
\end{align*}
\endgroup
where the last inequality is from the fact that
\begin{equation*}
    i_{m + 1} - j_{m + 1} \le i_{m + 1} - a^{(m + 1)}_{r_{m + 1} + 1}
    < \frac{n_{m + 1}}{2^{r_{m + 1}}} - 1 - \Big(\frac{n_{m + 1}}{2^{r_{m + 1} + 2}} - 1 \Big) = \frac{3}{4} \cdot \frac{n_{m + 1}}{2^{r_{m + 1}}}
\end{equation*}
when $r_{m + 1} \neq R_{m + 1}$, and
\begin{equation*}
    i_{m + 1} - j_{m + 1} \le i_{m + 1}
    \le \frac{n_{m + 1}}{2^{r_{m + 1}}} - 1 < \frac{n_{m + 1}}{2^{r_{m + 1}}}
\end{equation*}
when $r_{m + 1} = R_{m + 1}$.
If we apply the same argument to the $(m + 2)^{\text{th}}, \dots, d^{\text{th}}$ coordinates in turn, we obtain 
\begingroup
\allowdisplaybreaks
\begin{align*}
    (D^{(\mathbf{1}_m)} \theta)_{j_1, \dots, j_d} &\ge (D^{(\mathbf{1}_m)} \theta)_{i_1, \dots, i_{d}} - (d - m) \cdot \frac{C^{U}_{d, m + 1}}{\prod_{k = 1}^{m} n_k} \cdot \bigg[M + t \Big(\frac{2^{r_+}}{n}\Big)^{1/2} \cdot 2^{\sum_{k = 1}^{m} r_k}\bigg] \\ 
    &> \big(C_d - (d - m) C^{U}_{d, m + 1}\big) \cdot \frac{1}{\prod_{k = 1}^{m} n_k} \cdot \bigg[M + t \Big(\frac{2^{r_+}}{n}\Big)^{1/2} \cdot 2^{\sum_{k = 1}^{m} r_k}\bigg] \\
    &\ge (D^{(\mathbf{1}_m)} \theta^*)_{j_1, \dots, j_d} + \big(C_d - (d - m) C^{U}_{d, m + 1}\big) \cdot \frac{1}{\prod_{k = 1}^{m} n_k} \cdot t \Big(\frac{2^{r_+}}{n}\Big)^{1/2} \cdot 2^{\sum_{k = 1}^{m} r_k}.
\end{align*}
\endgroup
Here, the last inequality is attributed to \eqref{eq:bound-on-discrete-differences} and the fact that 
\begin{equation*}
    \widetilde{C}_d := C_d - (d - m) C^{U}_{d, m + 1} \ge B_d.
\end{equation*}
Hence, by Cauchy inequality, if $0 \le j'_1 < j_1 \le i_1, \dots, 0 \le j'_m < j_m \le i_m$, and if $a^{(m + 1)}_{r_{m + 1} + 1} \le j_{m + 1} \le i_{m + 1}, \dots, a^{(d)}_{r_{d} + 1} \le j_{d} \le i_{d}$, then 
\begingroup
\allowdisplaybreaks
\begin{align*}
    &\sum_{\delta \in \{0, 1\}^m} (\theta - \theta^*)^2_{(1 - \delta_1) j_1 + \delta_1 j'_1, \dots, (1 - \delta_m) j_m + \delta_m j'_m, j_{m + 1}, \dots, j_d} \\
    &\qquad > (j_1 - j'_1)^2 \times \cdots \times (j_m - j'_m)^2 \cdot \frac{\widetilde{C}_d^2 / 2^m}{\prod_{k = 1}^{m} n_k^2} \cdot t^2 \Big(\frac{2^{r_+}}{n}\Big) \cdot 2^{\sum_{k = 1}^{m} 2r_k},
\end{align*}
\endgroup 
as above.
As a result, 
\begingroup
\allowdisplaybreaks
\begin{align*}
    t^2 &\ge \|\theta - \theta^*\|_2^2 \ge \sum_{i_1/2 < j_1 \le i_1} \cdots \sum_{i_m/2 < j_m \le i_m} \sum_{j_{m + 1} = a^{(m + 1)}_{r_{m + 1} + 1}}^{i_{m + 1}} \cdots \sum_{j_{d} = a^{(d)}_{r_{d} + 1}}^{i_{d}} \\
    &\qquad \qquad \qquad \qquad \qquad \sum_{\delta \in \{0, 1\}^m} (\theta - \theta^*)^2_{(1 - \delta_1)j_1 + \delta_1 (i_1 - j_1), \dots, (1 - \delta_m)j_m + \delta_m (i_m - j_m), j_{m + 1}, \dots, j_d} \\
    &> \sum_{i_1/2 < j_1 \le i_1} \cdots \sum_{i_m/2 < j_m \le i_m} \sum_{j_{m + 1} = a^{(m + 1)}_{r_{m + 1} + 1}}^{i_{m + 1}} \cdots \sum_{j_{d} = a^{(d)}_{r_{d} + 1}}^{i_{d}} \\
    &\qquad \qquad \qquad \qquad \qquad (2j_1 - i_1)^2 \times \cdots \times (2j_m - i_m)^2 \cdot \frac{\widetilde{C}_d^2 / 2^m}{\prod_{k = 1}^{m} n_k^2} \cdot t^2 \Big(\frac{2^{r_+}}{n}\Big) \cdot 2^{\sum_{k = 1}^{m} 2r_k} \\
    &= \prod_{k = 1}^{m} \Big[\frac{1}{6}i_k(i_k + 1)(i_k + 2)\Big] \cdot \prod_{k = m + 1}^{d} \big(i_k - a^{(k)}_{r_{k} + 1} + 1\big) \cdot \frac{\widetilde{C}_d^2 / 2^m}{\prod_{k = 1}^{m} n_k^2} \cdot t^2 \Big(\frac{2^{r_+}}{n}\Big) \cdot 2^{\sum_{k = 1}^{m} 2r_k} \\
    &\ge \prod_{k = 1}^{m} (i_k + 1)^3 \cdot \prod_{k = m + 1}^{d} \big(i_k - a^{(k)}_{r_{k} + 1} + 1\big) \cdot \frac{\widetilde{C}_d^2 / 16^m}{\prod_{k = 1}^{m} n_k^2} \cdot t^2 \Big(\frac{2^{r_+}}{n}\Big) \cdot 2^{\sum_{k = 1}^{m} 2r_k} \\
    &> \prod_{k = 1}^{m} \Big(\frac{n_k}{2^{r_k + 1}}\Big)^3 \cdot \prod_{k = m + 1}^{d} \frac{n_k}{2^{r_k + 2}} \cdot \frac{\widetilde{C}_d^2 / 16^m}{\prod_{k = 1}^{m} n_k^2} \cdot t^2 \Big(\frac{2^{r_+}}{n}\Big) \cdot 2^{\sum_{k = 1}^{m} 2r_k} 
    = \frac{\widetilde{C}_d^2}{32^m \cdot 4^d} \cdot t^2 \ge t^2, 
\end{align*}
\endgroup
which leads to contradiction.
Here, for the last inequality, we use
\begin{equation*}
    i_k - a^{(k)}_{r_{k} + 1} + 1 > \frac{n_k}{2^{r_k + 1}} - 1 - \frac{n_k}{2^{r_k + 2}} + 1 = \frac{n_k}{2^{r_k + 2}}
\end{equation*}
for each $k = m + 1, \dots, d$.
    
From these results, we can conclude that \eqref{eq:finite-derivative-lower-bound} and \eqref{eq:finite-derivative-upper-bound} hold with $C_d = C^{L}_{d, m}$ and $C_d = C^{U}_{d, m}$, respectively, for every $\mathbf{p} \in \{0, 1\}^d$ with $|\mathbf{p}| = m$. 
Thus, by induction, \eqref{eq:finite-derivative-lower-bound} and \eqref{eq:finite-derivative-upper-bound} hold for every $\mathbf{p} \in \{0, 1\}^d \setminus \{\zerovec\}$, if we define $C_d = \max(\max_{m = 1, \dots, d} C^{L}_{d, m}, \max_{m = 1, \dots, d} C^{U}_{d, m})$.
\end{proof}

\subsubsection{Proof of Lemma \ref{lem:bounds-finite-derivatives-cont}}\label{pf:bounds-finite-derivatives-cont}
\begin{proof}[Proof of Lemma \ref{lem:bounds-finite-derivatives-cont}]
The proof proceeds similarly to that of Lemma 
\ref{lem:bounds-finite-derivatives}.
We prove \eqref{eq:bounds-finite-derivatives-cont-lower} and 
\eqref{eq:bounds-finite-derivatives-cont-upper} by induction on $|S|$, and 
Lemma \ref{lem:div-diff-mono-general} will be repeatedly used for the proof.

We first consider the case where $|S| = d$ ($S = [d] = \{1, \dots, d\}$).
Suppose, for contradiction, that there exist 
$\mathbf{r} = (r_1, \dots, r_d) \in \mathbb{N}^d$ and 
$1/2^{r_k + 1} \le s_k < t_k \le 1/2^{r_k}$ for $k \in [d]$ such that
\begin{equation*}
    \begin{bmatrix}
        s_1, t_1 & \multirow{3}{*}{; \ $f$ } \\ 
        \vdots             & \\
        s_d, t_d &
    \end{bmatrix}
    < - C_d t \cdot 2^{r_+/2} \cdot 2^{\sum_{k = 1}^{d} r_k},
\end{equation*}
where $C_d = (2\sqrt{3})^d =: C^L_{d, d}$.
By Lemma \ref{lem:div-diff-mono-general}, we have
\begin{equation*}
    \begin{bmatrix}
        u_1, v_1 & \multirow{3}{*}{; \ $f$ } \\ 
        \vdots             & \\
        u_d, v_d &
    \end{bmatrix}
    \le
    \begin{bmatrix}
        s_1, t_1 & \multirow{3}{*}{; \ $f$ } \\ 
        \vdots             & \\
        s_d, t_d &
    \end{bmatrix}
    < - C_d t \cdot 2^{r_+/2} \cdot 2^{\sum_{k = 1}^{d} r_k}
\end{equation*}
for every $t_k \le u_k < v_k$ for $k \in [d]$.
Applying Cauchy inequality, we thus obtain
\begin{align*}
    &\sum_{\delta \in \{0, 1\}^d} \Big(f\big((1 - \delta_1)v_1 + \delta_1 
    u_1, \dots, (1 - \delta_d)v_d + \delta_d u_d\big) \Big)^2 \\
    &\qquad \ge \frac{1}{2^d} \bigg(\sum_{\delta \in \{0, 1\}^d} 
    (-1)^{\delta_1 + \cdots + \delta_d} \cdot 
    f\big((1 - \delta_1)v_1 + \delta_1 u_1, \dots, (1 - \delta_d)v_d 
    + \delta_d u_d\big) \bigg)^2 \\
    &\qquad= \frac{1}{2^d} \prod_{k = 1}^d (v_k - u_k)^2 \cdot
    \begin{bmatrix}
        u_1, v_1 & \multirow{3}{*}{; \ $f$ } \\ 
        \vdots             & \\
        u_d, v_d &
    \end{bmatrix}^2
    > \prod_{k = 1}^d (v_k - u_k)^2 \cdot \frac{C_d^2}{2^d} \cdot t^2 
    \cdot 2^{r_+} \cdot 2^{2\sum_{k = 1}^d r_k}.
\end{align*}
Since $t_k \le x_k < t_k + 1 - x_k$ whenever $x_k \in [t_k, (t_k + 1)/2)$,
it follows that
\begin{align*}
    t^2 &\ge \|f\|_2^2 
    \ge \int_{\prod_{k = 1}^d [t_k, (t_k + 1)/2)} 
    \sum_{\delta \in \{0, 1\}^d} \Big(f\big((1 - \delta_1)
    (t_1 + 1 - x_1) + \delta_1 x_1, \dots, (1 - \delta_d) (t_d + 1 - x_d) 
    + \delta_d x_d\big) \Big)^2 \, dx \\
    &> \frac{C_d^2}{2^d} \cdot t^2 
    \cdot 2^{r_+} \cdot 2^{2\sum_{k = 1}^d r_k} \cdot\prod_{k = 1}^d 
    \int_{t_k}^{(t_k + 1)/2} (t_k + 1 - 2x_k)^2 \, dx_k \\
    &= \frac{C_d^2}{12^{d}} \cdot t^2 \cdot 2^{r_+} \cdot 
    2^{2\sum_{k = 1}^d r_k} \cdot \prod_{k = 1}^d (1 - t_k)^3
    \ge \frac{C_d^2}{12^{d}} \cdot t^2
    = t^2,
\end{align*}
which is a contradiction.
Here, the last inequality follows from
\begin{equation*}
    1 - t_k \ge 1 - \frac{1}{2^{r_k}} \ge \frac{1}{2^{r_k}},
\end{equation*}
which holds since $r_k \ge 1$.
    
Next, suppose, for contradiction, that there exist 
$\mathbf{r} = (r_1, \dots, r_d) \in \mathbb{Z}_{\ge 0}^d$ and 
$1/2^{r_k + 1} \le s_k < t_k \le 1/2^{r_k}$ for $k \in [d]$ such that
\begin{equation*}
    \begin{bmatrix}
        s_1, t_1 & \multirow{3}{*}{; \ $f$ } \\ 
        \vdots             & \\
        s_d, t_d &
    \end{bmatrix}
    > C_d t \cdot 2^{r_+/2} \cdot 2^{\sum_{k = 1}^{d} r_k},
\end{equation*}
where $C_d = (4\sqrt{6})^d =: C^U_{d, d}$.
By Lemma \ref{lem:div-diff-mono-general}, if $u_k < v_k \le s_k$ for 
$k \in [d]$, then
\begin{equation*}
    \begin{bmatrix}
        u_1, v_1 & \multirow{3}{*}{; \ $f$ } \\ 
        \vdots             & \\
        u_d, v_d &
    \end{bmatrix}
    \ge
    \begin{bmatrix}
        s_1, t_1 & \multirow{3}{*}{; \ $f$ } \\ 
        \vdots          & \\
        s_d, t_d &
    \end{bmatrix}
    > C_d t \cdot 2^{r_+/2} \cdot 2^{\sum_{k = 1}^{d} r_k}.
\end{equation*}
As in the previous case, applying Cauchy inequality gives
\begin{equation*}
    \sum_{\delta \in \{0, 1\}^d} \Big(f\big((1 - \delta_1)v_1 + \delta_1 
    u_1, \dots, (1 - \delta_d)v_d + \delta_d u_d\big) \Big)^2
    > \prod_{k = 1}^d (v_k - u_k)^2 \cdot \frac{C_d^2}{2^d} \cdot t^2 
    \cdot 2^{r_+} \cdot 2^{2\sum_{k = 1}^d r_k}.
\end{equation*}
Since $x_k < s_k - x_k \le s_k$ for all $x_k \in [0, s_k/2)$, it follows that
\begin{align*}
    t^2 &\ge \|f\|_2^2 
    \ge \int_{\prod_{k = 1}^d [0, s_k/2)} \sum_{\delta \in \{0, 1\}^d} 
    \Big(f\big((1 - \delta_1) (s_1 - x_1) + \delta_1 x_1, \dots, 
    (1 - \delta_d) (s_d - x_d) + \delta_d x_d\big) \Big)^2 \, dx \\
    &> \frac{C_d^2}{2^d} \cdot t^2 
    \cdot 2^{r_+} \cdot 2^{2\sum_{k = 1}^d r_k} \cdot \prod_{k = 1}^d 
    \int_{0}^{s_k/2} (s_k - 2x_k)^2 \, dx_k 
    = \frac{C_d^2}{12^{d}} \cdot t^2 \cdot 2^{r_+} \cdot 
    2^{2\sum_{k = 1}^d r_k} \cdot \prod_{k = 1}^d s_k^3
    \ge \frac{C_d^2}{96^{d}} \cdot t^2
    = t^2,
\end{align*}
which leads to contradiction.
Here, the last inequality is due to the fact that $s_k \ge 1/2^{r_k + 1}$.
    
Now, we assume that \eqref{eq:bounds-finite-derivatives-cont-lower} and 
\eqref{eq:bounds-finite-derivatives-cont-upper} hold with 
$C_d = C^L_{d, m + 1}$ and $C_d = C^U_{d, m + 1}$, 
respectively, for every subset $S \subseteq [d]$ with $|S| = m + 1$, 
and prove that they also hold with some $C_d$ for every subset 
$S \subseteq [d]$ with $|S| = m$. For notational convenience, we present 
the proof only for the case $S = [m] = \{1, \dots, m\}$; the other cases 
follow analogously.
    
Suppose, for contradiction, that there exist 
$\mathbf{r} = (r_1, \dots, r_d) \in \mathbb{N}^d$, 
$1/2^{r_k + 1} \le s_k < t_k \le 1/2^{r_k}$ for $k = 1, \dots, m$, and
$t_k \in [1/2^{r_k + 1}, 1/2^{r_k}]$ for $k = m + 1, \dots, d$ such that
\begin{equation*}
    \begin{bmatrix}
        s_1, t_1  & \multirow{6}{*}{; \ $f$ } \\ 
        \vdots    & \\ 
        s_m, t_m  & \\ 
        t_{m + 1} & \\ 
        \vdots    & \\ 
        t_d       & \\ 
    \end{bmatrix}
    < - C_d t \cdot 2^{r_+/2} \cdot 2^{\sum_{k = 1}^{m} r_k},
\end{equation*}
where $C_d = (3/2)(d - m)C^U_{d, m + 1} + (2\sqrt{3})^m =: C^L_{d, m}$.
If $t_k \le u_k < v_k$ for $k = 1, \dots, m$ and 
$t_k \le u_k \le 1/2^{r_k - 1}$ for $k = m + 1, \dots, d$, then we have
\begingroup
\allowdisplaybreaks
\begin{align*}
    \begin{bmatrix}
        u_1, v_1  & \multirow{6}{*}{; \ $f$ } \\ 
        \vdots    & \\ 
        u_m, v_m  & \\ 
        u_{m + 1} & \\ 
        \vdots    & \\ 
        u_d       & \\ 
    \end{bmatrix}
    &=
    \begin{bmatrix}
        u_1, v_1  & \\ 
        \vdots    & \\ 
        u_m, v_m  & \\ 
        t_{m + 1} &; \ f \\ 
        u_{m + 2} & \\ 
        \vdots    & \\ 
        u_d       & \\ 
    \end{bmatrix}
    + (u_{m + 1} - t_{m + 1}) 
    \begin{bmatrix}
        u_1, v_1  & \\ 
        \vdots    & \\ 
        u_m, v_m  & \\ 
        t_{m + 1}, u_{m + 1} &; \ f \\ 
        u_{m + 2} & \\ 
        \vdots    & \\ 
        u_d       & \\ 
    \end{bmatrix}
    \\
    &\le 
    \begin{bmatrix}
        u_1, v_1  & \\ 
        \vdots    & \\ 
        u_m, v_m  & \\ 
        t_{m + 1} &; \ f \\ 
        u_{m + 2} & \\ 
        \vdots    & \\ 
        u_d       & \\ 
    \end{bmatrix} 
    + \frac{3}{2^{r_{m + 1} + 1}} \cdot
    C^U_{d, m + 1} t \cdot 2^{r_+/2} \cdot 2^{\sum_{k = 1}^{m + 1} r_k},
\end{align*}
\endgroup
where the inequality follows from Lemma
\ref{lem:div-diff-mono-general}, the induction hypothesis, and the fact that 
\begin{equation*}
    u_{m + 1} - t_{m + 1} 
    \le \frac{1}{2^{r_{m + 1} - 1}} - \frac{1}{2^{r_{m + 1} + 1}} 
    \le \frac{3}{2^{r_{m + 1} + 1}}.
\end{equation*}
Iterating this argument for the 
$(m + 2)^{\text{th}}, \dots, d^{\text{th}}$ coordinates yields
\begingroup
\allowdisplaybreaks
\begin{align*}
    \begin{bmatrix}
        u_1, v_1 & \multirow{6}{*}{; \ $f$ } \\ 
        \vdots & \\ 
        u_m, v_m & \\ 
        u_{m + 1} & \\ 
        \vdots & \\ 
        u_d & \\ 
    \end{bmatrix}
    &\le 
    \begin{bmatrix}
        u_1, v_1 & \multirow{6}{*}{; \ $f$ } \\ 
        \vdots & \\ 
        u_m, v_m & \\ 
        t_{m + 1} & \\ 
        \vdots & \\ 
        t_d & \\ 
    \end{bmatrix} 
    + \frac{3}{2} \cdot (d - m) 
    C^U_{d, m + 1} t \cdot 2^{r_+/2} \cdot 2^{\sum_{k = 1}^{m} r_k} \\
    &\le 
    \begin{bmatrix}
        s_1, t_1 & \multirow{6}{*}{; \ $f$ } \\ 
        \vdots & \\ 
        s_m, t_m & \\ 
        t_{m + 1} & \\ 
        \vdots & \\ 
        t_d & \\ 
    \end{bmatrix} 
    + \frac{3}{2} \cdot (d - m) 
    C^U_{d, m + 1} t \cdot 2^{r_+/2} \cdot 2^{\sum_{k = 1}^{m} r_k} \\
    &< -\Big(C_d - \frac{3}{2} \cdot (d - m) C^U_{d, m + 1}\Big) 
    t \cdot 2^{r_+/2} \cdot 2^{\sum_{k = 1}^{m} r_k}.
\end{align*}
\endgroup
By Cauchy inequality, it follows that
\begin{align*}
    &\sum_{\delta \in \{0, 1\}^m} \Big(f\big((1 - \delta_1)v_1 + \delta_1 
    u_1, \dots, (1 - \delta_m)v_m + \delta_m u_m, u_{m + 1}, \dots, u_d\big) 
    \Big)^2 \\
    &\qquad> \prod_{k = 1}^m (v_k - u_k)^2 \cdot \frac{\widetilde{C}_d^2}{2^m} \cdot t^2 
    \cdot 2^{r_+} \cdot 2^{2\sum_{k = 1}^m r_k},
\end{align*}
where 
\begin{equation*}
    \widetilde{C}_d := C_d - \frac{3}{2} \cdot (d - m) 
    C^U_{d, m + 1} = (2\sqrt{3})^m.
\end{equation*}
Hence, integrating as before, we obtain
\begingroup
\allowdisplaybreaks
\begin{align*}
    t^2 &\ge \|f\|_2^2 
    \ge \int_{\prod_{k = 1}^m [t_k, (t_k + 1)/2) \times 
    \prod_{k = m + 1}^d [t_k, 1/2^{r_k - 1}]} \sum_{\delta \in \{0, 1\}^m} \\
    &\qquad \qquad \qquad 
    \Big(f\big((1 - \delta_1) (t_1 + 1 - x_1) + \delta_1 x_1, \dots,
    (1 - \delta_m) (t_m + 1 - x_m) + \delta_m x_m, 
    x_{m + 1}, \dots, x_d \big) \Big)^2 \, dx \\
    &> \frac{\widetilde{C}_d^2}{2^m} \cdot t^2 
    \cdot 2^{r_+} \cdot 2^{2\sum_{k = 1}^m r_k} \cdot \prod_{k = 1}^m 
    \int_{t_k}^{(t_k + 1)/2} (t_k + 1 - 2x_k)^2 \, dx_k  \cdot 
    \prod_{k = m + 1}^d \Big(\frac{1}{2^{r_k - 1}} - t_k\Big) \\
    &\ge \frac{\widetilde{C}_d^2}{12^{m}} \cdot t^2 \cdot 
    2^{r_+} \cdot 2^{2\sum_{k = 1}^m r_k} \cdot \prod_{k = 1}^m (1 - t_k)^3 
    \cdot \prod_{k = m + 1}^d \frac{1}{2^{r_k}}
    \ge \frac{\widetilde{C}_d^2}{12^{m}} \cdot t^2
    = t^2,
\end{align*}
\endgroup  
which is a contradiction.
Here, the second-to-last inequality uses the fact that
\begin{equation*}
    \frac{1}{2^{r_k - 1}} - t_k 
    \ge \frac{1}{2^{r_k - 1}} - \frac{1}{2^{r_k}} 
    = \frac{1}{2^{r_k}}
\end{equation*}
for $k = m + 1, \dots, d$.

Next, suppose, for contradition, that there exist 
$\mathbf{r} = (r_1, \dots, r_d) \in \mathbb{Z}_{\ge 0}^d$, 
$1/2^{r_k + 1} \le s_k < t_k \le 1/2^{r_k}$ for $k = 1, \dots, m$, and
$t_k \in [1/2^{r_k + 1}, 1/2^{r_k}]$ for $k = m + 1, \dots, d$ such that
\begin{equation*}
    \begin{bmatrix}
        s_1, t_1 & \multirow{6}{*}{; \ $f$ } \\ 
        \vdots & \\ 
        s_m, t_m & \\ 
        t_{m + 1} & \\ 
        \vdots & \\ 
        t_d & \\ 
    \end{bmatrix}
    > C_d t \cdot 2^{r_+/2} \cdot 2^{\sum_{k = 1}^{m} r_k},
\end{equation*}
where $C_d = (3/2)(d - m)C^U_{d, m + 1} + (2\sqrt{6})^m \cdot 2^d 
=: C^U_{d, m}$.
For $u_k < v_k \le s_k$ for $k = 1, \dots, m$ and 
$1/2^{r_k + 2} \le u_k \le t_k$ for $k = m + 1, \dots, d$, we have
\begingroup
\allowdisplaybreaks
\begin{align*}
    \begin{bmatrix}
        u_1, v_1 & \multirow{6}{*}{; \ $f$ } \\ 
        \vdots & \\ 
        u_m, v_m & \\ 
        u_{m + 1} & \\ 
        \vdots & \\ 
        u_d & \\ 
    \end{bmatrix}
    &=
    \begin{bmatrix}
        u_1, v_1  & \\ 
        \vdots     & \\ 
        u_m, v_m  & \\ 
        t_{m + 1} &; \ f \\ 
        u_{m + 2} & \\ 
        \vdots     & \\ 
        u_d       & \\ 
    \end{bmatrix}
    - (t_{m + 1} - u_{m + 1}) 
    \begin{bmatrix}
        u_1, v_1  & \\ 
        \vdots     & \\ 
        u_m, v_m  & \\ 
        u_{m + 1}, t_{m + 1} &; \ f \\ 
        u_{m + 2} & \\ 
        \vdots     & \\ 
        u_d       & \\ 
    \end{bmatrix}
    \\
    &\ge 
    \begin{bmatrix}
        u_1, v_1  & \\ 
        \vdots     & \\ 
        u_m, v_m  & \\ 
        t_{m + 1} &; \ f \\ 
        u_{m + 2} & \\ 
        \vdots     & \\ 
        u_d       & \\ 
    \end{bmatrix} 
    - \frac{3}{2^{r_{m + 1} + 2}} \cdot
    C^U_{d, m + 1} t \cdot 2^{r_+/2} \cdot 2^{(\sum_{k = 1}^{m + 1} r_k) + 1},
\end{align*}
\endgroup
where the inequality is due to Lemma \ref{lem:div-diff-mono-general}, the 
induction hypothesis, and the fact that
\begin{equation*}
    t_{m + 1} - u_{m + 1} 
    \le \frac{1}{2^{r_{m + 1}}} - \frac{1}{2^{r_{m + 1} + 2}} 
    \le \frac{3}{2^{r_{m + 1} + 2}}.
\end{equation*}
Repeating this argument with the $(m + 2)^{\text{th}}, \dots, d^{\text{th}}$ 
coordinates yields
\begingroup
\allowdisplaybreaks
\begin{align*}
    \begin{bmatrix}
        u_1, v_1 & \multirow{6}{*}{; \ $f$ } \\ 
        \vdots & \\ 
        u_m, v_m & \\ 
        u_{m + 1} & \\ 
        \vdots & \\ 
        u_d & \\ 
    \end{bmatrix}
    &\ge 
    \begin{bmatrix}
        u_1, v_1 & \multirow{6}{*}{; \ $f$ } \\ 
        \vdots & \\ 
        u_m, v_m & \\ 
        t_{m + 1} & \\ 
        \vdots & \\ 
        t_d & \\ 
    \end{bmatrix} 
    - \frac{3}{2} \cdot (d - m) 
        C^U_{d, m + 1} t \cdot 2^{r_+/2} \cdot 2^{\sum_{k = 1}^{m} r_k} \\
    &\ge
    \begin{bmatrix}
        s_1, t_1 & \multirow{6}{*}{; \ $f$ } \\ 
        \vdots & \\ 
        s_m, t_m & \\ 
        t_{m + 1} & \\ 
        \vdots & \\ 
        t_d & \\ 
    \end{bmatrix} 
    - \frac{3}{2} \cdot (d - m) 
    C^U_{d, m + 1} t \cdot 2^{r_+/2} \cdot 2^{\sum_{k = 1}^{m} r_k} \\
    &> \Big(C_d - \frac{3}{2} \cdot (d - m) C^U_{d, m + 1}\Big) 
    t \cdot 2^{r_+/2} \cdot 2^{\sum_{k = 1}^{m} r_k}.
\end{align*}
\endgroup
If we let
\begin{equation*}
    \widetilde{C}_d := C_d - \frac{3}{2} \cdot (d - m) 
    C^U_{d, m + 1} 
    = (2\sqrt{6})^m \cdot 2^d,
\end{equation*}
then we have
\begingroup
\allowdisplaybreaks
\begin{align*}
    t^2 &\ge \|f\|_2^2 
    \ge \int_{\prod_{k = 1}^m [0, s_k/2) \times 
    \prod_{k = m + 1}^d [1/2^{r_k + 2}, t_k]} \sum_{\delta \in \{0, 1\}^m} \\
    &\qquad \qquad \qquad 
    \Big(f\big((1 - \delta_1) (s_1 - x_1) + \delta_1 x_1, \dots,
    (1 - \delta_m) (s_m - x_m) + \delta_m x_m, 
    x_{m + 1}, \dots, x_d \big) \Big)^2 \, dx \\
    &> \frac{\widetilde{C}_d^2}{2^m} \cdot t^2 
    \cdot 2^{r_+} \cdot 2^{2\sum_{k = 1}^m r_k} \cdot \prod_{k = 1}^m 
    \int_{0}^{s_k/2} (s_k - 2x_k)^2 \, dx_k  \cdot 
    \prod_{k = m + 1}^d \Big(t_k - \frac{1}{2^{r_k + 2}}\Big) \\
    &\ge \frac{\widetilde{C}_d^2}{12^{m}} \cdot t^2 \cdot 
    2^{r_+} \cdot 2^{2\sum_{k = 1}^m r_k} \cdot \prod_{k = 1}^m s_k^3 
    \cdot \prod_{k = m + 1}^d \frac{1}{2^{r_k + 2}}
    \ge \frac{\widetilde{C}_d^2}{24^{m} \cdot 4^d} \cdot t^2
    = t^2,
\end{align*}
\endgroup
which is a contradiction.

For these reasons, \eqref{eq:bounds-finite-derivatives-cont-lower} and 
\eqref{eq:bounds-finite-derivatives-cont-upper} hold with 
$C_d = C^L_{d, m}$ and $C_d = C^U_{d, m}$, 
respectively, for every subset $S \subseteq [d]$ with $|S| = m$. 
By induction, it follows that \eqref{eq:bounds-finite-derivatives-cont-lower} and 
\eqref{eq:bounds-finite-derivatives-cont-upper} hold for every 
$\emptyset \neq S \subseteq [d]$ with 
$C_d = \max(\max_{m = 1, \dots, d} C^L_{d, m}, 
\max_{m = 1, \dots, d} C^U_{d, m})$.
\end{proof}

\subsubsection{Proof of Lemma \ref{lem:fvm-bracketing-entropy}}\label{pf:fvm-bracketing-entropy}
\begin{proof}[Proof of Lemma \ref{lem:fvm-bracketing-entropy}]
The proof of this lemma closely follows that of Lemma 11.18 of 
\cite{ki2024mars}. We therefore only describe the modifications required at 
the beginning of their proof in order to establish Lemma 
\ref{lem:fvm-bracketing-entropy}.

Suppose $f \in D(V, t)$ is of the form \eqref{intermodds}. 
For notational convenience, write $\beta_0 = \beta_{\emptyset}$.
Here, we prove that there exists a constant $C_d > 0$ such that
\begin{equation*}
    |\beta_T| \le C_d (V + t)
\end{equation*}
for every subset $T \subseteq [d]$ with $|T| \le s$.
Once this bound is established, the remainder of the proof follows exactly as 
in the proof of Lemma 11.18 of \cite{ki2024mars}.

Since $V(f) \le V$, we have
\begin{equation*}
    t \ge \|f\|_2 
    \ge \bigg\{\int_{[0, 1]^d} \bigg(\sum_{S : 0 \le |S| \le s} 
    \beta_S \prod_{j \in S} x_j \bigg)^2 \, dx\bigg\}^{1/2} - V,
\end{equation*}
which implies
\begin{equation}\label{eq:integral-bound-affine-squared}
    \int_{[0, 1]^d} \bigg(\sum_{S : 0 \le |S| \le s} 
    \beta_S \prod_{j \in S} x_j \bigg)^2 \, dx
    \le (V + t)^2.
\end{equation}
Now, fix $\delta \in (0, 1)$, and define 
$A = \sum_{S: 0 \le |S| \le s} \beta_S^2$.
Restricting the integral \eqref{eq:integral-bound-affine-squared} to 
$[0, \delta]^d$ and applying Cauchy inequality 
repeatedly, we obtain
\begin{align*}
    (V + t)^2 &\ge \int_{[0, \delta]^d} \bigg(\sum_{S : 0 \le |S| \le s} 
    \beta_S \prod_{j \in S} x_j \bigg)^2 \, dx 
    \ge \frac{1}{2} \beta_{\emptyset}^2 \delta^d 
    - \int_{[0, \delta]^d} \bigg(\sum_{S : 0 < |S| \le s} 
    \beta_S \prod_{j \in S} x_j \bigg)^2 \, dx \\
    &\ge \frac{1}{2} \beta_{\emptyset}^2 \delta^d 
    - A \int_{[0, \delta]^d} \sum_{S : 0 < |S| \le s} 
    \prod_{j \in S} x_j^2 \, dx 
    = \frac{1}{2} \beta_{\emptyset}^2 \delta^d - C_d \delta^{d + 2} \cdot A,
\end{align*}
and therefore,
\begin{equation*}
    \beta_{\emptyset}^2 
    \le 2(V + t)^2 \delta^{-d} + C_d \delta^2 \cdot A.
\end{equation*}

Next, assume that for every $T \subseteq [d]$ with $|T| < m \le s$, we have 
\begin{equation}\label{eq:bound-on-beta-T-squared}
    \beta_T^2 \le C_d (V + t)^2 \delta^{-d} + C_d \delta^2 \cdot A.
\end{equation}
We now show that the same bound, possibly with a different constant $C_d$, 
also holds for all subsets $T \subseteq [d]$ with $|T| = m$. 
Without loss of generality, it suffices to consider the case 
$T = [m] = \{1, \dots, m\}$, since the argument is identical for any
other subset of size $m$.

Restricting the integral \eqref{eq:integral-bound-affine-squared} 
to $[0, 1]^m \times [0, \delta]^{d - m}$ and applying Cauchy inequality 
yields
\begingroup
\allowdisplaybreaks
\begin{align*}
    (V + t)^2 &\ge \int_{[0, 1]^m \times [0, \delta]^{d - m}} 
    \bigg(\sum_{S : 0 \le |S| \le s} \beta_S \prod_{j \in S} x_j \bigg)^2 \, dx \\
    &\ge \frac{1}{3} \beta_{[m]}^2 \delta^{d - m} 
    - \int_{[0, 1]^m \times [0, \delta]^{d - m}}  
    \bigg(\sum_{S \subsetneq [m]} \beta_S \prod_{j \in S} x_j \bigg)^2
    + \bigg(\sum_{\substack{S: 0 \le |S| \le s \\ S \nsubseteq [m]}} 
    \beta_S \prod_{j \in S} x_j \bigg)^2 \, dx \\
    &\ge \frac{1}{3} \beta_{[m]}^2 \delta^{d - m} 
    - \bigg(\sum_{S \subsetneq [m]}
    \beta_S^2 \bigg) \cdot \int_{[0, 1]^m \times [0, \delta]^{d - m}}  
    \sum_{S \subsetneq [m]} \prod_{j \in S} x_j^2 \, dx \\
    &\qquad - A \int_{[0, 1]^m \times [0, \delta]^{d - m}}  
    \sum_{\substack{S: 0 \le |S| \le s \\ S \nsubseteq [m]}} 
    \prod_{j \in S} x_j^2 \, dx.
\end{align*}
\endgroup
By the induction hypothesis, we thus have
\begin{align*}
    (V + t)^2
    &\ge \frac{1}{3} \beta_{[m]}^2 \delta^{d - m} 
    - \sum_{S \subsetneq [m]} \big(C_d (V + t)^2 \delta^{-d} 
    + C_d \delta^2 \cdot A \big) \cdot \delta^{d - m} 
    - C_d \delta^{d - m + 2} \cdot A \\
    &\ge \frac{1}{3} \beta_{[m]}^2 \delta^{d - m} 
    - C_d (V + t)^2 \delta^{-m} - C_d \delta^{d - m + 2} \cdot A.
\end{align*}
Rearranging terms, we obtain
\begin{equation*}
    \beta_{[m]}^2 \le 3(V + t)^2 \delta^{m - d} + C_d (V + t)^2 \delta^{-d} 
    + C_d \delta^2 \cdot A 
    \le  C_d (V + t)^2 \delta^{-d} 
    + C_d \delta^2 \cdot A.
\end{equation*}
This proves that the bound \eqref{eq:bound-on-beta-T-squared} holds for $T = [m]$, 
and hence, for all $T \subseteq [d]$ with $|T| = m$.
By induction, it follows that the bound \eqref{eq:bound-on-beta-T-squared} 
holds for every $T \subseteq [d]$ with $|T| \le s$.

Summing the bounds \eqref{eq:bound-on-beta-T-squared} over all subsets 
$T \subseteq [d]$ with $|T| \le s$, we obtain
\begin{equation*}
    A \le C_d (V + t)^2 \delta^{-d} 
    + C_d \delta^2 \cdot A. 
\end{equation*}
Hence,
\begin{equation*}
    A \le \frac{C_d(V + t)^2 \delta^{-d}}{1 - C_d \delta^2}
\end{equation*}
provided $C_d \delta^2 < 1$.
Choosing $\delta \in (0, 1)$ sufficiently small and substituting this bound 
back into \eqref{eq:bound-on-beta-T-squared}, we obtain
\begin{equation*}
    \beta_T^2 \le C_d (V + t)^2 
\end{equation*}
for all $T \subseteq [d]$ with $|T| \le s$, as desired.
\end{proof}

\end{document}